%% file: KMoManinall.tex
\documentclass[11pt]{amsart}
\input{preamble}
\usepackage[margin=1in,marginpar=.8in]{geometry}
\usepackage{wrapfig}

\def\type{\it type}
\def\width{\it width} %or is it depth?
\def\depth{\it depth}
\def\gpd{\mathcal{G}}
\def\sqb{\sq_\act}
\def\sqbt{\sq_{\act_\ot}}

\def\Perm{\boldsymbol{\pi}}

\def\aca{\act^{op}\times \act}
\newcommand{\pher}{prehereditary}
\newcommand{\Pher}{Prehereditary}
\newcommand{\fact}{localizable}

\newcommand{\catplus}[1]{\mathcal{P}l({#1})}
\newcommand{\catplusp}[1]{\mathcal{P}l^{u}({#1})}

\newcommand{\catplusgp}[1]{\mathcal{P}l^{uc}({#1})}

\newcommand{\monplus}[1]{\mathcal{P}l_{\otimes}({#1})}
\newcommand{\monplusp}[1]{\mathcal{P}l_{\otimes}^{u}({#1})}

\newcommand{\monplusgr}[1]{\mathcal{P}l_{\otimes}^{graph}({#1})}
\newcommand{\monplusgcp}[1]{\mathcal{P}l_{\otimes}^{\it gcp}({#1})}
\newcommand{\monplusgcpgr}[1]{\mathcal{P}l_{\otimes}^{\it graph, gcp}({#1})}
\newcommand{\monplushyp}[1]{\mathcal{P}l_{\otimes}^{\it hyp}({#1})}
\newcommand{\monplushypgr}[1]{\mathcal{P}l_{\otimes}^{\it graph, hyp}({#1})}

\newcommand{\locmonplus}[1]{\mathcal{P}l_{\otimes}({#1})_{loc}}
\newcommand{\locmonplusp}[1]{\mathcal{P}l_{\otimes}^{u}({#1})_{loc}}

\newcommand{\redmonplus}[1]{\mathcal{P}l_{\otimes}({#1})_{red}}
\newcommand{\redmonplusp}[1]{\mathcal{P}l_{\otimes}^{u}({#1})_{red}}

\newcommand{\plus}[1]{{#1}^+}
\newcommand{\plusp}[1]{{#1}^{+, \it u}}

\newcommand{\plusgr}[1]{{#1}^{+, \it graph}}
\newcommand{\plusgcp}[1]{{#1}^{+, \it gcp}}
\newcommand{\plusgcpgr}[1]{{#1}^{+, \it graph, gcp}}
\newcommand{\plushyp}[1]{{#1}^{+, \it hyp}}
\newcommand{\plushypgr}[1]{{#1}^{+, \it graph, hyp}}

\newcommand{\anyplus}{\mathcal{P}_+}
\newcommand{\anyplusp}{\anyplus^{u}}
\newcommand{\anyplusg}{\anyplus^{c}}
\newcommand{\anyplusgp}{\anyplus^{uc}}

\newcommand{\anyplushyp}{\anyplus^{hyp}}

\newcommand{\gplus}[1]{\mathcal{G}r({#1})}

\begin{document}

% Abbreviated notation
\def\PlC{\catplus{\C}}
\def\PlM{\monplus{\M}}
\def\twocat{\mathbb{H}}

\title[Plus constructions]{Plus constructions, plethysm, and unique factorization categories with applications to graphs and operad--like  theories}
%[Plus constructions]{Baez-Dolan type Plus Constructions,  Plethysm Monoids and their Algebras}

\author{Ralph M.\ Kaufmann}
\email{rkaufman@math.purdue.edu}

\address{Purdue University Department of Mathematics, and Department of Physics \& Astronomy, West Lafayette, IN 47907
}

\author{Michael Monaco}
\email{monacom@purdue.edu}

\address{Purdue University Department of Mathematics, West Lafayette, IN 47907
}

\begin{abstract}
 Baez-Dolan type plus constructions  serve three main purposes:
 they
(1)  corepresent categorical bimodules that  are monoids with respect to a plethysm product,
 (2)  allow to define  functors as bimodule monoids, and thereby algebras over functors,
 (3)  provide a theory of twists of monads.
 Unital (monoidal) bimodule monoids yield (monoidal) categories and the corepresentation is for indexed enrichments of categories.
The original Baez--Dolan construction constructed algebras over operads. Operads are functors out of a particular Feynman Category (FC), viz.\ they are corepresented by it, and the plus construction was previously generalized to FCs.
 We now define several of these constructions in the general context of categories and (symmetric) monoidal categories, show that they are functorial  ---hence categorically good---, and prove their corepresentation properties.

One application is that the structures corepresented by an FC, like operads, props, etc.\  can be defined as plethysm monoids if and only if the corepresenting FC is a plus construction. In one direction, we prove that such a plus construction is based on a {Unique Factorization Category} (UFC)---a new notion that we introduce. In the other direction, we
prove that the resulting FC has special properties, like being cubical. This explains why there is no monoid formulation for cyclic or modular operads or props,  but there is for operads and properads. Additionally, using the bimodule monoids point of view,
we prove that as monoidal bimodule monoids FCs are characterized by the fact that the functor constructing free algebras preserves the property of being strongly monoidal.

We give a  local presentation of these constructions, as well as a global description of the morphisms, and a graphical version using decorated groupoid colored graphs. Our considerations are in an enriched setting.
The global presentation utilizes pasting diagrams from 2--categories or equivalently double categories, which is of independent interest. In the special case of a UFC, we also present a graphical formalism with groupoid colored graphs.
This allows us to identify our plus constructions as the step-by--step generalization of  the Baez--Dolan plus constructions.
%   We recall  the plus construction for Feynman categories and more generally for monoidal categories.
%   With this we obtain a characterization of $\F^+$--$\opcat$ or $M^+$--$\opcat$ as monoids in an underlying category monoidal category of $\V$--$\smodcat$ that are the plus construction of a monoidal category. The prime examples are operads and PROPs, where the former come from the plus construction of a Feynman category and the latter from the Feynman category for a monoidal category.
%   Another upshot is the characterization of the  $M^+$--$\opcat$ as operads in the sense of Berger.

%   Maybe: We furthermore link the plus construction and $B_+$ operators.
\end{abstract}
\dedicatory{To Yuri Ivanovich Manin on the occasion of his 85th birthday---et in memoriam}

\maketitle

% !TEX root = KMoManinv3.tex
%\tableofcontents

\section*{Introduction}

 The original source for the type of plus constructions we are considering is the opetopic construction of \cite{BaezDolan}. Plus constructions in related contexts can  be found in \cite{monads,feynman,BMplus,Bergermoment}.
This paper introduces more general plus constructions for arbitrary (monoidal) categories and proves fundamental structural results. It gives an {\em ab uovo} progression which builds the results step--by--step, clarifying the role of each structure and property individually along the way, e.g.\ the role of levels and units. The known constructions are recovered as special cases.

The successive constructions are as follows:
The first one is a plus construction for a pointed category $\C$ yielding a monoidal category $\catplus{\C,P}$---a pointing is a functor $P\in [\act,\C]$ that is bijective on objects. The second  $\monplus{\M,P}$, is a plus construction for  a monoidally pointed (symmetric) monoidal category $\M$ which adds morphisms reflecting the original monoidal structure. The third $\locmonplus{\M,P}$ is its localization with respect to these morphisms. It has a  companion, the reduced version $\redmonplus{\M,P}$, which is a subcategory that is equivalent to the ambient category. This last version  $\plus{\M}$ with the pointing given by the inclusion $Iso(\M)\to \M$, where $Iso(\M)$ is the underlying groupoid of $\M$ retaining only isomorphisms,  is called the standard plus construction. It is what the known constructions are equivalent to in the relevant specializations. The previous constructions were graphical in nature, and to make the identification, we develop a more general graphical framework for the standard plus construction.

 Plus constructions are designed to corepresent.
 A category over $\E$ is {\em corepresented} by a category $\mathcal{C}$ if it is equivalent to the category of functors, aka cosheaves, $[\mathcal{C},\mathcal{E}]$.  In the monoidal case, the categories of functors  we utilize are strict and strong monoidal functors.
 We show  that plus constructions corepresent  (monoidal) $\B$--bimodule monoids (M)BMs. Here   a $\B$--bimodule  is  a functor $[\aca,\E]$, where $\E$ is a symmetric monoidal enrichment category. The monoidal structure for bimodules, called plethysm is  an equivariant tensor product defined by a coend. Unital (monoidal) bimodule monoids  naturally arise as the morphisms of (monoidal) categories enriched over $\E$. The paradigmatic examples of (M)BMs are the morphisms of a (monoidal) category $\C$,
 considered as $\C$, $Iso(\C)$ or $\C^{disc}$ bimodules, where $\C^{disc}$ is the underlying discrete category retaining only identity morphisms. The identity morphisms can be viewed as given by the inclusion functor $\C^{disc}\to \C$. This and the inclusion are prime examples of a pointing. We define unital versions of the plus construction to corepresent bimodule units. This  requires the relative framework given by pointed categories. The localized versions corepresent strong monoidal structures.   All constructions  work in an enriched setting, which in the non--Cartesian case necessitates counital versions. Throughout there is a dichotomy between Cartesian enrichment, which poses not problems, and non--Cartesian enrichment, which requires extra care.

The main role ---and original motivation of Baez--Dolan--- of unital plus constructions is that they corepresent so--called indexed enrichments and allow to define algebras over modules in the following sense.
An indexed enrichment $\hat \C$ of $\C$ is a bijection--on--objects functor $b:\hat C\to\C$, which contrary to a pointing is allowed to be lax or strong monoidal. If $\anyplusp(\C)$ is a unital ---unital and counital in the non--Cartesian case--- plus construction of $\C$, it  defines  an indexed enrichment $\C_\D$ of $\C$. This allows one to define algebras over the module $\D$ as
$\C_\D$--modules thereby completely generalizing the notion of algebras over operads.
A module over $\C$ in this setting is simply a cosheaf.
The terminology ``module'' for a functor is natural from the bimodule perspective.
The term
``algebra'' is natural from the operadic perspective.
In the pointed case, a free $\C$--module on a $\act$--module, can be defined by Kan extension. The basic question is if  every free algebra for a unital monoidal bimodule monoid on a strong monoidal module is strong monoidal. We identify a necessary and sufficient condition for this called ``hereditary''. This specializes to the hereditary condition  a Feynman Category (FC) satisfies \cite{feynman}, with the name ``hereditary'' going back to \cite{Markl}.
An FC \cite{feynman} is a special type of monoidal category whose  underlying groupoid is free monoidal, whose morphisms are hereditary and whose slice categories are small.
They corepresent operad-like theories via strong monoidal functors, where corepresentation is generalized to corepresentation in a category $\hat E$ over $\E$.
This includes the usual generalizations of operads, modular operads, props, properads, and the whole zoo of them. The conditions are that the underlying groupoid is free monoidal, the bimodule of morphisms is hereditary and the slice categories are essentially small. Due to the hereditary condition, there is a free--forget adjunction which allows one to view these functors as algebras over a triple, aka.\ monad.
An instructive example is that of the FC $\F^{\rm operads}$ corepresenting operads, which is the Feynman plus construction of the FC of finite sets and surjections, $\F^{\rm surj}$, see~\cite{feynmanrep} for details. Thus algebras over an operad $\O$ can be identified with functors out of $F^{\rm surj}_\O$
The hereditary condition as defined in the current paper also clarifies the relationship to the notion of patterns  \cite{Getzoperad}.

One caveat for the construction of algebras remains, namely that not every FC is the plus construction of another FC, and the pertinent question left open is which Feynman categories are plus construction of what type of categories. To answer, we introduce the notion of a Unique Factorization Category (UFC), which is a generalization of Feynman categories and the notion of \pher{}, which is also a condition for the localized plus construction to be given by a roof calculus. We prove that  the standard  plus construction of a \pher{} UFC is a Feynman category and, vice--versa, if a Feynman category is equivalent to a standard plus construction of a monoidal category $\M$, then $\M$ is a \pher{} UFC. FCs are closely linked to all types of plus constructions as they appear as their output.
UFCs generalize FCs in the following way:
For both FCs and UFCs the underlying groupoid  is free. While the morphisms of an FC are generated by basic ``many--to--one'' morphisms, which is equivalent to the hereditary condition, for a UFC the hereditary condition is relaxed to allow for basic ``many--to--many'' maps which play the role of prime factors in a unique factorization.
In contrast to the FCs, this factorizablity of morphisms is not automatically compatible with the composition, but is guaranteed by an extra \pher{} condition.
Generalizing to ``many--to--many'', one loses the forget--free adjunction, but one still can recover many interesting features among them the plethysm description. The paradigmatic example of a UFC is that of cospans and of an FC is that of finite sets. Moreover, we define an indexing functor from a \pher{} UFC to the \pher{} UFC of cospans.
The standard plus construction of a \pher{} UFC is an FC which complements a previous result that the plus construction of an FCs is again an FC.
Hence ``plus'' is a kind of stabilization.
Moreover, many FCs occur naturally as standard plus constructions.
For example, it explains why there is a Feynman category for properads: it is the localized plus construction for cospans. Note that Cospans have also  been independently considered in \cite{Hackney} in a related situation.

A further use of plus constructions is to have a formulation of operads, and more generally the corepresented functors, as monoids in a monoidal category. Whereas FCs guarantee that the functors are algebras over the monad of a free--forget adjoint pair, the monoid description involves the plethysm product. More precisely: if $\F=\anyplusp(\M)$ given an $\F$--algebra $\O$, the $\Iso(M)$, i.e. the underlying groupoid of $\M$, bimodule defined by the morphisms of $\M_\O$ is the desired monoid and vice--versa. Dropping the unitality condition,
the non--unital versions of the plus construction corepresent the  (monoidal)  bimodule monoids over the plethysm unit, which are the morphisms of $\act$. Our results explain why there is such a formulation for operads and properads, but not for props, cyclic and modular operads as the latter are not the standard plus construction of a UFC, while the former are.

The second hidden origin of the plus construction, as explicated in \cite{feynman}, lies in their use in the definition of twists which are necessary to define bar and cobar construction and also appear in Koszul duality \cite{KW2}. In order for differentials to work properly signs need to be introduced, and the twists are the correct method for this. This facet first appeared in the form of hyper modular operads in \cite{GKmodular} and is  used widely, \cite{GKcyclic}, \cite{wheeledprops}, etc.\ see \cite{KWZ} and references therein.
Accordingly, there is a variation of the plus constructions called the \emph{hyper} version, which satisfies an additional condition needed for twists in the theory of operads.

\subsection*{Outline}
After laying the foundation in \S\ref{par:prelim}---including pointed categories in \S\ref{par:philosophy}---
we define the plus constructions locally, that is by by generators and relations, in \S\ref{par:plus}.
There are several types of generators, those coming from the pointing, $\g$--morphisms stemming from compositions and in the  monoidal case $\mu$-morphisms stemming from taking monoidal products. The localization is with respect to the $\mu$ morphisms.
The conditions for the localization to be given as a right roof calculus are in Definition \ref{df:localizable} and the result is Proposition \ref{prop:roof}. One of these is the \pher{} condition, making its appearance natural.
In \S\ref{par:corep},  we prove the first main result, the  corepresentation Theorem \ref{thm:1}.  The fact that the plus construction is a good categorical notion, namely it behaves well with equivalences, follows from Theorem \ref{thm:endo} showing the the plus construction is functorial.

In \S\ref{par:defs}, we turn our attention to modules and algebras. Here we prove Theorem \ref{thm:adjunction}, which ties the hereditary condition Definition \ref{df:herpointing} to the existence of a free--forget adjunction. This yields the recharacterization of FCs, cf.\ Proposition \ref{prop:fey=her}.
We define Unique Factorization Categories (UFC)s in Definition \ref{df:UFCdef} and provide a functor $\index$ to the UFC of cospans in \S\ref{sec:hereditary-UFC-indexing}.

In \S\ref{par:global}, we work out the details of the global representation using the notion of crossed categories, cf.\ Definition \ref{df:crossed}. This paragraph contains combinatorics relating to formulas, brick wall diagrams and several types of graphs ---among them generalizing string diagrams of 2--categories with decorations--- which may be of independent interest.
This is done in a step--by--step approach providing intermediate results of independent interest, e.g.\ for defining ``planar'' or non--Sigma versions of the construction. This gives a relation to decomposable little 2--cubes as they appear in \cite{Dunn,Brinkmeier,BFSV}.
The role of units is that of {\em levelling}.
 The results are Theorem~\ref{thm:cplus-equals-FC} which characterizes the plus construction for a category as a certain Feynman category and Theorem~\ref{thm:mncp-equals-FC}
which characterizes the monoidal plus  construction under the additional assumption of factorizable pointing (see Definition \ref{df:factpointing}).

In \S\ref{par:global}, we work out the details of the global representation using the notion of crossed categories, cf.\ Definition \ref{df:crossed}. This paragraph contains combinatorics relating to formulas, brick wall diagrams and several types of graphs which may be of independent interest. The combinatorial descriptions of the plus constructions  in Theorems \ref{thm:remonplusbox}.
The main results are Theorem~\ref{thm:cplus-equals-FC} which characterizes the plus construction for a category as a certain type of FC, Theorem~\ref{thm:mncp-equals-FC}
which characterizes the monoidal plus  construction under the additional assumption of factorizable pointing (see Definition \ref{df:factpointing}), and
 Theorem \ref{thm:remonplusbox} which gives the combinatorial description of the reduced plus construction.

Specializing to UFCs in \S\ref{par:ufcstructure}, this leads to one of the main results, Theorem \ref{thm:ufcp-equals-fc}, which characterizes \pher{} UFCs as those monoidal categories whose standard plus construction is equivalent to an FC, and, vice--versa, gives conditions on an FC to be in the image of this plus construction.

Finally, we give a graphical version of the  plus construction for UFCs \S\ref{sec:graphical-plus-construction}.
  The graphical language used is the
   groupoid extension of the one used in \cite{feynman}, which is grounded in the graph category introduced in \cite{BorMan}, given by graphs with  groupoid colored edges.
   This relates UFCs  to groupoid colored properads.
   The graph combinatorics  should be of independent interest.
  In the case of FCs, we recover the construction via groupoid colored trees of \cite{feynmanrep}.
  In this graphical description as well as in the string picture, the role of the units in leveling the graphs becomes apparent.
The non--decorated groupoid colored graphs are also at the heart of Koszul duality \cite{KW2}.
  The main result is Theorem \ref{thm:graphs}, which shows that graphical versions are equivalent to the ones defined by generators and relations. This  is what allows us to make contact with the previously existing plus constructions and prove that the new plus constructions are an extension of the known ones built from first principles.

%\subsection{Background}

% \begin{rmk} There are two straightforward generalizations. The first is to replace the groupoid simply by a category or some other algebraic structure and the second is in the linear case to consider group actions, viz.\ consider $\rho$ not to be enriched.
% \end{rmk}

\subsection*{Dedication and Acknowledgements}
It is a great pleasure and honor to dedicate this  paper to Yuri Ivanovich Manin, whom  we  thank for his continued guidance and support.
His example of regarding mathematical truths, formulating and sharing his insights have been a guiding light for mathematics. His unmistakable style is an aspirational goal for the field. His character, vision and overarching influence in and outside of mathematics have  been a constant inspiration.
The current results as well as many others in the field have a direct line to the themes and presentation of \cite{ManinBook,ManinBlack,ManinQG}.
During the review process, we had to mourn the passing of the great mathematical genius that Yuri Ivanovich Manin was, and we would like to also dedicate the paper to his memory.

%\subsection{Acknowledgments}
This has been a multi--year effort dating back to at least 2019. During this time RK acknowledges support from the MPIM in Bonn, the Czech Academy of Sciences, the CRM in Barcelona and the KMPB in Berlin   and thanks the hosts, especially Yuri Manin, Martin Markl, Carles Casacuberta and Dirk Kreimer for the invitations and discussions.  RK also acknowledges recent support from the Simons foundation.
We  especially wish to thank Clemens Berger for continued key discussions on the subject. RK would also like to thank Don Zagier and Alexei Davydov for related discussions.

\section{Plus constructions, corepresentation and algberas}
\subsection{Preliminaries}
\label{par:prelim}
\subsubsection{Notation and conventions}\label{par:freemonoidal}
For a category $\C$, the underlying groupoid $\Gpd=\Iso(\C)$ is given by the objects and isomorphisms.
We let $\C^{\rm disc}$ denote the discrete category with all objects of $\C$, but only the identity maps as morphisms.
We work in the setting of categories enriched over an enrichment category $\E$. For brevity and expository purposes, we will assume that $\E$ is Cartesian closed for the most part. In this case, the theory is completely parallel to that of $\E=\Set$. In particular, the unit is a terminal object if $\E$ is Cartesian.
One can also work in the non--Cartesian closed case with suitable modifications which we will describe. For a unital $\E$--monoid $A$, the corresponding $\E$--enriched category with one object will be denoted by $\underline{A}$.
For instance, the trivial monoid $*$ gives the category $\triv$ with one object and only its identity as morphisms.
% Similarly for a monoidal category $\M$, we denote by $\underline{M}$ the usual 2--category with one object whose horizontal morphisms are objects with monoidal product, and whose vertical morphisms are those of $\M$.

In a monoidal category, we will denote the monoidal unit by $\unit$ and the left and right unit constraints  by $u_l,u_r$, the associators by $A_{123}$, and the commutators by $C_{12}$ (and more generally by $C_{ij}$) in the symmetric case. Composition and monoidal structure are related by the interchange equation:
\begin{equation}
    \label{interchangeeq}
    (\phi\otimes \psi)\circ(\phi'\ot\psi')=(\phi\circ\phi')\ot (\psi\circ\psi')
\end{equation}
The unital monoid $R=(\Hom(\unit, \unit), \circ)$ plays the role of a ground monoid for the homomorphisms.   If $\M$ is a strict monoidal category, then $R$ is commutative and equal to $Hom(\unit,\unit)$ with product structure $\ot$ by the Eckmann--Hilton argument.
The action on the morphisms is given by using the left unit constraint: $
    \Hom(\unit,\unit)\times \Hom(X,Y)\stackrel{\ot}{\to} \Hom(\unit\ot X,\unit\ot Y)\stackrel{(u_l^{\ast -1},(u_l)_*)}{\longrightarrow} \Hom(X,Y)
$.

We will denote the strong/strict/lax/op-lax  monoidal functors between two
monoidal categories $\M$ and $\M'$ by
$[\M,\M']_{\ot}$, $[\M,\M']_{strict-\ot}$ $[\M,\M']_{lax-\ot}$, respectively $[\M,\M']_{op-\ot}$. The lax structure morphisms for a monoidal functor $F$ will be called $\mu^F_{X,Y}:F(X)\ot F(Y)\to F(X\ot Y)$.
The {\em trivial} strong monoidal functor $\final:\M \to \M'$ is defined by $\final(X)=\unit_{\M'}$
and $\final(\phi)=\id_{\unit_{\M'}}$ with the isomorphisms $f_{X,Y}$ given by the unit constraints of $\M'$: $\final(X)\ot \final(Y)=\unit_{\M'} \ot \unit_{\M'} \to \unit_{\M'}$.
The category of (symmetric) monoidal functors with (symmetric) monoidal transformations is (symmetric) monoidal with level--wise monoidal structure. The unit is the trivial functor $\final$.
A {\em counit} for a (symmetric) monoidal functor $\O$ is a  (symmetric) monoidal transformation $\eta:\final\to \O$, $\eps:\O\to \final$. A {\em counital (symmetric) monoidal functor} is such a  functor together with such a choice.
Note that if the target category is Cartesian monoidal, then $\unit$ is a final object so that $\eps$ is unique and every (symmetric) monoidal functor is automatically counital. However, this is not the case when $\E$ is not Cartesian.
% A monoidal category $(\C,\otimes)$ is said to be {\em indexed over} $(\D,\otimes)$
% if there is a strong monoidal functor $\index:\C\to \D$, such a functor is called an {\em indexing}.
%\RK{Check indexing}
 \label{par:groundring}

\subsubsection{Free monoidal categories}
$\C^{\bt}$ will denote the free (symmetric) monoidal category, see \cite{matrix} for more details and examples.  The context will specify if this is symmetric or not. Passing to strict  versions, a representation is given by a category whose objects are words (tuples) in the basic objects, i.e.\ objects of $\C$, and the morphisms are words in morphisms of $\C$. The empty word  is the unit: $\unit_{\C^\bt}=\emptyset$. Composition is defined by enforcing the interchange equation.
 In the symmetric case, there are also permutations of the letters as extra morphisms, which act in a wreath product fashion. We will tacitly fix a strict version and an equivalence to the strong version.
The free (symmetric) monoidal category  comes with an inclusion $j:\C\to \C^\bt$ satisfying the  universal property  that for any functor
$f:\C\to \E$ into a monoidal category $(\E,\otimes)$,
there is a  strong (symmetric) monoidal functor $f^\bt:\C^\bt\to \E$
 with $f=f^\bt i$ that is unique up to unique isomorphism, viz.\ $[\C^\bt,\E]_{strict\mdash\ot}\simeq [\C,\E]$.
The universal property for a free symmetric monoidal category is analogous.
We will say that a category $\act$ is a free monoidal category if there is a category $\V$ and a functor $\imath:\V\to \act$ such $\imath^\bt$ is an equivalence.
A free monoidal category has a natural {\em length} for objects and for morphisms given by the length of the word, i.e.\ the number of tensor factors, e.g.\ if $X=X_1\bt\cdots\bt X_n$ then the length of $X$ is $|X|=n$.

%\RK{More crossed here or not? Also string picture or not}
%\RK{$\ot$ to $\bt$}
If $\M$ is a monoidal category, by the universal property of $\M^\bt$, the
identity functor $id_\M:\M\to \M$ defines the functor $\mu=id_\M^\bt:\M^\bt\to \M$. It is given by sending $\boxtimes$ to $\ot$. I.e.\
$\mu(X\boxtimes Y)=X\ot Y$ and $\mu(\phi\boxtimes \psi)=\phi\ot \psi$.
If $\M$ is monoidal and $f$ is a strict monoidal functor, then $f^\bt=f\circ \mu$. If $f$ is lax or op-lax, then there is a natural transformation from one side to the other, and if it is strong, then the two functors in the equation are isomorphic. This informs the following construction:
Given a monoidal category $(\M,\ot)$, there is a non--connected (nc) free monoidal \cite{KWZ,feynman} category, aka.\ strings, \cite{baues}, aka.\ necklaces in the case of $\Delta$, see e.g.\cite{rigidification-necklaces}.
This is the category $\M^{nc}$ obtained from $\M^{\bt}$ by adjoining the data of the functor $\mu$,   viz.\ adding the morphisms $\mu_{X,Y}:X\boxtimes Y\to X\ot Y$ and the morphism $\eps:\unit_\boxtimes\to \unit_\M$.
There is an equivalence of categories between $[\M,\M]_{lax-\ot}$ and $[\M^{nc},\M]_{strict\mdash\ot}$, cf.\ e.g.\ \cite[\S 3.2.1]{feynmanrep}.
To obtain colax functors, one adjoins morphisms $\mu^{op}_{X,Y}:X\ot Y \to X\bt Y$, and to obtain strong functors one also enforces $\mu^{op}\mu=id_{X\bt Y}$ and $\mu\mu^{op}=id_{X\ot Y}$. This is the localization $\M^{nc}_{\it loc}$ of $\M^{nc}$ with respect to the morphisms $\mu$. Then $[\M^{nc}_{\it loc},\E]_{strict\mdash\ot}=[\M,\E]_{strong\mdash\ot}$. Notice that $\M$ is a subcategory of $\M^{nc}_{loc}$, which is monoidally equivalent to it, but not strictly monoidally. This is how strict functors with respect to $\bt$ are equivalent to strong functors with respect to $\ot$.

\subsubsection{Enrichment}
 \label{par:enrichphil}
 We will consider categories enriched over a closed symmetric monoidal category $\E$. Functors will be allowed to take values in any (symmetric) monoidal $\hat \E$ which is tensored over $\E$.
This means that they are enriched over $\E$, have fixed functor $\ot:\E\times \hat \E\to \hat \E$, and  there are fixed natural isomorphisms:
 \begin{equation}\label{eq:tensoradjuction}
\hat \E(E\ot X,Y)\simeq \E(E,\hat \E(X,Y))
\end{equation}
We will not dwell too much on the details in the following, since the versed reader will know how to make the adjustments, while the uninitiated would risk being confused by overly detailed exposition, but, we point out extra assumptions throughout the text. We also refer to \cite{feynmanrep} for examples.
The main guidelines are as follows. There are two cases, $\E$ is Cartesian, which is a straightforward generalization from $\E=\Set$, or $\E$ is not, e.g.\ $\E$ is linear.
The latter case needs some extra care.
 Following \cite{feynman,feynmanrep}, the appropriate notion of a groupoid $\Gpd$ in the non--Cartesian setting is a freely enriched groupoid $\Gpd=\Gpd'\odot\E$ of an (ordinary) groupoid $\Gpd'$; see \cite{kellybook} for the definition of free enrichment. A good example is the group algebra $k[G]$ as $\underline{G}\odot \Vect_k=\underline{k[G]}$.
To identify the groupoid as the isomorphisms in this case, $\C$ has to be {\em isomorphism split} meaning that
$\Hom_\C(X,Y)=(I(X,Y)\odot \E )\coprod \overline{\Hom}(X,Y)$ where $I(X,Y)$ is a set of isomorphisms and $\overline{\Hom}(X,Y)$ contains no isomorphisms.
 As a mnemonic, we often use $\oplus$ for the coproduct in a non--Cartesian setting, e.g.\ for an Abelian category like $\Vect_k$ or dg-$\Vect$.
%Note that if one does not insist on the identification of $\Gpd$ with $\Iso(\M)$,
%and would just like a set theoretic inclusion of the isomorphisms, then there is no problem.

\subsubsection{$R$--module categories}
\label{sec:rmodulecat}
When dealing with categories of morphisms of enriched categories, such as the arrow category or slice category, objects of the category now become objects in $\E$.
In particular, for monoidal categories of morphisms in a monoidal category, the objects are $R$--modules for the unital commutative ground monoid $R=\Hom(\unit,\unit)$ and the categories become not only enriched over $R$, but also subcategories of the category of $R$--modules. In this situation, the (symmetric) free monoidal category is the one with the universal property for $R$--module preserving functors and is given by the free $R$--modules on objects of $\V$ and morphisms of $\V$, viz.\ $T(\V)=\bigoplus_n \V^{\ot_R n}$,
 where the objects of $\V^{\ot_R n}$ are $X_1\ot_R\cdots\ot_R X_n$ and morphisms are $\phi_1\ot_R\cdots  \ot_R\phi_n$. If $\Indec$ is an ordinary category, we replace $\Indec$ by the free $R$--module category by taking formal extension of coefficients, that is objects are the free $R$--modules $R\times X$ for $X$ an object of $\V$, for brevity assume $\Indec$ is concrete, and the morphisms are $R\times \Hom(X,Y)$ with $R$ acting on the first factor.
 %This is the image of the internal category under the adjoint, free, to the forget functor $R$-mod$\to \Set$.

\subsubsection{Assumptions}
\label{sec:colimits}
 We will  assume that  the category $\C$ is isomorphism split if it is enriched. % and $\Gpd$ is taken to be the isomorphisms.
All colimits will be indexed colimits, cf. \cite{kellybook}.
 We will also assume that the ground ring $R$ is trivial or we are enriched over $R \mddash Mod$. In this case, the ground ring of the free monoidal category is taken to be $R$.
We also assume that whenever we take a colimit in a monoidal category $(\E, \ot)$, the tensor products $X \ot -$ and $- \ot Y$ will commute with colimits.
We will need this assumption in the coend computations of \S\ref{sec:monoidal-bimodules}.
This same condition also appears in many basic constructions in the theory of operads, see \cite[\S1.9]{MSS} for instance.
In practice, $\E$ will often be closed symmetric monoidal so that monoidal products will commute with colimits as a consequence of $X \ot -$ and $- \ot Y$ being left adjoints.
We will furthermore assume that for an $\mathcal I$-indexed colimit in a category $\E$, the smallness properties of $\mathcal I$ and completeness properties of $\E$ match up appropriately.
Since a coend can be equivalently expressed as a colimit of the subdivision category \cite[IX.5]{MacLane}, we make the same assumption for computing coends.
For example, if we have a coend $\int^X F(X) \otimes G(X)$ for the pair of functors $F: \act \to \E$ and $G: \act^{op} \to \E$ such that $\act$ is a finite category, then we will assume that $\E$ is at least finite-complete.

\subsubsection{Pointed categories}
Our plus constructions are relative to a base category via a pointing.
\label{par:philosophy}

\begin{df}
    \label{df:pointing} Given a category $\act$, called the \emph{base--category}, a \emph{$\act$--pointed category} $(\C, P)$ is a category $\C$ together with an identity on objects  functor $P: \act \to \C$ called a \emph{$\act$--pointing}. We will call the morphisms $P(\sigma)$ {\em base morphisms}.
    If $\B = \act_\ot$ is a monoidal category (we will use the subscript $\ot$ to emphasize the monoidal structure), then the functor $P$ is required to be {\em strict}.
    It is \emph{compatibly pointed} if $P$ is also faithful, where in the enriched case, we take faithful to also mean split, i.e.\ there is a splitting $\Hom_\C(P(X),P(Y))=P(\Hom_{\act}(X,Y)) \oplus \overline{\Hom}_\C(P(X),P(Y))$.
    A \emph{groupoid compatible pointing} (gcp) is additionally required to be pseudomonic, this means that $P$ also induces an isomorphism of categories $P:\Iso(\act) \simeq \Iso(\C)$.
The {\em standard gcp for a category} is the inclusion $\Iso(\C)\to \C$.

% MM: Found a repeat:
% When considering (symmetric) monoidal categories $(\C,\otimes)$, the base category $\act$ is also taken to be (symmetric) monoidal.
% If we want to emphasize this point, we will write $\act_\ot$.
% \begin{df}
% A {\em monoidal pointing} is a strong (symmetric) monoidal functor $P_\ot:\act_\ot\to \M$ between two (symmetric) monoidal categories which is bijective on objects. This is compatible/pseudomonic if the underlying pointing is.
% \end{df}

Pointed categories form a category where the morphisms between $\act \overset{P}{\to} \C$ and $\act' \overset{P'}{\to} \C'$ are pairs of functors $(f:\act\to \act',F:\C\to \C')$ such that $FP=P'f$.
In the monoidal case, we require that $f,F$ are strong.  The condition of $P$ being identity on objects can be relaxed to the categorically good notion of a {\em weak pointing} which is a functor $P$ that is equivalent to a pointing, i.e.\ there are equivalences $E_1,E_2$ such that $P=E_1 \bar P E_2$ for a pointing $\bar P$.
\end{df}

\begin{rmk}
\label{rmk:mueq}
Since $P,P'$ are strict and bijective on objects, the equation $FP=P'f$ as lax monoidal functors implies that
 $P'(\mu^f_{X, Y}) = \mu^F_{PX, PY}$.
\end{rmk}

A pointing $P:\act \to \C$ gives rise to an action given by pre- and post--composition on morphisms: $(\s^{op},\s')\phi=P(\s')\phi P^{op}(\s^{op})$, where we define $P^{op}:\act^{op}\to \C^{op}$ by $P^{op}(\s^{op})=P(\s)$. We will thus write $(\sigma,\s')(\phi)=\s'\phi\s$ by slight abuse of notation.
Categorically, the action is captured by the comma category $(P^{op}\downarrow P)$. This  is the relative twisted arrow category, which  is not quite a double category as there are left and right vertical morphisms with the left morphisms in $\act^{op}$ and the right morphisms in $\act$.
We will call this the action by base--morphisms.

Our philosophy in this aricle  is that categories are
constructed starting a base category. This means that there are three
steps:
(1) fix the pointing category $\B$, (2) add new
morphisms $\overline{\Hom}_\C$ with the
action of pre- and post-composition by the
morphisms from $\B$ and (3) add a
composition rule for the new morphisms
which is equivariant with respect to the specified
action. This results in a pointing.
The main examples will be: the discrete/set pointing $\C^{\rm disc}
\stackrel{i_{\C^{\rm disc}}}\hookrightarrow \C$, the standard gcp $\Iso(\C) \stackrel{i_{\Iso(\C)}}{\hookrightarrow} \C$, and the full pointing $\C\stackrel{id_\C}\to \C$.
The discrete pointing corresponds to the classical setting where one starts with objects and their identity morphisms and adds morphisms on which the identities act as identities. For applications, the standard gcp is the most important and corresponds to the philosophy going back to \cite{feynman,feynmanrep}.
Lastly, the full pointing, which is also groupoid compatible, is the maximal compatible pointing.

\begin{rmk}[Link to double categories]
\label{rmk:double}
If $\act$ is a groupoid $\Gpd$, then there is an equivalence of categories $\Gpd\simeq \Gpd^{op}$ given by $\s\to \s^{-1}$.
In this case, the action can be changed to the left action $\sds(\phi)=\s'\phi\s^{-1}=(\s^{-1},\s')$.
This action of isomorphisms naturally corresponds to 2--cells in a double category whose vertical category is $\C(\rho)$ and whose horizontal category is $\Gpd$, with the modified source and target given by $Ps$ and $Pt$ and we will denote this action by $\sds$, see \eqref{sdseq}.
Categorically speaking, this is the thin 2--category defined by the comma category $(P , P)$. This is the version used in \cite{feynman,feynmanrep}.
 We will use the notation $\sds$ for elements of $\Gpd\times \Gpd$
and use $(\s,\s')$ for those of  $\GtG$  if necessary, tacitly making the identification $\sds\leftrightarrow (\s^{-1},\s')$ where it is warranted.

\begin{equation}
\label{sdseq}
    \begin{tikzcd}[column sep=large]
X
\ar[r,"\phi",""{name=U,inner sep=5pt,below, near start}]
\ar[d,"\underset{\simeq}{\sigma}"']
&Y\ar[d,"\underset{\simeq}{\s'}"]\\
X'
\ar[r,"\phi'=\s'\phi\s^{-1}"',""{name=D,inner sep=5pt, near start}]
&Y'
\arrow[Rightarrow, from=U, to=D, "\sdsphi"]
\end{tikzcd}
\quad\text{where}\quad
\begin{aligned}
    X=P(s(\s)),\quad & Y=P(s(\s'))\\
    X'=P(t(\s), \quad &Y'=P(t(\s'))\\
\end{aligned}
\end{equation}

\end{rmk}

    Notice that since $P$ is bijective and strong, the associativity, unitary and commutativity constraints are always in the image of $P$, e.g.\ given $X_1,X_2,X_3\in \M$ there are $\hat X_i$ such that $P(\hat X_i) = X_i$ and then $P(a_{\hat X_1,\hat X_2,\hat X_3}) \simeq
    a_{X_1,X_2,X_3}$. (up to isomorphism)

\begin{df}
\label{df:factpointing}
    A monoidal pointing is {\em factorizable} if
    for  all  morphisms $\scs$ in $\aca$ acting on  $\phi=\phi_1\ot\phi_2$ in $\M$, there are morphisms  $(\sigma_1,\s_1'),(\s_2,\sigma'_2)$ in $\aca$ and $\psi_1,\psi_2$ in $\M$, such that  $\scs (\phi_1\ot \phi_2)=(\sigma_1,\s_1')\psi_1\ot (\s_2,\sigma'_2)\psi_2$ or,  in the symmetric case, alternatively, $\scs (C_{12},C_{12})(\phi_1\ot \phi_2)=(\sigma_1,\s_1')\psi_1\ot (\s_2,\sigma'_2)\psi_2$.
\end{df}
\noindent {\sc NB:} If the base category $\B$ is free monoidal, or rigid, then any pointing $P$ is factorizable.

\subsection{The plus constructions via generators and relations} \label{par:plus}
We treat the relative case here, as described in \S\ref{par:philosophy}.
In view of the applications and the philosophy, the case of $\act = \Iso(\C)$ is especially relevant.

\subsubsection{The  plus constructions}
%$\catplus{\C, P}$, $\monplus{\M, P_\ot}$, $\locmonplus{\M, P_\ot}$ and $\redmonplus{\M,P}$}
\label{par:cplus}

\begin{df}
\label{def:cplus}
Given a $\act$--pointed category $(\C, P)$, we define a (symmetric) monoidal $(P^{op}\da P)^\bt$ pointed category $\catplus{\C, P}$.
% %\Gpd^\bt
% %\begin{equation}
% %\Iso(\C\da\C)^\bt=\Iso(\C^\bt\da \C^\bt)
% %\end{equation}
That is, a basic object of $\catplus{\C, P}$ is a triple $(X, Y; \phi)$ where $X, Y \in \Obj(\act)$ and $\phi \in \Hom_{\C}(P(X), P(Y))$.
For the sake of brevity, we will simply write this object as ``$\phi$''.
A general object of $\catplus{\C, P}$ is then a word of basic objects $\Phi=\phi_1\bdb\phi_n$.
The morphisms realizing the pointing are words $(\s_1^{op},\s'_1) \bt \ldots \bt (\s_n^{op},\s'_n)$ with $\s_i, \s'_i \in \Mor(\act)$ where a morphism $(\s_i^{op},\s'_i)$ has a source $\phi_i$ and a target $\phi'_i := P(\s'_i) \circ \phi \circ P^{op}(\s^{op}_i)$.

The {\em general morphisms} of $\catplus{\C, P}$ and their compositions are given by additional generators and the following relations:

% \RK{Check the convention throughout!  }
\noindent{\sc  Generators:}
$\gamma_{\phi_1,\phi_0}:\phi_1\bt \phi_0\to \phi_1\circ\phi_0$ for a composable pair of morphisms $(\phi_1,\phi_0)$.

\noindent{\sc Relations:}
The relations are the usual relations for a (symmetric) monoidal category, that is associativity, identities, interchange, and the following additional relations:
\begin{enumerate}
\item {\it Equivariance\/}: For morphisms $\s:X_0\to X_0'$ and $\s':X_2\to X_2'$ in $\act$\\
(a) {\it ``Outer'' equivariance\/}:  $(\s,\s')\gamma_{\phi_1,\phi_0}=\g_{\s'\phi_1,\phi_0 \s}[(id,\s')\bt (\s,id)] $
 and \\
(b) {\it ``Inner'' equivariance\/}: $\g_{\phi_1,\s\phi_0}[(id,id)\bt (\s,id)]=\g_{\phi_1\s,\phi_0}[(\s,id)\bt (id,id)]$

%  The following diagrams commute:

% \begin{equation}
% \label{isoequi}
% %\label{innerequieq}\label{outerequieq}
% \begin{tikzcd}
% \phi_1\bt \phi_0
% \ar[r, "{(id,\s')\bt (\s, id)}"]
% \ar[d, "\g_{\phi_1,\phi_0}"']
% &(\s'\phi_1)\bt (\phi_0 \s)
% \ar[d, "\g_{\s'\phi_1,\phi_0 \s}"]\\
% \phi_1\phi_0
% \ar[r, "{(\s,\s')}"']
% &\s'\phi_1\phi_0\s
% \end{tikzcd}
% \hfill
%    \begin{tikzcd}
%     \phi_1\bt \phi_0
%     \ar[r, "{(\sigma, id)\bt(id, id)}"]   \ar[d,"({id,id)\bt (\s,id)}"]
%     & \phi_1 \s\bt  \phi_0
%     \ar[d, "\g_{\phi_1,\s\phi_0}"]\\
%    \phi_1\bt\s\phi_0 \ar[r, "\g_{\phi_1\s,\phi_0}"'] &\phi_1 \s \phi_0
%     &\end{tikzcd}
% \end{equation}
%
%
% That is as morphism $\phi_0\bt \phi_1\to \phi_0\circ\sigma_{X_1}\circ\phi'_1$:
%\begin{multline}
%\g_{\phi\circ \s_{X_1}, \psi}
%\circ( (\sigma^{-1}_{X_1} \bt id_{X_0}) \Da (\id_Z \bt \id_Y))\\
%=\g_{\phi, \s_Y\circ \psi}\circ
% ((\id_Y \bt \id_X)\Da (\id_Z\bt \sigma_Y))
% \end{multline}

% \item

% \begin{equation}

% \end{equation}

    %
%That is as   morphisms: $\phi\bt \psi\to \s'_Z\circ \phi\circ\psi\circ \s_X^{-1}$
% \begin{multline}
% \label{gammabimodeq}
% \g_{\s'_Z\circ \phi,\psi\circ \s_X^{-1}}
% \circ (\id_Y \Da \sigma'_Z)(\phi) \bt(\sigma_X\Da\id_Y)(\psi)\\
% =(\s_X,\Da \s'_Z)(\phi\circ\psi)\circ \g_{\phi,\psi}
%\end{multline}

\item
{\it Internal Associativity\/}:

For a composable triple of morphisms $(\phi_2,\phi_1,\phi_0)$:  $\gamma_{\phi_2\phi_1,\phi_0}[\gamma_{\phi_2,\phi_1}\bt id_{\phi_0}]=\g_{\phi_2,\phi_1\phi_0}[\id_{\phi_2}\bt \gamma_{\phi_1,\phi_0}]$

\end{enumerate}

\end{df}

As a short hand, we will use $\gamma_{\phi_0\kdk \phi_n}:\phi_0\bdb\phi_n\to \phi_0\circ\dots\circ \phi_n$ for the unique morphism resulting from any $n$--fold iteration of $\gamma$.
% In particular, the morphism in \eqref{assreleq} is $\g_{\phi_0,\phi_1,\phi_2}$.
We will also use the convention that
$\g_\phi:=\id_\phi=(id_{s(\phi)},id_{t(\phi)}):\phi\to \phi$.

\begin{nota} If no pointing is specified for $\C$, we will take $P = i_{\Iso(\C)}: \Iso(\C) \to \C$ as our default choice of pointing and use the notation $\catplus{\C} := \catplus{\C, i_{\Iso(\C)}}$.
\end{nota}

\label{par:mnc}
\begin{df}
\label{def:mnc}
Given a pointing $P: \act \to \M$ between (symmetric) monoidal categories $\act$ and $\M$, the (symmetric)  monoidal category $\monplus{\M, P}$ is obtained from $\catplus{\M, P}$ by adjoining additional  generators  and relations.

\noindent{\sc Additional Generators:}
\begin{enumerate}
\item $\mu_{\phi_0,\phi_1}:\phi_1\boxtimes \phi_2\to \phi_1 \ot \phi_2$ for a pair of morphisms $(\phi_1, \phi_2)$
\item $\mu_\unit:id_{\unit_\bt}\to id_{\unit_\M}$ where $\unit_\bt$ is the empty word and $\unit_\M$ is the monoidal unit for $\M$
\end{enumerate}

 \noindent{\sc Additional Relations:}
 \begin{enumerate}
 \setcounter{enumi}{2}
\item\label{eq:muequi} {\it Equivariance\/}:  $\mu[(\s_1,\s_2')\bt (\s_2,\s'_2)]=[\s_1\ot \s_2,\s'_1\ot \s'_2]\mu$.
=

\item {\em Compatibility with unit constraints} $u_{\M, l}$ and $u_{\M, r}$ of $\M$:\\  $(u^{-1}_{\M, l},u_{\M, l}) \mu_{\phi, id_{\unit_\M}}[id_\phi \bt \mu_\unit]= u_l$, and  $(u_{\M,r}^{-1},u_{\M, r})\mu_{\id_{\unit_\M}}[\mu_\unit\bt id_\phi]=u_r$.

\item {\em Internal Associativity for $\mu$}:
$\mu_{\phi_1\ot\phi_2,\phi_3 }
[\mu_{\phi_1,\phi_2}\bt id_{\phi_3}]=
\mu_{\phi_1,\phi_2\ot \phi_3}[id_{\phi_1}\bt \mu_{\phi_2,\phi_3}]
$.

% \begin{equation}
% \label{eq:muassoc}
% \begin{tikzcd}
% \phi_1\bt\phi_2\bt\phi_3\ar[r,"\mu_{\phi_1,\phi_2}\bt id_{\phi_3}"]\ar[d,"id\bt \mu_{\phi_2,\phi_3}"]&\phi_1\ot\phi_2\bt\phi_3\ar[d,"\mu_{\phi_1\ot\phi_2,\phi_3}"]\\
% \phi_1\bt\phi_2\bt\phi_3\ar[r,"\mu_{\phi_1,\phi_2\ot\phi_3}"]&\phi_1\ot\phi_2\ot\phi_3
% \end{tikzcd}
% \end{equation}

\item \label{eq:intinter} {\it Internal Interchange\/}:  $\g_{\phi_1\ot \psi_1,\phi_0\ot \psi_0} [\mu_{\phi_1,\psi_1}\bt\mu_{\phi_0,\psi_0}]\t^{23}=\mu_{\phi_1\phi_0,\psi_1\psi_0}[\g_{\phi_1\phi_0}\bt\g_{\psi_1\psi_0}]$.
% \begin{equation}
% \begin{tikzcd}
% \phi_0\bt\psi_0\bt\phi_1\bt\psi_1
% \ar[r,"\t^{23}"]
% \ar[d,"(\mu\bt\mu)"]
% &\phi_0\bt\phi_1\bt\psi_0\bt\psi_1
% \ar[r,"(\g_{\phi_0,\phi_1}\bt\g_{\psi_0,\psi_1})"]
% &\phi_0\circ\phi_1\bt\psi_0\circ\psi_1\ar[d,"\mu"]\\
% (\phi_0\ot\psi_0)\bt(\phi_1\ot\psi_1) \ar[r,"\g_{\phi_0\ot\phi_1,\psi_0\ot\psi_1}"]&
% (\phi_0\ot \psi_0)\circ (\phi_1\ot \psi_1)\ar[r,equal]
% &(\phi_1\circ\phi_0)\ot(\psi_0\circ\psi_1)
% \end{tikzcd}
% \end{equation}
\end{enumerate}
And in the symmetric case:
\begin{enumerate}
\setcounter{enumi}{6}
\item \label{commu} {\it Compatibility with commutators\/}:
$\mu_{\psi,\phi}\tau^{12}=(C_{12},C_{12})\mu_{\phi, \psi}$, where $C_{12}$ is the commutativity constraint in $\M$ and $\t^{12}$ is the commutativity constraint for $\bt$.
% When $\M$ and $\act$ are symmetric monoidal,
% let $C_{12}$ be the commutativity constraint on objects. We impose the following quadratic relations:
% \begin{equation}
%
% \begin{tikzcd}
%     \phi_1\bt \phi_2
%     \ar[d,"\tau^{12}"]
%     \ar[r,"\mu_{\phi_1,\phi_2}"]&
%     \phi_2\ot \phi_1
%     \ar[d,"(C_{12}{,} C_{12})"] \\
%     \phi_2 \bt \phi_1 \ar[r,"\mu_{\phi_2,\phi_1}"] & \phi_2\ot \phi_1
% \end{tikzcd}
% \end{equation}
\end{enumerate}
\end{df}

\begin{rmk}
\label{identitiesrmk}
There are several technical remarks:
\begin{enumerate}

\item
A general morphism in $\catplus{\C, P}$ and $\monplus{\M, P}$ is obtained by concatenating and forming tensor products of identities,  generating morphisms, and isomorphisms, modulo the relations of a monoidal category ---associativity, units, interchange--- and  the additional relations above.

\item The morphisms in a monoidal category always have a symmetric structure.
% \noindent {\sc NB:} in the non--symmetric plus construction, the  symmetric structure on morphisms is ignored.
% In the symmetric case, we use the constraints $C_{12}^+$.
The isomorphisms contain the associativity, unit and in the symmetric case commutativity  constraints on the level of morphisms.

\item The unit, associators,  and in the symmetric case, the commutators descend to the quotients.
  The unit is $id_\unit$. In the strict case, $\phi\ot\id_\unit=\phi$ as a morphism $X\ot\unit=X\to Y$.
  In the non--strict case, the unit constraints and associativity constraints descend---as do the commutativity constraints.

\item Associativity and, in the symmetric case, interchange hold automatically in the quotient and do not lead to additional identifications there.

\end{enumerate}
\end{rmk}

%\subsection{The localized plus construction   $\locmonplus{\M}$ and its reduced version $\M^+$}
%\label{par:strongplus}
\begin{df}
\label{def:mloc}
For a (symmetric) monoidal category $\M$, the {\em localized plus construction} $\locmonplus{\M} = \monplus{\M}[\mu^{-1}]$ is given by  the localization with respect to the morphisms $\mu$.
%In particular, this is given by inverting the morphisms $\mu$ in $\monplus{\M}$ fiberwise.
That is, adjoin  generators $\mu^{-1}_{\phi; \phi_1, \phi_2}:\phi\to \phi_1\bt \phi_2$
whenever $\phi$ decomposes as $\phi=\phi_1\ot \phi_2$ and $\mu_\unit^{-1}:id_{\unit_\M}\to id_{\unit_\bt}$ then quotient by the relations $\mu^{-1}\mu=\mu\mu^{-1}=id$.

\end{df}

In Propositions \ref{prop:loc-red-equiv} and \ref{lem:roof}, we will need to consider iterations of $\mu$ morphisms, so we will develop the following short hand notation. Given a source $\phi_1\bt\dots\bt\phi_{n+1}$ for $I\subset \{1,\dots,n\}$, we let $\mu_I$ be the map that converts the $i$th occurrence of $\bt$ to $\ot$ for all $i\in I$. Explicitly,  if $I=\{i_1,\dots,i_k\}$ with $i_1<\dots < i_k$, then inductively set $\mu_I=\mu_{I\setminus \{i_k\}}\circ (id_{\phi_1}\odo id_{\phi_{i_k}-1}\ot \mu_{\phi_{i_k},\phi_{i_k+1}}\ot\id_{\phi_{i_k+2}}\odo id_{\phi_{n+1}})$. We also set $\mu_n=\mu_{\{1\kdk n\}}$.

\begin{rmk}
\label{rmk:loc}
In the localization,
an object $\phi_1\bdb \phi_n$ of $(P\da P)^\bt$ is isomorphic to the image of the object $\phi_1\odo \phi_n$ under $(P \da P) \to (P \da P)^\bt$.
There are  new morphisms $\bar \gamma_{\phi_0\phi_1}:=\gamma_{\phi_0,\phi_1}\mu^{-1}:\phi_0\ot\phi_1\to \phi_0\circ \phi_1$,
and new isomorphisms $(\s,\s')\mu:
\psi_1\bt \psi_2\stackrel{\sim}{\to} \psi_1\ot\psi_2\stackrel{\sim}{\to} \s'(\psi_1\ot\psi_2)\s$, and in case of $\s'(\psi_1\ot\psi_2)\s=\psi_1\ot\psi_2$, new automorphisms $\mu^{-1}(\s',\s)\mu$ of $\psi_1\bt \psi_2$.
Furthermore, $\unit=id_{\unit_\bt}$ and $id_{\unit_\M}$ become isomorphic  and the ground ring becomes the group ring $End(\unit)$ of the plus construction.
\end{rmk}

% \begin{df}
% %\label{dfprop:mp}
% \label{df:pred}

% \end{df}

\begin{prop} \label{prop:loc-red-equiv} Let $\redmonplus{\M,P}$ be the full (symmetric) monoidal subcategory of $\locmonplus{\M,P}$ whose objects are those of $(P\da P)$, then $\redmonplus{\M,P}$ is equivalent to $\locmonplus{\M,P}$ and the functor $m:\locmonplus{\M,P}\to \redmonplus{\M,P}$ that extends
  $m:(P \da P)^\bt\to (P \da P)$ by $m(\g_{\phi_0,\phi_1})=\bar \g_{\phi_0,\phi_1}$ and $m(\mu_{\psi_1,\psi_2})=id_{\psi_1\ot\psi_2}$
  witnesses the equivalence.

  Furthermore $\redmonplus{\M,P}$ can be identified with the (symmetric) monoidal category obtained from $(P\da P)$ by adjoining the morphisms $\bar \gamma:\phi_0\ot \phi_1\to \phi_0\phi_1$ which satisfy the relations of the $\gamma$ morphisms, associativity, inner and outer equivalence as in Definition~\ref{def:cplus}, and interchange in the form
  $\bar\gamma_{\phi_0\ot \psi_0,\phi_1\ot\psi_1}=\bar\gamma_{\phi_0,\phi_1}\ot \bar \gamma_{\psi_0\psi_1}$.
\end{prop}

\begin{proof}
  The equivalence is clear, since  by definition the inclusion $i:\redmonplus{\M,P}\to \locmonplus{\M,P}$ is full and faithful, and by Remark~\ref{rmk:loc} is essentially surjective.
The composition $mi$ is the identity, and there is a natural transformation $n$ providing the isomorphism from the identity to $im$ whose components are $\mu_{k-1}: \phi_1\bdb\phi_k \to \phi_1\odo \phi_k$.
This is natural in the generators $(\s,\s')$ by definition, natural in the generators $\gamma$ as $im(\gamma)=\bar\gamma$ and $\gamma=\bar \gamma \mu$, and natural with respect to $\mu$ and $\mu^{-1}$ since $im(\mu)=id$. The last statement readily follows from this, since any morphism in the subcategory is a word in the  generators which under $m$ gets sent to a word in the $\bar \gamma$s. That the $\bar\gamma$ morphisms satisfy the same relations as the $\gamma$ morphisms is clear, since these are related by isomorphisms $\mu$. The inner interchange relations \ref{def:mnc} (6) maps to the stated interchange relation under $m$.
%For this notice that by definition of the functor $m:(\mu_{\phi_0\phi_1,\psi_0\psi_1}(\gamma_{\phi_0,\phi_1}\bt \g_{\psi_0,\psi_1})\mu^{-1}_{\phi_0,\phi_1,\psi_0,\psi_1})=\bar\gamma_{\phi_0,\phi_1}\ot \bar\gamma_{\psi_0,\psi_1}$.
It is straightforward to check that there are no additional relations, coming from the morphisms $\mu, \mu^{-1}$, by going through the list of relations.
\end{proof}

\begin{df} \label{standard-plus} The {\em standard} plus construction $\M^+$ for a symmetric monoidal category $\M$, is defined to be $\M^+ = \redmonplus{\M,P}$ for the standard gcp $P$.
\end{df}
This is equivalent to the plus construction found in the literature \cite{BaezDolan,feynman,feynmanrep}
by Theorem~\ref{thm:graphs}.
If certain factorizability conditions are met, then the localization can be computed using a roof calculus. These hold for the standard plus constructions found in the literature

\begin{df}
\label{df:localizable}
A monoidally pointed category $\M$
\begin{enumerate}
    \item   has {\em localizable base morphisms}, if given a morphism $\phi$ in $\M$ and a morphism $(\s,\s')$ in $\act^{op}\times \act$ such that $\scs(\phi)=\psi_1\ot \psi_2$
    for morphisms $\psi_1,\psi_2$ in $\M$, there are morphisms $\phi_1,\phi_2$ in $\M$  and morphisms
    $\scsi{1},\scsi{2}$ in $\act^{op}\times \act$ such that $\phi=\phi_1\ot \phi_2$, and $\psi_i=\scsi{i}\phi_i$, or $\psi_1=\scsi{1}\phi_2$ and $\psi_2=\scsi{2}\phi_1$.
\item has {\em common monoidal factorizations}, if given morphisms $\phi_1,\phi_2, \psi_1,\psi_2$, satisfying $\phi_1\ot \phi_2=\psi_1\ot \psi_2$ there is a morphisms $\chi$ in $\M$ s.t.\
$\phi_1=\psi_1\ot \chi$ and $\psi_2=\chi \ot \phi_2$.
\item  is {\em prehereditary}, if given morphisms $\phi_1,\phi_2, \psi_1,\psi_2$, satisfying $\phi_1\phi_2=\psi_1\ot\psi_2$ there are morphisms $\chi_1,\chi_2,\chi_3,\chi_4$, s.t.\
$\phi_1=\chi_1\chi_2,\phi_2=\chi_3\chi_4$ and $\psi_1=\chi_1\chi_3,\psi_2=\chi_2\chi_4$.
 \end{enumerate}
Such an $\M$ is called  {\em \fact{}} if it satisfies all three conditions, and is called  {\em fully factorizable} if is factorizable and localizable.
\end{df}

\begin{lem}
\label{lem:locfac}
 If the base morphisms are  invertible, then  the pointing being factorizable is equivalent to  the base morphisms being localizable. I.e.\ in this situation, fully factorizable and localizable coincide.
\end{lem}
\begin{proof}
    Invert the horizontal arrows in the first diagram of \eqref{eq:factorizable}.
\end{proof}
\noindent{\sc NB:} This is the case for the standard gcp.
%%%%%%%%%%%%%

% \begin{ex}
% Examples are the discrete pointing, as $id_X\ot id_Y=\id_{X\ot Y}$, and when $\B_\ot=\V^\bt$ is free monoidal, as then all objects and morphisms  in products of objects and morphisms of $\V$.

% Note this is sufficient, but not necessary. Namely, if
% $\B$ has $\B(X,Y)=\emptyset$ for $X\neq Y$. For $\s\in \B(X,X)$ setting
% $P(\s)=id_X$ is a split pointing.
% \end{ex}
\begin{prop}
\label{prop:roof}
If $\M$ has a \fact{} pointing, then
the localization can be realized by a right roof calculus.
\end{prop}
\begin{proof}
The subcategory generated by morphisms $\mu_{\phi_1,\phi_2}$ is closed under composition and contains all identities as $\mu_{id_\unit,\phi}=id_\phi$ by convention.  Writing out the conditions of Definition \ref{df:localizable} in terms of morphisms in $\monplus{\M,P}$ ---the two cases for (1) are given by $j=0,1$ in the first diagram---
we precisely  obtain the right Ore conditions for the three types of generators \eqref{eq:factorizable}. In the symmetric case, the last Ore condition is the compatibility condition (7) of Definition \ref{def:mnc}. There are no non--trivial cancelability conditions to check,  since ---due to the nature of the relations--- there are no non--identical parallel morphisms coequalized by  a combination of $\mu$'s.
%First $\mu\circ\sigma_{12}=\sigma_{12}\circ \mu$.
%\RK{Michael, can we make these fit on one line? Maybe by making the equalities vertical}
\begin{equation}
\label{eq:factorizable}
%\label{eq:sigmaroof}
%\label{eq:sigmafac}
\begin{tikzcd}[column sep=large]
\phi_1\bt \phi_2\ar[r,dotted,"\scsi{1}\bt\scsi{2}"]\ar[d,dotted,"\mu"]&\psi_1\bt\psi_2
\ar[d,"\mu"]\\
\phi\ar[r,"\scs {(C_{12},C_{12})^j}"]&\psi
\end{tikzcd},
\begin{tikzcd}
%\label{eq:cm}
\psi_1\bt\chi\bt\phi_2\ar[r,dotted,"id\bt\mu"]
\ar[d,dotted,"\mu\bt id"]
&\psi_1\bt \psi_2\ar[d, "\mu"]\\
\phi_1\bt\phi_2\ar[r,"\mu"]&\phi
\end{tikzcd},
\begin{tikzcd}
%\label{eq:hereditary}
\chi_1\bt \chi_2\bt \chi_3\bt \chi_4\ar[r,dotted,"(\g\bt\g)\sigma_{23}"]
\ar[d,dotted,"\mu\bt\mu"]
&\psi_1\bt \psi_2\ar[d, "\mu"]\\
\phi_1\bt\phi_2\ar[r,"\g"]&\psi
\end{tikzcd}
\end{equation}

\end{proof}

%
%\begin{rmk} In case that $\M$ is isomorphism split, the
%factorization of isomorphisms becomes unique.
%%  $\sds\mu^{-1}_{\phi;\phi_1,\phi_2}=(\s\ot)$
% \end{rmk}

%May happen that $\phi\ot\tilde\psi=\phi\ot \psi$ or that $\phi_2\ot \phi_1\phi_0=\phi_2\ot\phi_1\ot \phi_0$. but the morphisms are different.

%{\sc NB} This reflects the semi-simplicial structure of $\M^{nc}$.

\begin{lem}
\label{lem:roof}
Two roofs $(\mu_1,f_1)$ and $(\mu_2,f_2)$ are equivalent if and only if there is a roof $(\mu_l, \mu_r)$ such that the diagram
\eqref{eq:roofdiag} commutes
\begin{equation}
\label{eq:roofdiag}
\begin{tikzcd}
&&\Phi_a\ar[dr,"\mu_r"]\ar[dl,"\mu_l"']&&\\
&\Phi_v\ar[dl,"\mu_1"']\ar[drrr,"f_1", very near start]&&\Phi_w\ar[dlll,"\mu_2"', very near start]\ar[dr,"f_2"]\\
\Phi_s&&&&\Phi_t
\end{tikzcd}
\end{equation}
\end{lem}
\begin{proof}
  If $(\mu_1,f_1)$ and $(\mu_2,f_2)$ are equivalent, then by definition there is a roof $(\mu_l, f)$ which makes \eqref{eq:roofdiag} commute.
  Since $\mu_1\mu_l=\mu_I$ for some index set $I$, we see that $f$ must be $\mu_r$ with $\mu_2\mu_r=\mu_I$, since no $\g$ and hence no maps $\scs$ may appear.
\end{proof}

%{\sc \noindent {\sc NB:}} This is equivalent to the decompsability of $\sds$ or $\sdsphi\circ (C_{12}\Da C{12})$.

% \begin{definition}
% \label{df:mp}
% For a monoidal category $\M$ with factorizable isomorphisms, let $\plus{\M}$ be the category whose underlying groupoid
% is $\Iso(\M\da \M)$ with additional generating
% %morphisms   the images of the $\g$'s, which are
% morphisms $\bar \g_{\phi_0,\phi_1}:\phi_1\ot\phi_2\to \phi_0\circ \phi_1$ that satisfy equivariance with respect to isomorphisms and internal associativity
% % MM: Refers to somting commented out
% % , as given by the formulas \eqref{isoequi} and \eqref{assreleq},
% where all occurrences of $\bt$ are replaced by $\ot$.
% \end{definition}

\subsubsection{Unital and counital plus construction}
\label{par:unitalconst}
\label{par:plusnotation}
In many situations, units are part of the structure, while counits are important in the non--Cartesian case.
To adapt the plus construction as corepresenting objects for these structures, the morphisms that become  the
units/counits under functors must be added as generators---with new relations.
To shorten statements, we will use the generic notation $\anyplus$ for one of the plus constructions introduced below. That is $\anyplus$ can be $\catplus{\C,P}$,
$\monplus{\M,P}$, $\locmonplus{\M,P}$, $\redmonplus{\M,P}$ or $\M^+$.

\begin{definition}
\label{gcpdef}
\label{df:hyp}
Starting with any of the
plus constructions $\anyplus$, define the {\em unital plus construction} $\anyplusp$ for a  pointing to be the
monoidal category obtained by first
freely monoidally adjoining the morphism $i_\sigma: \unit_\bt \to P(\sigma)$, one for each morphism $\sigma$ in the underlying category $\act$ to $\anyplus$.
% Here the $\unit_\bt = \emptyset$ is the unit with respect to the $\bt$--product (or $\ot$ in the case of $\M^+$).
Here freely monoidally means that the data of the $i_\s$ is extended by
 $
 i_{\s_1\bdb\s_n}:=\unit_\bt
\to \unit_\bt \bdb \unit_\bt \stackrel {i_{\s_1} \bdb i_{\s_n}}{\longrightarrow}\s_1\bdb \s_n
$. \\
Let $u_\bt: \unit_\bt \bt \unit_\bt \to \unit_\bt$ be a unit constraint
and mod out by the following relations:
 \begin{itemize}
     \item [(a)]
 The compatibility with the pointing structure:
$i_\s=i_{\s'}$ if $\P(\s)=P(\s')$, and
$(\s,\s')(i_\tau)=i_{ \sigma' \tau \sigma}$.
%\end{equation}

\item[(b)]  The compatibility with the generating morphisms: $\gamma_{\s,\s'}[i_\s\bt i_{\s'}]=i_{\s\s'}u_\bt$
\item[(c)] The unit conditions:
  $\gamma_{\s,\phi}[i_\s \bt id_\phi]u^{-1}_l = (\s, id): \phi \to \phi\s$ and
  $\gamma_{\phi, \s'}[id_\phi \bt i_\s] u^{-1}_r = (id, \s'): \phi\to \s'\phi$,
  viz $i_{\s'}\phi i_\s=(\s,\s')\phi$.
% $\gamma_{\phi,\s}[id_\phi\bt i_\s]=(\s,id)u_r$ and $\gamma_{\s',\phi}[i_{\s'}\bt id_\phi]=(id,\s')u_l$.
\end{itemize}
% \begin{equation}
% \label{icompateq}
% \begin{tikzcd}
%     \phi_0 \bt \sigma_1
%     \ar[r, "\gamma"]
%     & \phi_0 \circ \sigma_1 \\
%     \phi_0 \bt \unit_{\anyplus}
%     \ar[u, "id_{\phi_0} \bt i_{\sigma_1}"]
%       & \phi_0
%     \ar[u, "(\sigma_1 {,} id)"', "\sim"]  \ar[l, "\text{unit constraint}", "\sim"']
% \end{tikzcd}
% \hspace{1.5cm}
% \begin{tikzcd}
%     \sigma_0\bt \phi_1
%     \ar[r, "\gamma"]
%     &\sigma_0 \circ \phi_1 \\
%     \unit_{\anyplus} \bt \phi_1
%     \ar[u, "i_{\sigma_0} \bt id_{\phi_1}"]
%     & \phi_1
%     \ar[l, "\text{unit constraint}", "\sim"']
%     \ar[u, "(id{,}\sigma_0)"', "\sim"]
% \end{tikzcd}
%\end{equation}
For the monoidal unital cases $\monplusp{\M,P},\locmonplusp{\M,P},\redmonplus{\M,P}$, we further mod out by the relations:
\begin{itemize}
    \item [(d)]Monoidal compatibility: $\mu_{\s,\s'}[i_\s\bt i_{\s'}]u_{\bt}=i_{\s\ot \s'}$.
\end{itemize}
In the non--Cartesian case, we also require $i_\s$ to be split.
% \begin{equation}
%  \label{eq:unitcompat2}
%   \begin{tikzcd}
%     \unit\bt \unit\ar[r,"i_\s\bt i_{\s'}"]&\s\bt\s''\ar[d,"\mu_{\s,\s'}"]\\
%     \unit\ar[r,"i_{\s\ot\s'}"]\ar[u,"{\rm unit}"]&\s\ot\s'
%   \end{tikzcd} \hspace{.5cm} \text{and  by} \hspace{.5cm}
%   \begin{tikzcd}
%     \unit\ot \unit\ar[r,"i_\s\ot i_{\s'}"]&\s\ot\s\\
%     \unit\ar[ur,"i_{\s\ot\s'}"'] \ar[u,"{\rm unit}"]&
%   \end{tikzcd}
%   \text{ for $\monplusgcp{\M}$}
% \end{equation}
\noindent The {\em counital plus construction} $\anyplus^{c}(\C,P)$ based on $\anyplus(\C,P)$
is obtained by freely monoidally adding generators
$r_\phi:\phi\to \unit_\bt$ which satisfy the
equations:
\begin{itemize}
\item[(a')] Equivariance: $r_{ \sigma'\phi\sigma}(\s,\s')=r_\phi$.
\item [(b')] Multiplicatively: $r_{\phi_1\phi_0}\gamma_{\phi_1,\phi_0}=u_\bt[r_\phi\bt r_\psi]$.
\item[(c')] Counitality:
$(id, \s')u_l[r_{\s'} \bt id_\phi]=\g_{\s',\phi}$ and
$(\s, id)u_r[id_\phi \bt r_\s]=\g_{\phi,\s}$
\end{itemize}
And in the monoidal case
\begin{itemize}
    \item [(d')] Monoidal compatibility: $u_\bt[r_\phi\bt r_\psi]=r_{\phi\ot \psi}\mu_{\phi,\psi}$.
\end{itemize}
In the unital and counital case, the equation $r_{P(\s)} i_\s=id_{\unit_\bt}$ is enforced.

\noindent The {\em hyper--version} $\anyplus^{hyp}(\C,P)$ is the unital version with inverses $r_{\s}$ to the $i_{\s}$. In the unital and counital case, hyper means that the inverses are part of the counital structure.

There are natural inclusion functors $Pt: \anyplus \to \anyplusp$, $B:\anyplus \to \anyplusg$, $BPt:\anyplus\to \anyplusgp$  by inclusion and $Hyp:\anyplusgp\to \anyplushyp$ as a quotient.
\end{definition}

To make contact with previous definitions of the unital plus construction, in the case of the standard gcp we will use the superscript $gcp$ and omit $P$ for the unital plus construction, e.\ g.\ $\plusgcp{\M}$.

\begin{lem} A functor $\D$ from $\anyplusp$ factors through $\anyplushyp$ if and only if   $\D(\phi)=\unit_\E$ for all $\phi\in \anyplus$. Such a functor will be called {\em hyper}.
\end{lem}
\begin{proof}
   This is straightforward in the Cartesian case. In the non-Cartesian case, this follows from the fact that $i_\s$ is taken to be split.
\end{proof}

\begin{ex}
\label{ex:algebra}
Consider the trivial category $\triv$, with the pointing $\triv\to \triv$. The sole basic object of $\catplus{\triv}$ is $id_*$ with general objects being words $n=id_*^{\bt n}$. The pointing gives the morphism $(id_*,id_*)=id_{id_*}$
and there is a generating morphism $\g_{id_*,id_*}: 2 \to 1$ iterations of which yield morphisms $n \to 1$ for each $n$. A general morphism from $n$ to $m$ is then given by surjections $n = n_1 \bt \dots \bt n_m \to m$. This can be seen as the skeleton of ordered surjections, aka.\ the subcategory of surjection in the simplicial category $\Delta$.

Endowing the category $\triv$ with the monoidal structure $*\ot *=*$, the monoidal $\monplus{\triv}$ has additional morphisms $\mu_{id_*, id_*}: 2 \to 1$ and $\mu_{\unit}: 0 \to 1$ where $0$ is the empty word.
Applying the Eckmann-Hilton argument to internal interchange shows that $\g_{id_*,id_*} = \mu_{id_*, id_*}$.
Using the monoidal structure, we can factor a general morphism $f$ into $f = i \circ s$ where $i$ is monoidally generated by $id_{id_*}$ and $\mu_\unit$, and $s$ is generated by $id_{id_*}$ and $\mu_{id_*, id_*}$.
The unitality relation allows one to exchange the two types of morphisms: $s \circ i = i' \circ s'$ in such a fashion that morphisms in $\monplus{\triv}$ behave the same way as maps of finite sets do. Hence this category is ordered finite sets in the non--symmetric version and finite sets in the symmetric version.
% Localizing doesn't give anything interesting.
\end{ex}

\begin{ex}
\label{ex:monoidalgebra}
Considering the free symmetric monoidal category on  $\triv$, which is equivalent to $\SS$, the basic objects of $\catplus{\SS,id_\SS}$ are $\sigma \in \SS_n$. The action is by $(\t,\t')(\s)=\t'\s\t$ and
% Considering the free symmetric monoidal category $\triv^\bt=\SS$, for $\monplus{\C}$ with pointing  by $\SS$,
each object $\sigma \in \SS_n$ is isomorphic to $id_*^{\bt n},n  \in \N$,
with $\Aut(id_n)\simeq \SS_n$.
Up to isomorphism, a general object is then represented by a word $id_{n_1} \bt \ldots \bt id_{n_k}$, that is simply is a tuple $(n_1,\dots, n_l)$.
For $\monplus{\SS}$, there are generating morphisms $\gamma_n: id_n \bt id_n \to id_n$ for each $n \in \N$.
Iterating morphisms of this type yields ordered surjections and general morphisms are given by $\bt$ products of these and isomorphisms.
Moving on to $\monplus{\SS}$, we get a new generating morphism $\mu: id_n \bt id_m \to id_{n+m}$.

In $\locmonplus{\SS}$, the object $(n_1\kdk n_l)$ becomes isomorphic to $n=\sum_{i=1}^l n_i$ and it is shown in \cite{feynman} that one obtains ordered surjection in the non--symmetric case and surjections with orders on the fibers in the non--symmetric case. Adding units upgrades the skeletal version to $\Delta_+$ and non--commutative sets, that is finite sets with maps where the maps have orders on the fibers.
\end{ex}

\begin{ex}
\label{ex:operad}
Consider the category $\C=\Surj$ that is finite sets with surjections with the pointing by $\Iso(\C) \simeq \SS$ where the latter is the skeleton of the category of finite sets and bijections. The basic objects of $\catplus{\C}$ are surjections $f:S\twoheadrightarrow T$. The morphisms are generated by pairs of bijections $(\s:S'\leftrightarrow S,\t:T\leftrightarrow T')$ and the basic compositions for two surjections $f\bt g\to f\circ g$.
A surjection $f$ is given by its fibers $\amalg_{t\in T}f^{-1}(t)$. Each fiber can be depicted by a corolla with inputs, aka.\ leaves, $f^{-1}(t)$ and output, aka.\ root, $t$. In this way, $f$ is a collection/forest/aggregate of corollas. A general object of $\catplus{\C}$ is a collection of forests.
The composition of two such forests of corollas $f$ and $g$ is possible by a matching of the roots of $g$ with the leaves of $f$ and contracting the new edges. This is given by the composability $S\stackrel{g}{\twoheadrightarrow}T\stackrel{f}{\twoheadrightarrow}U$,  cf.\ e.g.\ \cite{KWZ} for pictures.
Inner equivariance means that one can change $T$ by an isomorphism, which is taken care of pictorially by the matching. Outer equivariance means that one can pre-- and post--compose with isomorphisms relabeling the leaves and root of the output. In this fashion, the $\g_{f_1\kdk f_n}$ are $n$--level forests whose target is the total contraction.

The category $\C$ has a symmetric monoidal product given by disjoint union, which yield additional morphisms $f\bt g\to f\amalg g$ in $\monplus{\C}$ that send two forests of corollas to their union. This is similar to the unboxing of \cite{KWZ}.
The morphism in the category $\C$ are monoidally generated by the surjections to one element $\pi_{S,t}:S\to \{t\}$.
The skeleton of $\locmonplus{\C}$ then has just these corollas as basic objects. When a forests for $\g_{f_1\kdk f_n}$ has a corolla as a target, $f_1$ must be a corolla so that these forests become trees under gluing roots to leaves.
With the results of the next chapters, this category is equivalent to the Feynman category for operads, cf.\  \cite{feynmanrep} for full details.
Units play a special role here as they allow one to level and de--level the trees, see section B.16 {\it loc.\ cit.}.
\end{ex}

Other examples can be found in \cite{feynman,feynmanrep} for the plus construction for Feynman categories and in \cite{Michael}.

\subsubsection{Equivalence principle}
\label{par:goodnotion}
We will extend the plus construction to the category of pointed categories as defined in Definition \ref{df:pointing}.
A consequence is that the equivalent categories have equivalent plus constructions making it a categorically ``good'' (not evil) notion.

\begin{lem}
\label{lem:functor}
On objects, define $\anyplus(f, F)(\phi_1 \boxtimes \ldots \boxtimes \phi_n) = F(\phi_1) \boxtimes \ldots \boxtimes F(\phi_n)$. \\
On generating morphisms define:
\begin{itemize}
\item [(a)] $\anyplus(f, F)(\sigma , \sigma') = (f(\sigma) , f(\sigma')): F(\phi) \to P'(f(\s))F(\phi)P'(f(\s))=
F(P(\s')\phi P (\s))$
\item [(b)] $\anyplus(f, F)(\gamma_{\phi_1,\phi_0}) = \gamma_{F(\phi_1), F(\phi_0)}: F(\phi_1) \boxtimes F(\phi_0) \to F(\phi_1) \circ F(\phi_0) = F(\phi_1 \circ \phi_0)$,
\item[(c)] In the monoidal case, writing $\phi_1 = (X_1, Y_1; \phi_1)$ and $\phi_2 = (X_2, Y_2; \phi_2)$
  \[ \monplus{f, F}(\mu_{\phi_1, \phi_2})
  = ( (\mu^f_{X_1, X_2})^{-1}, \mu^f_{Y_1, Y_2}) \mu_{F(\phi_1), F(\phi_2)}: F(\phi_1) \boxtimes F(\phi_2) \to F(\phi_1 \otimes \phi_2) \]
Note by functorality we must have $\locmonplus{f,F}(\mu^{-1})=\monplus{f,F}(\mu)^{-1}$.
\end{itemize}
Then $\anyplus(f, F): \anyplus(\C, P) \to \anyplus(\C', P')$ is a  functor which in the (symmetric) monoidal case is strict (symmetric) monoidal functor.
\end{lem}

\begin{proof}
  We need to verify that the assignments of $\anyplus(f,F)$ preserve the relations of the plus categories.
  First, we check the equivariance of $\mu$.
  Let $(\sigma_1, \sigma_1'): (X_1, Y_1; \phi_1) \to (W_1, Z_1; \psi_1)$ and $(\sigma_2, \sigma_2'): (X_2, Y_2; \phi_2) \to (W_2, Z_2; \psi_2)$ be two isomorphisms and consider the outer rectangle of Figure~\ref{fig:plus-functor-mu}.
  The top square commutes by equivariance of $\mu$ with respect to isomorphisms.
  Equality (1) is an application of interchange.
  Equality (2) is a consequence of naturality.
  The commutativity of the middle square is straightforward.
  Equalities (3a) and (3b) follow from $P'(\mu^f_{X, Y}) = \mu^F_{PX, PY}$, cf.\ Remark \ref{rmk:mueq}.
  The final square commutes tautologically.

  \begin{figure}
    \centering
    \begin{equation*}
    \begin{tikzcd}[column sep = 1.3cm]
      F(\phi_1) \boxtimes F(\phi_2)
      \arrow[d, "\mu_{F(\phi_1), F(\phi_2)}"']
      \arrow[r, "{(f \sigma_1, f \sigma'_1) \boxtimes (f \sigma_2, f \sigma'_2)}"']
      & F(\psi_1) \boxtimes F(\psi_2)
      \arrow[d, "\mu_{F(\psi_1), F(\psi_2)}"]
      \\
      F(\phi_1) \otimes F(\phi_2)
      \arrow[r, "{(f \sigma_1, f \sigma'_1) \otimes (f \sigma_2, f \sigma'_2)}"']
      \arrow[ddd, "{((\mu^f_{X_1, X_2})^{-1}, \mu^f_{Y_1, Y_2})}"']
      & F(\psi_1) \otimes F(\psi_2)
      \arrow[d, equals, "\text{(1)}"]
      \\
      {}
      & P'(f \sigma'_1 \otimes f \sigma'_2)
      (F\phi_1 \otimes F\phi_2)
      P'(f \sigma_1 \otimes f \sigma_2)
      \arrow[d, "{((\mu^f_{W_1, W_2})^{-1}, \mu^f_{Z_1, Z_2})}"]
      \\
      {}
      & P'(\mu^f_{Z_1, Z_2})
      P'(f \sigma'_1 \otimes f \sigma'_2)
      (F\phi_1 \otimes F\phi_2)
      P'(f \sigma_1 \otimes f \sigma_2)
      P'(\mu^f_{W_1, W_2})^{-1}
      \arrow[d, equals, "\text{(2)}"]
      \\
      P'(\mu^f_{Y_1, Y_2} )
      (F\phi_1 \otimes F\phi_2)
      P'(\mu^f_{X_1, X_2})^{-1}
      \arrow[r, "{( f(\sigma_1 \otimes \sigma_2), f(\sigma'_1 \otimes \sigma'_2) )}"']
      \arrow[d, equals, "\text{(3a)}"]
      & P'f(\sigma'_1 \otimes \sigma'_2)
      P'(\mu^f_{Y_1, Y_2})
      (F\phi_1 \otimes F\phi_2)
      P'(\mu^f_{X_1, X_2})^{-1}
      P'f(\sigma_1 \otimes \sigma_2)
      \arrow[d, equals, "\text{(3b)}"]
      \\
      F(\phi_1 \otimes \phi_2)
      \arrow[r, "{( f(\sigma_1 \otimes \sigma_2), f(\sigma'_1 \otimes \sigma'_2) )}"']
      & P'f(\sigma'_1 \otimes \sigma'_2)
      F(\phi_1 \otimes \phi_2)
      P'f(\sigma_1 \otimes \sigma_2)
    \end{tikzcd}
    \end{equation*}
    \caption{Equivariance relation of $\mu$ under $\anyplus(f, F)$}
    \label{fig:plus-functor-mu}
  \end{figure}

  For the remaining verifications, we suppress the subscript of $\mu^f_{X,Y}$ and simply write $\mu^f$.
  Next, observe that an internal interchange diagram in $\monplus{\M}$ gets mapped to the following diagram in $\monplus{\M'}$:
  \begin{equation}
    \begin{tikzcd}[column sep = 0.5cm]
      F\phi_0 \boxtimes F\phi_1 \boxtimes F\psi_0 \boxtimes F\psi_1
      \arrow[rr, "\gamma \boxtimes \gamma"] \arrow[d, "(\mu \boxtimes \mu) \circ \sigma_{23}"']
      & {}
      & F\phi_1 \circ F\phi_0 \boxtimes F\psi_1 \circ F\psi_0 \arrow[d, "\mu"] \\
      F\phi_0 \otimes F\psi_0 \boxtimes F\phi_1 \otimes F\psi_1
      \arrow[d, "{((\mu^f)^{-1}, \mu^f) \boxtimes ((\mu^f)^{-1}, \mu^f)}"']
      \arrow[r, "\gamma"]
      & (F\phi_0 \otimes F\phi_1) \circ (F\psi_0 \otimes F\psi_1)
      \arrow[r, equal]
      & (F\phi_1 \circ F\phi_0) \otimes (F\psi_1 \circ F\psi_0)
      \arrow[d, "{((\mu^f)^{-1}, \mu^f)}"] \\
      F(\phi_0 \otimes \phi_1) \boxtimes F(\psi_0 \otimes \psi_1)
      \arrow[r, "\gamma"']
      & F(\phi_0 \otimes \phi_1) \circ F(\psi_0 \otimes \psi_1)
      \arrow[r, equal]
      & F(\phi_1 \circ \phi_0 \otimes \psi_1 \circ \psi_0)
    \end{tikzcd}
  \end{equation}
  The top commutes by internal interchange in $\monplus{\M'}$ and the bottom commutes by equivariance of $\gamma$.
 When it applies, the compatibility with commutators diagram gets sent to \eqref{eq:compat-com}.
  \begin{equation}
    \label{eq:compat-com}
    \begin{tikzcd}[column sep = 1.5cm]
    {F\phi_1\bt F\phi_2}
    \arrow[r, "{\mu_{\phi_1,\phi_2}}"]
    \arrow[d, "\tau^{12}"]
    & {F\phi_1\ot F\phi_2}
    \arrow[d, "{(FC_{12}, FC_{12})}"]
    \arrow[r, "{((\mu^f)^{-1}, \mu^f)}"]
    & {F(\phi_1\ot\phi_2)}
    \arrow[d, "{F(C_{12}, C_{12})}"] \\
    {F\phi_2\bt F\phi_1}
    \arrow[r, "{\mu_{\phi_2,\phi_1}}"]
    & {F\phi_2\ot F\phi_1}
    \arrow[r, "{((\mu^f)^{-1}, \mu^f)}"]
    & F(\phi_2\ot \phi_1)
    \end{tikzcd}
  \end{equation}
  The left commutes by compatibility with commutators in $\monplus{\M', P'}$ and the right commutes because the functor is symmetric monoidal.
\end{proof}

\begin{lemma}
\label{lem:functorcomp}
  For a composible pair $(f, F)$ and $(g, G)$ of pointed functors,
 $\catplus{g \circ f, G \circ F} = \catplus{g, G} \circ \catplus{f, F}$.
 In the monoidal case, $\monplus{g \circ f, G \circ F} = \monplus{g, G} \circ \monplus{f, F}$.
\end{lemma}\begin{proof}
  The left and right hand sides affect basic isomorphisms in exactly the same way.
  The two sides also agree for $\gamma$-morphisms by functorality.
  For $\mu$-morphisms, this is a consequence of composition of monoidal functors:
 \begin{equation}
  \begin{aligned}
       &\monplus{g, G}(\monplus{f, F}(\mu))
      = ((g \mu^f_{X_1, X_2})^{-1}, g \mu^f_{Y_1, Y_2})
      \circ \monplus{G}(\mu_{F(\phi_1), F(\phi_2)})    \\
      &= ((g \mu^f_{X_1, X_2})^{-1}, g \mu^f_{Y_1, Y_2})
      \circ ((\mu^g_{FX_1, FX_2})^{-1}, \mu^g_{FY_1, FY_2})
      \circ \mu_{GF(\phi_1), GF(\phi_2)}    \\
      &= ((\mu^{g \circ f}_{X_1, X_2})^{-1}, \mu^{g \circ f}_{Y_1, Y_2})
      \circ \mu_{GF(\phi_1), GF(\phi_2)}
      = \monplus{g \circ f, G \circ F}(\mu)
  \end{aligned}
 \end{equation}

\end{proof}

\begin{df}
  \label{df:pointing-2mor}
  Building on the category structure defined in Definition~\ref{df:pointing}, we define additional 2--morphisms $(\beta, \alpha)$ where $\alpha: F \to G$ and $\beta: f \to g$ are natural transformations such that $\beta$ whiskered with $P'$ is the same as $\alpha$ whiskered with $P$.
  In the monoidal case, these natural transformations are monoidal as well.
\end{df}

\begin{lemma}
  \label{plus-natural-transformation}
  For a 2--isomorphism $(\beta, \alpha)$ as in Definition~\ref{df:pointing-2mor}, the family of morphisms $\catplus{\beta, \alpha}_\phi := (\beta_{s(\phi)}^{-1}, \beta_{t(\phi)})$ indexed over the morphisms of the original category constitutes a monoidal natural transformation $\catplus{\beta, \alpha}: \catplus{f, F} \to \catplus{g, G}$.
  Moreover when $(\beta, \alpha)$ are monoidal natural transformations, this defines a monoidal natural transformation $\monplus{\beta, \alpha}: \monplus{f, F} \to \monplus{g, G}$.
\end{lemma}

\begin{proof}
  Like before, it is enough to prove naturality for generators.
  We start by considering a morphism $(\s, \s'): \phi \to \psi$ and verifying that \eqref{eq:action-natural} commutes.
   \begin{equation}
     \label{eq:action-natural}
     \begin{tikzcd}[column sep = 2cm]
       F(\phi)
       \arrow[d, "{\anyplus(\beta, \alpha)_\phi}"']
       \arrow[r, "{\anyplus(f, F)(\sigma,\sigma')}"] % need braces for some reason
       & F(\psi)
       \arrow[d, "{\anyplus(\beta, \alpha)_\psi}"]
       \\
       G(\phi)
       \arrow[r, "{\anyplus(g, G)(\sigma, \s')}"]
       & G(\psi)
     \end{tikzcd}
   \end{equation}
   Unpacking the clockwise direction and applying the whiskering assumption amounts to appending the $\bullet$-marked arrows in Diagram~\eqref{eq:unpack} onto $F(\phi)$.
   Likewise, the counterclockwise morphisms amounts to appending the $\circ$-marked arrows onto $F(\phi)$.
   Diagram~\eqref{eq:unpack} commutes by naturality of $\alpha$, which implies that clockwise and counterclockwise directions coincide.
   Hence Diagram~\eqref{eq:action-natural} commutes.
  \begin{equation}
    \label{eq:unpack}
    \begin{tikzcd}
      FP(X')
      \arrow[d, leftarrow, "\alpha_{X'}^{-1}", "\bullet" marking]
      \arrow[r, "FP(\sigma)", "\bullet" marking]
      & FP(X)
      \arrow[r, "F(\phi)"]
      \arrow[d, leftarrow, "\alpha_{X}^{-1}", "\circ" marking]
      & FP(Y)
      \arrow[d, "\alpha_{Y}^{-1}", "\circ" marking]
      \arrow[r, "FP(\s')", "\bullet" marking]
      & FP(Y')
      \arrow[d, "\alpha_{Y'}", "\bullet" marking]
      \\
      GP(X')
      \arrow[r, "GP(\sigma)", "\circ" marking]
      & GP(X)
      \arrow[r, "G(\phi)"']
      & GP(Y)
      \arrow[r, "GP(\s')"', "\circ" marking]
      & GP(Y')
    \end{tikzcd}
    \end{equation}
  Naturality of $\gamma$-morphisms is a straightforward consequence of inner and outer equivariance:
  \begin{equation}
    \begin{tikzcd}[column sep = large]
      F(\phi_1) \boxtimes F(\phi_0)
      \arrow[d, "{\anyplus(\beta, \alpha)_{\phi_1} \boxtimes \anyplus(\beta, \alpha)_{\phi_0}}"']
      \arrow[r, "{\gamma_{F(\phi_1), F(\phi_0)}}"]
      & F(\phi_1 \circ \phi_0)
      \arrow[d, "{\anyplus(\beta, \alpha)_{\phi_1 \circ \phi_0}}"] \\
      G(\phi_1) \boxtimes G(\phi_0)
      \arrow[r, "{\gamma_{G(\phi_1), G(\phi_0)}}"']
      & G(\phi_1 \circ \phi_0)
    \end{tikzcd}
  \end{equation}
 To show naturality for the $\mu$-morphisms in the monoidal case, consider the following diagram with $\monplus{F}(\mu)$ on top and $\monplus{G}(\mu)$ on the bottom:
  \begin{equation}
    \begin{tikzcd}[column sep = 2cm]
      F(\phi_1) \boxtimes F(\phi_2)
      \arrow[r, "\mu_{F(\phi_1), F(\phi_2)}" {yshift=5pt}]
      \arrow[d, "\monplus{\alpha}_{\phi_1} \boxtimes \monplus{\alpha}_{\phi_2}"]
      & F(\phi_1) \otimes F(\phi_2)
      \arrow[r, "{((\mu^f_{X_1, X_2})^{-1}, \mu^f_{Y_1, Y_2})}" {yshift=5pt}]
      \arrow[d, "\monplus{\alpha}_{\phi_1} \otimes \monplus{\alpha}_{\phi_2}"]
      & F(\phi_1 \otimes \phi_0)
      \arrow[d, "\monplus{\alpha}_{\phi_1 \otimes \phi_2}"] \\
      G(\phi_1) \boxtimes G(\phi_2)
      \arrow[r, "\mu_{G(\phi_1), G(\phi_2)}"' {yshift=-5pt}]
      & G(\phi_1) \otimes G(\phi_2)
      \arrow[r, "{((\mu^g_{X_1, X_2})^{-1}, \mu^g_{Y_1, Y_2})}"' {yshift=-5pt}]
      & G(\phi_1 \otimes \phi_0)
    \end{tikzcd}
  \end{equation}
  The left square commutes by equivariance with respect to isomorphisms.
  The right square commutes because $\alpha$ is a monoidal natural transformation.
\end{proof}

\begin{theorem}  \label{thm:endo}
  The plus construction $\catplus{-}$ is a 2--functor from the 2--category of pointed categories, pointed functors, and pointed natural isomorphisms to the 2--category of symmetric monoidal categories, symmetric monoidal functors, symmetric monoidal natural isomorphisms.

  Similarly, the plus constructions $\monplus{-}$, $\locmonplus{-}$, and $\redmonplus{-}$
  as well as their unital, counital, gcp and hyper versions
  are 2--functors from monoidal categories to symmetric monoidal categories and endofunctors on symmetric monoidal categories.
 \end{theorem}
\begin{proof}
The statement for $\catplus{-}$ and $\monplus{-}$ is proven in Lemmata \ref{lem:functor}, \ref{lem:functorcomp} and \ref{plus-natural-transformation}.
  Now consider the localized version.
  For a pointed functor $(f, F)$, define $\locmonplus{f, F}: \locmonplus{\M, P} \to \locmonplus{\M', P'}$ so that $\locmonplus{f, F}(\mu_{\phi_1, \phi_2}^{-1})=\mu_{F(\phi_1), F(\phi_2)}^{-1}(\mu^f_{s(\phi)_1, t(\phi_1)}, (\mu^f_{s(\phi_2), t(\phi_2)})^{-1})$
  and define the functor on the remaining generators like we did with $\monplus{f, F}$.
  % $F(\phi_1 \otimes \phi_2) \to F(\phi_1) \boxtimes F(\phi_2)$.
 % to be the composition
 %    \begin{equation}
 %      \begin{tikzcd}
 %        F(\phi_1) \boxtimes F(\phi_2)
 %        & F(\phi_1) \otimes F(\phi_2)
 %        \arrow[l, "\mu_{F(\phi_1), F(\phi_2)}^{-1}" {yshift=-5pt}]
 %        & F(\phi_1 \otimes \phi_0)
 %        \arrow[l, "(f_{X_1, X_2}^{-1} \Da f_{Y_1, Y_2}^{-1})" {yshift=-5pt}]
 %      \end{tikzcd}
 %    \end{equation}
 %  \end{description}
  For a pointed natural isomorphism $(\beta, \alpha)$, define $(\locmonplus{\beta, \alpha})_\phi$ the same way as ${\monplus{\beta, \alpha}}_{\phi}$.
  The commutativity of the diagrams for $\mu$-morphisms implies the commutativity of the diagrams for $\mu^{-1}$-morphisms.
  Hence $\locmonplus{f, F}$ and $\locmonplus{\beta, \alpha}$ are both well-defined.
  Therefore the argument for $\monplus{-}$ carries over to $\locmonplus{-}$ and $\redmonplus{-}$.

  Finally, for the unital, gcp and hyper versions, define
     $\anyplusp(f, F)(i_\sigma)$ as
        $F(\unit_{\M}) \stackrel{\sim}{\to}\unit_{\M'} \stackrel{i_{F(\sigma)}}{\to}  F(\sigma)$.
In the counital and hyper cases define $\anyplusg(F)(r_\sigma)$ by
$F(\unit_{\M}) \stackrel{\sim}{\leftarrow} \unit_{\M'} \stackrel{r_{F(\sigma)}}{\leftarrow}
 F(\sigma)$.
  The rest of the argument is now routine.
 \end{proof}
\begin{cor}
  \label{cor:plusgood}
  If $(\C,P)$ and $(\C',P')$ are equivalent pointed categories, then
 $\catplus{\C,P} \cong \catplus{\C',P}$.
  If $(\M,P)$ and $(\M',P')$ are monoidally equivalent pointed monoidal categories, then their plus constructions are also monoidally equivalent $\anyplus(\M,P)\simeq \anyplus(\M',P')$

  % $\anypu{\M} \cong \monplus{\M'}$,
  %   $\locmonplus{\M} \cong \locmonplus{\M'}$, $\redmonplus{\M} \cong \redmonplus{\M'}$.
 The same holds for the unital, counital, gcp and the hyper versions of these functors.
\end{cor}

%%% MM cut: I think this is clear
%\begin{proof}
%  We start with $\monplus{\M}$
%  Suppose the functors $F: \M \to \M'$ and $G: \M' \to \M$ and the natural transformations $\alpha: GF \Rightarrow Id_{\M}$ and $\beta: FG \Rightarrow Id_{\M'}$ induce a monoidal equivalence.
%  By the lemmata, we have functors $\monplus{F}: \monplus{\M} \to \monplus{\M'}$ and $\monplus{G}: \monplus{\M'} \to \monplus{\M}$ and the natural transformations $\monplus{\alpha}: \monplus{G}\monplus{F} \Rightarrow \monplus{Id_{\M}}$ and $\monplus{\beta}: \monplus{F} \monplus{G} \Rightarrow \monplus{Id_{\M'}}$.
%  The only thing that remains is to show that $\monplus{Id_{\M}} = Id_{\monplus{\M}}$, but this is immediate. The other cases are analogous.
%\end{proof}

%\subsection{Special monoidal categories}

\subsection{Corepresentation and indexing}
\label{par:corep}

\subsubsection{Categories as unital bimodule monoids} \label{prelim-cat-bimodule}
The  point of view of bimodules is useful for the three step process above and necessary for the  corepresentation properties of the plus construction.
 For categories $\act$ and $\act'$,
a $\act$--$\act'$ bimodule in $\E$ is a functor $\act^{op}\times \act'\to \E$.\footnote{Both ways of ordering $\act^{op}$ and $\act'$ are reasonable. We adopt the categorical version which agrees with how the $Hom$ functor is usually written. The alternative is the algebraic version of functors $[\act'\times \act^{op},\E]$ which agrees with the convention that the left action is covariant and the right action is contravariant.} By convention, a $\act$-bimodule is a $\act$--$\act$ bimodule.
One can also consider the enriched analog.
In that case,
if $\act=\underline{R}$ is the one object $\Ab$-enriched category given by a ring $R$, then a $\act$-bimodule is an $R$--bimodule in the usual sense, whence the name.
$\act$--bimodules form a category $\act\mdash Bimod$ whose morphisms are natural transformations.

% $\act$--bimodules also form a monoidal category under levelwise tensor product which we write as $(\rho\times \rho')(X,Y)=\rho(X,Y)\ot_\E \rho'(X,Y)$ and monoidal natural transformations.

\subsubsection{Plethysm monoids}
\label{sec:plethysm-monoids}

The \emph{plethysm product} of a $\act^{op}$--module $\a$ and a $\act$--module $\beta$ is defined to be the coend $\int^{X\in \B}\a(X)\ot \beta(X)$. In particular for an
$\act$--$\act'$--bimodule $\rho$ and a
$\act''$--$\act$--bimodule $\rho'$, the plethysm product is the $\act''$--$\act'$--bimodule defined by the coend $(\rho\pl_{\act}\rho')(X'',X')=\int^{X\in \act}\rho(X,X')\ot_\E\rho'(X'',X)$.
Specializing to $\act$---bimodules, the plethysm product
$\pl_{\act}$ yields a monoidal structure on the category of bimodules whose unit is $u= \Hom_{\act}(-,-)$.
A bimodule $\rho$ is coaugmented  if it is endowed with a map, i.e.\ natural transformation, $\eta:u \to \rho_\C$.
We will write $u_\s$ for the element $\eta(\s)$ of $\rho$.

\begin{df}
A $\act$--bimodule monoid is a bimodule $\rho$ together with an associative multiplication map $\g:\rho\pl_{\act}\rho\to \rho$.
We will write $\gamma (\phi_1\sq_\act \phi_0)=\phi_0\phi_1$.
A unit is a coaugmentation which we assume that the image is split in the enriched case, that is $\rho(X,Y)=\eta(u(X,Y))\oplus \bar \rho(X,Y)$.
Split unital means that $\eta$ is an isomorphism onto its image, i.e. $\rho(X,Y)\simeq \rho_\act(X,Y))\oplus \bar \rho(X,Y)$. A split unit is a {\em groupoid compatible} (gc) if
(a) $\bar\rho(X,Y)$ contains no isomorphisms and
(b) $\eta(\s)$ is invertible only if
$\s$ is.
\end{df}

Paradigmatic examples of unital bimodule monoid are the $\Hom$ functor $\rho_\C := \Hom_\C: \C^{op} \times \C \to \E$ its restriction to the
    discrete subcategory $\C^{disc}$,
    $\rho_{\C^{disc}}: (\C^{disc})^{op} \times \C^{disc} \to \E$, and its restriction to the underlying groupoid $\Iso(\C)$.
Composition in the category makes $\rho_{\C^{disc}}$ into
a unital monoid in the category of $\C$ bimodules, with the unit given by the inclusion of the identity maps $id_X$ into $\Hom_\C(X,X)$. This is simply a reinterpretation of the requirement that for
$\phi_0$ and $\phi_1$ to be composable, one needs that $t(\phi_0)=s(\phi_1)$ where $s,t$ are the source and target maps. Note that by dropping the unital condition, one will only need a category without units, aka.\ semi--category \cite{MitchellSemiCat}.
Considering the full $\rho_\C$, the monoidal structure is still composition which descends to the coend of the plethysm product due to the associativity equation $\phi_1\circ(\psi\circ \phi_2)=(\phi_1\circ\psi)\circ \phi_2$. In this case, the unit is the identity natural transformation.
Restricting $\Hom$ to the groupoid $\Gpd=\Iso(\C)$, this is still a unital monoid.
All of these are compatibly pointed and the restriction to $\Gpd$ is groupoid compatible.

More generally, any functor $P:\act \to \C$  gives rise to a unital $\act$--bimodule  monoid via pull--back
$\rho_P =P^*\rho_{\C}$ defined by
$P^*(\rho_{\C})(-,-) = \Hom_\C(P^{op}(-), P(-)): \act^{op} \times \act\to \E$.
The functorial action is by pre-- and post--composition, $\rho_P(\s^{op},\s')\phi=P(\s')\phi P^{op}(\s^{op})$.
The monoid structure is given by the composition
$\rho_P(Y,Z)\pl_{\act}\rho_P{(X,Y)}\to \rho_P(X,Z): \phi\pl_{\act}\psi\to \phi\circ \psi$ and the unit is
$\eta=P:\rho_{\act}\to \rho_P$, i.e.~$u$ is the natural transformation defined by $P$ on morphisms, viz.\ $P:\Hom_{\act}(X,Y)\to \Hom_\C(P(X),P(Y))$.
There is a connection to the relative twisted arrow category as $\el(\rho_P)=(P^{op}\da P)$, where $\el$ is the category of elements.

 The trivial  bimodule monoid is given by $\final(X,Y)=\unit$ with the composition map given by the unit constraints.  This is the terminal bimodule monoid in the Cartesian case. Indeed a natural transformation is given by $\rho(X,Y)\to \T(X,Y)=\unit$. Such maps are unique in the Cartesian case and extra data in the general case.
 A unit for the functor $\final$
 must send any $\s\in \act(X,Y)$ to $\unit \in \final(X,Y)$, this is unique in the Cartesian case and otherwise extra data.

\begin{prop}\label{prop:bimodcat}
    The category of (gcp/ split) unital $\act$--bimodule monoids is equivalent to $\act$ (gcp/ compatibly) pointed categories $\C$.
    Without the unital condition, the equivalance is between bimodule monoids and semi--categories, whose objects are those of $\B$ and whose morphisms have an two--sided action of $\B$.
\end{prop}
\begin{proof}
Given a unital bimodule $\rho$, define the category $\C(\rho)$ to have the same objects as $\act$ and  $\Hom_{\C(\rho)}(X,Y)=\rho(X,Y)$ with composition and the compatible pointing given by the monoid structure and the unit. The composition is given by $\phi_0\circ\phi_1=\mu(\phi_0,\phi_1)$. The identity morphisms for the category are given by the images $\eta(id_X)$.
If the unit is split, so that $\Hom_{\C(\rho)}=\Hom_{\act}\oplus \bar \rho$, then the inclusion of the first factor defines a faithful functor $P:\act\to \C$.

Vice--versa, given a pointing $P:\act\to \C$, the corresponding  bimodule is $\rho_P$ with the unit provided by the pointing. If the functor is faithful, then $\Hom_{\C(\rho)}=\Hom_{\act}\oplus \bar \rho$ and the pointing is split.
The conditions of being gcp clearly match up on both sides.

Finally, dropping the unit, forgets the pointing and thereby the units in the category. The remaining structure is that of a semi--category.
\end{proof}

\label{par:bimodmon}

For the trivial bimodule $\final:\act^{op} \times \act \to \E$, the category
$\C(\final)$ is the complete groupoid on $\Obj(\act)$, viz.\ exactly one isomorphism,  between any two objects. This is also the groupoid associated to the complete graph on the vertices labeled by objects of $\C$.

\begin{rmk}
Rephrasing the philosophy using bimodules:
Step~(1) is the specification of an underlying bimodule monoid $\rho_{\act}$, coding a base category $\act$.
%We will denote the morphisms of $\Gpd$ as $\sigma$.
 In step~(2), the putative new morphisms are then given by a $\act$ bimodule
 $\rho$ in $\E$.
For step~(3), one specifies the structure of a unital monoid for the
plethysm product $\pl_\act$.
This perspective is explored
% in \S\ref{par:pley} and
more deeply for $\act = \Iso(\C)$ in \cite{plethysm}.
\end{rmk}

\subsubsection{Monoidal bimodules}
\label{sec:monoidal-bimodules}

In addition to the plethysm monoid structures, a bimodule can be equipped with a second kind of multiplication analogous to the monoidal product in a monoidal category if the base category $\act_\ot$ has a monoidal structure.

\begin{df}
For a (symmetric) monoidal category $\act_\ot$,  a $\B_\ot$ bimodule $\rho$ is called {\em monoidal} if it is a lax monoidal  functor.
\end{df}
Being a lax monoidal functor means that the is a natural transformation of functors $\mu_\rho:\rho\ot \rho \to \rho (\mu_\act\ot \mu_\act)\t_{23}:\aca\times \aca\to \E$ that is in elements a natural family of morphisms
\begin{equation}
    \mu_\rho:\rho(X,Y)\ot \rho(X',Y')\to \rho(X\ot X',Y\ot Y')
\end{equation}
  and a unit $\unit_\E\in \rho(\unit_\act,\unit_\act)$ which satisfies the usual conditions.
This family being natural with respect to the morphisms $\scs$ means:
 $\mu_\rho(\rho(\s_1{,} \s_2)\times (\s'_1,\s'_2))
=(\rho(\s_1 \ot \s'_1),(\s_2 \ot  \s'_2)) \mu_\rho$. The multiplication
$\mu_\rho$ is  associative $\mu_\rho[\mu_\rho\ot id_\rho]\simeq\mu_\rho[id\ot \mu_\rho]$ and unital with respect to the a given unit morphism $\unit_\E \to \rho(\unit_{\act}, \unit_{\act})$.

 \begin{ex}
 \label{ex:hommon}
 If $\M$ is a monoidal category then $\rho_\M=\Hom_\M(-,-):\M^{op}\times \M \to \Set$, has a  monoidal structure $\mu_{\rho_\M}$. The transformation $\mu$ is given by the monoidal structure:  $\mu_{\rho_\M}:\Hom(X,Y)\times \Hom(X',Y')\to \Hom(X\ot X',Y\ot Y')$ is given by $(\phi,\psi)\mapsto \phi\ot \psi$. Note the compatibility %\eqref{eq:sdsotcompat}
 holds since
 $(\s_1, \s_2)\ot(\s'_1,\s'_2)=((\s_1\ot\s'_1),(\s_2\ot \s'_2))$
 and as $u$ is the inclusion $u_\s=\s$ and $\mu(\s\ot \s')=\s\ot\s'$ so that  the unit conditions
% \eqref{eq:unitcond}
holds tautologically. The associativity is given by the associators in $\M$.
This structure restricts to $\act_\ot=\M^{disc}$ and $\act_\ot=\Iso(\M)$.
 \end{ex}

  \begin{lem}
  \label{lem:rhorhomonoidal}
If $\rho$ is monoidal then so is $\rho\sq_\act\rho$.
% The monoidal structure is induced by the maps $\mu_\rho\ot \mu_\rho: (\rho(Y,X)\ot \rho(Y,Z))\ot
%      (\rho(Y',X')\ot \rho(X',Z'))
%     \to (\rho(Y\ot Y',Z\ot Z')\ot \rho(X\ot X',Y\ot Y')$,
 The lax structure is given by the sequence of maps
$
%\label{eq:rhorhomonoidal}
    \int^{Y\in \act}\rho(Y,Z)\ot \rho(X,Y)\ot \int^{Y'\in \act}\rho(Y',Z')\ot \rho(X',Y')\stackrel{I}{\to}$ $ \int^{Y\times Y'\in \B\times \B} \rho(Y\ot Y',Z\ot Z')\ot \rho(X\ot X',Y\ot Y')\to \int\limits^{\hat Y\in \B}\rho(\hat Y, Z\ot Z')\ot \rho(X\ot X',\hat Y)$,
%\end{multline}
where $I: (\rho \pl_{\act_\ot} \rho)\ot (\rho \pl_{\act_\ot} \rho) \to (\rho\ot\rho)\pl_{\act\times \act}(\rho \ot \rho)$ is the interchange defined by applying Fubini and exchanging the second and third term and the second limit is defined by $\mu_\act: \act \ot \act \to \act$ and the universality of colimits.
    \end{lem}

 \begin{proof}
 Straightforward.
\end{proof}

\begin{df}
A monoidal bimodule $\rho$ is a {\em monoidal bimodule monoid} (MBM) if $\g:\rho\sq_\act\rho\to \rho$ is  monoidal natural transformation.
    In the unital case,
  $\g$ is required to be compatible with the unit in the sense that $u_{id_\unit}\in \rho(\unit,\unit)$ is the lax unit of $\rho$ and
  $\mu\circ (u\ot u)=u\circ\ot$.
  A morphisms from a $\B_\ot$ MBM to a $\C_\ot$ MBM is given by a strong monoidal functor $\B_\ot\to \C_\ot$. This applies to units and counits.
\end{df}
The unital condition in elements reads
   % \begin{equation}
  %   \label{eq:unitcond}
  $\mu(u_\s,u_{\s'})=u_{\s\ot\s'}$.
  % \end{equation}
  Note that $(\s_1,\s_2)\ot (\s'_1,\s'_2)=((\s_1\ot\s'_1),(\s_2\ot \s'_2))$ by definition.  Being a monoidal natural transformation translates to the following interchange equation: $
    \mu_\rho (\gamma \ot \gamma)=\gamma\circ(\mu_{\rho} \ot \mu_\rho)\circ I:(\rho\pl_{\act}\rho)\ot_\E (\rho \pl_{\act_\ot}\rho)\to \rho
  $.

\noindent {\sc NB:} More generally one can use pairs of functors $(F,G)$ to define the morphism  by precomposing with $F^{op}\times G:\B^{op}\times \B\to \C^{op}\times \C$ where now $F$ is oplax and G is lax ($F$  being oplax means $F^{op}$ is lax).

\begin{prop}
    The category of lax monoidal  unital $\act_\ot$--bimodule monoids is equivalent to $\act_\ot$ compatibly pointed monoidal categories $\M$.

    Without the unital condition, this is an equivalence between MBMs and monoidal semi-categories, whose objects and their monoidal structure are those of $\B$ and whose morphisms have a monoidal action of $\B$.
\end{prop}
\begin{proof}
    Analogous to Proposition~\ref{prop:bimodcat} with the addition of  the operation tensoring morphisms and monoidal units provided by the extra data.
    In the non--unital case, this is just a reinterpretation of the data.
\end{proof}
The monoidal category defined by a monoidal $\rho$ will be called $\M(\rho)$.

{\sc NB:} For a unital MBM $\rho$ to be strong monoidal means that the morphism  $\Hom_{\M(\rho)}(X,Y) \ot \Hom_{\M(\rho)}(X',Y)' \to \Hom_{\M(\rho)}(X\ot Y, X'\ot Y')$ is an isomorphism, which is one of the conditions of being rigid in the definition of \cite{DelMilne}.

\subsubsection{Indexing Data, Algebras, and Functors}
In order to identify the corepresented objects, functors are encoded as indexing data.
\begin{df}
\label{df:indexingdata}
Given a pointing $P$,  suppressing the obvious associators  in $\E$, the following data will be called \emph{indexing data}:
\begin{enumerate}
\item A functor $\D: (P^{op} \da P) \to \E$, viz.\
an object $\D(\phi)$ of $\E$ for each $\phi\in \Mor(\C)$ together with an action of $\act^{op} \times \act$:
$\D((\s,\s')):\D(\phi)\to \D(P(\s')\phi P(\s))$.
\item Multiplication maps:
$\gamma^\D_{\phi_1,\phi_0}:\D(\phi_1) \ot_\E \D(\phi_0)\to \D(\phi_1\circ\phi_0)
$ that are:\\
(a) associative: $\g_{\phi_2\phi_1,\phi_0}^\D[\gamma^D_{\phi_2,\phi_1}\bt id_{\phi_0}]=\g^D_{\phi_2,\phi_1\phi_0}[id_{\phi_2}\bt \g^D_{\phi_1,\phi_0}]$,\\
(b) outer equivariant: $\g^\D [\D((\s , \id)) \bt \D((\id, \s'))] = \D((\s, \s')) \g^\D$, and\\
(c) inner equivariant: $\g^\D [\D((\id , \s)) \bt \D((\id , \id))] = \g^\D[\D((\id , \id)) \bt \D((\s,\id))]$
\end{enumerate}
\emph{Monoidal indexing data} includes additional maps:
\begin{enumerate}
\setcounter{enumi}{2}
\item  This includes a family of maps  $\mu^\D_{\phi,\psi}:\D(\phi\bt\psi):=\D(\phi)\ot \D(\psi)\to \D(\phi\ot\psi)
$ and a map
$\mu_\unit^\D:\D(id_{\unit_\bt})=\unit_\E\to \D(id_{\unit_\M})$, these maps are required to be\\
(a) associative: $\mu^\D_{\phi_1,\phi_2\ot\phi_3}[id_{\phi_1}\bt \mu^\D_{\phi_2,\phi_3}h]=\mu^\D_{\phi_1\ot\phi_2,\phi_3}[\mu^\D_{\phi_1,\phi_2}\bt id_{\phi_3}]$,
\\
(b) equivariant:
$\mu^\D[ \D(\scs)\ot \D(\tct)]=\D((\s\ot \t,\s'\ot \t')) \mu^\D$,\\
(c) distributive $
\g^\D_{\phi_0\ot \psi_0,\phi_1\ot \psi_1}
[\mu^\D_{\phi_0,\psi_0}\bt\mu^\D_{\phi_1,\psi_1}]=\mu^\D_{\phi_0\phi_1,\psi_0\psi_1}
[\g^\D_{\phi_0,\phi_1}\bt
\g^\D_{\psi_0,\psi_1}] \tau^{23}$, and\\
(d) symmetric in the symmetric case:
$\mu_{\phi,\psi}^\D=\D((C_{12},C_{12})) \mu_{\psi,\phi}^\D\tau^{12}$.\\
Monoidal data is \emph{strong} if the maps $\mu^\D_{\phi,\psi}$ are isomorphisms otherwise it is called {\em lax}.
\end{enumerate}
\emph{Unital indexing data} is a
 natural  transformation $\eta$ from the trivial functor $\final:P(\act^{op}\da \act)\to \E$ to the restriction of  $\D$, i.e.\
 \begin{itemize}
\item[(u)]
 For each
 $\phi\in \Mor(\act)$, an element $u_\phi:\unit_\E\to  \D(P(\phi))$   such that
\begin{itemize}
     \item [(a)]  $\D(\scs)(u_\phi)= u_{\s'\phi \s}$.
 \end{itemize}
 with the additional conditions that:
\begin{itemize}
%\item [(b)] if $P(X)=s(\phi)$, $u_{id_X}$ is a right identity for $\D(\phi)$ under $\g^D_{\phi,id_{P(X)}}$, and
% if $P(Y)=t(\phi)$, $u_{id_Y}$ is a left identity for $\g^D_{id_{P(Y)},\phi}$, and
% \item[(c)]  the following compatibility holds: $(\s,\s')^\D  =\g^D(\g^D\ot id)[u_{\s'}\ot \D(\phi)\ot u_\s ]u^{-1}_{\E,l} u^{-1}_{\E,r}: \D(\phi)\to \D(P(\s') \phi P(\s))$.
 \item[(b)] $\gamma_{\phi_0\phi_1}(u_{\phi_0} \otimes u_{\phi_1}) = u_{\phi_0\phi_1}:\unit_\E \to \D(\phi_0\phi_1)$ under the identification $\unit_\E\ot \unit_\E\simeq \unit_\E$.
 \item[(c)] $(\s, id)^\D = \g^D [\D(\phi)\ot u_\s ]u^{-1}_{\E,r}: \D(\phi)\to \D(\phi P(\s))$ and \\ $(id,\s')^\D = \g^D[u_{\s'}\ot \D(\phi)]u^{-1}_{\E,l}: \D(\phi)\to \D(P(\s')\phi)$
\end{itemize}
%the following diagrams commute:
% %there is an induced action in the sense that the $\D(\sdsphi)$ defined by
% \begin{equation}
% \begin{tikzcd}
% \D(\phi)
% \ar[d,"\scs^\D"']
% \ar[r,"\sim"]&\unit_\E \ot_\E \D(\phi)\ot_\E \unit_\E
% \ar[d,"u_\s\ot id\ot u_{\s'}"]\\
% \D(P(\s')\phi P(\s))&\D(P(\s)) \ot_{\E} \D(\phi) \ot_{\E} \D(P(\s'))\ar[l,"\g(\g\ot id)"]
% \end{tikzcd}
% \end{equation}\\
In the monoidal case, it is also required that
\begin{itemize}
    \item [(d)]
$\mu^{\D} [u_\phi\ot u_\psi] \simeq u_{\psi\ot\phi}$ under the identification $\unit_\E\ot \unit_\E\simeq \unit_\E$.
\end{itemize}
 In the enriched case, unital data is required to be split, that is $\D(\s)=u_{\s}\oplus \bar \D(\s)$.
 Unital data is {\it compatible\/} if the natural transformation consists of injections and {\it groupoid compatible\/} if $\bar \D(\s)$ contains no isomorphisms and $u_\s$ is an isomorphism if and only if $\s$ is. Unital data is {\it reduced\/} if $\D(\s)\simeq \unit_\E$ for all $\s\in \Mor(\act)$.
\end{itemize}
\emph{Counital indexing data} is specified by a natural transformation $\D\to \final$, i.e.\ the additional data
\begin{itemize}
    \item[(c)] Morphisms  $\eps_\phi:\D(\phi)\to \unit_\E$ which satisfy:
\begin{itemize}
\item [(a')] $\eps_{\scs\phi} \scs=\eps_\phi:\D(\phi)\to \unit_\E$.
\item[(b')]
  $\eps_{\phi_1\phi_0}\gamma^\D_{\phi_1,\phi_0}=\eps_{\phi_1}\ot\eps_{\phi_0}: \D(\phi_0)\ot \D(\phi_1)\to \unit_\E$ under the identification $\unit_\E\ot\unit_\E\simeq \unit_\E$.
\item[(c')] $(\s,id)u_r[id_\phi \bt \eps_\s] = \g_{\phi,\s}$ and $(id,\s')u_l[\eps_{\s'}\bt id_\phi] = \g_{\s',\phi}$.
% $(\s,\s')u_lu_r[\eps_{\s'}\bt id_\phi\bt \eps_\s] = \g_{\s',\phi,\s}$.
\end{itemize}
In the monoidal case,  we also require that
\begin{itemize}
    \item [(d')]$\eps_{\phi\ot \psi}\simeq (\eps_\phi\ot \eps_\psi)\mu^\D_{\phi,\psi}$ under the identification $\unit_\E\ot\unit_\E\simeq \unit_\E$.
\end{itemize}

\end{itemize}
In case of unital and counital data, it is required that  $\eps_\phi u_\phi=id_{\unit_\E}$.
\end{df}
\noindent As $u_\s=\D(\id_{s(\s)}, \s) u_{\id_{s(\s)}}$ %=\D((\s^{-1}\Da id_{t(\s)})u_{t(\s)})$,
the unital data is already fixed to be the data on the unit morphisms.

% Note that the two conditions of (u.c) imply that morphisms of the form $u_{id}$ are left and right units. For instance, the right unit is given by $\g^D [\D(\phi)\ot u_{id_{s(\phi)}} ]u^{-1}_{\E,r} = (id, id)^\D = id_{\D(\phi)}$.
% However, if we combine the two conditions of (u.c) into one condition $(\s,\s')^\D  =\g^D(\g^D\ot id)[u_{\s'}\ot \D(\phi)\ot u_\s ]u^{-1}_{\E,l} u^{-1}_{\E,r}: \D(\phi)\to \D( P(\s')\phi P(\s))$, then we cannot draw the same conclusion.

\begin{prop}\label{prop:datafun}
Fix a (symmetric) monoidal category $\E$.
\begin{enumerate}
\item For a category $\C$, the collection of strict (symmetric) monoidal functors $\D: \Pl(\C,P) \to \E$ is in bijection with sets of (symmetric) indexing data.

\item For a {\em (symmetric) monoidal} category $\M$, the collection of strict (symmetric) monoidal functors $\monplus{\M} \to \E$ are in bijection with sets of (symmetric) monoidal indexing data.

\item The strict (symmetric) monoidal functors $\D:\locmonplus{\M}\to \E$ are in bijection with strong (symmetric) monoidal indexing and is equivalent to strong monoidal functors $\redmonplus{\M}\to \E$.

\end{enumerate}
 Functors from the unital plus constructions of a category are in bijection with unital data, those from the counital plus construction to counital data, and those from the unital and counital construction to unital and counital data, where the data being reduced is equivalent to the functor factoring through the hyper plus construction.

 Defining (symmetric) (monoidal) natural transformations between indexing data via these identifications, upgrades the bijections to isomorphisms of categories.

\end{prop}

\begin{proof}
  Giving a strict (symmetric) monoidal functor from $(P^{op}\da P)^\bt$ is, by the universality of the free  (symmetric) monoidal product, equivalent to giving a  functor from $(P^{op} \da P)$.
 The extension to the plus constructions is given  by $\D(\g_{\phi_0,\phi_1})=\g_{\phi_0,\phi_1}^D$, and in the monoidal case $\D(\mu_{\phi,\psi})=\mu^\D_{\phi,\psi}$. The compatibilities are straightforward.
 Reading these definitions backwards yields indexing data from a functor. It is straightforward that this is a  bijection.

 By the universal property of localization,
$[\locmonplus{\M}, \E]_{strict\mdash\ot}$ is equivalent to the subcategory $[\monplus{\M}, \E]_{strict\mdash\ot}$ such that the maps $\D(\mu_{\phi,\psi})$ are isomorphisms.
That this translates to the condition of being strong for $\redmonplus{\M}$ is straightforward.

In the unital, counital and hyper cases, the identification is given by $\D(i_\s)=u_\s$, $\D(r_\phi)=\eps_\phi$ which is an isomorphism if the data is reduced.
 The additional claims are straightforward.
\end{proof}

\begin{ex}
    In the case of the trivial category, strict monoidal functors from $\catplus{\triv}$ are given by the data of an object $A = \D(id_*)$ and a multiplication map $\g^\D:A\ot A\to A$, that is an algebra.
    In the unital case, there is a unit $1\in A$ given by $u_{id}$ and a counit adds a morphism $A \to \unit_\E$ (e.g.\ $k$ in the case $\E=kVect$).
    That is, these functors are (unital) (augmented) monoids.

A strong monoidal functor $\catplus{\SS} \to \C$ determines an $\SS$-module $M = \{M_i\}_{i\in\N}$ with an extra structure that makes $M_i$ into an $\SS_i$ equivarient monoid in $\C$.
Further, a strong monoidal functor $\monplus{\SS} \to \C$ adds a new family of morphisms $\mu_{ij}: M_i \otimes M_j \to M_{i+j}$ making $(M_i, \mu_{ij})$ into a twisted commutative algebra where ``twisted commutativity'' is a consequence of Condition~\ref{commu} of Definition~\ref{def:mnc}.
The terminology of a ``twisted algebra'' can be traced back to \cite{barratt-twisted}.
See \cite{ginzburg-schedler-diffop-bv} for an application of twisted commutative algebras in the context of noncommutative algebra or the expository article \cite{sam-snowden-tca}.

By localizing, we get $id_n \cong (id_1)^{\bt n}$ so that $id_1$ is the only basic object $\locmonplus{\SS}$ and the generating morphisms $\g: id_1 \bt id_1 \to id_1$.
Carrying out a similar analysis as \cite{feynmanrep}, we see that $\locmonplus{\SS}$ is equivalent to the category of ordered finite sets and order preserving maps.
\end{ex}

\begin{ex}
\label{ex:Doperad} For $\locmonplus{\Surj}$, the strict monoidal functors are determined by data $D(\pi_n)=:\D(n)$ which have an action by $(\s,\t)$ that is a right action by bijections and a left action by bijections of one--element sets. Restricting to the skeleton, the maps $\g^\D$ define
$\g^\D_{f,\pi_m}:\D(n)\ot \D(|f^{-1}(1)|)\odo \D(|f^{-1}(m)|)\to \D(m)$,
for a map $f:n\to m$.
These operations need to be associative and equivariant. Adding units
    there is an element $u\in \D(1)$ which is a left unit under composition and for which $u^{\ot n}$ is a right unit.
    This is a definition of operads in terms of generators and relations, see \cite{feynmanrep} for details.
 \end{ex}

\begin{rmk}[Enrichment functors]

 In the standard gcp case $\B=\Gpd=Iso(\C)$, indexing data can be viewed as defining a double functor, called an enrichment functor \cite[Appendix]{feynmanrep}.
The source is the  double category $\underline{\C}$ (actually the double groupoid) whose groupoid of objects is $\gpd$, and whose groupoid of morphisms is $P(\gpd\da\gpd)$ as in Remark~\ref{rmk:double}.
The target $\underline{\E}$ is the usual  double category for a monoidal category $\E$:  one object $\Mor_h(\D(\E))=Obj(\E)$ with $\circ_h=\ot$ as composition $\Mor_v(\D(\E))$ is trivial and $2 \mdash Mor(\D(\E))=Mor(\E)$.
An {\em enrichment functor} is then a horizontally lax and vertically strict functor of double categories $\underline{\C}\to \underline{\E}$.

 {\em Mutatis mutandis} the above isomorphisms extend to include enrichment functors in the case $\B$ is a groupoid.
We refer to {\it loc.~cit.\/} for the link of units and counits to holonomy and connections.

    \end{rmk}

\subsubsection{Corepresentation}

The corepresenting objects are most naturally thought of as bimodules.
These can be thought of as categories in the unital case and as a semi-categories in the non-unital case, see \S\ref{par:bimodmon}.

\begin{df}
A bimodule  $\rho$ is {\em indexed} over a pointed category $(\C, P)$ if it is in the slice category of $P^*\rho_\C$, that is it comes equipped with a natural transformation to $P^*\rho_\C$. To connect to the indexing data, note that in the Cartesian case, we have the pull--backs $\D(\phi)$ defined as
\begin{equation}
\label{eq:Dphidef}
\begin{tikzcd}[column sep = small]
\D(\phi):=\{X,Y,\phi\}\bisub{i}{\times}{b}
\rho(X,Y)\ar[r]\ar[d]&\rho(X,Y)\ar[d,"b"]\\
\{X,Y,\phi\}\ar[r,"i"]&P^*\rho_\C(X,Y)
\end{tikzcd}
\end{equation}
and in this way $\rho$ splits as $\rho(X,Y)=\coprod_{\phi\in P^*\rho(X,Y)} \D(\phi)$ using the shorthand $\phi$ for $(X,Y,\phi:P(X)\to P(Y))$.
% \begin{equation}
% \label{eq:homindexed}
% \begin{aligned}
%   % \label{eq:rhodef}
%   \Hom_{\C(\rho)}(X,Y)
%   &= \rho(X,Y) \\
%   &= \coprod_{\phi\in \Hom_\C(X,Y)}\D(\phi)
% \end{aligned}
% \end{equation}
In the general enriched setting, our general assumptions are that the coproducts exist in $\E$, the natural transformation is to  $P^*\rho_C\odot \E$, that is
\begin{equation}
\label{eq:phisplit}
 \rho(X,Y)=\bigoplus_{\phi\in P^*\rho(X,Y)} \D(\phi), \quad P^*\rho_\C\odot \E(X,Y)=\bigoplus_{\phi\in
P^*\rho_\C(X,Y)}\unit_\E \text{ and }
b = \bigoplus_{\phi\in P^*(\rho_\C)} \eps_\phi
\end{equation}
with $\eps(\phi):\D(\phi)\to \unit_\E $.
A unital bimodule $\rho$ indexed over $P^*\rho_\C$
is called {\em reduced} if $\D(P(\s))\simeq \unit$ for every $P(\s)\in P^*\rho_\C$.

A (symmetric (monoidal)) bimodule monoid indexed over $P^*\rho$ is
an indexed bimodule whose structure maps commute with
the base map:
$\g_{\rho_{\act_\ot}}(b\sq_{\act_\ot} b)=b\g_{\rho},
\mu_{\rho_{\act_\ot}}(b\sq_{\act_\ot} b)=b\mu_{\rho_{\act_\ot}}$
and similarly for units/counits.

\end{df}
%
%\begin{lem}
%There is a functor $b:\C(\rho)\to \C$ given by identity on objects and $\D(\phi)\to \{\phi\}$.
%\end{lem}
%By adjunction \cite{Kelly}, The functor can be viewed as from the underlying category of $\C_\D$ or to the freely enriched category $\C\odot\E$.
\noindent {\sc NB:}
In the Cartesian case, the $\eps(\phi)$ are not extra data, in the sense that they uniquely exist, and every bimodule is in the slice category of $\final$. In the non--Cartesian case, the natural transformations to $\final$ are given by morphisms $\eps:\rho(X,Y)\to \unit$; this is exactly the counital data.

\begin{prop}
\label{prop:data-equals-bimod}
Let $\C$ be a ((symmetric) monoidal) category with a pointing $P:\act\to \C$ and let $\E$ be a symmetric monoidal enrichment category.
There are isomorphisms between the categories of
\begin{enumerate}
    \item ((Symmetric) (lax) monoidal), indexing data, which is also counital in the non--Cartesian case.
    \item ((Symmetric) (lax) monoidal) bimodule monoids indexed over $P^*\rho$.
\end{enumerate}
Unital  indexing data corresponds to unital bimodules.
These are (groupoid) compatible, and/or reduced, when the data is.
In the non--Cartesian case, without a counit the decomposition \eqref{eq:phisplit} holds, but there is no natural transformation to $P^*\rho_\C$.
\end{prop}
We will write $\rho_\D$ for the bimodule defined by the data $\D$.
\begin{proof}
The equation \eqref{eq:phisplit}
establishes the bijection between the indexed bimodule and the data (1) of the functor $\D$.
In the non--Cartesian case, the data needs to be counital for $b$ to exist.
The data (2) is then equivalent to a monoid structure.
The data (3) is given by the monoidal monoid structure.
By inspection, unital data is equivalent to giving a natural transformation $\rho_\act\to \rho$.
That the conditions transform into each other is then straightforward.
\end{proof}

%\subsubsection{Corepresentation in the unital/counital case: indexed enrichments.}
In the unital case (and the unital and counital case if $\E$ is non--Cartesian), the plus constructions  corepresent indexed enriched categories.
\label{par:indexing}
\begin{df} A  Cartesian-enriched category $\hat\C$ is said to be \emph{indexed enriched} over a category $\C$ if there is a functor $b:\hat\C\to \C$,
which is identity on objects. For (symmetric) monoidal categories, the functor needs to be (lax) (symmetric) monoidal. The indexing is lax or strong if the functor is.
A \emph{section} over a pointing $P:\act\to \C$ is a lift of $P$ against $b$, i.e.\ a functor $\hat P: \act \to \hat \C$ such that $P=b\hat P$.
If $P$ is a compatible pointing, then a section is compatible if it is also a compatible pointing.

If $\E$ is non--Cartesian and $\act$ and $\C$ are not yet enriched, we postulate the following additional  restrictions for indexed enrichments:  $\Hom_{\hat \C}$ splits as \eqref{eq:indexcoprod}
 and the functor $b:\hat \C\to \C\odot \E$ splits accordingly as $b=\bigoplus_\C \eps_\phi$ where $\eps_\phi$ maps to the component of $\phi$ in $\Hom_{\C\odot \E}(X,Y)=\bigoplus_{\phi \in \Hom_{\hat \C}(X,Y)}\unit_\E$.
After enriching, the section is given by $\hat P:\act\odot \E\to \hat \C$ such that $b\hat P=P\odot\E$ where $P\odot\E$: $\act\odot \E\to\C\odot \E$ induced by the pointing $P:\act\to \C$.

A compatible section is reduced if on morphisms $b$ splits as $b_p\oplus \bar b: \
\hat P(Hom_\act) \odot \E \oplus  \overline{ Hom}_{\hat \C}\to  (P(Hom_\act)\odot \E) \oplus (\overline{Hom}_{\C}\odot \E)$ with $b_p$ induced by an isomorphism $\hat P(Hom_\act) \simeq P(Hom_\act)$.
\end{df}

\noindent {\sc NB:} Although both a pointing and an indexed enrichment are bijective on objects, a pointing needs to be strict, while and indexed enrichment may be lax or strong. In particular, if $P:\B\to \C$ is a pointing, then $\B$ is also index enriched over $\C$, but not vice--versa.

\begin{dfprop}
\label{dfprop:index-equals-data} Fix the enrichment category $\E$ and
assume that coproducts exist in $\E$.
A set of unital indexing data $\D$ for a pointing $P:\act\to \C$ defines a category $\C_\D$
 with the same objects as $\C$
% $\Obj(\C_\D):=\Obj(\C)$
and morphisms
\begin{equation}
\label{eq:indexcoprod}
\Hom_{\C_\D}(X,Y)=\coprod_{\phi\in \Hom_\C(X,Y)}\D(\phi)
\end{equation}
%with the $\D(\phi)$ objects of an enrichment category $\E$,
where
the source and target maps are $s(\D(\phi))=s(\phi)$,
$t(\D(\phi))=t(\phi)$. The unit morphisms are given by the elements $u_{id_X}$. The composition is given by $\circ=\g^\D$ and in the monoidal case the monoidal product for morphisms is given by $\mu^\D$. This is lax/strong if the data is.
The unitors, and in the symmetric case, commutators  are given by the monoidal indexing data.

The unital data also defines a section of $P$, by $\hat P(\s)=u_{\s}$. If $P$ is compatible, so is $\hat P$. If $P$ is groupoid compatible and
the data is groupoid compatible, then so is $\hat P$.

If the data is counital, then there is a functor $b=\amalg \eps_\phi:\C_\D\to \C$ and $\C_\D$ is indexed over $\C$. For the pointings, $\eps\hat P(\phi)=P(\phi)$ by definition.
\end{dfprop}

\begin{proof}
What needs to be verified is that $\C_\D$
is indeed a category. This is routine since source, target, identities and composition, respectively tensor product are explicitly given. The functor $\hat P$ is surjective on objects and clearly faithful. In the case of a groupoid compatible pointing, $u_\s$ where $\s$ is an isomorphism is still invertible with inverse $u_{\s^{-1}}$. As the identities lie in $\D(id_X)$, any invertible element must lie in some $\D(\s)$ with $\s$ invertible. By definition, there are no further invertible elements outside of $u_\s$.
\end{proof}

\noindent{\sc NB:}
In the Cartesian case, every category $\C(\rho)$ is naturally index enriched over $\C(\final)$.
In the non--Cartesian case, such an indexing is given by a functor $\eps:\Hom_\C(X,Y)=\rho(X,Y)\to \unit$.

\begin{prop}
\label{prop:unitaldatacat}
%\label{thm:unitaldatacat}
Let $\C$ be a ((symmetric) monoidal) category, pointed by $P:\act\to \C$ and $\E$ be a symmetric monoidal enrichment category.
There are isomorphisms between the categories of
\begin{enumerate}
    \item Unital (counital) ((symmetric) (lax) monoidal), indexing data.
      \item Unital (counital) ((symmetric) (lax) monoidal) bimodule monoids.
      \end{enumerate}
      These are also  equivalent to
      \begin{enumerate}
    \setcounter{enumi}{2}
    \item $\E$ enriched ((symmetric) monoidal)
    categories index enriched over $\C$ with a section over the pointing, where in the Cartesian case, only unital is required and in the non--Cartesian case unital and counital is required.
\end{enumerate}
The section is compatible,  groupoid compatible or reduced on one side if the data is or equivalently the indexed bimodule is.

Dropping the counit in the non--Cartesian case, there is no canonical functor to $\C$.  The equivalence is then to categories where equation \eqref{eq:indexcoprod} respectively bimodules where \eqref{eq:phisplit}  holds.

Dropping the unit, the equivalence is to semi--categories as in Proposition \ref{prop:bimodcat}.
\end{prop}
% {\sc NB: note that a natural transformation to $\T$ always exists uniquely in the Cartesian case.}

\begin{proof}
The equivalence of (1) and (2) immediately follows from Proposition~\ref{prop:data-equals-bimod} with the added data of units.
To go from (1) to (3) is the content of the Definition-Proposition above. One can also use  Proposition~\eqref{prop:bimodcat} to get the equivalence from (2) to (3).
To get the correspondences of conditions is a straightforward tracing through the data.
To go from (3) to (1) directly: reading \eqref{eq:indexcoprod} from right to left defines the data of $\D$, where the morphisms $\g$ come from composition in $\hat C$ and the morphisms $\mu$ from the monoidal product if present. The counit is part of the data and the unit is given by the section of the pointing, that is the equation $u_\phi=\hat P(\phi)$.
\end{proof}

\begin{rmk}
The construction and the theorem generalize to the case that $\C$ is already enriched over $\E$ and $\hat \E$ is an enrichment category tensored over $\E$.
The construction goes through {\em mutatis mutandis} in this case as well using indexed colimits, cf. \cite{kellybook}.
\end{rmk}

Combining  Propositions \ref{prop:datafun}, \ref{prop:data-equals-bimod}
 and \ref{prop:unitaldatacat} we obtain an omnibus theorem, where parentheses indicate possible restrictions that match.

\begin{thm}
\label{thm:1}
Let $(\C,P)$ be a ((symmetric) monoidal) category, with base category $\B$, and let $\E$ be  a symmetric monoidal category. In the Cartesian case,
there are isomorphisms between the categories of
\begin{enumerate}
    \item (Unital) ((symmetric) (lax) monoidal) indexing data.
 \end{enumerate}
\begin{enumerate}
    \item [(2)] (Unital)   ((symmetric)  (lax) monoidal) $\B$--bimodule monoids in $\E$  over $\rho_P$.
\end{enumerate}
%\RK{categories are unital}
%With the same restriction, the above are equivalent to
\begin{enumerate}
    \item[(3)]  $\E$ enriched ((symmetric) monoidal) semi--categories, categories in the unital case, which are (lax) index enriched over $\C$, with a section over the pointing in the unital case.
\end{enumerate}
The conditions of being (groupoid) compatible, and/or reduced match.

The above are equivalent to strict monoidal functor categories corepresented by $\anyplus$ and $\anyplusp$ in the unital case, where $\anyplus$ is $\catplus{\C,P}$  in the non--monoidal case, $\monplus{\C,P}$ in the lax (symmetric) monoidal case and $\locmonplus{\C,P}$ or equivalently in the strong (symmetric) monoidal case. Here $\redmonplus{\C,P}$ corepresents via strong monoidal functors.

The statements hold for non--Cartesian $\E$ if one adds the condition of being counital.
This means that the corepresentation is by $\anyplusg$ or by $\anyplusgp$ in the  unital case.
 \qed
\end{thm}
% MM: This was commented out
% \begin{enumerate}
% \item[(4)] If $\B$ is a groupoid then
% these are equivalent to
% (Unital) (counital) ((symmetric) monoidal)  enrichment functors to $\underline{\E}$.
% \end{enumerate}

\noindent Writing out the corepresentation explicitly for the unital case:
\begin{enumerate}
\item The category of strict monoidal functors $[\catplusp{\C, P}, \E]_{strict\mdash\ot}$ is equivalent to the category of  unital $\act$ bimodule monoids over $\rho_P$ in the Cartesian case, as well as to the category of categories index enriched over $\C$ with a section over the pointing. In the non--Cartesian case, the equivalence is with
$[\catplusgp{C,P},\E]_{strict\mdash\ot}$.
\item The category of
 strict monoidal functors $[\monplusp{\M, P}, \E]_{strict\mdash\ot}$ is equivalent to the category of unital (symmetric) lax--monoidal  $\act_\ot$--bimodules  monoids over $\rho_P$, as well as to the category of (symmetric) monoidal categories with lax indexing over $\C$ in the Cartesian case. In the non--Cartesian case, the equivalence is with $[\catplusgp{\M,P},\E]_{strict\mdash\ot}$.
\item The categories of
 strict  monoidal functors $[\locmonplusp{\M, P}, \E]_{strict\mdash\ot}$ and strong monoidal functors $[\redmonplusp{\M, P}, \E]_{\ot}$ are both equivalent to the category of unital strong--monoidal  $\act_{\otimes}$--bimodules monoids over $\rho_P$  as well as to the category of (symmetric) monoidal categories with strong indexing over $\C$ in the Cartesian case. In the non--Cartesian case, the equivalence is with the unital and counital versions.
% \item The category of
%  strong monoidal functors $[\plusgcp{\act_\ot}, \E]_{\ot}$ is equivalent to the category of unital strong--monoidal  $\Iso(\act_\ot)$--bimodules  monoids over $\Hom_\act_\ot$.
\end{enumerate}

%\RK{What does this mean for algebras?}
% !TEX root =  KMoManinv3.tex

\section{Algebras and special monoidal categories}
\label{par:defs}

\subsection{Algebras and Hereditary condition}
One main reason for Baez--Dolan type plus constructions is that they allow one to define algebras over functors, like algebras over operads, and furthermore again see them as functors.
The classical formulation of algebras over operads generalizes to the following:
\begin{df}
    A ((strong) monoidal) (unital) algebra in $\E$ over ((strong) monoidal) (unital) indexing data is given by a
    functor $\a:\act \to \hat \E$ together with a collection $m$ of morphisms $m_\phi:\D(\phi)\ot \a(s(\phi))\to \a(t(\phi))$ which are
    \begin{itemize}
    \item[(a)] actions: $m[id\ot m]=m[\g^D\ot id]$.
    \item[(b)] equivariant: $m[\D(\s,id)\ot id] =m[id\ot \a(\s)]$.
    \item[(c)] (strong) monoidal in the (strong) monoidal case:
    $m[\mu^\D\ot \mu^\a]=\mu^\a (a\ot a)\t^{23}$, where $\mu^\A_{X,Y}:\a(X)\ot \a(Y)\to \a(X\ot Y)$
    is the lax/strong structure.
    \end{itemize}

\end{df}

For monoids, the natural definition of modules is:

\begin{df} Let $\hat \E$ be tensored over $\E$ and let $\rho$ be a  $\act$--bimodule monoid in $\E$. Then an $\rho$ module $\alpha$ in $\hat \E$ is a functor $\alpha:\act\to \hat \E$ together with natural transformations
   $m_\a:\rho \pl_{\act} \alpha := \int^{X \in \act}\rho(X,-) \otimes \alpha(X) \to \alpha$
which satisfies the module property $m_\a[\g_\rho \pl_\act id_\a]=m_\a[id_\rho\pl_\act m_\a]: \rho \pl_\act \rho \pl_\act \a\to \rho $.

Moreover, if $\rho$ is a unital bimodule monoid, then $\a$ is a unital module if $m_\a[\eta \pl_\act id]=id_\a$, i.e.\ $a(u_\s)=\a(\s)$.

A (symmetric) strong/lax  monoidal module over a monoidal $\act_\ot$ (symmetric) bimodule monoid $\rho$ with a monoidal $\B$ is a strong/lax (symmetric) monoidal functor $\alpha:\act_\ot\to \hat \E$ with an associative (symmetric) monoidal natural transformation $m_\a:\rho\pl_{\act_\ot}\alpha\to \alpha$.

The natural transformations that intertwine the action imbue $\rho$--algebras of all flavors with a categorical structure.

\end{df}

% Note that given $\a,m_\a$,  using \eqref{eq:tensoradjuction}, the data is equivalent to an equivariant action as follows: for every $\phi \in \rho(X,Y)$ there is a morphism $m(\phi):\a(X) \to \a(Y)$ of $\B$--modules.
\begin{prop}
\label{prop:data=algebra}
    There is an equivalence of categories between
    (1) Algebras over indexing data $\D$, and
    (2) Modules over $\rho_\D$.
    This is true for all flavors: regular, (symmetric) strong/lax monoidal.

\end{prop}
\begin{proof}
    This is a straightforward unraveling of the definitions.
\end{proof}

    \begin{ex} Continuing Examples \ref{ex:operad} and \ref{ex:Doperad},
    a strong algebra over an operad thought of as indexing data $\D$ is given by a strong monoidal functor $\a:\Surj\to \hat\E$.
    Up to isomorphism, this is fixed by $\A(1)=A$ as $\A(S)\simeq \A(|S|)\simeq \A(1)^{\ot |S|}=A^{\ot |S|}$.
    The additional data up to isomorphisms is given by (a)---(c) as an action $\D(n)\ot_{\SS_n}A^{\ot n}\to A$. If $\hat \E$ satisfies the adjunction \eqref{eq:tensoradjuction}, then this is a map $\D(n)\to Hom(A^{\ot n},A)$. This is the traditional way of defining an algebra over an operad.
    If the operad is unital, then there is an element $u\in \D(1)$, such that its action is $id_A$.

        Using the category $\Surj_\D$, a strong algebra over a unital operad $\D$ becomes a functor from $\Surj_\D$. In the nonunital case, $\alpha$ is a strong monoidal functor from $\SS$,
         which is given by $\a(1)=A\in \hat \E$ and an $\SS_n$ equivariant action by $\rho_\D$. Setting $D(n)=\rho_\D(n,1)$ this is again defined by equivariant maps $D(n)\ot_{\SS_n} A^{\ot n}\to A$
    \end{ex}

\begin{prop}
\label{prop:modulefunctor}
If $\rho$ is unital,
the category of unital $\rho$ modules  is equivalent to the category of functors $\O\in [\C(\rho), \hat\E]$. Similarly, this is true for (strong)(symmetric) monoidal algebras and
the selected type of functors.
%from $P^*\rho_\C\to \hat\E$.
    % The category of strong/lax monoidal modules over a (symmetric) monoidal monoid is equivalent to the lax/strong (symmetric) monoidal functors from  $\act_\ot(\rho)$.
\end{prop}
\begin{proof}
  This is an unwinding of definitions. Let the module $\a$ be given.
  Define the functor $\O:\C(\rho)\to \E$ on objects by $\O(X)=\a(X)$ using the fact that $\act$ and $\C(\rho)$ have the same objects.
  A morphism $\phi \in \C(\rho)(X,Y)$ is by definition an element of $\rho(X,Y)$.
  Setting $\O(\phi):=m( \phi):\O(X)\to \O(Y)$, as defined by \eqref{eq:tensoradjuction}, defines the functor on morphisms.
The functionality $\O(\phi\psi)=m(\phi\psi)= m(\g(\phi\bt\psi))=m(\phi)a(\psi)$ follows from  the module property, and unitality implies that $\O(id_X)=id_{\O(X)}$.

Vice--versa given a functor $\O: \C(\rho)\to \hat \E$,  let $P:\act\to \C(\rho)$ be the corresponding pointing from Proposition \ref{prop:bimodcat}. Then set $\a=P^*\O$ and define the natural transformation $\rho \pl_\act \a\to \a$ using the morphism $\O(\phi):\O(s(\phi))\to \O(t(\phi))$ and the adjunction
\eqref{eq:tensoradjuction}. This will be unital.

The remaining statements are analogous using the fact that $\a$ is strong/lax monoidal as well as the natural transformation defining the action.
\end{proof}

% \begin{ex}
% In particular, the category of unital $\rho$-modules in $\hat \E$ is equivalent to the category of functors $\C\to \hat \E$.
% % Indeed, a $\rho_{\C^{disc}}$ module defines a functor $\alpha$ on objects by $\alpha(X)=\rho(X)$. On the morphisms level, using  the adjunction \eqref{eq:tensoradjuction}: $\alpha(\phi)=m(\phi)$. Being unital ensures that $\alpha(id_X)=id_{\alpha(X)}$.
% \end{ex}
For MBMs, the free $\rho$ algebras given by strong monoidal $\B$--modules need not be strong and this involves an extra condition.
\begin{df}
\label{df:herpointing}
We call a  pointing \emph{hereditary} if the  following morphism of monoidal functors $\act_\ot^{op}\times \act_\ot\times \act_\ot\to \E$ is an isomorphism.
% \begin{equation}
%    ( \rho_{\act_\ot}\otimes \rho_\act_\ot)\pl_{\act_\ot\times \act_\ot}\rho_{\act_\ot}(id\times \mu_{\act_\ot})\to\rho_{\act_\ot}(id\times \mu_\act_\ot)\pl_{\act_\ot\times \act_\ot}\a\mu_{\act_\ot}
% \end{equation}
\begin{multline}
    \label{eq:hereditarypointed}
  (\rho\otimes \rho) \pl_{\act_\ot\times \act_\ot}(\rho_{\act_\ot}(id\times \mu))\stackrel{(1)}{\to}
 (\rho(\mu\ot \mu))\pl_{\act_\ot \times \act_\ot}(\rho_{\act_\ot}(id\times \mu))\\\stackrel{(2)}{\to}\rho(id\times \mu)\sq_{\B_\ot}\rho_{\B_\ot}\stackrel{(3)}{\to}\rho(id\times \mu)
\end{multline}
Here (1) is the morphism given by the monoidal $\mu^F$ structure of $\rho$, (2) is defined by the universal property of the coend, as in Lemma \ref{lem:rhorhomonoidal}, and (3) is the isomorphism coming from the fact that $\rho_{\act_\ot}$ is a unit for $\sq_{\B_\ot}$.
\end{df}

\begin{ex}
\label{ex:homhereditary}
For $\rho=Hom_\C$, \eqref{eq:hereditarypointed}  reads
$ \int^{Y,Y'}Hom_\C(Y,Z)\ot Hom_\C(Y',Z')\ot  Hom_\B(X,Y\ot Y')
\to
 \int^{Y,Y'}Hom_\C(Y\ot Y',Z\ot Z')\ot Hom_\B(X,Y\ot Y')
\to\int^{\hat Y} Hom_\C(\hat Y,Z\ot Z') \ot Hom_\B(X,\hat Y) \simeq Hom_\C(X,Z\ot Z')$. Where the penultimate map is induced by including the elements $Y\ot Y'=\hat Y$ into coend. Hereditary means that this is an isomorphism. Written in this way, this is the condition (ii') of a Feynman category, cf.\ \cite[\S 1.8.5]{feynman}.

    \end{ex}

\begin{thm}
   \label{thm:adjunction}
   For a unital bimodule $\rho$, let $G$ be the forgetful functor $G:\rho$--modules $\to$ $\rho_\act$--modules by letting $G(\a)$ be the underlying $\rho_\B$--module, and define $F:$ $\rho_\act$--modules $\to$ $\rho$--modules by  $F(\alpha):=\rho \pl_\act \a$ with the $\rho$--action given by $\mu_\rho\pl_\act id:\rho\pl_\B\rho\pl_\B\a\to \rho\pl_\B\a$.
In the (symmetric) monoidal case, if $\a$ is lax, then $F(\a)$ is lax. If $\alpha$ is strong and $\rho$ is hereditary, then $F(\a)$ is strong. Moreover, if $F(\a)$
is strong for every $\a$, then $\rho$ is hereditary.
Also,
$F\dashv G$, i.e.\ $F$ is a left adjoint for $G$, in the following situations:
\begin{enumerate}
    \item $F$:~$\rho_{\act}$--modules in $\hat \E$ $\leftrightarrows$ $\rho$--modules in $\hat\E$:~$G$.
    \item $F$:~(symmetric) lax monoidal
$\rho_{\B_\ot}$--modules in $\hat\E$ $\leftrightarrows$ (symmetric) lax monoidal $\rho_{\act_\ot}$--modules in $\hat\E$:~$G$.
\item If $\rho$ is hereditary, then
$F$:~(symmetric) strong monoidal
$\rho_{\B_\ot}$--modules $ \leftrightarrows$ (symmetric) strong  monoidal $\rho_{\act_\ot}$--algebras:~$G$.
\end{enumerate}

Or, equivalently, using Proposition~\ref{prop:modulefunctor}, adjunctions for the categories of functors:
\begin{itemize}
    \item [(1')] $F$:~$[\act,\hat\E]\leftrightarrows [\C(\rho),\hat\E]$:~$G$.

    \item[(2')]  $F$:~$[B_\ot,\hat \E]_{lax-\ot}  \leftrightarrows [\C(\rho_{\act_\ot}),\hat\E]_{lax-\ot}$:~$G$.
    \item  [(3')]
      $F$:~$[\B_\ot,\hat \E]_\ot \leftrightarrows [\C(\rho),\hat \E]_\ot$:~$G$, in the case that  $\rho$ is hereditary.
\end{itemize}
\end{thm}

\begin{proof}
Thinking of $F(\a)$ as a functor $\O(F(\a))$ by using Proposition \ref{prop:modulefunctor} from $\C(\rho): \O(F(\a))(X)=\int^{Y\in \B}\rho(Y,X)\ot \a(Y)= Lan_P\a$, that is the pointwise left Kan extension defined by the coend, which is the left adjoint of the ``forgetful'' functor $P^*:[\C(\rho),\hat \E]\to [\B,\hat E]$, proving (1').
Using the Proposition in the other direction now proves (1).
By assumption,  $(\rho\sq_{\act_\ot}\a)\ot(  \rho \sq_{\act_\ot} \a)
    \simeq (\rho\ot\rho)
    \sq_{\act_\ot \times \act_\ot}(\a\ot \a)$ and
there is a sequence of natural transformations of functors $\act\times \act\to \hat \E$
\begin{equation}
\label{eq:alpha}
 (\rho\ot\rho)
    \sq_{\act\ot\times \act_\ot}(\a\ot \a)
    \stackrel{id\sq\mu^\a}{\to}\\
(\rho\ot\rho)
    \sq_{\act\ot\times \act_\ot}\a\mu_\act
    \stackrel{\mu_\rho\sq\id}{\to}
   \rho( \mu_\act\ot \mu_\act)\sq_{\act_\ot\times \act_\ot} \a\mu_{\act}
   \to(\rho\sq_{\act_\ot}\a)\mu_\act
\end{equation}
where $\mu_\act$ is the monoidal structure of $\act$, the first transformation is given by commuting the coend with the tensor product,  $\mu^\a$ is the lax structure of $\a$, $\mu^\rho$ is the  lax structure of $\rho$  and
the last map is the defined by the universality of the coend as in Lemma~\ref{lem:rhorhomonoidal}. This defines the lax structure of $F(\a)$. Together with (1) this implies (2) and (2').

If $\rho$ is hereditary, note that, as $\rho_\B$ is the unit for $\sqbt$, tensoring the sequence \eqref{eq:hereditarypointed} with $\sqbt \a$ yields the sequence
\eqref{eq:alpha}.
Vice--versa the fact that  \eqref{eq:alpha} holds for modules
$\alpha_T(-)=\rho_\act(T,-)$ for all $T\in \act$   yields \eqref{eq:hereditarypointed}.
% $[(\rho\otimes \rho) \pl_{\act_\ot\times \act_\ot}(\rho_{\act_\ot}(id\times \mu))]\sq_{\act_\ot}\a=(\rho\otimes \rho) \pl_{\act_\ot\times \act_\ot}[\rho_{\act_\ot}\sq_{\act_\ot}(\a(id\times \mu))]$
\end{proof}

Theorem \ref{thm:1} can be seen as the Monadicity Theorem for Feynman categories \cite[Theorem 1.5.6]{feynman}. It also clarifies the relationship to patterns \cite{Getzler}, which are basically equivalent to hereditary pointings, cf. \cite[1.11.14]{feynman}, and in another language to substitudes \cite{BKW}.

In the following, we will mostly consider the case where $\B$ is a groupoid and use the notation $\sds$ instead of $\scs$.

\begin{prop}
\label{prop:fey=her}
A Feynman category (FC) as in \cite{feynman} is a (symmetric) monoidal category $\F$ together with a groupoid $\V$ and a functor $\imath:\V\to \F$, satisfying the conditions:
\begin{itemize}
    \item[(i)] The functor $\imath$ induces an equivalence $\Iso(\F)\simeq \V^{\bt}$.
    Thus the functor exhibiting the equivalence is weak pointing equivalent to the standard gcp pointing.
    \item [(ii'')] For the standard gcp  $\rho_\F$ is hereditary.
    \item [(iii)] The slice categories of $\F$ are essentially small.
\end{itemize}
\end{prop}
\begin{proof}
    The conditions (i) and (iii) are those appearing in the definition of a Feynman category \cite[Definition 1.1.1]{feynman}. The condition (ii'') is a reformulation of the condition (ii'), which states that $Hom_\F(X,Z\ot Z')=\int^{Y,Y'\in \B\times \B} Hom_\F(Y,Z)\ot Hom_\F(Y',Z') \ot Hom_\V(X,Y\ot Y')$ of \cite{feynman}, which in light of the Example \ref{ex:homhereditary} is the stated hereditary condition.
The equivalence is given by \cite[Proposition 1.8.9]{feynman} which states, that conditions (i),(ii),(iii) defining a Feynman category are equivalent to conditions (i),(ii'),(iii), where the original condition (ii) states that $Iso(\F\da \F)\simeq Iso(\F\da \V)^\bt$.
\end{proof}
The non--symmetric FC are  alternatively called non-sigma FCs.

%\subsection{Factorizing objects and morphisms (up to isomorphisms)}

%\subsubsection{Reduced version of the strong plus construction: $\plus{\M}$}
\begin{ex}[Finite sets and variations]
\label{ex:finset}
 Let  $\FinSet$ be the category of finite sets with the monoidal product of disjoint union $\amalg$.
 Let $\triv$ be the category with one object $*$ and its identity $id_*$.
 We will denote the unique map $T\to \{*\}$ by $\pi_T$.
 Then $\FinSet$ is a Feynman category by $\F=\FinSet$, $\V = \triv$, and $\imath(*)=\{1\}$.
 Note that $\V^\ot$ has the groupoid $\SS$ as its strict version and skeleton.

The  condition on morphisms is guaranteed by the fact that maps of sets have fibers, that is any map $F:S\to T$  in the arrow category is isomorphic to  the map: $\coprod f|_{f^{-1}(i)}:\coprod_{i=1}^{|T|} f^{-1}(i)\to \coprod_{i=1}^{|T|} \{*\}$. Here
the fibers are the decomposition into irreducibles. The objects in $\Indec=\Iso(\M\da\V)$
 are the surjections $\pi_S:S\to \{*\}$
if $S\neq \emptyset$ and the empty map $i_\emptyset:\emptyset\to \{*\}$ if the source is empty.
% A sequence of morphisms $(f_0\kdk f_n)$ is connected if and only if $|t(f_n)|=1$.
We set $\FFinSet=(\triv,\FinSet,\imath)$.

Working with the skeleton $sk(\FinSet)$, this has the ``same'' objects as $\SS$ under the identification $n\leftrightarrow \underline{n}=\{1,\dots,n\}$,
 where the equivalence class of a finite set $T$ is $\underline{|T|}$ in the skeleton. Under the above identification,
 $\SS=\Iso(sk(\FinSet))=sk(\Iso(\FinSet))$.
 In the skeleton, the monoidal structure is then given by addition or $\un\amalg \um=\underline{n+m}$
and $\underline{0}=\emptyset$ is the monoidal unit. This identifies $sk(\triv^\ot)=\SS$ and establishes the equivalence.
Working with the unbiased strict monoidal category, $\triv^{\ot\Set}=\Iso(\FinSet)$.

The restrictions to injections, $FI$, and surjections, $FS$, are also Feynman categories, \cite{feynman,feynmanrep}.
 There is a non--Sigma version given by finite ordered sets, see \cite{feynmanrep} for details, and the respective subcategories $OI$ and $OS$.
     A graphical representation for the surjections, see e.g.\ \cite{matrix} is given by depicting a morphism as a collection of rooted corollas. In particular, the basic morphism $\pi:S\to \{t\}$ is depicted by the corolla with root $t$ and leaves $S$. The composition is given by grafting corollas to a level forests and then contracting the edges, see \cite{woods}.
   The connected components of a pair of composable morphisms, discussed in the next paragraph,  are the  2-level trees of the 2-level forest.
   In the ordered case, these corollas and forests are planar, cf.\cite{feynmanrep}.
 \end{ex}

We adapt \cite[Definition 5.2.5]{feynman} to our situation.
\begin{df}
\label{df:deg}
A {\em degree function} for morphisms of a monoidal category $\M$ is a map $\Mor(\M)\to \N_0$, which is  additive under tensor product and composition, and the degree $0$ and $1$ morphisms generate all morphisms under tensor product and composition. It is {\em proper} if additionally the degree $0$ morphisms are precisely the isomorphisms.

A category with a proper degree function is {\em cubical} if, up to concatenation with base morphisms, each degree $n$ morphism its decompositions $\phi=\phi_1\codco \phi_n$ into degree $1$ morphisms is an $\SS_n$ torsor which is natural under composition.

The latter means that for the upon composition of two decompositions the torsors embed under $\SS_n\times \SS_m\subset \SS_{n+m}$.
Here everything is understood in terms of $R$--modules in the presence of a ground ring.

In terms of bimodules, this can be stated as follows.
Given a proper degree function $\rho\to \N_0$, where $\N_0$ is considered with trivial bimodule action, let $\rho^{(n)}$ be the degree $n$ morphisms, then the cubical condition says that the $(\mu^\rho)^{(n)}: (\rho^{(1)})^{\sqb n} \to \rho^{(n)}$ is a principal
$\SS_n$-bundle which is natural under composition $(\rho^{(1)})^{\sqb n}\sqb (\rho^{(1)})^{\sqb m}\to (\rho^{(1)})^{\sqb n+m}$.
\end{df}

\subsection{UFCs}

The question about defining algebras for functors $\O$
from an FC then turns to the question of if $\F=\M^+$ for some monoidal category $\M$.
The main result is Theorem~\ref{thm:ufcp-equals-fc} which says that this is the case if and only if $\M$ is a \pher{} UFC, defined below. We will work in the standard gcp case and use the language of double categories, cf.\ Remark \ref{rmk:double}.

\label{df:conditions}

\begin{df}
  \label{df:factorizable}
  A (symmetric) monoidal category $(\M,\ot)$ has {\em essentially uniquely factorizable objects}, if there is a groupoid $\V$, s.t.\ $\V^\bt\simeq Iso(\M)$, in other words the
  standard gcp is free monoidal.
It  has {\em essentially uniquely factorizable morphisms} if there is a groupoid $\Indec$ such that $\Indec^\bt \simeq \Iso(\M \da \M)$, viz.\ that the groupoid of the arrow category is free monoidal.
 Note, here we are using the conventions of \S\ref{sec:rmodulecat}.

 % Having factorizable objects means that there is a groupoid $\V$ of basic objects together with a functor $\imath:\V\to \M$, which induces an equivalence.
 % % \begin{equation}
 % %   \label{objectcond}
 % $   \imath^\bt:\V^\bt\stackrel{\sim}{\to}\Iso(\M)$,
 % % \end{equation}
 %   where $\V^\bt$ is
 %   the free (symmetric) monoidal category and $\imath^\bt$ is the functor defined by the universal property.
   A choice of  a pair $(\V,\imath:\V \to Iso(\M))$ witnessing the equivalence will be called a {\em basis of objects} and its elements will be called {\em irreducibles} or {\em basic objects}.
%  Having factorizable objects says
%  there is a groupoid $\Indec$ and
% a functor $\jmath:\Indec\to \Iso(\M\da \M)$ for which $\jmath^\bt$ is an equivalence:
% %  \begin{equation}
%   %\label{eq:decomp}
%  $  \jmath^\bt:\Indec^\bt \simeq \Iso(\M\da \M)$.
% %\end{equation}
A choice of  a pair $(\Indec,\jmath:\Indec \to Iso(\M\da \M))$ witnessing the equivalence will be called a {\em basis of morphisms} and its elements will be called {\em basic morphisms}.
 \end{df}

\begin{rmk}
\label{rmk:factorizable}
\label{item:iso}
 The condition  of factorizable objects guarantees that any object $X$ can be decomposed as $X\simeq \bigotimes_{v \in V} \imath(*_v)$ up to isomorphism with $*_v \in \V$.   Moreover, the only isomorphisms are given by wreath products of isomorphism groups, that is isomorphisms on each $*_v$ and a bijection of indexing sets, inducing permutations of the isomorphic factors.
 In particular, any isomorphism $\sigma:X\to X'$  decomposes as
$
         \label{eq:isodecomp}
         \sigma\simeq \Perm \circ \bigotimes_{i=1}^n \sigma_i=\bigotimes_{i=1}^n \sigma_{\Perm(i)} \circ \Perm
$, for some permutation $\Perm\in \SS_n$,
    where $\sigma_i:*_i\to *'_i$ for chosen decompositions $X\simeq \bigotimes_{i=i}^n \imath(*_i)$ and $X'\simeq \bigotimes_{i=1}^n\imath(*'_i)$.
The factorization condition for morphisms  means that any $\phi \in \M(X,Y)$ can, up to isomorphism $\sds$, be decomposed as a monoidal product. That is there are isomorphisms $\s,\s'$ and a set of morphisms $\phi_v\in \Indec\in V$ such that
$\sds(\phi)=\bigotimes_{v\in V}\jmath(\phi_v)$.
%  \begin{equation}
% \label{eq:decomp}
% \xymatrix{
% X\ar[r]^{\phi}\ar[d]^{\simeq}_{\sigma}&Y\ar[d]_{\simeq}^{\sigma'}\\
% \bigotimes_{v\in V}X_v\ar[r]^{\bigotimes_{v\in V}\phi_v}&\bigotimes_{v\in V}Y_v
% }
% \end{equation}
 This decomposition is essentially unique, in the sense that it is unique up to  isomorphisms ---in the arrow category--- on the factors $\phi_v$ and
permutations of these factors. We will write $\phi\simeq \bigotimes_{v\in V} \phi_v$ as a shorthand.
\end{rmk}

\begin{df}

%\label{par:facdeg}
\label{df:degrees}
If a category has essentially uniquely factorizable objects, then
each object has a degree $|X|$ which is the length of any isomorphic object in $\V^\bt$.
This gives a natural {\em bidegree} or {\em type} of a morphism $\type(\phi)=(|s(\phi)|,|t(\phi)|)$ and a {\em degree}, the {\em length decrease}, $|\phi| = |s(\phi)| - |t(\phi)|$.

If $\M$ has essentially uniquely factorizable morphisms there is a natural degree called {\em depth} for morphisms: $\depth(\phi)$  is the length of a monoidal decomposition in $\Indec^\bt$ isomorphic to $\phi$.
\end{df}
Both type and length decrease are  invariant under isomorphism.
Both are additive under $\ot$ and the length decrease is also additive under $\circ$ as well. Depth is additive under $\ot$.

\begin{df}
\label{df:UFCdef}
A {\em unique factorization category} (UFC) is a symmetric monoidal category
with uniquely factorizable morphisms, whose slice categories are essentially small, together with a choice of basis, viz.\ a triple $(\M,\Indec,\jmath)$.
We will call a  UFC {\em strict} if  $\jmath^\bt$ is the identity.

A morphism of  UFCs is a pair of a functor $e:\Indec\to \Indec'$ and a strong monoidal functor $f:\M \to \M'$ that commute with the functors $\jmath^{\bt}$ and $\jmath^{\prime \bt}$.
\end{df}

\begin{df}
\label{df:ufc}
{\em A compatibility} between a choice of basic objects and a choice of basic morphisms is a choice of functor $\imath_\Indec:\Indec\to (\V^\bt\da\V^\bt)$, such that $\jmath = (\imath^\bt, id, \imath^\bt) \circ \imath_\Indec$.
This means that the following diagram commutes:

\begin{equation}
    \xymatrix{
    \Indec\ar[d]_{\imath_P}\ar[r]\ar@(ur,ul)[rr]^\jmath&\Indec^\bt\ar[r]^{\jmath^\bt}\ar[d]_{\imath_\Indec^\bt}&
    \Iso(\M\da\M)\ar@{=}[d]\\
   \Iso(\V^\bt\da\V^\bt)\ar[r]\ar@(dr,dl)[rr]_{(\imath^\bt,id,\imath^\bt)}&
   \Iso(\V^\bt\da\V^\bt)^\bt\ar[r]_{(\imath^\bt,id,\imath^\bt)^\bt}&\Iso(\M\da\M)
   }
\end{equation}
A compatible choice of basis is a choice of basic objects, basic morphisms and a compatibility.
\end{df}
\begin{nota}
\label{nota:index}
To simplify the notation, we will  often work in $\V^{\otimes Set}$ and $\Indec^{\otimes Set}$, using the formalism of \cite{TannakaDel},
fix a presentation, and choose pseudo--inverses $\tilde \imath$ and $\tilde \jmath$.
Thus $\tilde\imath(X)=\bigotimes_{s\in S}*_s$ for a collection of objects $*_s$ in $\V^S$
and $X\simeq \bigotimes_{s\in S}\imath(*_S)$ in a chosen fashion  $\sds$. As a short hand for these choices, we will simply write $X\simeq \bigotimes_{s\in S}*_s$.
Any isomorphism of objects $\sigma:X=\bigotimes_{s\in S}*_S \stackrel{\sim}{\to} \tilde X=\bigotimes_{t\in T}*_T$ is then given by
the image of a bijection $\bar\sigma:S\leftrightarrow T$ and fixed isomorphisms $\sigma_s:*_s\to *_{\bar \sigma (t)}$. We will simply write $S\leftrightarrow T$ for such a map.
We define the index of a such a tensor product as $\index(\bigotimes_{s\in S}*_s)=S$ and $\index(X)=\index(\tilde\imath(X))$.
\end{nota}

\begin{prop}
\label{prop:facobj}
%\label{objectgpdlem}
If a (symmetric) monoidal category has essentially uniquely factorizable morphisms, then it has essentially uniquely factorizable objects. In particular, it has factorizable pointing and common factorizations of morphisms.

Concretely, if $\Indec \subset Iso(\M\da \M)$ is a maximal set of basic morphisms, which can be obtained
after replacing some choice of basic morphisms with its essential image,
$\V$ can be chosen to be the  full subgroupoid of $\Iso(\M)$ whose objects are those $X$
 which $id_X\in  \Indec$.
The inclusion $\imath_\Indec$ which is given by the source and target is a compatibility.
\end{prop}
\begin{proof}
%
%Let $\V$ be the full subgroupoid defined above. We have to show that it is equivalent to $\Iso(\M)$.
%First, we will show that any objects $X$  factorizes essentially uniquely. By the assumption for any object $X$ of $\M$, $id_X$ factorizes essentially uniquely as $\bigotimes_{v\in V}\jmath(\phi_v)$. Let $X_v=s(\jmath(\phi_v))$,
% then $X$ factorizes as $X\simeq \bigotimes_{v\in V} {X_v}$. For any such factorization  $id_X\simeq \bigotimes_{v\in V} id_{X_v}$ is  unique up to isomorphisms and permutations of the identities $id_{X_v}$, and the $id_{X_v}$ are irreducible. This in turn means that the decomposition into the $X_v$ is essentially unique up to isomorphisms and permutations and that the $X_v$ are irreducible.

Let $\V$ be the full subgroupoid defined above. We have to show that this is equivalent to $\Iso(\M)$.
First, we will show that any object $X$ of $\M$  factorizes essentially uniquely. By the assumption for any object $X$ of $\M$, $id_X$ factorizes essentially uniquely as $\bigotimes_{v\in V}\phi_v$, so that $\depth(id_X)=|V|$. Let $X_v=s(\phi_v)$,
 then $X$ factorizes as $X\simeq \bigotimes_{v\in V} {X_v}$ and $id_X\simeq \bigotimes_{v\in V}id_{X_v}$ is another decomposition of the morphism $id_X$. Hence $\depth(id_{X_v})=1$ and there are irreducible and
   unique up to isomorphisms and permutations of the identities.
 This in turn means that the decomposition into the $X_v$ is essentially unique up to isomorphisms and permutations and that the $X_v$ are irreducible. Indeed, if $X_v\simeq X'_v\ot X''_v$ then $id_{X_v}\simeq id_{X'_v}\ot id_{X''_v}$ and one of the $X'',X'=\unit$ as otherwise,
 the identity morphisms of the other would not be in $\Indec^\ot$.

Second, we have to show that $\V^\ot\simeq Iso(\M)$ also on the level of morphisms.
Consider an isomorphism $\sigma:X\stackrel{\sim}{\to} Y$.
Let $\id_Y\simeq \bigotimes_{v\in V}id_{Y_v}$ be a decomposition of $id_Y$ into irreducibles, then
$(\sigma\Da id_Y):\sigma\to \bigotimes_{v\in V}id_{Y_v}$  is an isomorphism in $\Iso(\M\da\M)$, see \eqref{sigmadecompeq}.
\begin{equation}
\label{sigmadecompeq}
  \begin{tikzcd}
    X \arrow[r, "\sigma"] \arrow[d, "\sigma"']
    & Y \arrow[r, "\simeq"] \arrow[d, "id_Y"]
    & \bigotimes_{v \in V} Y_v \arrow[d, "\bigotimes_v id_{Y_v}"] \\
    Y \arrow[r, "id_Y"']
    & Y \arrow[r, "\simeq"']
    & \bigotimes_{v \in V} Y_v
\end{tikzcd}
\end{equation}
Thus any isomorphism in $\Obj(\Indec^\ot)=\Iso(\M\da \M)$ is isomorphic to tensor products of depth one isomorphisms which are between objects of $\V$. Any $\sigma$  thus decomposes essentially uniquely as $\bigotimes_{v\in V} \sigma_v$ with $\sigma_v \in  \V(X_v,Y_v)$.
In fact, essentially uniquely as identities of irreducible objects.
That the inclusion is a compatibility is clear.
\end{proof}

% Note that by the proof of the proposition above, there is always a compatible choice of basis and, up to equivalence, we can always assume that the choices are compatible. It is, however, useful to have presentations given by choices. ``Good'' (vs.\ evil) categorical notions, that is those up to equivalence are of course independent of such a choice, but concrete applications and constructions warrant them.

%\begin{rmk}

% \begin{cor}
% For any $\sigma\in \Indec$ $\jmath(\sigma):X\stackrel{\sim}{\to}Y$ is an isomorphism of irreducible objects.
% \end{cor}

 %\subsubsection{Feynman categories}

% \subsection{Presentations and Feynman categories}
%
\begin{df}
    A presentation of a  UFC is a category with essentially factorizable morphisms together with  a choice of basic objects, and a compatibility.
This is thus a tuple
$\mathfrak{M} = (\V,\Indec,\M,\imath,\jmath,\imath_\Indec)$.
As a convention, we use the fraktur font to refer to the presentation and use the calligraphic font for the categories.
 We call a presentation strict, if the equivalences $\imath^\ot,\jmath^\ot$ are identities.

A morphism between  UFCs with a presentation $\ff:(\M,\V,\Indec,\imath,\jmath)\to(\M',\V',\Indec',\imath',\jmath')$
is a triple of functors $\ff=(v,p,f)$,
% with $f$ (symmetric) monoidal,
which fit into the following commutative diagrams.
\begin{equation}
  \xymatrix{
    \V \ar[r]^\imath\ar[d]_v&\M\ar[d]^f\\
    \V'\ar[r]^{\imath'}&\M'
    }
    \quad
    \xymatrix{
    \Indec\ar[r]^-\jmath\ar[d]_p&\Iso(\M\da\M)\ar[d]^{(f,f,f)}\\
  \Indec'\ar[r]^-{\jmath'}&\Iso(\M'\da\M')\\
    }
    \quad
    \xymatrix{
    \Indec\ar[d]_{\imath_\Indec}\ar[r]^p&\Indec'\ar[d]_{\imath_{\Indec'}}\\
    \Iso(\V^\ot\da\V^\ot)\ar[r]^{(v^\ot,f,v^\ot)}&\Iso((\V')^\ot\da(\V')^\ot)
    }
\end{equation}
% In the case of strict \pher{} UFCs, it suffices to specify $f$ and for Feynman categories it suffices to specify $(v,f)$, cf.\ \cite[\S1]{feynman}.
% A morphism of presentations of  UFCs is a (strong/strict) {\em indexing} if $f$ is a (strong/strict) indexing.
\end{df}
%This generalizes the notions of  \cite[Definition 2.1.1] {feynmanrep}.

\begin{lem}
%\begin{df}
%\label{df:fey}
\label{cor:fey}
\label{lem:fcn1}Feynman categories are equivalent to  UFCs with a presentation whose basic morphisms are of the type $(n,1)$.\qed
\end{lem}

  \begin{proof}
%  \label{FCidentificationlem}
Given a Feynman category $(\V,\imath,\F)$ one obtains  a UFC all of whose basic morphisms are type $(n,1)$ by setting  $\Indec=\Iso(\V^\ot \da \V)$,  and defining $\imath_P$ by the canonical map $\Iso(\V^\ot \da \V)^\ot\to \Iso(\V^\ot \da \V^\ot)$ sending a formal product of morphisms to the monoidal product of morphisms. Vice--versa given a presentation UFC with morphisms of type $(n,1)$ forgetting the data $\Indec$, yields the FC.
  \end{proof}

%\subsubsection{Hereditary condition and roof calculus}
%The computation of a localization is greatly simplified in the presence of a roof calculus.
%\begin{df}

%
% {\sc \noindent {\sc NB:}} In particular this means that
%\begin{equation}
%\label{eq:backcrossed}
% \sdsphi\circ\mu_{\phi_1,\phi_2}=\mu_{\tilde{\phi_1},\tilde{\phi_2}}\circ
%[\sdsp{\s_1}{\s'_1}{\phi_1}\bt\sdsp{\s_2}{\s'_2}{\phi_2}]
% \end{equation}

\begin{rmk}
\label{rmk:UFCbimod}
    There is a definition of factorizability of morphisms in  terms of bimodules used throughout \cite{plethysm}. Consider the functors $m_{\act^{op}\times \act}:(\act^{op}\times \act)^\bt\to \aca$ and
    $m_\E:\E^\bt\to \E$. For a $\B$--bimodule $\nu$, define
    $\nu^\ot=Lan_{m_{\aca}}m_\E\nu^\bt$.
    A $\act$--bimodule is {\em factorizable} if there exists a $\act$--bimodule $\nu$, such that $\rho=\nu^\ot$. The category
     $\C(\rho)$ has factorizable morphisms if $\rho$ does.
     The case at hand, is factorizability for the standard gcp.
   \end{rmk}
\subsection{\Pher{} UFCs}

\begin{prop}
\label{lem:ufcisocm} A category with essentially unique factorization of objects has factorizable ismorphisms. A category with essentially unique factorization of morphisms has a $\ot$--factorization of morphisms.  In particular, any UFC has these properties.
 \end{prop}
\begin{proof}
The first statement is straightforward from  Remark~\ref{rmk:factorizable}. Given two factorizations of a morphisms $\phi=\phi_1\ot\phi_2=\psi_1\ot\psi_2$,
decomposing $\phi$ into irreducibles yields the desired decomposition.
\end{proof}

\begin{lem}
\label{lem:objectpushout}
In a category with essentially uniquely factorizable objects, given any two decompositions of an object $X$ into not necessarily basic objects $X\simeq \bigotimes_{s\in S}X_s$ and $X\simeq \bigotimes_{t\in T}X_t$ for any
decomposition $X$ into basic objects $X\simeq \bigotimes_{v\in V}*_v$ there is a span
$S\stackrel{l}{\leftarrow} V\stackrel{r}{\rightarrow} T$, whose connected components (as a span) yield isomorphisms:
$\bigotimes_{s\in S_c}X_s \simeq \bigotimes_{v\in V_c} *_v\simeq \bigotimes_{t\in T_c}X_t$.

\end{lem}
\begin{proof}
Decomposing as above,  by essentially unique factorization there are maps $V\to S$ and $V\to T$ such that
$X_s\simeq \bigotimes_{v\in l^{-1}(s)}*_v, X_t\simeq \bigotimes_{v\in r^{-1}(t)}*_v$.
Considering the connected components of the span,
we see that the restriction $S_c\leftarrow V_c\rightarrow T_c$ yields the isomorphisms
$\bigotimes_{s\in S_c}X_s \simeq \bigotimes_{v\in V_c} *_v\simeq \bigotimes_{t\in T_c}X_t$.
\end{proof}

\begin{cor}
In a UFC, given two composable morphisms, $X_0\stackrel{\phi_0}{\to} X_1\stackrel{\phi_1}{\to}X_2$,
choosing a presentation and pseudo-inverses
 yields the diagram \eqref{eq:connected}.

\begin{equation}
\label{eq:connected}
    \begin{tikzcd}
    &X_1
    \ar[dl, "\simeq"']
    \ar[d, "\eta"', "\simeq"]
    \ar[dr, "\simeq"] &\\
   \bigotimes_{v\in V} X_{1,v}
   \ar[dr, "\mu'", "\simeq"']
   &\bigotimes_{s\in S} \imath(*_s)
   \ar[r, "\simeq"]
   \ar[d, "\simeq"]
   \ar[l, "\simeq"']
   &\bigotimes_{w\in W} X_{1,w}
   \ar[dl, "\simeq"]\\
   &\bigotimes_{u\in U} X_{1,u}
    \end{tikzcd}
\end{equation}
where $\phi_0:X_0\to X_1$ is decomposed as $\phi_0= \bigotimes_{v\in V}\jmath(\phi_v)$
inducing decompositions of $X_0, X_1$  into  $X_{0,v}= s(\jmath\phi_v)$ and $X_{1,v}=t(\jmath \phi_v)$.
%and then the $X_{0,v},X_{1,v}$ into basic objects according to $\tilde \imath$ and
%then further into objects $*_{0,s},S\in S_{0,v}$ and $*_{1,t} t\in T_{0,v}$.
The middle product is the chosen decomposition of $X_1$, the $X_{1,w}$
and $X_{1,w}$ come from decomposing $\phi_1\simeq \bigotimes_{w\in V}\phi_{1,w}$
and $U$ is the push--out of the span $V\leftarrow S \rightarrow W$ which defines the $X_{1,u}$ via Lemma~\ref{lem:objectpushout}.

\end{cor}

\begin{rmk}
\label{rmk:degeneratespan}
In the special case that $X_1=\unit$, $S=\emptyset$, viz.\ morphisms that factor through $\unit$.
More generally,  since we are decomposing up to isomorphism and suppressed the unit constraints $\unit\ot X\simeq X\simeq X\ot \unit$ in the above decomposition some of the $X_{1,v}$ or the $X_{1,w}$ can be equal to $\unit$. In this case, we will take the span to be $V'\leftarrow S \rightarrow W'\amalg V''\leftarrow \emptyset \rightarrow W''$ where $V',W'$ index those factors that are not $\unit$ and $V'',W''$ index those that are. This is consistent with the category of spans, see \S\ref{par:spans}, and $\bigotimes_\emptyset=\unit$.
\end{rmk}

\begin{lem}
\label{lem:UFCcommonfact}
%\label{prop:conected-equals-irred}
Given a pair of composable morphisms $(\phi_0,\phi_1)$ in a  UFC $(\M,\Indec,\jmath)$ together with decompositions   $\phi_0\simeq \bigotimes_{v\in V}\phi_{0,v}$ and  $\phi_1\simeq\bigotimes_{w\in W} \phi_{1,w}$,
there exists a  partition of $V\amalg W=\amalg_{u\in U} P_u$ indexed by $U$, such that for each $u\in U$ there is a composable pair $(\phi_{0,u},\phi_{1,u})$ with $\phi_{1,u}\circ \phi_{0,u}=\phi_u$ and $\phi_{0,u}\simeq \bigotimes_{v\in P_u}\phi_{0,v}, \phi_{1,u}\simeq\bigotimes_{v\in P_u}\phi_{1,v}$, cf.\ \eqref{eq:connected}. These fit into the diagram in Figure~\ref{Udecompeq} and are essentially unique.
%In particular, a UFC has common factorizations.
\end{lem}

\begin{figure}
  \begin{tikzcd}[row sep = 0.6cm, column sep = 1cm]
    X_0
    \arrow[rr, "\phi_0"']
    \arrow[d, "\simeq"', "\sigma"]
    \arrow[rrrr, "\phi", shift left=3]
    \arrow[dd, "\chi"', bend right=49]
    & & X_1
    \arrow[rr, "\phi_1"']
    \arrow[ld, "\simeq", "\sigma'"']
    \arrow[rd, "\simeq"', "\nu" ]
    & & X_2
    \arrow[d, "\simeq", "\nu"']
    \arrow[dd, "\chi'", bend left=49]
    \\
    {\displaystyle {\bigotimes_{v \in V} X_{0,v}}}
    \arrow[d, "\simeq"', "\mu"]
    \arrow[r, "{\bigotimes_{v \in V} \phi_{0,v}}"']
    & {\displaystyle {\bigotimes_{v \in V} X_{1,v}}}
    \arrow[rd, "\simeq", "\mu'"']
    & & {\displaystyle {\bigotimes_{w \in W} X_{1,w}}}
    \arrow[ld, "\simeq"', "\tau"]
    \arrow[r, "{\bigotimes_{w \in W} \phi_{1,w}}"']
    & {\displaystyle {\bigotimes_{w \in W} X_{2,w}}}
    \arrow[d, "\simeq", "\tau'"']
    \\
    \displaystyle
    \bigotimes_{u \in U} \bigotimes_{v \in P_u} X_v
    \arrow[rr, "{\bigotimes_{u \in U} \phi_{0,u}}"]
    \arrow[rrrr, "\bigotimes_{u \in U} \phi_u"', shift right=8]
    & & {\displaystyle {\bigotimes_{u \in U} X_{1,u}}}
    \arrow[rr, "{\bigotimes_{u \in U} \phi_{1,u}}"]
    & & {\displaystyle {\bigotimes_{u \in U} X_{2,u}}}
  \end{tikzcd}
\caption{
  \label{connecteddiag}
  The diagram for the condition of \pher{} and the definition of connected components
}
  \label{Udecompeq}
\end{figure}
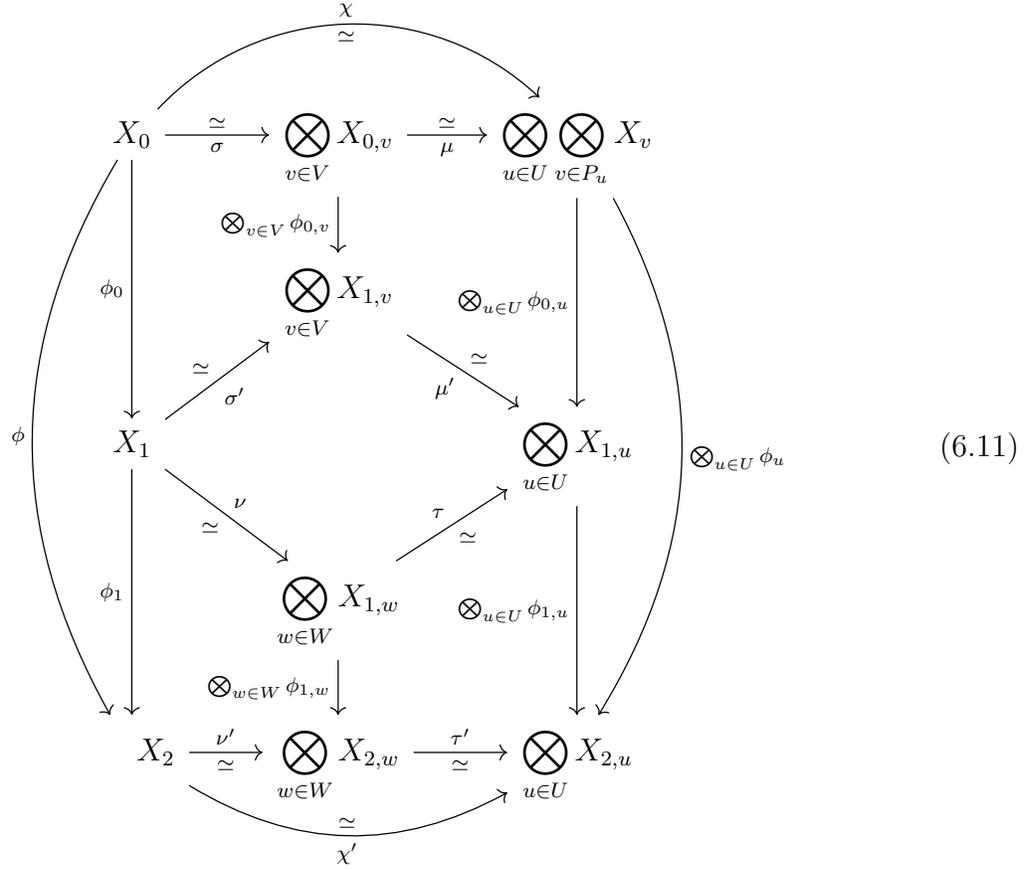

%  The morphisms are captured in the diagram in Figure~\ref{Udecompeq}
% $\phi_1\circ\phi_0=\phi=\bigotimes_{u\in U}\phi_u$
\begin{proof}

To do the computation, fix compatible bases and pseudo--inverse functors $\tilde \imath,\tilde \jmath$, defining $\sds$ for $\phi_0$ and $(\nu\Da\nu')$ for $\phi_1$ in Figure~\ref{Udecompeq}.
This also fixes $\tilde\imath(X_1)\simeq \bigotimes_{s\in S}*_s$ then we obtain a decomposition that defines a ``middle square''  via the push--out $U$ of $V\leftarrow S\rightarrow W$
%$\begin{tikzcd}
%    &S\ar[dl]\ar[dr]&\\
%    V\ar[dr]&&W\ar[dl]\\
%    &U&
%\end{tikzcd}
%$
according to the Corollary above. This
 defines the remaining morphisms in Figure~\ref{Udecompeq}, namely the composable pairs of morphisms $(\phi_{0,u},\phi_{1,u})$ for each $u \in U$ and their compositions $\phi_u=\phi_{1,u}\circ\phi_{0,u}$.
 The construction also shows that the $\phi_u$ are essentially unique given that we are in a UFC.
\end{proof}

The essentially unique morphisms $\phi_u$ above will be called the \emph{connected components} of the decomposition $\phi=\phi_0\phi_1$. Components for which $\phi_{u}$ is in the ground ring will be called trivial.

% By definition in a \pher{} UFC these are irreducible, and vice--versa that the $\phi_u$ are irreducible is the condition for a \pher{} UFC.
% The property of having common factorizations is not automatic for UFCs.
% In the special case of Feynman categories, this property of composition is implicit,  but
% in the more general case of UFCs, it is necessary and we will formulate this property explicitly in this section.

% \begin{cor}
% \label{cor:fullyheriditary}
%   A category with essentially unique  factorizations of objects that is \pher{} is fully \pher{}. \qed
% \end{cor}

\begin{prop}
  A UFC is \pher{} iff all the non--trivial connected components are irreducible.
\end{prop}

\begin{proof}
If the non--trivial connected components are irreducible, then these give the morphisms needed for the preherediatry condition. If there is a connected component $\phi_u=\phi_{0,u}\phi_{1,u}$ that is not irreducible, then it is a monoidal product $\phi_u=\psi_1\ot \psi_2$, but being connected means that $\phi_{0,u}$ and $\phi_{1,u}$ cannot be compatibly decomposed into morphisms $\chi_i$ as required by the preherditary condition, thus it does not hold.
\end{proof}

\begin{prop}
    The localized plus construction $\locmonplus{\M,P}$ for a UFC $\M$ with the standard gcp can be calculated via a roof calculus.
\end{prop}
\begin{proof}
    In light of Lemma \ref{lem:locfac}, Lemma \ref{lem:ufcisocm} and Lemma \ref{lem:UFCcommonfact}, the statement follows from Proposition \ref{prop:roof}.
\end{proof}

\subsubsection{A UFC that is not \pher{}}
\label{par:counterex}
For a finite set $X$, a \emph{tie} is given by a pair  $\mathcal{T}= (X_0, \{X_1,\kdk X_k\})$ in which $X_0\subseteq X$ is a possibly empty subset, and
$\{X_1,\kdk X_k\}$ is a partition of $X\setminus X_0$ into nonempty subsets.
$X_0$ is the subset of {\em untied} elements and each $X_i$ for $i>0$ is a tied subset.
Note that a single element $x$ that can be untied $x\in X_0$, a part of a tied subset, or tied by itself $U_i=\{x\}$ for some $i$.
Let $T(X)$ be the set of ties for $X$.

Consider the category $\Ties$ whose objects are finite sets and whose
morphisms will be {\em isomorphisms} of finite sets together with a choice of tie.
Let $\Hom(X,Y) = \Iso(X,Y) \times T(Y)$.
Note that an isomorphism $\sigma: X \to Y$ canonically determines a bijection of ties $T(X) \to T(Y)$.
Now define the composition of $(\sigma, \mathcal{T}) \in \Iso(X,Y) \times T(Y)$ and  $(\sigma', \mathcal T') \in \Iso(Y,Z) \times T(Z)$ to be $(\sigma' \circ \sigma, \mathcal T' \circ \mathcal T)$ where
$
  \mathcal T' \circ \mathcal T
  = (\sigma(U_0) \cap V_0, \{\sigma(U_i) \cap V_j : (i,j) \neq (0,0) \} \setminus \{\emptyset\})
$.
This  formula says that the resulting partition is the least common refinement of the partition for the tied elements with the untied elements
acting as identity on the ties.
In particular, the identity of $X$ in this category is $(\id_X, (X, \emptyset))$.

\begin{figure}
  %\centering
  \begin{subfigure}{0.16\linewidth}
    \centering
    \begin{tikzpicture}[scale=0.4]
      \node [style=Black, inner sep = 0.5mm] (0) at (-1, 1) {};
      \node [style=Black, inner sep = 0.5mm] (1) at (0, 1) {};
      \node [style=Black, inner sep = 0.5mm] (2) at (1, 1) {};
      \node [style=Black, inner sep = 0.5mm] (3) at (-1, -1) {};
      \node [style=Black, inner sep = 0.5mm] (4) at (0, -1) {};
      \node [style=Black, inner sep = 0.5mm] (5) at (1, -1) {};
      \node [style=none] (6) at (-1, 0) {};
      \node [style=none] (7) at (0, 0) {};
      \node [style=none] (8) at (1, 0) {};
      \node [style=none] (9) at (-1.25, -0.25) {};
      \node [style=none] (10) at (0.25, -0.25) {};
      \draw [in=90, out=-90] (1) to (8.center);
      \draw [in=90, out=-90, looseness=0.75] (2) to (7.center);
      \draw (7.center) to (4);
      \draw (8.center) to (5);
      \draw (0) to (6.center);
      \draw (6.center) to (3);
      \draw [bend right=150, looseness=1.25] (9.center) to (10.center);
    \end{tikzpicture}
    \caption{}
  \end{subfigure}
  \hspace{0.2\linewidth}
  \begin{subfigure}{0.35\linewidth}
    \centering
    \begin{tikzpicture}[scale=0.4]
      \node [style=Black, inner sep = 0.5mm] (0) at (-2, 2) {};
      \node [style=Black, inner sep = 0.5mm] (1) at (-2, 0) {};
      \node [style=Black, inner sep = 0.5mm] (2) at (-1, 2) {};
      \node [style=Black, inner sep = 0.5mm] (3) at (-1, 0) {};
      \node [style=Black, inner sep = 0.5mm] (4) at (0, 2) {};
      \node [style=Black, inner sep = 0.5mm] (5) at (0, 0) {};
      \node [style=none] (6) at (-2.25, 1) {};
      \node [style=none] (7) at (-0.75, 1) {};
      \node [style=none] (8) at (-0.25, 1) {};
      \node [style=none] (9) at (0.25, 1) {};
      \node [style=none] (10) at (-1.25, -1) {};
      \node [style=none] (11) at (0.25, -1) {};
      \node [style=none] (12) at (-2.25, -1) {};
      \node [style=none] (13) at (-1.75, -1) {};
      \node [style=Black, inner sep = 0.5mm] (14) at (-2, -2) {};
      \node [style=Black, inner sep = 0.5mm] (15) at (-1, -2) {};
      \node [style=Black, inner sep = 0.5mm] (16) at (0, -2) {};
      \node [style=none] (17) at (1, 0) {$=$};
      \node [style=none] (18) at (-0.25, -1) {};
      \node [style=none] (19) at (-0.25, -1) {};
      \node [style=Black, inner sep = 0.5mm] (21) at (2, 2) {};
      \node [style=Black, inner sep = 0.5mm] (22) at (2, -2) {};
      \node [style=Black, inner sep = 0.5mm] (23) at (3, 2) {};
      \node [style=Black, inner sep = 0.5mm] (24) at (3, -2) {};
      \node [style=Black, inner sep = 0.5mm] (25) at (4, 2) {};
      \node [style=Black, inner sep = 0.5mm] (26) at (4, -2) {};
      \node [style=none] (29) at (3.75, 0) {};
      \node [style=none] (30) at (4.25, 0) {};
      \node [style=none] (31) at (2.75, 0) {};
      \node [style=none] (32) at (3.25, 0) {};
      \node [style=none] (33) at (1.75, 0) {};
      \node [style=none] (34) at (2.25, 0) {};
      \draw (0) to (1);
      \draw (2) to (3);
      \draw (4) to (5);
      \draw [bend right=165, looseness=1.50] (6.center) to (7.center);
      \draw [bend right=150, looseness=2.25] (8.center) to (9.center);
      \draw [bend right=165, looseness=1.50] (10.center) to (11.center);
      \draw [bend right=150, looseness=2.25] (12.center) to (13.center);
      \draw (1) to (14);
      \draw (3) to (15);
      \draw (5) to (16);
      \draw (21) to (22);
      \draw (23) to (24);
      \draw (25) to (26);
      \draw [bend right=150, looseness=2.25] (29.center) to (30.center);
      \draw [bend right=150, looseness=2.25] (31.center) to (32.center);
      \draw [bend right=150, looseness=2.25] (33.center) to (34.center);
    \end{tikzpicture}
    \caption{}
  \end{subfigure}
  \caption{
    Pictorially, one can write a  morphisms as a line diagram for a bijection with at most one tie around each strand, where we think of $T(Y)$ as bands tied at the bottom, as shown in (A).
    Here, the strand on the right is untied.
    Composition  amounts to joining the strands and retying the ties if they overlap as we see in (B).
    }
  \label{fig:tie-diagram}
\end{figure}

The category $\Ties$ is monoidal.
On objects, define $X \ot Y = X \amalg Y$.
For the morphisms, define $(\sigma, \mathcal T) \ot (\sigma', \mathcal T') = (\sigma \amalg \sigma', \mathcal T \ot \mathcal T')$ where $\ot: T(X) \times T(Y) \to T(X \amalg Y)$ is given by
\begin{equation}
  (X_0 , \{X_i\}_{i=1}^n) \ot (Y_0 , \{Y_j\}_{j=1}^m) = (\iota_XX_0 \cup \iota_YY_0 , \{\iota_XX_i\}_{i=1}^n \cup  \{\iota_YY_j\}_{j=1}^m)
\end{equation}
here $\iota_X: X \to X \amalg Y$ and $\iota_Y: Y \to X \amalg Y$ are the natural inclusion maps.

%The commutativity constraint $X \ot Y \to Y \ot X$ is $(C_{XY}, \mathcal (Y \amalg X, \emptyset))$ where $C_{XY}: X \amalg Y \to Y \amalg X$ is the ordinary commutativity constraint in $\Set$.
\begin{prop} The category $\Ties$ together with subset of morphisms   $\Indec$ for which the tie is either $(\{x\},\emptyset)$ for a singleton set,
or $(\emptyset, \{X\})$ for any set
is a  UFC, but not a \pher{} UFC.
\end{prop}
\begin{proof} Note that the isomorphisms of $\Ties$ are of the form $(\sigma, (X,\emptyset))$ and in the arrow category $(\sigma,\mathcal{T})$ is equivalent to $(id_X,\T)=\bigotimes_{x\in X_0} (id_{\{x\}},(\{x\},\emptyset)) \ot \bigotimes_{i=1}^k(id_{X_i},(\emptyset,\{X_i\}))$, if
 $\T=(X_0,\{X_1\kdk X_k\})$ then $(id_X,\T)$. This is unique up to unique isomorphisms proving $\Ties$ is a  UFC.

The fact that this is not a \pher{} UFC is illustrated in Figure~\ref{fig:tie-diagram}. The decomposition of the composition of the two morphisms into indecomposables is $(id_{\{x\}},(\emptyset,\{x\})^{\ot 3})$, while the decomposition of the constituent morphisms each have two factors and the composable decomposition of these involve both factors, so that there is no decomposition into three tensor factors as required for being \pher.
\end{proof}

% \begin{df}
%     A basis for morphisms $(\Indec,\jmath)$ is  {\em \pher{}} if for every pair of composable morphisms $(\phi_0,\phi_1)$, with $\phi_1\phi_0=\phi$, and decomposition into irreducible morphisms
% \begin{equation}
% \phi_0\simeq \bigotimes_{v\in V}\phi_{0,v}, \quad  \phi_1\simeq \bigotimes_{w\in W} \phi_{1,w},  \text{ and }\phi=\bigotimes_{u\in U}\phi_u
% \end{equation}
%       there exists a  partition of $V\amalg W=\amalg_{u\in U} P_u$ indexed by $U$,
%   such that for each $u\in U$ there is a decomposition pair $(\phi_{0,u},\phi_{1,u})$ of the $\phi_u$,  viz.\ $\phi_{1,u} \phi_{0,u}=\phi_u$, such that
%   \begin{equation}
% \label{eq:hereditarycond}
%     \phi_{0,u} \simeq \bigotimes_{v\in P_u\cap V}\phi_{0,v} \text{ and } \phi_{1,u} \simeq \bigotimes_{w\in P_u\cap W}\phi_{1,w}
%   \end{equation}

% \end{df}

%\subsubsection{The FC of finite Sets}

%  \begin{proof}
%   This follows directly from the definitions.
%   Dropping $(\Indec,\jmath)$ the triple $(\V,\M,\imath)$ is a Feynman category.
%   Vice--versa given $\FF=(\V,\F,\imath)$ setting
%   makes the quadruple
%   $(\V,\Iso(\V^\ot \Da \V),\F,\imath,\jmath)$, where $\jmath$ is as above is a hereditary UFC.
%   The hereditary condition is proven in \cite[??]{feynman}.
%  This is only up to equivalence, as one can replace $\Indec$ by an equivalent subcategory.
% \end{proof}

\subsection{Cospans and the Indexing functor}
 \label{sec:hereditary-UFC-indexing}
\label{cospanexpar}
\subsubsection{The category of cospans}
The  \pher{} UFC analogous to $\FinSet$ for FCs is the category of cospans of finite sets  $\Cospan$.
The objects are finite sets  and morphisms are isomorphism classes of cospans.
A cospan is a diagram of the form \eqref{cospaneq} which we write as $(l,r)$ as a short hand notation.
An isomorphism of cospans is given by an isomorphism in the middle as in \eqref{cospaneq} so that $(l,r)\simeq (\sigma\circ l,\sigma \circ r)$ for an isomorphism $\sigma: V \to V'$. We will denote by $[(l,r)]$  the isomorphism class containing $(l,r)$.
The composition of two classes cospans is given by push--out, see below:
  \begin{equation}
    \label{cospaneq}
    \begin{tikzcd}
      &V& \\
      S \arrow[ru, "l"]
      & & T \arrow[lu, "r"']
    \end{tikzcd} \quad
    \begin{tikzcd}[row sep = small]
      & V \ar[dd, "\sigma", "\simeq"'] & \\
      S \ar[dr, "l'"]
      \ar[ur, "l"'] & &
      T\ar[dl, "r'"']
      \ar[ul, "r"]\\
      &V'&
    \end{tikzcd}\quad
        \begin{tikzcd}
        & & U& & \\
        & V
        \arrow[ru, "f'", dashed] & &
        W \arrow[lu, "g'"', dashed] & \\
        S \arrow[ru, "f_1"] & & C \arrow[lu, "g_1"'] \arrow[ru, "f_2"] & & T \arrow[lu, "g_2"']
        \end{tikzcd}
  \end{equation}

 Note this is well-defined and associative on isomorphism classes due to the universal property of cospans.
 The identity morphism for the set $X$ is the cospan $[(id_X, id_X)]$.
 The {\em isomorphisms} in the category $\Cospan$ are the morphisms $[(\sigma, \tau)]$ where
 $\sigma$ and $\tau$ are bijections.

  This category is naturally {\em monoidal}.
  For a pair of object $X$ and $Y$, the product is their disjoint union $X \amalg Y$.
  Moreover, for morphisms we define $[(l,r)]\amalg [(l',r')]=[(l\amalg l',r\amalg r')]$.
 The monoidal unit is $\emptyset$ which is strict. So, in particular, for morphisms we have $[(l,r)] \amalg [(id_{\emptyset}, id_{\emptyset})] = [(l,r)]$.
$\Cospan$ has a ground ring as $Hom(\emptyset,\emptyset)=\{\emptyset\to \underline{k} \leftarrow \emptyset\vert k\in \N\}\simeq \N$ with the operation of $+$. The action on morphisms is given by disjoint union.

\begin{rmk}
If one does not pass to isomorphism classes, one ends up with
a 2--category, see \cite{Benabou}. Here a 2--morphism between two cospans $(l,r)$ and $(l',r')$
with the same source and target is given by any morphism $m:V\to V'$ such that $l'=lm,r'=rm$.
Such morphisms, more precisely surjections $m$, are natural when considering mergers for directed graphs, see below.
\end{rmk}

We will call a cospan {\em connected} if $|V|=1$.  This notion is well-defined under isomorphisms of cospans and such a map is given by the class
 $[(\pi_S,\pi_T)]$, where for any set $X:\pi_{X}:S\to \{*\}$ is the unique map to a final object.   %This notion is invariant under isomorphisms in $\Iso(\Cospan\da\Cospan)$.
     A {\em singleton map} is a map in which $|S|=|T|=|V|=1$.

 A connected cospan is {\em non--degenerate} if it is not in the ground monoid, i.e.\ not both $S=\emptyset$ and $T=\emptyset$.
  Requiring $S \amalg T \to V$ to be surjective and modifying composition leads to the category of \emph{corelations}.
  This category has a trivial ground monoid and unique decomposition of morphism.
 \begin{rmk}
 A cospan is the same as a $V$-partition of $S\amalg T$ given by $l\amalg r$.
  This induces an equivalence relation on $S \amalg T$ whose classes are in 1--1 correspondence with the union of the images of the maps $l$ and $r$. However,
the elements not in this image are empty classes that cannot be recovered from the equivalence relation alone.
  In Figure~\ref{fig:cospan-composition}, the boxes represent the subsets of the partition and a edge represents membership.
  In this graphical representation, the composition is obtained by collecting the boxes and middle dots that are connected into one ``big box'' and retaining the box as a new vertex erasing the inside.

The push--out $U$ of a cospan  is  the relative coproduct
$U=V \bisub{g_1}{\amalg}{f_2} W=(V \amalg W)/\sim$ where $\sim$ is the equivalence relation given by $g_1(c)\sim f_2(c)$ for $c\in C$.
The resulting equivalence relation on $S\amalg T$ is given by the image of $f_1 \circ f'\amalg g_1 \circ g'$.

There is a left and a right injection of $\FinSet$ to $\Cospan$ given by identity on objects and by sending $f$ to $(f,id)$ or $(id,f)$.
%1/2-Prop is not correct. It is somehow related to factorizing $(n,m)$ through 1, but somehow more general
 Note that $[(\sigma,\tau)]=[(\tau^{-1}\sigma,id)]$ and $[(\sigma,id_T)]=[(id_S,\sigma^{-1})]$ where $\sigma:S\stackrel{\sim}{\to} T$.
Thus one can identify
$\Iso(\Cospan)\simeq \Iso(\FinSet)$ by the identity on objects and sending an isomorphism $\sigma$ to $[(\sigma,id)]$.
\end{rmk}

% The category structure of cospan is a bit more intricate than one would assume {\em prima vista}.
The following facts are directly verifiable from the definitions.
\begin{lem} \mbox{}
  \label{lem:cospan}
  \begin{enumerate}
    \item \label{lem:cospan-decomp} Any morphism $\phi\in Hom_{\Cospan}(S,T)$ is isomorphic to a disjoint union $(S \stackrel{l}{\rightarrow} V \stackrel{r}{\leftarrow} T) \simeq \amalg_{v\in V} (\pi_{l^{-1}(v)},\pi_{r^{-1}(v)})$ of connected morphisms.

%    \item The isomorphisms in $\Cospan$ are given by disjoint unions of singleton maps
\item Any morphism uniquely decomposes as $\phi_{deg}\amalg \bar\phi$, where $\phi_{deg}=\emptyset\to \underline{k}\leftarrow \emptyset$ is degenerate and $\bar\phi$ is union of the  non--degenerate connected morphisms.
Moreover, as the union of two cospans $\phi\amalg \phi'$ satisfies $(\phi\amalg \phi')_{deg}=\phi_{deg}\amalg \phi'_{\deg}$ and $\overline{\phi\amalg \phi'}=\bar\phi\amalg \bar\phi'$ the monoidal product $\amalg$ is compatible with the $R$--module structure $r\phi \amalg r'\psi=rr'(\phi\amalg \psi)$.
    \item A basic morphism has $\Aut([\pi_S,\pi_T])=\Aut(S)\times \Aut(T)$. More generally, an isomorphism class of cospans has a residual action of $(\Aut(S)\times \Aut(T))/\Aut(V)$ on $\Hom(S,V) \times \Hom (T,V)$ where $\Aut(V)$ acts diagonally permuting the fibers of $l$ and $r$. This identifies $\Aut([(l,r)])$ as the classes of $\Aut(S) \times \Aut(T)$ which simultaneously preserve the fibers of $l$ and $r$ over each $v$. In the degenerate case, there are no non--identity automorphisms, since there is only one map $\emptyset\to \emptyset$. \qed
    % \item The decomposition of (1) for a cospan $[l, r]$ is unique up to unique isomorphism given by simultaneous permutation of fibers and isomorphisms fixing the fibers of $l$ and $r$ when the induced map $S \amalg T \to V$ is surjective.
    \end{enumerate}
\end{lem}

\begin{prop} The category
  $\Cospan$ is part of a \pher{} UFC $\CCospan$ with $\V=\trivial$ and $\Indec$ being the full subgroupoid $\Ctd$ of $\Iso(\Cospan\da\Cospan)$ whose objects are non--degenerate connected cospans, with $\jmath$ being the inclusion.
%move !!!
%  The functor $\imath$  is defined by $\V\to\V^\ot\simeq \Iso(\FinSet)\simeq \Iso(\Cospan)$.
%  A compatibility map $\imath_{\Ctd}$ is given by choosing a particular basis of objects. \qed

%We fix bases $\tilde \imath$ which defines
%$i_{\Ctd}$ and consider
%the \pher{} UFC
%$\CCospan=(\triv,\Ctd,\Cospan,\imath,\jmath,\imath_{\Ctd})$.
\end{prop}

\begin{proof}
Using the convention of \S\ref{sec:rmodulecat},
the fact that $\Cospan$ is a UFC with basis $\Ctd$ is clear from Lemma~\ref{lem:cospan}. Indeed, any cospan is up to isomorphism a disjoint union of a non--degenerate cospan and a degenerate one.
This means that it can uniquely be written as an element of the ground ring times a non--degenerate cospan.  Considering only the non--degenerate cospans, the map $j^\ot$ sends their external product to the disjoint union $(\pi_{S_1},\pi_{t_1})\bdb (\pi_{S_n},\pi_{T_n})\to  (\pi_{S_1},\pi_{t_1})\amalg \cdots \amalg (\pi_{S_n},\pi_{T_n})$. This is clearly an equivalence of $R=Hom(\emptyset,\emptyset)$ modules.
The inverse functor is given by choosing an order for the  non--degenerate components, i.e.\ for the set $V$.

To show the \pher{} condition, let $\phi_0=[(f_1,g_1)]$ and $\phi_1=[(f_2,g_2)]$ be a pair of composable cospans as in the rightmost diagram of \eqref{cospaneq}.
The decomposition of $\phi:=\phi_1\phi_0=
\amalg_{u\in U} \phi_u$ is given by  $\phi_u=[(f'f_1|_{(f'f_1)^{-1}(u)},g'g_2|_{g'g_2^{-1}(u)})]$.

Setting $I_u=f'^{-1}(u)\amalg g'^{-1}(u)$ gives the partition of $V\amalg W$ and each $\phi_u$  decomposes into $\phi_{0,u}=\amalg_{v\in f'^{-1}(u)}\phi_{0,v}$
and  $\phi_{1,u}=\amalg_{w\in g'^{-1}(u)}\phi_{0,w}$ which constitute the desired decomposition.
\end{proof}

% MM: refers to something commented out
% A graphical representation of a morphism is given in Figure~\ref{fig:coRel-morphism}.
% The morphism in Figure~\ref{fig:coRel-morphism} has four connected components.

\begin{rmk}
\label{rmk:cospanintepretation}
Cospans appear in many guises:
\begin{enumerate}
\item {\it Graph interpretation.} For this, one interprets $V$ as a set of vertices of an aggregate of corollas, $S$ as the set of input flags and $T$ as the set of output flags. The map $\del = l\amalg r:S\amalg T\to V$ is the attaching maps of the flags. If $V$ is a singleton set, then the irreducible $S \rightarrow\{*\}\leftarrow T$ is a directed corolla.
The composition matches input and output flags and then contracts them, generalizing the usual corolla interpretation of the associative and commutative operad, see \cite{woods}.
  \item {\it PROP interpretation.} This composition is not by chance reminiscent of the composition in PROPs.
  It corresponds to the terminal functor $\final$ for the Feynman category for PROPs $\FF_{PROP}$, see \cite[\S2]{feynman}, with values in $\Set$.
  \item {\it Surface  and cobordism interpretation.}
  \label{item:cobordism}
  In this interpretation, consider the cospan $S \rightarrow\{*\}\leftarrow T$ to be a surface with $S$--labeled input boundaries and $T$--labeled output boundaries. Upon gluing the cobordisms, one stabilizes by removing all extra handles.
This corresponds to the contraction above and corresponds to  the stabilized arc structures of \cite{postnikov}. To recover the ``lost genus'', one can use a push--forward as in \cite{BergerKaufmann,DDecDennis}.
\end{enumerate}

\end{rmk}

\begin{prop}
[Bimodule interpretation] \label{bimodule-interpretation} Let  $\B = (\Iso(\FinSet),\amalg)$
with $\amalg$ as the monoidal structure and take $\E = (\Set,\times)$ and $\nu$ to be the trivial bimodule defined  by $\nu(S,T) = \{ * \}$, then $\nu^\otimes$ is isomorphic to the bimodule of cospans $\rho(S,T)=\{[S\to V\leftarrow T]\}$.
\end{prop}
\begin{proof} Computing the left Kan extension on $(\unit_{\B}, \unit_{\B})$ for the ground ring: \\
\begin{equation}
   \nu^{\otimes}(\emptyset, \emptyset)
    = Lan_{m_{\B^{op} \times \B}}(m_{\Set}\nu^\boxtimes) (\emptyset, \emptyset)
    = \colim_{m_{\aca}( -) \downarrow (\emptyset, \emptyset))}m_{\Set}\nu^\boxtimes
\end{equation}
  The colimit is indexed by the category of words $w \in (\B^{op} \times \B)^\bt$ over $(\emptyset, \emptyset)$. These words are
  \ $(\emptyset, \emptyset)^{\bt n}, n\in \N_0$, with $n=0$ corresponding to the empty word $()$. As $m_{\Set}\nu^{\ot}(w)$ is a singleton set, the computation then reduces to $\colim_{\N}\final$, where $\final(n)=\{*\}$. We conclude that $\nu^\otimes(\emptyset, \emptyset) \cong \N_0\cong\{ \emptyset \rightarrow \underline{k} \leftarrow \emptyset | k \geq 0 \}$.

  {\sc Non--degenerate elements.}
  Consider the computation for $(S, T)$ in general:
  \begin{equation}
    \nu^{\otimes}(S, T)
    = \colim_{(m_{\aca} \downarrow (S, T)} m_{\Set}\nu^\bt
  \end{equation}
  This time, the indexing comma category has objects $(\bt_{i=1}^n(S_i, T_i), d)$ where $d: ( \coprod_{i=1}^n S_i, \coprod_{i=1}^n T_i ) \overset{\sim}{\to} (S, T)$ is an isomorphism.
  Each object in this indexing category can be associated to a cospan $S \rightarrow \underline{n} \leftarrow T$ by composing $S \to \coprod_{i=1}^n S_i \to \underline{n}$ and $T \to \coprod_{i=1}^n T_i \to \underline{n}$.
  Moreover, isomorphic objects $(\bt_{i=1}^n(S_i, T_i), d) \to (\bt_{i=1}^n(S'_{\sigma(i)}, T'_{\sigma(i)}), d')$ will be associated to the same cospan.
  Computing the colimit then gives us $\nu^\otimes(S, T) \cong \{ \text{iso classes } S \rightarrow V \leftarrow T \}$.
  A more detailed coend computation is given in \cite{plethysm}.
\end{proof}
\begin{rmk}
  Let $\bar \rho$ be the sub--bimodule of non-degenerate cospans and let $\rho_{deg}$ be the sub--bimodule of degenerate cospans.
  Then $\nu^\ot(S, T) \cong \rho_{deg}(S, T)\amalg \bar \rho(S, T)$.
  Note that $\bar \rho$ itself is factorizable, but it is not a submonoid under the natural monoid structure of $\nu^\ot$.
  To get a monoid structure for $\bar \rho$ (and hence a category structure), one could take $\bar \rho \pl \bar \rho \to  \rho_{deg} \amalg \bar \rho  \to \bar \rho$ where the second map sends everything in $\rho_{deg}$ to the unique element in $\bar \rho(\emptyset, \emptyset)$.
  This is equivalent to considering corelations. To recover the ``missing part'' one should consider curved monoids/categories.
\end{rmk}

\begin{figure}
    \centering
\begin{tikzpicture}[scale=.4]
	\node [style=Black, inner sep = 0.5mm] (1) at (-4, 3) {};
	\node [style=Black, inner sep = 0.5mm] (2) at (-4, 1) {};
	\node [style=Black, inner sep = 0.5mm] (3) at (-4, -1) {};
	\node [style=Black, inner sep = 0.5mm] (5) at (1, 3) {};
	\node [style=Black, inner sep = 0.5mm] (6) at (1, 1) {};
	\node [style=Black, inner sep = 0.5mm] (8) at (-1.5, 3) {};
	\node [style=Black, inner sep = 0.5mm] (9) at (-1.5, 2) {};
	\node [style=Black, inner sep = 0.5mm] (10) at (-1.5, 0) {};
	\node [style=Box] (13) at (-2.5, 2.5) {};
	\node [style=Box] (14) at (-2.5, 0) {};
	\node [style=Box] (15) at (-0.5, 3) {};
	\node [style=Box] (16) at (-0.5, 2) {};
	\node [style=Box] (17) at (-0.5, 0) {};
	\node [style=none] (18) at (-3, 3.5) {};
	\node [style=none] (19) at (0, 3.5) {};
	\node [style=none] (20) at (0, 1.5) {};
	\node [style=none] (21) at (-3, 1.5) {};
	\node [style=none] (22) at (-3, 0.5) {};
	\node [style=none] (23) at (0, 0.5) {};
	\node [style=none] (24) at (-3, -0.5) {};
	\node [style=none] (25) at (0, -0.5) {};
	\node [style=Black, inner sep = 0.5mm] (26) at (3, 3) {};
	\node [style=Black, inner sep = 0.5mm] (27) at (3, 1) {};
	\node [style=Black, inner sep = 0.5mm] (28) at (3, -1) {};
	\node [style=Black, inner sep = 0.5mm] (29) at (6, 3) {};
	\node [style=Black, inner sep = 0.5mm] (30) at (6, 1) {};
	\node [style=Box] (35) at (4.5, 0) {};
	\node [style=Box] (37) at (4.5, 2) {};
	\node [style=none] (38) at (2, 5) {};
	\node [style=none] (39) at (2, -2) {};
	\node [style=none] (40) at (-1.5, 4.75) {\large Composition};
	\node [style=none] (41) at (4.5, 4.75) {\large Result};
	\draw (13) to (8);
	\draw (13) to (9);
	\draw (9) to (16);
	\draw (16) to (5);
	\draw (6) to (16);
	\draw (1) to (13);
	\draw (2) to (14);
	\draw (14) to (10);
	\draw (10) to (17);
	\draw (3) to (14);
	\draw (8) to (15);
	\draw [style=Dash] (18.center) to (21.center);
	\draw [style=Dash] (20.center) to (21.center);
	\draw [style=Dash] (20.center) to (19.center);
	\draw [style=Dash] (19.center) to (18.center);
	\draw [style=Dash] (22.center) to (24.center);
	\draw [style=Dash] (24.center) to (25.center);
	\draw [style=Dash] (25.center) to (23.center);
	\draw [style=Dash] (23.center) to (22.center);
	\draw (37) to (29);
	\draw (30) to (37);
	\draw [in=150, out=-30] (27) to (35);
	\draw (28) to (35);
	\draw (26) to (37);
	\draw (38.center) to (39.center);
    \end{tikzpicture}
    \caption{A composition of two morphisms.}
    \label{fig:cospan-composition}
\end{figure}

The decomposition of a morphism for a UFC is governed by $\Cospan$ in a functorial fashion.
\begin{lem}  A compatible choice allows one to index morphisms of $\M$ by cospans, extending the indexing by type. In particular,
\label{lem:cospanindex}

\begin{equation}
\label{eq:hereditaryUFCdecomp}
\Hom_\M
\left(
\imath(\bigotimes_{s\in S} *_s),
\imath (\bigotimes_{t\in T}*'_t)
\right)
\simeq
\\ \bigoplus_{k}
\bigotimes_{S\stackrel{l}{\to} \underline{k} \stackrel{r}{\leftarrow}T}
\bigotimes_{1\leq i\leq k}
\Indec
\left(
\bigotimes_{s\in l^{-1}(i)} *_s,\bigotimes_{t\in r^{-1}(i)}*'_t
\right)
%\ot \Indec^{\ot \underline{k}\setminus (im(l)\cap im(r))}(0,0)
\end{equation}
where the product is over all skeletal cospans,
we used the short hand notation $\Indec(X,Y)$ to be the morphisms in $\imath_\Indec(\Indec)$ from $\imath(X)$ to $\imath(Y)$.
\end{lem}
If $\underline{k}=0=\emptyset$ then there is only the empty cospan $\emptyset \to \underline{0}\leftarrow \emptyset$. This contribution only appears in $\Hom_\M(\unit,\unit)=\jmath(\unit_\Indec)$ which are the irreducibles of type $(0,0)$.

\begin{proof}
Given a morphism $\tilde \jmath(\phi)=\bigotimes_{v\in V}\phi_v$ with $\phi_v:\tilde\imath(X_v)\to \tilde\imath(Y_v)$, let  $\tilde \imath (X)=\bigotimes_{s\in S} *_s$
 and $\tilde\imath(Y)=\bigotimes_{t\in T} *'_t$, so that  $\index(X)=S$ and $\index(Y)=T$,
 then there is a $V$--partition of $S$ and a $V$--partition of $T$ such that $\tilde \imath(X_v)=
 \bigotimes_{s\in S_v} *_s$ and $\tilde\imath(Y_v)=\bigotimes_{t\in T_v} *'_t$. This defines a $V$--partition of $S$ and $T$ and hence a maps $s_V:S\to V$
 and $t_V: T\to V$.
  Thus a compatible choice of basis yields a pair of maps which forms a cospan: $\index(\phi): (S\stackrel{s_V}{\to} V \stackrel{t_V}{\leftarrow} T)$.
\end{proof}

% \begin{rmk}
% \label{fsindexrmk}
% This generalizes the formulas for a Feynman category from maps of finite sets to correspondences as for Feynman categories $k=|T|$ in which case \eqref{eq:hereditaryUFCdecomp} specializes  to
% %\begin{equation}
% %\label{eq:feydecomp}
% $\Hom(\bigotimes_{s\in S}*_s,\bigotimes_{t\in T} *'_t)\simeq \bigotimes_{f:S \to T}\bigotimes_{t\in T}\Hom(\bigotimes_{s\in f^{-1}(t)}*_s,*'_t)$
% %\end{equation}
% cf. \cite{feynman} and the reformulation \cite{BKW}.
% The reason being that the connected components are of type $(n_j,1)$ where in particular $n_j=|f^{-1}(j)|$.

% More generally,  a set indexed presentation defines morphism of Feynman categories $\findex_{\tilde\imath}:\FF\to \FFinSet$:
%   On $\V$, this is the only possible functor $v:\V\to \triv$.
%   On objects of $\F$, the functor is $f(X)=\index(\tilde\imath(X))$ where $\index$ is defined in \ref{lem:cospan}.
%    On the morphisms of $f$, the functor is defined by the formula \eqref{eq:feydecomp} given by applying $\tilde{j}$ to the morphisms and picking the factor corresponding to $\F$ in which it lies.
%    Note that even in the skeletal case, $S=\un$ and $T=\um$ the fibers are unordered and
% the monoidal products are most naturally indexed by sets.
%   \end{rmk}

\begin{prop}
\label{hereditary-UFC-Cospan-index}
If $\M$ has essentially uniquely factorizable objects and is \pher{}, then a presentation defines a functor $\index:\M\to \Cospan$ which extends to a functor
 \pher{} UFCs $\findex: \FM\to \CCospan$.
\end{prop}

\begin{proof}
Given $\M$, fix a compatible basis $\tilde \imath,\tilde \jmath$. The
putative functor is given by $\index$ on objects and $\index$ on morphisms by Lemma \ref{lem:cospan}.
 This is clearly strong monoidal and $\index(\unit_\M)=\unit_{\Cospan}$.

 For the units:
 $\index(id_{\imath(\bigotimes_{s\in S}(*_s))})=(S\stackrel{id}{\to}S\stackrel{id}{\leftarrow}S)=id_S\in \Cospan(S,S)$.
Thus it remains to check that
$
      \index(\phi_0\circ\phi_1)=\index(\phi_0)\circ \index(\phi_1)
$. This is the case, as the connected components are indexed by the composition of the cospans and the vertex $U$ is the index set of the irreducibles of the decomposition.
The functor $\index$
automatically extends to a UFC indexing of $\FM$ by $\CCospan$ as the restriction to $\Indec$ takes values in $\Ctd$ as $|V|=1$ for irreducibles, and
on $\V$ this restricts to  the trivial functor $v:\V\to\triv$.
\end{proof}
%By abuse of notation,  the functor $f$ will be called $\index$ as well.

\begin{lem}
\label{lem:feyhereditaryUFC}
Being a Feynman category is equivalent to being a \pher{} UFC with elementary morphisms only of type $(n,1)$.
\end{lem}
\begin{proof}
  By Corollary \ref{cor:fey}, a Feynman category is a UFC where all of the elements of $\Indec$ are type $(n,1)$.
  Since the $\phi_{0,v}$ are of type $(n_v,1)$, the left arrow in the span indexing the diagram is a bijection that is $V\leftrightarrow S\rightarrow W$.  The pushout are thus just the fibers of $r$ and are in bijection with $W$. Thus the $\phi_u$ are of type $(n_u,1)$ and hence irreducible since the type is additive.
\end{proof}

\subsubsection{The category of spans}
\label{par:spans}
Just like cospans, one can consider equivalence classes of spans $(l,r)= S\leftarrow V \rightarrow T$ by reversing the arrows.
The composition is then by pull--back. The disjoint union again gives a monoidal structure.
A span is \emph{connected} if the pushout $S\amalg_V T$ has a single element. The ground ring is trivial in this case as $S=T=\emptyset$ implies that $V=\emptyset$.

\begin{prop}
  $\Span$ is a \pher{} UFC with basis $\Ctd_S \subseteq (\Span \da \Span)$ the subgroupoid of connected spans.
\end{prop}
\begin{proof}
We start with the case where both $l$ and $r$ are surjections.
Let $C=S\amalg_VT$ and let $\pi:S\amalg T\to S \amalg_V T$ be the projection.
For $c\in C$, let $S_c$ be the preimage of $c$ under $S \hookrightarrow S \amalg T \overset{\pi}{\to} C$.
Likewise, let $T_c$ be the preimage of $c$ under $T \hookrightarrow S \amalg T \overset{\pi}{\to} C$.
Set $V_c=l^{-1}(S_c)$ then $r(V_c)=T_c$ so that the restriction $(l_c,r_c)$ is well-defined.
Since $l$ and $r$ are surjective, each $V_c$ is non--empty and
$(l,r)=\amalg_{c\in C}(l_c,r_c)$.
As is easily checked,
this decomposition is unique up to unique isomorphisms.
Dropping the surjectivity requirements, we decompose the source as $S=im(l)\amalg S_0$ and the target as $T=im(r)\amalg T_0$.
Then $C$ can be written as $C= (im(l) \amalg_V im(r)) \amalg S_0 \amalg T_0$
and $(l,r)=\amalg_{c\in C}(l_c,r_c)$, where for $s\in S_0 \subseteq C$ we define $\{s\} \overset{l_s}{\leftarrow} \emptyset \overset{r_s}{\rightarrow} \emptyset$ and similarly for $t\in T_0 \subseteq C$.

For the \pher{} condition, note that the connected components of the composition naturally decompose into connected components of the two constituent morphisms by the universal properties of push--forwards.
\end{proof}

From the proof we see that if both maps are surjective, then we can think of both $S$ and $T$ giving partitions of $V$ via their fibers, and the connected components give the greatest common partition of $V=\amalg_{c\in C}V_c$.

\begin{rmk}
  Unlike the case of cospans, $\Hom(\emptyset, \emptyset)$ only contains the span $\emptyset \leftarrow \emptyset \rightarrow \emptyset$ so the ground monoid is trivial. There is a special subcategory, $\mathcal{DS}$  of spans of the form $S \leftarrow \emptyset \rightarrow T$.
  The two generating sets of morphisms $l_S=S\leftarrow \emptyset\rightarrow \emptyset$ and $r_T=\emptyset\leftarrow \emptyset\rightarrow T$  commute under the monoidal structure $l_S\amalg r_T=r_T\amalg l_S$ if one uses a strict version of the definition of the disjoint union with the empty-set $S\amalg \emptyset=S$.
  However, it is not necessary to work in the strict case, but if one does it is still equivalent to a free symmetric monoidal category, as in the following example.

  Let $\Indec$ be the groupoid with objects $l$ and $r$ and define a functor $\imath: \Indec \to \Iso(\mathcal{DS} \da \mathcal{DS})$ such that $\imath(l) = (\{\ast\} \leftarrow \emptyset \rightarrow \emptyset)$ and $\imath(r) = (\emptyset \leftarrow \emptyset \rightarrow \{\ast\})$.
  This extends to a functor $\imath^{\bt}: \Indec^{\bt} \to \Iso(\mathcal{DS} \da \mathcal{DS})$.
  Since any degenerate span $S \leftarrow \emptyset \rightarrow T$ is isomorphic to $(S \leftarrow \emptyset \rightarrow \emptyset) \amalg (\emptyset \leftarrow \emptyset \rightarrow T)$, it follows that $\imath^{\bt}$ is essentially surjective.
  The functor $\imath^{\bt}$ is also full and faithful.
  Consider, for instance, the words $lrrl$ and $lrlr$ and observe that there is a one-to-one correspondence between $\Hom_{\Indec^{\bt}}(lrrl, lrlr)$ and  the set of morphisms between $\imath^\bt(lrrl) = (\{1,4\} \leftarrow \emptyset \rightarrow \{2,3\})$ and $\imath^\bt(lrlr) = (\{1,3\} \leftarrow \emptyset \rightarrow \{2,4\})$ in $\Iso(\mathcal{DS} \da \mathcal{DS})$.
  In the strict case, these objects become identified, but we still get the correct correspondence.
In the strictly commuting case, one also can pass to a larger ground ring, which includes these morphisms. Similar to cospan, where the ground ring introduced lone elements in $V$, here there are lone elements on the source and target side, which commute with everything.

To explicitly exhibit the \pher{} property in the ``degenerate'' situation, consider the composable pair $(S \leftarrow \emptyset \rightarrow T, T \leftarrow \emptyset \rightarrow R)$.
Using the notation of Figure~\ref{connecteddiag}, $U = S \amalg T \amalg R$, and
the pairs $(\phi_{0,u}, \phi_{1,u})$ correspond to $
(\{u\} \leftarrow \emptyset \rightarrow \emptyset, \emptyset \leftarrow \emptyset \rightarrow \emptyset)$
when $u \in S$, to $(\emptyset \leftarrow \emptyset \rightarrow \{u\},\{u\}  \leftarrow \emptyset \rightarrow \emptyset)$ when $u \in T$,
and to $(\emptyset \leftarrow \emptyset \rightarrow \emptyset,\emptyset \leftarrow \emptyset \rightarrow \{u\})$ when $u \in R$.
For $u \in S$, composition evaluates to $\phi_u = (\{u\} \leftarrow \emptyset \rightarrow \emptyset)$ which is a basic span and irreducible.
For $u \in T$,  composition evaluates to $\phi_u = (\emptyset \leftarrow \emptyset \rightarrow \emptyset)$, which corresponds to the empty word in $\W^\bt$ and in the ground ring which is suppressed in the decomposition into basic objects via the unit constraints.
For $u \in R$, composition evaluates to $\phi_u = (\emptyset \leftarrow \emptyset \rightarrow \{u\})$ which is another basic span and hence irreducible.
\end{rmk}

\section{Structural Results, decompositions and standard forms}
\label{par:global}
%In order to compute plus constructions, one can use several structural results presented here.
% \label{par:ufcstructurpar}
% In the case of a   UFC, there is more structure which allows us to reduce the number of  generators, and hence data for the plus constructions. The index is a tool for these considerations.
% We will choose a presentation. The results are easily checked to be independent of this choice.
% %as we know that the connected components are irreducible
% %and one can use the index functor to show that the irreducible generators are the connected ones.

\subsection{Monoidal categories generated by subcategories and the plus constructions}

%\subsection{Standard form for two (crossed) generating subcategories}

\begin{df}
\label{df:crossed}
Given two monoidal subcategories $\I(\M)$ and $\spure(\M)$ which generate a monoidal category $\M$, we call $\M$
$(\spure(\M),I(\M))$-{\em crossed} if  for composable pair $(\Sigma,\G)$ such that $\Sigma\in \I(\M)$ and $\G\in \spure(\M)$, there are morphisms $\Sigma'\in \I(\M)$ and $\G'\in \spure(\M)$ such that
$\Sigma\Gamma=\G'\Sigma'$.
\end{df}
\begin{prop}
\label{prop:alternate}
If $\M$ is generated by two monoidal subcategories $\I(\M)$ and $\spure(\M)$, then
any morphism in $\M$ can be written as $\G_n\circ\Sigma_n \codco \Gamma_1\circ \Sigma_1$ with $\Sigma_i\in \I(\M)$ and $\G_i\in \spure(\M)$.
In particular, if $\M$ is $(\spure(\M),\I(\M))$-{\em crossed}, then any morphism is of the form $\Gamma\Sigma$.
\end{prop}
\begin{proof}
Any morphism in $\M$ can be written as a concatenation of tensor products
%or vice--versa as a tensor product of concatenations.
by first adding identities to equate tensor lengths if necessary and then using the interchange equation to rewrite
 $(f\circ g)\ot h=(f\circ g)\ot (h\circ id)=(f\ot h)\circ(g\ot id)$ and its symmetric rewriting iteratively.

A convenient way to encode this is as a type of ``brick wall'' picture where the rows correspond to tensoring and the columns to concatenation, see Figure~\ref{fig:compositionsfig}. Note that the horizontal seams go through, by the preparation step above, while the vertical ones can be interrupted. The entries $\Sigma$ are for elements of $\I(\M)$ and $\Gamma$ for an element in $\spure(\M)$.
Due to the interchange relations,
$
  \Gamma\otimes \Sigma = (\Gamma \ot id)\circ (id\ot \Sigma)$ and
  $
   \Sigma\ot\G =(id \ot \G)\circ(\Sigma\ot id )$,
one can ``pull apart'' the
rows into only $\G$ or $\Sigma$ by ``pulling up'' the factors of $\Sigma$,
starting with a row of $\G$ tensoring together the rows gives the desired form, see Figure~\ref{fig:compositionsfig}.
\end{proof}

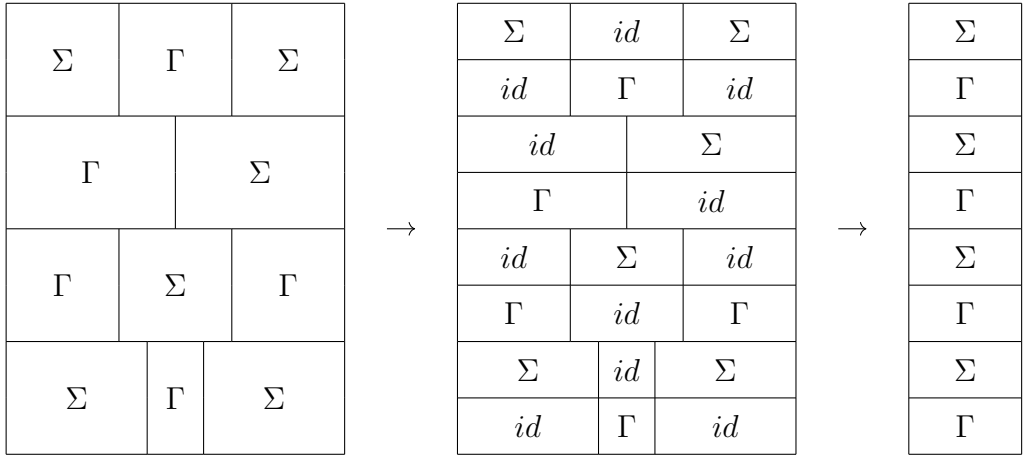
\begin{figure}
    \centering
    \begin{tikzpicture}[scale=.6]
		\node [style=none] (0) at (-3, 2) {};
		\node [style=none] (1) at (-3, 1) {};
		\node [style=none] (2) at (-1, 2) {};
		\node [style=none] (3) at (-1, 1) {};
		\node [style=none] (4) at (1, 2) {};
		\node [style=none] (5) at (1, 1) {};
		\node [style=none] (6) at (3, 2) {};
		\node [style=none] (7) at (3, 1) {};
		\node [style=none] (8) at (-1, 0) {};
		\node [style=none] (9) at (1, 0) {};
		\node [style=none] (10) at (3, 0) {};
		\node [style=none] (11) at (-3, 0) {};
		\node [style=none] (12) at (0, 0) {};
		\node [style=none] (13) at (0, -1) {};
		\node [style=none] (14) at (0, -2) {};
		\node [style=none] (15) at (-3, -1) {};
		\node [style=none] (16) at (-3, -2) {};
		\node [style=none] (17) at (3, -1) {};
		\node [style=none] (18) at (3, -2) {};
		\node [style=none] (20) at (-1, -2) {};
		\node [style=none] (21) at (1, -2) {};
		\node [style=none] (22) at (1, -3) {};
		\node [style=none] (23) at (-1, -3) {};
		\node [style=none] (24) at (-1, -4) {};
		\node [style=none] (25) at (-3, -3) {};
		\node [style=none] (26) at (-3, -4) {};
		\node [style=none] (27) at (3, -3) {};
		\node [style=none] (28) at (3, -4) {};
		\node [style=none] (29) at (1, -4) {};
		\node [style=none] (30) at (-0.5, -4) {};
		\node [style=none] (31) at (0.5, -4) {};
		\node [style=none] (32) at (-3, -5) {};
		\node [style=none] (33) at (-0.5, -5) {};
		\node [style=none] (34) at (0.5, -5) {};
		\node [style=none] (35) at (3, -5) {};
		\node [style=none] (36) at (0.5, -6) {};
		\node [style=none] (37) at (-0.5, -6) {};
		\node [style=none] (38) at (3, -6) {};
		\node [style=none] (39) at (-3, -6) {};
		\node [style=none] (40) at (-2, 1.5) {$\Sigma$};
		\node [style=none] (41) at (2, 1.5) {$\Sigma$};
		\node [style=none] (42) at (1.5, -0.5) {$\Sigma$};
		\node [style=none] (43) at (-1.75, -4.5) {$\Sigma$};
		\node [style=none] (44) at (0, -2.5) {$\Sigma$};
		\node [style=none] (45) at (1.75, -4.5) {$\Sigma$};
		\node [style=none] (46) at (-1.75, -5.5) {$id$};
		\node [style=none] (51) at (0, -5.5) {$\Gamma$};
		\node [style=none] (52) at (2, -3.5) {$\Gamma$};
		\node [style=none] (53) at (-2, -3.5) {$\Gamma$};
		\node [style=none] (54) at (-1.5, -1.5) {$\Gamma$};
		\node [style=none] (55) at (0, 0.5) {$\Gamma$};
		\node [style=none] (56) at (1.75, -5.5) {$id$};
		\node [style=none] (57) at (-2, -2.5) {$id$};
		\node [style=none] (58) at (0, -3.5) {$id$};
		\node [style=none] (59) at (2, -2.5) {$id$};
		\node [style=none] (60) at (-1.5, -0.5) {$id$};
		\node [style=none] (61) at (1.5, -1.5) {$id$};
		\node [style=none] (62) at (-2, 0.5) {$id$};
		\node [style=none] (63) at (2, 0.5) {$id$};
		\node [style=none] (64) at (0, 1.5) {$id$};
		\node [style=none] (65) at (0, -4.5) {$id$};
		\node [style=none] (66) at (3.75, -2) {};
		\node [style=none] (67) at (4.25, -2) {};
		\node [style=none] (68) at (5, 2) {};
		\node [style=none] (69) at (5, 1) {};
		\node [style=none] (70) at (5, 0) {};
		\node [style=none] (71) at (5, -1) {};
		\node [style=none] (72) at (5, -2) {};
		\node [style=none] (73) at (5, -3) {};
		\node [style=none] (74) at (5, -4) {};
		\node [style=none] (75) at (5, -5) {};
		\node [style=none] (76) at (5, -6) {};
		\node [style=none] (77) at (7, 2) {};
		\node [style=none] (78) at (7, 1) {};
		\node [style=none] (79) at (7, 0) {};
		\node [style=none] (80) at (7, -1) {};
		\node [style=none] (81) at (7, -2) {};
		\node [style=none] (82) at (7, -3) {};
		\node [style=none] (83) at (7, -4) {};
		\node [style=none] (84) at (7, -5) {};
		\node [style=none] (85) at (7, -6) {};
		\node [style=none] (86) at (6, 0.5) {$\Gamma$};
		\node [style=none] (87) at (6, 1.5) {$\Sigma$};
		\node [style=none] (88) at (6, -0.5) {$\Sigma$};
		\node [style=none] (89) at (6, -1.5) {$\Gamma$};
		\node [style=none] (90) at (6, -2.5) {$\Sigma$};
		\node [style=none] (91) at (6, -4.5) {$\Sigma$};
		\node [style=none] (94) at (6, -3.5) {$\Gamma$};
		\node [style=none] (95) at (6, -5.5) {$\Gamma$};
		\node [style=none] (96) at (-11, 2) {};
		\node [style=none] (97) at (-11, 1) {};
		\node [style=none] (98) at (-9, 2) {};
		\node [style=none] (99) at (-9, 1) {};
		\node [style=none] (100) at (-7, 2) {};
		\node [style=none] (101) at (-7, 1) {};
		\node [style=none] (102) at (-5, 2) {};
		\node [style=none] (103) at (-5, 1) {};
		\node [style=none] (104) at (-9, 0) {};
		\node [style=none] (105) at (-7, 0) {};
		\node [style=none] (106) at (-5, 0) {};
		\node [style=none] (107) at (-11, 0) {};
		\node [style=none] (108) at (-8, 0) {};
		\node [style=none] (109) at (-8, -1) {};
		\node [style=none] (110) at (-8, -2) {};
		\node [style=none] (111) at (-11, -1) {};
		\node [style=none] (112) at (-11, -2) {};
		\node [style=none] (113) at (-5, -1) {};
		\node [style=none] (114) at (-5, -2) {};
		\node [style=none] (115) at (-9, -2) {};
		\node [style=none] (116) at (-7, -2) {};
		\node [style=none] (117) at (-7, -3) {};
		\node [style=none] (118) at (-9, -3) {};
		\node [style=none] (119) at (-9, -4) {};
		\node [style=none] (120) at (-11, -3) {};
		\node [style=none] (121) at (-11, -4) {};
		\node [style=none] (122) at (-5, -3) {};
		\node [style=none] (123) at (-5, -4) {};
		\node [style=none] (124) at (-7, -4) {};
		\node [style=none] (125) at (-8.5, -4) {};
		\node [style=none] (126) at (-7.5, -4) {};
		\node [style=none] (127) at (-11, -5) {};
		\node [style=none] (128) at (-8.5, -5) {};
		\node [style=none] (129) at (-7.5, -5) {};
		\node [style=none] (130) at (-5, -5) {};
		\node [style=none] (131) at (-7.5, -6) {};
		\node [style=none] (132) at (-8.5, -6) {};
		\node [style=none] (133) at (-5, -6) {};
		\node [style=none] (134) at (-11, -6) {};
		\node [style=none] (135) at (-10, 1) {$\Sigma$};
		\node [style=none] (136) at (-6, 1) {$\Sigma$};
		\node [style=none] (137) at (-6.5, -1) {$\Sigma$};
		\node [style=none] (138) at (-9.75, -5) {$\Sigma$};
		\node [style=none] (139) at (-8, -3) {$\Sigma$};
		\node [style=none] (140) at (-6.25, -5) {$\Sigma$};
		\node [style=none] (142) at (-8, -5) {$\Gamma$};
		\node [style=none] (143) at (-6, -3) {$\Gamma$};
		\node [style=none] (144) at (-10, -3) {$\Gamma$};
		\node [style=none] (145) at (-9.5, -1) {$\Gamma$};
		\node [style=none] (146) at (-8, 1) {$\Gamma$};
		\node [style=none] (147) at (-4.25, -2) {};
		\node [style=none] (148) at (-3.75, -2) {};
		\draw (0.center) to (2.center);
		\draw (2.center) to (3.center);
		\draw (2.center) to (4.center);
		\draw (4.center) to (5.center);
		\draw (6.center) to (4.center);
		\draw (6.center) to (7.center);
		\draw (0.center) to (1.center);
		\draw (1.center) to (11.center);
		\draw (3.center) to (8.center);
		\draw (5.center) to (9.center);
		\draw (7.center) to (10.center);
		\draw (1.center) to (3.center);
		\draw (3.center) to (5.center);
		\draw (5.center) to (7.center);
		\draw (8.center) to (12.center);
		\draw (9.center) to (12.center);
		\draw (12.center) to (13.center);
		\draw (13.center) to (14.center);
		\draw (13.center) to (15.center);
		\draw (13.center) to (17.center);
		\draw (8.center) to (11.center);
		\draw (11.center) to (15.center);
		\draw (15.center) to (16.center);
		\draw (10.center) to (9.center);
		\draw (10.center) to (17.center);
		\draw (18.center) to (17.center);
		\draw (16.center) to (20.center);
		\draw (20.center) to (14.center);
		\draw (14.center) to (21.center);
		\draw (18.center) to (21.center);
		\draw (27.center) to (22.center);
		\draw (22.center) to (23.center);
		\draw (23.center) to (25.center);
		\draw (24.center) to (26.center);
		\draw (30.center) to (24.center);
		\draw (30.center) to (31.center);
		\draw (31.center) to (29.center);
		\draw (29.center) to (28.center);
		\draw (35.center) to (34.center);
		\draw (34.center) to (33.center);
		\draw (33.center) to (32.center);
		\draw (38.center) to (36.center);
		\draw (36.center) to (37.center);
		\draw (37.center) to (39.center);
		\draw (16.center) to (25.center);
		\draw (25.center) to (26.center);
		\draw (32.center) to (26.center);
		\draw (39.center) to (32.center);
		\draw (27.center) to (18.center);
		\draw (28.center) to (27.center);
		\draw (35.center) to (28.center);
		\draw (38.center) to (35.center);
		\draw (20.center) to (23.center);
		\draw (23.center) to (24.center);
		\draw (21.center) to (22.center);
		\draw (22.center) to (29.center);
		\draw (30.center) to (33.center);
		\draw (31.center) to (34.center);
		\draw (33.center) to (37.center);
		\draw (34.center) to (36.center);
		\draw [style=Arrow] (66.center) to (67.center);
		\draw (68.center) to (69.center);
		\draw (69.center) to (70.center);
		\draw (70.center) to (71.center);
		\draw (71.center) to (72.center);
		\draw (72.center) to (73.center);
		\draw (73.center) to (74.center);
		\draw (74.center) to (75.center);
		\draw (75.center) to (76.center);
		\draw (76.center) to (85.center);
		\draw (85.center) to (84.center);
		\draw (84.center) to (83.center);
		\draw (83.center) to (82.center);
		\draw (82.center) to (81.center);
		\draw (81.center) to (80.center);
		\draw (80.center) to (79.center);
		\draw (79.center) to (78.center);
		\draw (78.center) to (77.center);
		\draw (77.center) to (68.center);
		\draw (69.center) to (78.center);
		\draw (79.center) to (70.center);
		\draw (80.center) to (71.center);
		\draw (81.center) to (72.center);
		\draw (82.center) to (73.center);
		\draw (83.center) to (74.center);
		\draw (84.center) to (75.center);
		\draw (96.center) to (98.center);
		\draw (98.center) to (99.center);
		\draw (98.center) to (100.center);
		\draw (100.center) to (101.center);
		\draw (102.center) to (100.center);
		\draw (102.center) to (103.center);
		\draw (96.center) to (97.center);
		\draw (97.center) to (107.center);
		\draw (99.center) to (104.center);
		\draw (101.center) to (105.center);
		\draw (103.center) to (106.center);
		\draw (104.center) to (108.center);
		\draw (105.center) to (108.center);
		\draw (108.center) to (109.center);
		\draw (109.center) to (110.center);
		\draw (104.center) to (107.center);
		\draw (107.center) to (111.center);
		\draw (111.center) to (112.center);
		\draw (106.center) to (105.center);
		\draw (106.center) to (113.center);
		\draw (114.center) to (113.center);
		\draw (112.center) to (115.center);
		\draw (115.center) to (110.center);
		\draw (110.center) to (116.center);
		\draw (114.center) to (116.center);
		\draw (119.center) to (121.center);
		\draw (125.center) to (119.center);
		\draw (125.center) to (126.center);
		\draw (126.center) to (124.center);
		\draw (124.center) to (123.center);
		\draw (133.center) to (131.center);
		\draw (131.center) to (132.center);
		\draw (132.center) to (134.center);
		\draw (112.center) to (120.center);
		\draw (120.center) to (121.center);
		\draw (127.center) to (121.center);
		\draw (134.center) to (127.center);
		\draw (122.center) to (114.center);
		\draw (123.center) to (122.center);
		\draw (130.center) to (123.center);
		\draw (133.center) to (130.center);
		\draw (115.center) to (118.center);
		\draw (118.center) to (119.center);
		\draw (116.center) to (117.center);
		\draw (117.center) to (124.center);
		\draw (125.center) to (128.center);
		\draw (126.center) to (129.center);
		\draw (128.center) to (132.center);
		\draw (129.center) to (131.center);
		\draw [style=Arrow] (147.center) to (148.center);
\end{tikzpicture}

    \caption{Schematic of compositions. First tensor along rows then compose the rows in the column direction. Note the target of a row being the source of a column allows for different tensor decompositions, which in general can be of different length. The interstices are to mark the relative ``cuts'' of these tensors. In a \pher{} UFC the tensor lengths agree ---per unique decomposition.}
    \label{fig:compositionsfig}
\end{figure}

%\begin{nota}\mbox{}
%  \label{rmk:brick}
  In light of the previous proposition, we will use the following convenient and suggestive notation
  \begin{equation}
    \gent{\Gamma}{\Sigma}{n}:=\G_n\circ\Sigma_n \codco \Gamma_1 \circ \Sigma_1
  \end{equation}
  For later purposes, it will be convenient to think of this as a column as in Figure~\ref{fig:compositionsfig}.
  Note that these standard forms are not unique, as there are relations.
  For instance, some of the $\G_i$ or $\Sigma_i$ can be identities. A standard form for an identity is $id\stackrel{id}{-}$.

The composition of generators is given by concatenating the diagrams or columns.
\begin{equation}
\begin{aligned}
    &(\gent{\Gamma}{\Sigma}{n})
    \circ
    (\gent{\Gamma'}{\Sigma'}{m})\\
    &=
    \gent{\Gamma}{\Sigma}{n} \gent{\Gamma'}{\Sigma'}{m}
\end{aligned}
\end{equation}

%\label{item:seam}
In terms of the brick wall pictures, the monoidal structure is given by placing two brick walls side-by-side and including them into a bigger diagram, where one adds extra identities if the diagrams are not the same length.
In particular, if the two generators have the same length, the product is given by
\begin{eqnarray}
\label{Gammaoteq}
  && (\gent{\Gamma}{\Sigma}{n}) \ot (\gent{\Gamma'}{\Sigma'}{n}) \nn\\
  &=& \G_n \ot \G'_n
      \stackrel{\Sigma_n \ot \Sigma'_n}{-}
      \G_{n-1} \ot \G'_{n-1}
      \stackrel{\Sigma_{n-1} \ot \Sigma'_{n-1}}{-} \cdots \stackrel{\Sigma_{2}\ot \Sigma'_{2}}{-}
      \G_{1}\ot \G'_1
      \stackrel{\Sigma_{1}\ot \Sigma'_{1}}{-} \nn \\
  &=&\gent{\Gamma''}{\Sigma''}{n}
\end{eqnarray}
On the other hand, given generators $\gent{\Gamma}{\Sigma}{n}$ and  $\gent{\Gamma'}{\Sigma'}{m}$ with say $n>m$ formally add $n-m$ identity rows/entries
$\gent{\G}{\Sigma}{m} id\cdots \stackrel{id}{-}id$ to get a length $n$ generator and use the formula above.
  Note that there is an ambiguity of where to add the identities as all these formal expressions   actually represent the same morphism.
  The well-definedness is guaranteed by the interchange equation.

\begin{lem}
\label{lem:irred}
A generator $\gent{\Gamma}{\Sigma}{n}$ of Proposition~\ref{prop:alternate} is $\bt$-reducible if and only if it has a brick--wall representative with a vertical seam. \qed
\end{lem}

For our intended application to the plus constructions, we take the morphisms $\gamma,\mu$ on one hand and the morphisms $(\s,\s')$ on the other.

\begin{df}
   The subcategory $\spure(\catplus{\C,P})\subset \catplus{\C,P}$ is the wide monoidal subcategory with only the $\g$ morphisms.
   The subcategory $\spure(\monplus{\M,P})\subset \monplus{\M,P}$ is the wide  monoidal subcategory which only has the $\g$ and $\mu$ morphisms.

    The subcategory $\spure(\locmonplus{\M,P})\subset \locmonplus{\M,P}$ is the wide  monoidal subcategory which contains  the morphisms $\g$, $\mu$ and $\mu^{-1}$. We furthermore let $\overline{\spure}(\locmonplus{\M,P})$
    be the subcategory which is the image of $\spure(\monplus{\M,P})$ and let $\Loc$ be the wide category generated by the $\mu^{-1}$ morphisms. Together they generate $\spure(\locmonplus{\M,P})$.
The subcategory $\spure(\redmonplus{\M,P})\subset \redmonplus{\M,P}$ is the wide monoidal subcategory which only has the morphisms $\bar \g$.
 In all cases, $\I(\anyplus)=Mor(P^{op}\da P)^\bt$, where for $\locmonplus{\M}$ this means  the image in the localization.  These will be called the {\em base morphisms of the plus construction} or action by base morphisms.  The morphisms generated by the $\scs$ will be called {\em internal} base morphisms, while the morphisms given by the $\tau^{12}$ (cf.\ Definition \ref{def:mnc} \eqref{commu}) will be called {\em external}.
 %let $\I$ be the image of $\Iso(\monplus{\M})$, viz.\ it does not contain the morphisms $\mu$ or $\mu^{-1}$.
\end{df}

  It is clear that the monoidal subcategories $\spure(\anyplus)$ and $\I(\anyplus)$ generate. Note that $\spure(\anyplus)$ is not symmetric, since all of the commutators are  in $\I(\anyplus)$. The notation $\spure$ should remind of ``planar'' and $\I$ of isomorphisms or 2-morphisms.

%\begin{df}
% \begin{equation}
%     \begin{tikzcd}
%     \Phi'
%     \ar[r, "\G"]
%     \ar[d, dashed, "\Sigma'"]
%     &\Psi
%     \ar[d, "\Sigma"] \\
%     \Phi
%     \ar[r, dashed, "\Gamma'"]
%     &\Psi'
%     \end{tikzcd}
% \end{equation}
%\end{df}

\begin{lem}
\label{lem:mcrossed}
%\label{prop:mcrossed}
%The plus categories have crossed pairs of generators.
The category $\catplus{\C, P}$ is $(\spure(\catplus{\C,P}),\I(\catplus{\C,P}))$--crossed.
If a monoidal category $\M$ has factorizable pointing, e.g.\ for the standard gcp, then \begin{enumerate}
    \item  $\monplus{\M,P}$ is
    $(\spure(\monplus{\M,P}),\I(\monplus{\M},P))$--crossed.
    \item $\locmonplus{\M,P}$ is $(\spure((\locmonplus{\M,P}),\I(\locmonplus{\M,P}))$--crossed, and
     \item $\redmonplus{\M,P}$ is $(\spure(\redmonplus{\M,P}),\I(\redmonplus{\M,P})$--crossed. In particular, this is true for the standard plus construction $\M^+$.
\item If $P$ is a factorizable pointing, then  $\spure(\locmonplus{\M,P})$ is $(\overline{\spure}(\locmonplus{\M,P},\Loc)$--crossed.
\end{enumerate}
\end{lem}
\begin{proof} The first statement follows from outer equivariance in Definition \ref{def:cplus}.
% MM: isoequi was commented out
% follows from \eqref{isoequi}
The second statement part~(a) follows from  the first diagram in \eqref{eq:factorizable}  under the inversion of the horizontal arrows in the first diagram, which is possible by assumption. Part~(b) follows from  equivariance (Definition~\ref{def:mnc}~\eqref{eq:muequi}) by inverting the morphisms $\mu$.
This restricts to $\redmonplus{\M,P}$, from whence we obtain Part~(c).
The last statement follows from the third diagram in  \eqref{eq:factorizable}.
\end{proof}

\begin{cor}\mbox{}
  \label{cor:decomp}
 Any morphism in $\catplus{\C,P}$ can be written as $\G\Sigma$ with $\G\in \spure(\catplus{\C,P})$
and $\Sigma\in \catplus{\C,P}$.
Furthermore, if $\M$ has factorizable pointing then
\begin{enumerate}
    \item  Any morphism in $\monplus{\M,P}$ can be written as $\G\Sigma$ with $\G\in \spure(\monplus{\M,P})$ and $\Sigma\in \I(\monplus{\M,P})$.
    \item Any morphism in $\locmonplus{\M,P}$ can be written as $\G\Sigma$ with $\G\in \spure(\locmonplus{\M,P})$ and $\Sigma\in \I(\locmonplus{\M,P})$.
    \item Any morphism in $\redmonplus{\M,P}$ can be written as $\G\Sigma$ with $\G\in \spure(\redmonplus{\M,P})$ and $\Sigma\in \I(\redmonplus{\M,P})$. In particular, this is true for $\M^+$.
  \item If  $P$ is additionally a factorizable pointing then any morphism can be written as $\G \Sigma M^{-1}$ with $\G\in \overline{\spure}(\locmonplus{\M})$, $\Sigma\in \I(\locmonplus{\M,P})$ and $M^{-1} \in \Loc(\locmonplus{\M,P})$.
    \qed
\end{enumerate}
\end{cor}

In the symmetric case,   $\I(\anyplus)$ is generated by the $\scs$ morphisms and the commutators $\t^{12}_{\bt}:\phi_1\bt\phi_2\stackrel{\sim}{\to} \phi_2\bt \phi_1$, which satisfy the relation $\tau_\bt^{12}[(\s_1,\s'_1)\bt(\s_2, \s'_2)]\tau_\bt^{12}=(\s_2,\s'_2)\bt(\s_1, \s'_1)$.

\begin{df}
Define the following subcategories of $\I(\monplus{\M,P})$: $\I_{ext}$ is generated by the external morphisms $\t_\bt^{i \, i+1}$ and $\I_{int}$ is generated by the internal morphisms $\scs$.
We define $\spure(\anyplus)\I_{\it int/ext}$ to be the subcategory generated by $\spure(\anyplus)$ and
$\I_{int}(\anyplus)$ respectively $\I_{ext}(\anyplus)$ of $\anyplus$.
\end{df}

The following results are straightforward:
\begin{prop}
  \label{prop:sigma}
 $\I(\anyplus)$ is both $(\I_{ext},\I_{int})$-crossed and $(\I_{int},\I_{ext})$--crossed. Hence, any isomorphism can be written as $\Sigma  \boldsymbol{\s}$ or $\boldsymbol{\s}\Sigma$ where $\Sigma\in \I_{ext}$  and $\boldsymbol{\s}\in \I_{int}$. Letting $\I_{ext}$ and $\I_{int}$ be the corresponding subcategories in $\I(\anyplus)$,
this decomposition refines all the decompositions in Corollary~\ref{cor:decomp}.

 Any morphism in $\spure(\anyplus)\I_{ext/int}$ can be written as $\G\Sigma$ with $\G\in \spure(\anyplus)$ and $\Sigma\in \I_{ext}(\anyplus)$ respectively $\Sigma \in \I_{int}(\anyplus)$. \qed
\end{prop}

 \subsection{The structure of $\catplus{\C,P}$}
 \label{par:cpgen}
 Using the decomposition, we will assemble the category by building it up from the three subcategories.
%\subsubsection{The structure of $\spure(\catplus{\C})$}
For $\spure(\catplus{\C,P})$, removing the associativity brackets leaves the monoidal generators
 $\g_{\phi_1\kdk\phi_n}$.
These compose as follows.
Given composable tuples $(\phi_1^i\kdk \phi_{n_i}^i)$ for $i=1\kdk k$, set
$\psi_i=\phi_1^i\codco\phi_{n_i}^i$ then
%\begin{equation}
%\label{gammacompeq}
$\g_{\psi_1\kdk\psi_k}\circ [\gamma_{\phi_1^1\kdk\phi_{n_1}^1}\bt\dots\bt\g_{\phi^k\kdk\phi_{n_k}^k}  ]=\gamma_{\phi_1^1\kdk\phi_{n_1}^1\kdk \phi_1^k\kdk\phi_{n_k}^k}$.
%\end{equation}
% The analogous statement holds in $\spure(\plus{M})$
% with $\bt$ replaced by $\ot$ and $\g$ by $\bar \g$
If one thinks of the generators as directed, or equivalently rooted, linear graphs whose vertices are labeled with $\phi_i$, then this composition agrees with the behavior of graph insertion, see e.g.\ \cite{feynman} for a concrete definition. The monoidal product $\bt$ becomes disjoint union in this interpretation.
\begin{prop}
\label{prop:spurecplus}
The generators above freely generate $\spure(\catplus{\C,P})$ under the monoidal product $\bt$ and $\spure(\catplus{\C,P})$ with the basis of objects $\V=\Mor(M)$ as a discrete groupoid is a cubical non--Sigma Feynman category. The  $\gamma_{\phi_0,\phi_1}$ morphisms are degree $1$ generators.
\end{prop}
\begin{proof} These morphisms clearly generate and are closed under composition.
There are no relations between the different generators, associativity has been incorporated, and
all other relations involve the isomorphisms, see \S\ref{par:cplus}. The statement about being a Feynman category follows readily. There are no isomorphisms and the $\g_{\phi_0,\phi_1}$ are by definition generators.
The generator $\g_{\phi_0\kdk \phi_n}:\phi_0\bdb \phi_n\to \phi_0\codco \phi_n$ has degree $n$ and there are indeed $n!$ decompositions of it by stipulating an order of turning the $n$ symbols $\bt$ into $\circ$ using one morphism $
id \bdb id \bt \gamma_{\phi_i, \phi_{i+1}}\bt id \bdb id$.
\end{proof}

\begin{prop}
\label{prop:spurecplusext}
The monoidal generators of morphisms of  $\spure(\catplus{\C,P})\I_{ext}$ can be written as pairs $(\g_{(\phi_n\kdk \phi_1)},\s)$ with $\s\in \Sn$. The composition is given by the wreath product.
These can be alternatively thought of as a decorated linear rooted tree whose vertices are labeled by morphisms of $\C$ and an enumeration of the vertices.
$\spure(\catplus{\C,P})\I_{ext}$  is a cubical Feynman category with discrete $\V=\Mor(\C)$ and proper degree function in which the $\g_{\phi_0,\phi_1}$ have degree $1$.
\end{prop}
\begin{proof}
The first statement follows from Proposition~\ref{prop:sigma}. For the second, we interpret the morphism
\begin{equation}
    \phi_1\bdb\phi_{n}\stackrel{\s}{\to}
\phi_{\s^{-1}(1)}\bdb\phi_{\s^{-1}(n)}
\xrightarrow{\g_{\phi_{\s^{-1}(1)}\kdk\phi_{\s^{-1}(n)}}} \phi
\end{equation}
as a linear rooted tree whose vertices are decorated by the $\phi_i$. This fixes the map $\g$.
The enumeration of the vertices  gives the source, namely the $\bt$ product of
the morphisms decorating the vertices in that order.
The statement about being a Feynman category follows readily. For the degree functions, the  added morphisms are indeed isomorphisms and have degree $0$. Up to these isomorphisms the presentation of a morphisms in generators is again given by enumerating the symbols $\bt$ in the source.
\end{proof}

%\subsubsection{Adding internal isomorphisms}

To add in internal basic morphisms, recall that  $(\s,\s')\phi=P(\s')\phi P(\s)$ and $(s,s')=(\s,id)(id,\s')=(\s',id)(id,\s)$. Thus the action can be split up into pre-- and post--composing,
and  utilizing inner and outer equivariance the generators $\gent{\Gamma}{\Sigma}{n}$ take on the form:
\begin{multline}
\gen{\phi}{\sigma}{n}:= \\\quad
\gamma_{\s_0 \circ \phi_1\circ\s_1\kdk \phi_{n-1}\circ\s_{n-1},\phi_n}\circ
[(\id ,\s_0)(\phi_1)\odo(id, \s_{n-2})(\phi_{n-1})\ot (\s_n^{-1}, \s_{n-1})(\phi_n)]
\end{multline}
{\sc NB:} Due to the equivariance w.r.t.\ base morphisms, there are several other ways to write this morphism, by ``distributing'' the action of the base morphisms, e.g.:\\
$
        \gen{\phi}{\sigma}{n}
        = (\s^{-1}_n, \s_{0})(\phi_1\s_1\codco\s_{n-1}\phi_{n-1}\phi_n)$\\
\hspace*{\fill}$\gamma_{\phi_1, \s_1\phi_2.\kdk \s_{n-1}\phi_{n}}
[(\id,\id)(\phi_1)\ot (\id,\s_1)(\phi_2)\odo(id, \s_{n-1})(\phi_{n})]
    $.

\noindent The following identities hold:\\
(i) Concatenation:
 $
         (\gen{\phi}{\sigma}{n})(\gen{\phi'}{\sigma'}{m})   =$\\
\hspace*{\fill}  $    \stackrel{\sigma_0}{-}\phi_1\stackrel{\sigma_1}{-}\phi_2\stackrel{\sigma_2}{-} \cdots \stackrel{\sigma_{n-1}}{-}\phi_n\stackrel{\sigma'_0  \sigma_{n}}{-}\phi'_1\stackrel{\sigma'_1}{-}\phi'_2\stackrel{\sigma'_2}{-} \cdots \stackrel{\sigma'_{m-1}}{-}\phi'_m\stackrel{\sigma'_{m}}{-}
  $.\\
 (ii) Action of isomorphisms:
$       (\nu',\nu')  \gen{\phi}{\sigma}{n}  (\t_1, \t'_1)\bdb (\t_n,\t'_n)=$\\
$        \stackrel{\nu\sigma_0}{-}\phi_1\stackrel{\sigma_1\tau'_1}{-}\phi_2\stackrel{\tau_1\sigma_2\tau'_2}{-}\cdots \stackrel{\tau_{n-2}\sigma_{n-1}\tau'_{n-1}}{-}\phi_n\stackrel{\sigma_{n}\nu}{-}
$.\\
(iii) Interchange:
%\label{interchagegeneq}
$\mu
\left[
\stackrel{\sigma_0}{-}\phi_1\stackrel{\sigma_1}{-}\phi_2\stackrel{\sigma_2}{-} \cdots \stackrel{\sigma_{n-1}}{-}\phi_n\stackrel{\sigma_{n}}{-}\bt
\stackrel{\t_0}{-}\psi_1\stackrel{\t_1}{-}\psi_2\stackrel{\t_2}{-} \cdots \stackrel{\t_{n-1}}{-}\psi_n\stackrel{\t_{n}}{-}
\right]
=$\\ \hspace*{\fill} $
\stackrel{\sigma_0\ot \t_0}{-}\phi_1\ot \psi_1 \stackrel{\sigma_1\ot \t_2}{-}\phi_2\ot \psi_2\stackrel{\sigma_2\ot \t_2}{-} \cdots \stackrel{\sigma_{n-1}\ot\t_{n-1}}{-}\phi_n\ot \psi_n
\stackrel{\sigma_{n}\ot\t_n}{-}[\mu^{\bt n}] c^+_{n,n}
$,\\
where $C^+_{n,n}$ is the $(n,n)$--shuffle  shuffling in the $\psi_i$ to the right of the $\phi_i$.

Thinking of the generators as elements in $\rho^{\sqb n}$ (i) and (ii) are the fact that this composition descends to a map $\rho^{\sqb n}\sqb\rho^{\sqb m}\to \rho^{\sqb n+m}$,
and (iii) is a consequence of $\rho$ being an MBM.
Graphically, these can be regarded as linear rooted
b/w bipartite trees with a black root, whose  vertices are decorated by the $\phi_i$ and the edges are decorated by the $\sigma_i$. The enumeration left to right corresponds to the enumeration starting at the root, due to the function notation for composition.
Alternatively, these can be seen as b/w bipartite linear rooted trees,
where the black vertices are decorated by elements $\s_i$ and white vertices are decorated by elements $\phi_i$.

\begin{prop}
\label{prop:cplusint}
The morphisms of $\spure(\catplus{\C,P})\I_{int}$ are monoidally freely generated by the $\gen{\phi}{\sigma}{n}$.
$\spure(\catplus{\C})\I_{int}$  is a  non--Sigma Feynman category with groupoid $\V=\Iso(P\da P)$ and a degree function in which $\g_{\phi_0,\phi_1}$ are the degree $1$ generators and base--morphisms are of degree $0$. If these are invertible, then the degree function is proper and the category cubical.
\end{prop}
\begin{proof}
By Proposition~\ref{prop:sigma} and Proposition~\ref{prop:spurecplus},
all the morphisms are of the type $\g_{\phi_1\dots\phi_n}\boldsymbol{\s}$.
If $\tilde\phi_i=(\s_i,\s'_i)(\phi_i)$ and
$\boldsymbol{\s}=\scsp{\s_1}{\phi_i}{\s'_n}\bdb\scsp{\s_n}{\phi_n}{\s'_n}$,
then
$
    \g_{\tilde \phi_n\kdk \tilde\phi_1}\boldsymbol{\s}=\stackrel{\s'_1}{-}\phi_1\stackrel{\s_1\s'_2}{-}{\phi_2}
    \stackrel{\s_2\s'_2}{-}\dots \stackrel{\s_{n-1}\s_n}{-}\phi_n\stackrel{\s_n}{-}
$.
Thus the purported elements generate. The only new relations present in $\spure(\catplus{\C})\I_{int}$ are inner equivariance, which is taken care of in each individual generator. Hence these generators are indeed independent. Note that there are no non--trivial isomorphisms in $\spure(\catplus{\C}$. The statement about being a Feynman category follows readily after identifying $(\s,\s')$ with $(\s^{-1}\Da \s)$.
The statement about the degree function is clear. To be proper, the base--morphisms, which are of degree $0$ need to be isomorphisms. The
cubical structure then is as in \ref{prop:spurecplusext}.
\end{proof}

%   It and $\tilde\phi_i=\sdsp{\s_i}{\phi_i}{\s'_i}$ then
%   \begin{equation}
%       \g_{\tilde \phi_n\kdk \tilde\phi_1}\boldsymbol{\s}:\phi_n\btb \phi_1\to \tilde \phi_n\codco\tilde \phi_1
%   \end{equation}
%  but due to outer equivalence

%  \subsubsection{General morphisms}
\begin{thm}
\label{thm:cplus-equals-FC}
%Result!
The morphisms of $\catplus{\C,P}$ are freely monoidally generated by linear rooted trees together with a labeling of
the vertices by morphisms of $\C$ and edges by morphisms $P(\sigma)$ and an enumeration of the vertices.
 $\catplus{\C,P}$ is a  Feynman category with groupoid $\V=\Iso(P\da P)$  and a degree function in which $\g_{\phi_0,\phi_1}$ are the degree $1$ generators and base--morphisms are of degree $0$. If these are invertible, then the degree function is proper and the category cubical.
These statements also hold for the hyp version.  The standard gcp version is cubical.

The unital versions are Feynman categories with additional generators of degree $-1$, type $(0,1)$ generators corresponding to the units. Counits are new type $(1,0)$ generators which are part of a hereditary UFC.
\end{thm}

\begin{proof}
Adding general base morphisms, we can write any morphism as $\G \boldsymbol{\sigma} P$ or $\G P\boldsymbol{\sigma}$ with  $\G\in \spure$, $\sigma \in \I_{int}$ and $P \in \I_{ext}$. The arguments are then parallel to the results above.

    Setting the degree of $\g$'s, the  degree function for the plus categories $\catplus{\C,P}$ is well defined as all relations are homogeneous with respect to this degree. It is also clearly proper,
    as the $\g$ morphisms except for the identity are not invertible, the isomorphisms are the isomorphisms in the relative twisted arrow category, which is equivalent to $\Iso(P\da P)$ via identifying $scs$ with $(\s^{-1}\da \s)$.

    In $\catplus{\C,P}$, any linear tree represented by a morphism $\gamma_{\phi_0,\kdk,\phi_n}$
  up to isomorphisms has $n!$ decompositions as iterations of $id\bdb id \bt\g_{\phi_i,\phi_{i+1}}\bt id\bdb id$ where nearest neighbors are composed. This corresponds to an edge contraction in the linear tree.

  The statements about the unital version is straightforward. For the counital version, generators of that type are not allowed in an FC, but do yield \pher{} UFC.
  In hyp, the degree $-1$ morphisms become isomorphisms of degree $0$. The degree function is then again proper and if the base morphisms are invertible cubical by the previous arguments. This is the case for the standard gcp plus construction.
\end{proof}

\noindent {\sc NB:} In the unital and counital case, we have generators corresponding to the spans $\emptyset \to \{*\} \leftarrow \{*\}$ and $ \{*\} \rightarrow \{*\}\leftarrow\emptyset$ as in section \ref{cospanexpar}.
 Proposition~\ref{prop:cplusint} means that the generators form a groupoid colored non--Sigma operad given a suitably defined notion of groupoid colored in the enriched case.
  Similarly, Theorem~\ref{thm:cplus-equals-FC} says that the generators form a groupoid colored  operad, see \cite[\S1.11.2]{feynman}, in the enriched case.
 These constructions can be seen as a categorification of the plus construction on the trivial category which yields the FC for monoids, cf.\ \cite[Proposition 3.20, Proposition 3.34]{feynmanrep}.

\subsection{Formulas, cells and graphs}
To describe the morphisms in $\monplus{\M,P}$, we define several formalisms.
%\subsubsection{Formulas}
A {\em fully bracketed irreducible pre--formula} is a formal expression
of the two formal binary operations $\circ$ and $\ot$. For instance, $(-\circ (-\circ -))\ot (-\circ -)$. The {\em arity} of a formula is the number of the symbol ``$-$'' which is the number of operations minus one.
A fully bracketed  pre--formula gives rise to a flow chart and vice--versa. This is a planar planted binary tree with black vertices for the binary operations $\ot$ and a white vertex for the operation $\circ$. The level of nesting of brackets is the distance to the root vertex plus one, if the outside parenthesis are of level $1$.
To model associative binary operations, we use reduced pre--formulas.
That is, use $(-\circ-\circ -)$ to represent both $((-\circ-)\circ-)$ and $(-\circ(-\circ-))$ and likewise for $\ot$.
% for instance $(-\circ-\circ-)\ot (-\circ -)$.
Replacing a nested expression of brackets of the same type by just one bracket defines reduced pre--formulas.
In the flow chart, this becomes particularly transparent and one obtains a planar planted b/w bi--partite tree:
Call an edge black (resp.\ white) if it is between two black (respectively white) vertices, the edges between a black and white vertex will be called \emph{mixed}.
The associativity equation acts as usual by edge collapses and expansions of black and white edges, see e.g.\ \cite{woods,del}. Contracting all black and white edges, one is left with only mixed edges, that is a black and white bipartite  tree.

Note that irreducible pre--formulas naturally form a non--Sigma operad by substitution and hence generate a non--Sigma Feynman category.
As flow charts, this is gluing the leaves/inputs to the root flag/output.

A {\em pre--formula} (fully bracketed or reduced) is a formal conjunction of fully bracketed irreducible pre--formulas by a associative formal binary operator $\bt$ ---for instance
\begin{equation}
\label{eq:preformula}
f=((-\circ (-\circ -))\ot (-\circ -))\bt(-\circ-)\bt (-\ot-)
    \end{equation}
in the fully bracketed case. The arity is the sum of arities.
As flow charts, pre--formulas are planar, viz.\ ordered, forests of trees of the given type

\begin{dfprop}
\label{dfprop:formulascubicalFC}
 Consider the category $\F_{b/w-b}$
whose objects are the natural numbers and whose morphisms $\Hom(n,m)$
are ordered forests of b/w binary planar planted trees with $m$ trees having a total of $n$ leaves.  This is a non--Sigma Feynman  category with composition given by gluing roots to leaves.
This has one basic object $1$ and the basic $(n,1)$-morphisms are b/w binary planar planted trees.
Isomorphically, the set of basic $(n,1)$-morphisms are the $n$--ary irreducible fully bracketed pre-formulas.

 Similarly, we obtain a category $\F_{\it form}$  by taking the same objects but letting the basic $(n,1)$-morphisms be the planar planted bipartite trees and letting the general morphisms be ordered forests.
  Isomorphically, we can use reduced formulas as morphisms with the $\boxtimes$-irreducible reduced formulas being the basic morphisms.

In both cases, the  morphisms are generated by $\gamma$ and $\mu$ in $Hom(2,1)$.
In the fully bracketed case, they generate freely and the length decrease is a proper degree function.
In the reduced case, they have quadratic relations and the degree function descends as a proper degree function.

Contracting all edges that connect two vertices of identical color ends a binary b/w tree to a b/w bipartite tree and this operation defines a degree preserving functor $\F_{b/w-b}\to \F_{\it form}$.
\end{dfprop}

\begin{proof}
Straightforward.
\end{proof}

\begin{rmk} The appearance of b/w bipartite trees suggests a connection to little 2-cubes by \cite{del,LucasThesis,Brinkmeier}.This is made precise in \S\ref{par:boxpar} below.
The fact that the $(2,1)$ generate can be taken to mean that there are two $B_+$ operators, cf.\ \cite{KZhang}.
\end{rmk}
%\subsubsection{Valid formulas}

Let $S$ be a set with two possibly colored   associative binary  operations $\circ$ and $\ot$ which satisfy the interchange equation. Here colored means that there are source and target maps for $S$ and the binary operations are only defined if they coincide, cf.\ e.g.\ \cite{KYM}.
An example anticipating the next section is furnished by the 2--morphisms of a double category.
Consider the free associative monoid $S^\bt$ on the set $S$.
The \emph{population} of an $n$--ary formula $f$ (fully bracketed or reduced)
by elements $\phi_1\kdk \phi_n$ of $S$
is the formal substitution of the $\phi_i$ into the $i$--th slot of the formula.

\begin{df}
A \emph{valid formula} (fully bracketed or reduced) is a population of a pre--formula whose evaluation is possible.
We will denote this evaluation by $eval(f)(\phi_1\kdk \phi_n)$ and call it the target of
$f(\phi_1\kdk \phi_n)$.
The source of $f(\phi_1\kdk \phi_n)$ is defined to be the expression
$\phi_1\bdb \phi_n$.
\end{df}

The morphism corresponding to the formula $\phi_1\circ \phi_2$ will be called
$\g_{\phi_1,\phi_2}:\phi_1\bt\phi_2\to \phi_1\circ \phi_2$, and
the morphism corresponding to the formula $\phi_1\ot \phi_2$ will be called
$\mu_{\phi_1,\phi_2}:\phi_1\bt\phi_2\to \phi_1\ot \phi_2$. The class of $\gamma$-morphisms and the class of $\mu$-morphisms generate under composition. They satisfy associativity relations and the interchange relation.

\begin{ex}
The expression
\begin{equation}
\label{eq:phi8}
f(\phi_1\kdk\phi_8)=\phi_1\circ(\phi_2\ot (\phi_3\circ\phi_4))\circ (\phi_5\ot \phi_6)\bt (\phi_7 \circ \phi_8)
\end{equation}
is a fully bracketed valid formula given by a population of \eqref{eq:preformula} if the source and target maps align properly for $\phi_1\circ(\phi_2\ot (\phi_3\circ\phi_4))\circ (\phi_5\ot \phi_6)$ and for $\phi_7 \circ \phi_8$.
The formula \eqref{eq:phi8} then
specifies a morphism from
$\phi_1\bdb\phi_6\to \psi_1\bt\psi_2$ where $\psi_1= \phi_1\circ(\phi_2\ot (\phi_3\circ\phi_4))\circ (\phi_5\ot \phi_6)$ and $\psi_2=\phi_7\circ \phi_8$. The morphism can be read off as
\begin{equation}
[\g_{\phi_1,\phi_2\ot (\phi_3\circ\phi_4),\phi_5\ot \phi_6)}\circ (id_{\phi_1}\bt\mu_{\phi_2,\phi_3\circ\phi_4}\bt \mu_{\phi_5,\phi_6})\circ (id_{\phi_1}\bt id_{\phi_2}\bt \g_{\phi_3,\phi_4}\bt id_{\phi_5}\bt id_{\phi_6}]
\bt \g_{\phi_7,\phi_8}
\end{equation}
\end{ex}

A formula such as $(\phi_1 \bt\phi_2)\circ \phi_3$ is never valid as there are no elements of $S$ so that the source of $\phi_3$ is given by $t(\phi_1)\bt t(\phi_2)$. The  pre--formula $(-\bt-)\ot -$ is also not valid.

For the flow chart/tree picture, a population corresponds to a
decoration of the inputs/leaves by  the $\phi_i$ in their order and the decoration of the output/root by the target,
see Figure~\ref{fig:population}.  We call the decoration valid if the corresponding formula is valid.
This entails that each edge is naturally labeled by the output of the compositions above it.
The following is readily checked.

\begin{figure}
  \centering
  \input{population.tikz}
  \caption{A flow chart representation of \eqref{eq:phi8}}
  \label{fig:population}
\end{figure}

\begin{dfprop}
\label{dfprop:frmcubical}
Valid formulas, both fully bracketed and reduced, form  a non-Sigma Feynman category whose basis of objects is given by $S$ and
whose basis of morphisms is given by valid populated irreducible formulas with the source and target as defined above.
We denote the monoidal category of reduced valid formulas by $\cF(S)$ and the corresponding Feynman category by $\FF_{\it form}(S)$.

 If $\F$ is the underlying category, i.e.\ $S=\Mor(\F)$, the  morphisms are generated  by $\gamma_{\phi_1,\phi_0}\in \F(\phi_1,\phi_0;\phi_1\circ\phi_0)$, whenever the pairs are composable and $\mu_{\psi_1,\psi_2} \in \F(\psi_1,\psi_2;\psi_1\ot\psi_2)$.
 In the fully bracketed case, they generate freely, and, in the reduced case, they have quadratic relations. For both Feynman category, length decrease is a proper degree function.
\end{dfprop}

\begin{proof}
This is just a rephrasing of structures.
 Being a non--Sigma Feynman category is checked readily. Here $\V=S$ as a discrete category and the underlying category $\cF$ of $\FF$ has objects $S^\bt$.  $\Indec$ is freely monoidally generated by the $(n,1)$-morphisms $\cF(\phi_1\kdk \phi_n;\phi)$. The statement about generation follows from the definition.  The statement about the degree function is straightforward.
 \end{proof}

%RK: Lost somehow
% \begin{ex}
% The operadic decomposition  into the iterated substitution of the  two formulas $\phi_1\ot \phi_2$ and $\phi_2\ot \phi_1$ with the renumbering given by the usual operadic renumbering, in short hand, omitting the colors,  is
% $(((\g\circ_1\mu)\circ_2 \g)\circ_3\mu)\circ_4\g$.
% \end{ex}

The morphism can be read off from a  valid formula as follows:
Given a valid formula, the target of the corresponding morphism is the value of the formula and the source is obtained by replacing all occurrences of $\circ$ and $\ot$ by $\bt$. The morphism expressed in generators is the combination of $\g$'s and $\mu$'s that changes the respective occurrences of $\bt$ to $\circ$ and $\ot$ whose iteration is determined by the nesting.
The FCs above as colored structures are obtained from their uncolored counterparts by a decoration and restriction, cf.\  \cite[\S2.5]{feynman} and more generally \cite[6.1.4]{decorated}.
    The decoration yields a Feynman category of populated formulas with a restriction to the valid ones.
 The  irreducible reduced formulas form  non--Sigma $S$-colored operad $\cF$ freely generated by two generators. Composition is given by substitution:
    For example, $(\phi_1\circ\phi_2)\bt \phi_3$ corresponds to $\mu_{\phi_1\circ\phi_2,\phi_3}\circ_1 \gamma_{\phi_1,\phi_2}$.  This colored operad is again quadratic.

\subsection{Cells and diagrams}

A particularly helpful way to think about fully bracketed valid formulas is as a
diagram or pasting scheme in a 2--category  $\twocat$ with a  composition ordering.
The two operations are horizontal $\circ_h$ and vertical composition $\circ_v$.
Equivalence classes  defined by the associativity relations for the decomposition correspond to compatible enumerations of the cells. There are again three levels: unpopulated diagrams, populated diagrams and valid diagrams. In the application $\twocat=\underline{\M}$, the 2-category has one object, the 1-morphisms are given by the objects of $\M$ and the two--morphisms are given by morphisms of $\M$.
 The vertical composition of the 2-category is composition of morphisms and the horizontal composition is the tensor product. String diagrams provide a graphical version.

\begin{nota}
With a view towards the application to the 2--category $\underline{\M}$, we will write $\circ:=\circ_v$ for the vertical composition and $\ot:=\circ_h$ for the horizontal composition. The maps for these compositions of 2--cells will be called $\g$ and $\mu$.
\end{nota}

 Considering the associated double category leads to decomposable tight square arrangements considered up to isotopies called \emph{basic decomposable box diagrams}.
 These basic box diagrams can be generalized to account for different sides of the
interchange equation.
There is a dual b/w bipartite graph for a basic box diagram and a b/w bipartite suspension graph. The former is related to string diagrams, while the latter contains
finer information about interchanges.

%\subsubsection{Pasting diagrams}
\label{par:algpar}
Starting from the 2--categorical view, unpopulated diagrams can be seen as abstract diagrams in the categorical sense, where the index 2-category does not need  $0$ or $1$ cell labels or enumeration, as the $0$ and $1$--cells then need to have compatible labels, so this information can be omitted from the enumeration.
However, we do label and enumerate the 2--cells as  $1,\dots, n$.
Given a diagram in a double category $\D$, a decomposition into elementary composition steps involving only one horizontal or vertical composition defines a full enumeration of the cells. An enumeration of cells is {\em compatible} if and only if it is the image of the morphism.

A 2-functor $D\to \D$ of such a diagram into a 2--category   is directly  a  population where sources and targets match.
To obtain a formula, one has a 2--cell  $\phi_i$ for each $i$ and uses the following algorithm:
If two consecutively numbered cells $i$ and $i+1$ are composable either   horizontally or vertically, first compose them, this is the new decoration, and replace the number by $i$ and renumber $j$ to $j-1$ for $j>i$, and secondly
 record this in the formula as the morphism
    $\mu_{\phi_i,\phi_{i+1}}$ for horizontal composition respectively $\g_{\phi_i,\phi_{i+1}}$ for vertical composition.
 Repeat as long as possible.
 If there is only one 2--cell labeled by one morphism left, the enumeration is said to be {\em compatible} and the diagram is called {\em valid}.
  This can be seen as a map from the free monoid on 2--morphisms to 2--morphisms.
 The composition is the target of the formula and the source is $\phi_1\bdb \phi_n$,
 where $\bt$ is the free product. Ordered collections of such diagrams then provide maps from the free product to the free product.

%This corresponds to the linear order of the morphisms in the formula.
%\subsubsection{String diagrams}
It is well-known that the diagrams of a 2--category can be equivalently represented by string diagrams. These are the dual graphs. Given a diagram, there
is a vertex for each 2-cell and an edge for each 1-cell connecting the two vertices of the 2-cells it bounds with the inherited labeling.
The outer edges, those that are on the boundary of only one 2-cell, are leaves or tails. Each edge/tail is directed  induced by the source to target maps. The monoidal product is formal disjoint union and composition is by insertion into vertices. Note, the string diagrams need not be connected, which is why the monoidal product is formal.

\begin{ex}
The   formula $(\phi_3\circ\phi_2\circ\phi_1)\bt (\psi_2\circ \psi_1)$ corresponding to the morphism
given by $\mu_{\phi_3\circ\phi_2\circ\phi_1,\psi_2\circ\psi_1}\circ (\gamma_{\phi_0,\phi_1,\phi_2)}\bt \gamma_{\psi_2,\psi_1}):\phi_3\bt\phi_2\bt\phi_1\bt \psi_2\bt \psi_1\to (\phi_3\circ\phi_2\circ\phi_1)\bt (\psi_2\circ \psi_1)$ is given in Figure~\ref{fig:2cat} (A) with the decomposition given by first doing the vertical composition and then the horizontal decomposition.
If the first step in the composition is given by (B) and then the standard decomposition the formula reads $((\phi_3\circ\phi_2)\ot \psi_2)\circ(\phi_1\ot \psi_1)$. This has a different source and enumeration, namely $\phi_3\bt \phi_2\bt \psi_2\bt \phi_1\bt \psi_1$.
 The formula obtained by evaluating a different first horizontal composition as in (C) and then using the standard decomposition is $(\phi_3\ot \psi_2)\circ((\phi_2\circ\phi_1)\ot \psi_1)$ which again has a different source.
\end{ex}

\begin{figure}[h]
  %\centering
  \begin{subfigure}[c]{0.30\linewidth}
    \centering
    \begin{tikzcd}
      \bullet & \bullet & \bullet
      \arrow[""{name=L1}, curve={height=-30pt}, from=1-1, to=1-2]
      \arrow[""{name=L2}, curve={height=-10pt}, from=1-1, to=1-2]
      \arrow[""{name=L3}, curve={height=10pt}, from=1-1, to=1-2]
      \arrow[""{name=R1}, curve={height=-20pt}, from=1-2, to=1-3]
      \arrow[""{name=R2}, from=1-2, to=1-3]
      \arrow[""{name=R3}, curve={height=20pt}, from=1-2, to=1-3]
      \arrow[""{name=L4}, curve={height=30pt}, from=1-1, to=1-2]
      \arrow["\psi_1" marking, rotate=90, draw=none, from=R1, to=R2]
      \arrow["\phi_2" marking, rotate=90, draw=none, from=L2, to=L3]
      \arrow["\phi_3" marking, rotate=90, draw=none, from=L3, to=L4]
      \arrow["\phi_1" marking, rotate=90, draw=none, from=L1, to=L2]
      \arrow["\psi_2" marking, rotate=90, draw=none, from=R2, to=R3]
    \end{tikzcd}
    \begin{tikzpicture}[baseline=(current bounding box.center)]
      \tikzmath{\h = 0.8;}
      \node (4a) at (-0.4,4*\h) {};
      \node (4b) at (0.4,4*\h) {};
      \node(3a)[draw, minimum size=5mm] at (-0.4, 3*\h) {};
      \node(2b)[draw, minimum size=5mm] at (0.4, 2.5*\h) {};
      \node(2a)[draw, minimum size=5mm] at (-0.4, 2*\h) {};
      \node(1b)[draw, minimum size=5mm] at (0.4, 1.5*\h) {};
      \node(1a)[draw, minimum size=5mm] at (-0.4, 1*\h) {};
      \node(0a) at (-0.4, 0*\h) {};
      \node(0b) at (0.4, 0*\h) {};
      \begin{scope}[on background layer]
        \draw (4a) -- (3a) -- (2a) -- (1a) -- (0a);
        \draw (4b) -- (2b) -- (1b) -- (0b);
      \end{scope}
    \end{tikzpicture}
    \caption{}
  \end{subfigure}
  \begin{subfigure}[c]{0.30\linewidth}
    \centering
    \begin{tikzcd}
      \bullet & \bullet & \bullet
      \arrow[""{name=L1R1}, bend left = 60, from=1-1, to=1-3]
      \arrow[""{name=L2}, curve={height=-10pt}, from=1-1, to=1-2]
      \arrow[""{name=L3}, curve={height=10pt}, from=1-1, to=1-2]
      \arrow[""{name=R2}, from=1-2, to=1-3]
      \arrow[""{name=R3}, curve={height=20pt}, from=1-2, to=1-3]
      \arrow[""{name=L4}, curve={height=30pt}, from=1-1, to=1-2]
      \arrow["\phi_1 \otimes \psi_1" marking, rotate=90, draw=none, from=L1R1, to=1-2]
      \arrow["\psi_2" marking, rotate=90, draw=none, from=R2, to=R3]
      \arrow["\phi_2" marking, rotate=90, draw=none, from=L2, to=L3]
      \arrow["\phi_3" marking, rotate=90, draw=none, from=L3, to=L4]
    \end{tikzcd}
    \begin{tikzpicture}[baseline=(current bounding box.center)]
    \tikzmath{\h = 0.8;}
    \node (4a) at (-0.4,4*\h) {};
    \node (4b) at (0.4,4*\h) {};
    \matrix(level3) [draw, fill=white, column sep={0.8cm,between origins}, nodes={draw=none, minimum size=5mm}, inner sep = 0] at (0,3*\h)
    {
      \node(3a) {}; & \node(2b){}; \\
    };
    \node(2a)[draw, minimum size=5mm] at (-0.4, 2*\h) {};
    \node(1b)[draw, minimum size=5mm] at (0.4, 2*\h) {};
    \node(1a)[draw, minimum size=5mm] at (-0.4, 1*\h) {};
    \node(0a) at (-0.4, 0*\h) {};
    \node(0b) at (0.4, 0*\h) {};
    \begin{scope}[on background layer]
      \draw (4a) -- (3a) -- (2a) -- (1a) -- (0a);
      \draw (4b) -- (2b) -- (1b) -- (0b);
    \end{scope}
  \end{tikzpicture}
    \caption{}
  \end{subfigure}
  \begin{subfigure}[c]{0.30\linewidth}
    \centering
    \begin{tikzcd}
      \bullet & \bullet & \bullet
      \arrow[""{name=L1}, curve={height=-30pt}, from=1-1, to=1-2]
      \arrow[""{name=L2}, curve={height=-10pt}, from=1-1, to=1-2]
      \arrow[""{name=L3}, curve={height=10pt}, from=1-1, to=1-2]
      \arrow[""{name=R1}, curve={height=-20pt}, from=1-2, to=1-3]
      \arrow[""{name=R2}, from=1-2, to=1-3]
      \arrow[""{name=L4R3}, bend right = 60, from=1-1, to=1-3]
      \arrow["\psi_1" marking, rotate=90, draw=none, from=R1, to=R2]
      \arrow["\phi_2" marking, rotate=90, draw=none, from=L2, to=L3]
      \arrow["\phi_3 \otimes \psi_2" marking, rotate=90, draw=none, from=1-2, to=L4R3]
      \arrow["\phi_1" marking, rotate=90, draw=none, from=L1, to=L2]
    \end{tikzcd}
    \begin{tikzpicture}[baseline=(current bounding box.center)]
    \tikzmath{\h = 0.8;}
    \node (4a) at (-0.4,4*\h) {};
    \node (4b) at (0.4,4*\h) {};
    \node(3a)[draw, minimum size=5mm] at (-0.4, 3*\h) {};
    \node(2b)[draw, minimum size=5mm] at (0.4, 2*\h) {};
    \node(2a)[draw, minimum size=5mm] at (-0.4, 2*\h) {};
    \matrix(level1) [draw, fill=white, column sep={0.8cm,between origins}, nodes={draw=none, minimum size=5mm}, inner sep = 0] at (0,1*\h)
    {
      \node(1a) {}; & \node(1b){}; \\
    };
    \node(0a) at (-0.4, 0*\h) {};
    \node(0b) at (0.4, 0*\h) {};
    \begin{scope}[on background layer]
      \draw (4a) -- (3a) -- (2a) -- (1a) -- (0a);
      \draw (4b) -- (2b) -- (1b) -- (0b);
    \end{scope}
  \end{tikzpicture}
  \caption{}
  \end{subfigure}
  \caption{\label{fig:2cat}A diagram (A) and two different first compositions (B) and (C) represented on the left by pasting diagrams and on the right by string diagrams.}
\end{figure}

\begin{prop}
\label{prop:associnv}
The compatible enumeration of the morphism, as cells in the diagram or letters, is a complete invariant of the associativity relations. Thus, compatibly enumerated populated valid diagrams are in bijection with reduced valid formulas.
\end{prop}
\begin{proof}
The fact that it is an invariant of the associativity transformations is straightforward. The fact that it is complete follows from the fact that the only other relation possibly resulting in different  (de)compositions is the interchange relation \eqref{eq:intinter}. This however changes enumeration of the order and hence the source thereby changing the morphism.
\end{proof}

\begin{df}[Standard numbering]
Given a non--enumerated diagram,
there is a standard way to give a decomposition by prioritizing the vertical compositions over the horizontal ones. For this, first perform all possible vertical compositions then all possible horizontal ones, continue in this manner. This gives a particular formula and enumeration which we call the {\em standard} enumeration.
An irreducible formula is in {\em standard} form or a {\em standard formula}, if its cell diagram has a standard enumeration.
\end{df}

An example is given in Figure~\ref{fig:boxcell}.
In general, there may be more than one compatible enumeration. This is due to the interchange relation.
In particular, \eqref{interchangeeq} corresponds to the two enumerations of cells
$
 \begin{tikzcd}
      \bullet & \bullet & \bullet
      \arrow[""{name=L1}, curve={height=20pt}, from=1-1, to=1-2]
      \arrow[""{name=L2},  from=1-1, to=1-2]
      \arrow[""{name=L3}, curve={height=-20pt}, from=1-1, to=1-2]
      \arrow[""{name=R1}, curve={height=-20pt}, from=1-2, to=1-3]
      \arrow[""{name=R2}, from=1-2, to=1-3]
      \arrow[""{name=R3}, curve={height=20pt}, from=1-2, to=1-3]
      \arrow["2" marking, rotate=90, draw=none, from=R1, to=R2]
      \arrow["1" marking, rotate=90, draw=none, from=L2, to=L3]
      \arrow["3" marking, rotate=90, draw=none, from=L1, to=L2]
      \arrow["4" marking, rotate=90, draw=none, from=R2, to=R3]
    \end{tikzcd}
$
respectively
$
 \begin{tikzcd}
      \bullet & \bullet & \bullet
      \arrow[""{name=L1}, curve={height=20pt}, from=1-1, to=1-2]
      \arrow[""{name=L2},  from=1-1, to=1-2]
      \arrow[""{name=L3}, curve={height=-20pt}, from=1-1, to=1-2]
      \arrow[""{name=R1}, curve={height=-20pt}, from=1-2, to=1-3]
      \arrow[""{name=R2}, from=1-2, to=1-3]
      \arrow[""{name=R3}, curve={height=20pt}, from=1-2, to=1-3]
      \arrow["3" marking, rotate=90, draw=none, from=R1, to=R2]
      \arrow["1" marking, rotate=90, draw=none, from=L2, to=L3]
      \arrow["2" marking, rotate=90, draw=none, from=L1, to=L2]
      \arrow["4" marking, rotate=90, draw=none, from=R2, to=R3]
    \end{tikzcd}
$
 see also Figure~\ref{fig:interchangebox}.
 This means that the left hand side of \eqref{interchangeeq} is in standard form, while the right hand side is not.

\begin{lem}
\label{lem:interchangeenum}
All compatible enumerations are obtained from the standard enumeration by permutations stemming from the interchange relation.
\end{lem}
\begin{proof}
This is clear, since the interchange relation is the only relation other than associativity between the operations of horizontal and vertical composition.
\end{proof}

\begin{df}
We call a diagram  horizontally or ---in the case of $\underline{\M}$--- $\ot$-decomposable if its string diagram is disconnected. The horizontal irreducible, respectively $\ot$-irreducible, components are consequently the connected components of the string diagram.
\end{df}

In terms of the original diagram as a graph, this means that by removing a vertex, the irreducible diagram remains connected.
Dually, one can think of a reducible diagram as merged from irreducible ones by merging vertices.
Another related graphical  realization can be found in   \cite{JoyalStreetGeotensor}.
For the diagrams of Figure~\ref{fig:2cat}, (A) is horizontally reducible while (B) and (C) are horizontally irreducible. This can be read off from the string diagrams in Figure~\ref{fig:2cat}.

%\subsubsection{Box diagrams and their graphs}
\label{par:boxpar}
Reinterpreting the 2--category as a double category, the only vertical morphism is $id_*=\unit$.
Hence one obtains a box diagram of tight decomposable rectangles by expanding the nodes to vertical lines with segments according to the 2--cells.
This is closely related to the considerations of \cite{Dunn,Brinkmeier,BFSV} in that there is a non--Sigma operad structure which is a quotient of  the suboperad of decomposable cubes of tight little 2--cubes, see Remark~\ref{rmk:decompcube}.
``Tight'' means that the little squares fill out the entire cube and ``decomposable''
means precisely that the configuration is in the image of the two basic operations, see below. The quotient is by isotopies of moving line segments.
Conversely, given such a diagram, collapsing the horizontal line segments yields a two--cell diagram,
see Figure~\ref{fig:boxcell}. We will now make this precise.

\begin{figure}[h]
  \begin{subfigure}[c]{0.24\linewidth}
  \centering
  \begin{tikzpicture}[scale=0.9]
    \tikzstyle{dot}=[draw, circle, inner sep=2pt, fill=black]
    % width
    \tikzmath{\w = 1.2;}
    \useasboundingbox (-0.25,-0.25) rectangle (3*\w + 0.25, 4 + 0.25);
    % Squares
    \draw (0*\w,4) rectangle (2*\w,3);
    \node at (1*\w, 3.5) {$\phi_1$};
    \draw (2*\w,4) rectangle (3*\w,3);
    \node at (2.5*\w, 3.5) {$\phi_2$};
    \draw (0*\w,3) rectangle (1*\w,1);
    \node at (0.5*\w, 2) {$\phi_3$};
    \draw (1*\w,3) rectangle (3*\w,2);
    \node at (2*\w, 2.5) {$\phi_4$};
    \draw (1*\w,2) rectangle (3*\w,1);
    \node at (2*\w, 1.5) {$\phi_5$};
    \draw (0*\w,1) rectangle (3*\w,0);
    \node at (1.5*\w, 0.5) {$\phi_6$};
    % black dots
    \node[dot] at (0*\w,3) {};
    \node[dot] at (0*\w,1) {};
    \node[dot] at (1*\w,3) {};
    \node[dot] at (1*\w,2) {};
    \node[dot] at (1*\w,1) {};
    \node[dot] at (2*\w,3) {};
    \node[dot] at (3*\w,3) {};
    \node[dot] at (3*\w,2) {};
    \node[dot] at (3*\w,1) {};
  \end{tikzpicture}
  \caption{}
  \end{subfigure}
  \hfill
  \begin{subfigure}[c]{0.24\linewidth}
  \begin{tikzpicture}[scale=0.9]
    \tikzstyle{dot}=[draw, circle, inner sep=2pt, outer sep=1pt, fill=black]
    % width
    \tikzmath{\w = 1.2;}
    \useasboundingbox (-0.25,-0.25) rectangle (3*\w + 0.25, 4 + 0.25);
    % 2-cells
    \node at (1*\w, 3.5) {$\phi_1$};
    \node at (2.5*\w, 3.5) {$\phi_2$};
    \node at (0.5*\w, 2) {$\phi_3$};
    \node at (2*\w, 2.5) {$\phi_4$};
    \node at (2*\w, 1.5) {$\phi_5$};
    \node at (1.5*\w, 0.5) {$\phi_6$};
    % 0-cells
    \node[dot] (a) at (0*\w,2) {};
    \node[dot] (b) at (1*\w,2) {};
    \node[dot] (c) at (2*\w,3) {};
    \node[dot] (d) at (3*\w,2) {};
    % 1-cells
    \draw[-to, in=135, out=90, looseness=2] (a) to (c);
    \draw[-to, in=90, out=75, looseness=2.5] (c) to (d);
    \draw[-to, in=120, out=60, looseness=1] (a) to (b);
    \draw[-to, in=180, out=60, looseness=1] (b) to (c);
    \draw[-to, in=120, out=0, looseness=1] (c) to (d);
    \draw[-to,] (b) to (d);
    \draw[-to, in=-120, out=-60, looseness=1.5] (b) to (d);
    \draw[-to, in=-120, out=-60, looseness=1] (a) to (b);
    \draw[-to, in=-90, out=-90, looseness=1.75] (a) to (d);
  \end{tikzpicture}
  \caption{}
 \end{subfigure}
   \begin{subfigure}[c]{0.24\linewidth}
    \centering
    \begin{tikzpicture}[scale=.6]
      \tikzstyle{mor}=[draw, circle, inner sep=2pt]
      \tikzstyle{iso}=[draw, circle, inner sep=2pt, fill=black]
      \tikzmath{\h = -0.6;}
      \tikzmath{\w = 0.8;}
      \node [iso] (1a) at (-0.75*\w, 1*\h) {};
      \node [iso] (1b) at (1*\w, 1*\h) {};
      \node [mor] (2a) at (-0.75*\w, 2*\h) {};
      \node [mor] (2b) at (1*\w, 2*\h) {};
      \node [iso] (3a) at (-1.5*\w, 3*\h) {};
      \node [iso] (3b) at (0*\w, 3*\h) {};
      \node [iso] (3c) at (1*\w, 3*\h) {};
      \node [mor] (4a) at (-1.5*\w, 4*\h) {};
      \node [mor] (4b) at (0.5*\w, 4*\h) {};
      \node [iso] (5)  at (0.5*\w, 5*\h) {};
      \node [mor] (6)  at (0.5*\w, 6*\h) {};
      \node [iso] (7a) at (-1.5*\w, 7*\h) {};
      \node [iso] (7b) at (0.5*\w, 7*\h) {};
      \node [mor] (8)  at (-0.5*\w, 8*\h) {};
      \node [iso] (9)  at (-0.5*\w, 9*\h) {};
      % Edges
      \draw (1a) -- (2a); \draw (1b) -- (2b);
      \draw (2a) -- (3a); \draw (2a) -- (3b); \draw (2b) -- (3c);
      \draw (3a) -- (4a); \draw (3b) -- (4b); \draw (3c) -- (4b);
      \draw (4a) -- (7a);
      \draw (4b) -- (5) -- (6) -- (7b);
      \draw (7a) -- (8); \draw (7b) -- (8);
      \draw (8) -- (9);
      % Lines
      \node (Guide1L) at (-2.25*\w, 1*\h) {};
      \node (Guide1R) at (1.75*\w, 1*\h) {};
      \node (Guide2L) at (-2.25*\w, 3*\h) {};
      \node (Guide2R) at (1.75*\w, 3*\h) {};
      \node (Guide3L) at (-2.25*\w, 7*\h) {};
      \node (Guide3R) at (1.75*\w, 7*\h) {};
      \draw[dashed] (Guide1L.center) to (Guide1R.center);
      \draw[dashed] (Guide2L.center) to (Guide2R.center);
      \draw[dashed] (Guide3L.center) to (Guide3R.center);
    \end{tikzpicture}
    \caption{}
  \end{subfigure}
  \begin{subfigure}[c]{0.24\linewidth}
    \centering
    \begin{tikzpicture}[scale=.6]
      \tikzstyle{mor}=[draw, circle, inner sep=2pt]
      \tikzstyle{iso}=[draw, circle, inner sep=2pt, fill=black]
      \tikzmath{\h = -0.6;}
      \tikzmath{\w = 0.8;}
      \node [iso] (1) at (-0.2*\w, 1*\h) {};
      \node [mor] (2a) at (-1.5*\w, 2*\h) {};
      \node [mor] (2b) at (1*\w, 2*\h) {};
      \node [iso] (3) at (-0.2*\w, 3*\h) {};
      \node [mor] (4a) at (-1.5*\w, 4*\h) {};
      \node [mor] (4b) at (0.5*\w, 4*\h) {};
      \node [iso] (5)  at (0.5*\w, 5*\h) {};
      \node [mor] (6)  at (0.5*\w, 6*\h) {};
      \node [iso] (7) at (-0.5*\w, 7*\h) {};
      \node [mor] (8)  at (-0.5*\w, 8*\h) {};
      \node [iso] (9)  at (-0.5*\w, 9*\h) {};
      % Edges
      \draw (1) -- (2a); \draw (1) -- (2b);
      \draw (2a) -- (3); \draw (2a) -- (3); \draw (2b) -- (3);
      \draw (3) -- (4a); \draw (3) -- (4b); \draw (3) -- (4b);
      \draw (4a) -- (7);
      \draw (4b) -- (5) -- (6) -- (7);
      \draw (7) -- (8); \draw (7) -- (8);
      \draw (8) -- (9);
      % Lines
      \node (Guide1L) at (-2.25*\w, 1*\h) {};
      \node (Guide1R) at (1.75*\w, 1*\h) {};
      \node (Guide2L) at (-2.25*\w, 3*\h) {};
      \node (Guide2R) at (1.75*\w, 3*\h) {};
      \node (Guide3L) at (-2.25*\w, 7*\h) {};
      \node (Guide3R) at (1.75*\w, 7*\h) {};
      \draw[dashed] (Guide1L.center) to (Guide1R.center);
      \draw[dashed] (Guide2L.center) to (Guide2R.center);
      \draw[dashed] (Guide3L.center) to (Guide3R.center);
    \end{tikzpicture}
    \caption{}
  \end{subfigure}
  % \caption{\label{fig:bwgraphs} corresponding to the example in
  % \protect{Figure \ref{fig:boxcell}}. The orientation is downward. There is a merging/splitting of vertices when going from one diagram to the other.}
  \caption{\label{fig:boxcell} (A) The box diagram, (B) 2--cell, (C) planar directed b/w bipartite graph and  (D) suspension graph  realization described in Definition~\ref{df:planar-suspension} of the formula $f=(\phi_6\circ(\phi_3\ot (\phi_5\circ \phi_4)))\circ (\phi_1\ot \phi_2)$ in the standard enumeration. The underlying diagram is horizontally irreducible and has only one compatible enumeration}
\end{figure}

\begin{df}
 A {\em basic box diagram} is a class of decomposable
arrangements of boxes (little squares) obtained iteratively by starting with a unit square and subdividing  horizontally or vertically
 modulo   isotopies on the vertical and horizontal line segments.
 {\em The horizontal line segments can be moved} but are not allowed to cross each other. One also may not break triple crossing points or ``T-junctions'' but one {\em is} allowed to move the horizontal line segments to possibly align in a four point crossing and to break  such a crossing, compare Figure~\ref{fig:movesegments}.
{\em The vertical line segments may be moved}, and are not allowed to not cross each other {\em or} any of the intersection points, see Figure~\ref{fig:movesegments}.
 Given a box diagram, one obtains a cell diagram by shrinking all horizontal lines, see Figure~\ref{fig:boxcell} for an example.

A {\em generalized   box diagram} is  a class of box diagrams up to isomorphism,
 where  the horizontal isotopies are {\em not allowed to break four point intersections}, aka.\ cross, or ``+''--intersections,  or to create them.
Such a diagram is called {\em generic} if it has only ``T-intersections''.

\begin{figure}[h]
    \centering
    \input{movesegments.tikz}
\caption{\label{fig:movesegments} Different boxes obtained by moving the horizontal line segments. In the top row, moving the line segments into different non--generic positions is an equality in basic diagrams. The diagrams are different as generalized box diagrams.
If the line segments are aligned, one can perform horizontal compositions changing the diagram as indicated. The diagram in the bottom row are neither equivalent as basic nor as generalized box diagrams. The steps correspond to the (de)compositions in Figure \protect{\ref{fig:2cat}}}
\end{figure}
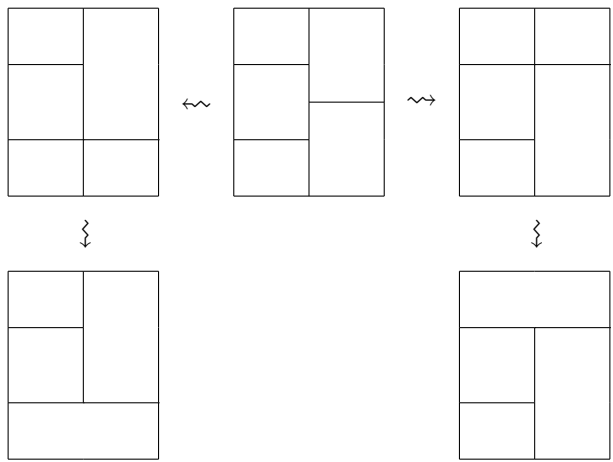

A {\em population} or coloring is given by labeling the boxes with 2--morphisms of $\twocat$ ---that is morphisms of $\M$ if $\twocat=\underline{\M}$. A population is {\em valid} if the compositions line up correctly.  That is for each full line segment, the tensor product of the target boundaries of the two cells bordering this line segment from above agree with the tensor product of the sources of the 2--cells bordering this line segment from below.
\end{df}
For example, for the population to be valid in Figure~\ref{fig:boxcell}, the following equality must hold  $t(\phi_1)\ot t(\phi_2)= s(\phi_3) \ot s(\phi_4)$.

As box diagrams, the generators $\mu_{\phi_1,\phi_2}$ and $\g_{\phi_1,\phi_2}$ are  a box with a vertical line respectively a box with a  horizontal line, and the following valid decorations:
\begin{equation}
\label{eq:boxgenerators}
    \begin{tikzpicture}[baseline={(current bounding box.center)}]
    % unit
    \tikzmath{\u = 0.9;}
    % shift
    \tikzmath{\s = 3.5*\u;}
    % picture 1
    \draw    (0*\u, 0.5*\u) rectangle (2*\u, 1.5*\u);
    \draw    (1*\u, 0.5*\u) -- (1*\u, 1.5*\u);
    \node at (0.5*\u, 1*\u) {$\phi_1$};
    \node at (1.5*\u, 1*\u) {$\phi_2$};
    % text
    \node at ({(2*\u)/2 + (0.5*\u + 1*\s)/2}, 1*\u) {and};
    % picture 2
    \draw    (0.5*\u + 1*\s, 0*\u) rectangle (1.5*\u + 1*\s, 2*\u);
    \draw    (0.5*\u + 1*\s, 1*\u ) -- (1.5*\u + 1*\s, 1*\u);
    \node at (1*\u + 1*\s, 1.5*\u) {$\phi_1$};
    \node at (1*\u + 1*\s, 0.5*\u) {$\phi_2$};
  \end{tikzpicture}
\end{equation}
Or just unlabeled boxes in the case of unpopulated diagrams.

Being decomposable means tautologically that it stems from an iteration as above.
The recursive recognition definition is that there is either a horizontal of a vertical line which cuts the diagram into two nonempty parts each of which has the same property.
The  box diagrams obtained in this way are decomposable
by definition  and valid formulas are  obtained by this  iterated substitution.
The ambiguity in bracketing  operations leading to  multiple parallel horizontal or vertical line segments in the choice is taken care of by associativity of composition respectively  the strict monoidal structure.
The basic   box diagrams   also already incorporate the interchange relation, while the generalized box diagram encodes the different sides differently
--- see Figure~\ref{fig:interchangebox}. The following is straightforward:

\begin{figure}[h]
  \begin{tikzpicture}[baseline={(current bounding box.center)}]
    % unit
    \tikzmath{\u = 0.8;}
    % shift
    \tikzmath{\s = 3*\u;}
    % left/right picture
    \draw (0*\u + 2*\s, 0*\u) rectangle (2*\u + 2*\s, 2*\u);
    \draw (0*\u + 2*\s, 1*\u ) -- (2*\u + 2*\s, 1*\u);
    \draw[blue, very thick] (1*\u + 2*\s, 0*\u ) -- (1*\u + 2*\s, 2*\u);
    \node at (0.5*\u + 2*\s, 1.5*\u) {1};  \node at (1.5*\u + 2*\s, 1.5*\u) {3};
    \node at (0.5*\u + 2*\s, 0.5*\u) {2};  \node at (1.5*\u + 2*\s, 0.5*\u) {4};
    % equals
    \node at ({(2*\u+0*\s)/2+(0*\u+1*\s)/2}, 1*\u) {$\rotatebox{180}{$\leadsto$}$};
    % four squares picture
    \draw (0*\u + 1*\s, 0*\u) rectangle (2*\u + 1*\s, 2*\u);
    \draw (0*\u + 1*\s, 1*\u ) -- (2*\u + 1*\s, 1*\u);
    \draw (1*\u + 1*\s, 0*\u ) -- (1*\u + 1*\s, 2*\u);
    \draw (1*\u + 1*\s, 0*\u ) -- (1*\u + 1*\s, 2*\u);
    % equals
    \node at ({(2*\u+1*\s)/2+(0*\u+2*\s)/2}, 1*\u) {$\leadsto$};
    % up/down picture
    \draw (0*\u + 0*\s, 0*\u) rectangle (2*\u + 0*\s, 2*\u);
    \draw[blue, very thick] (0*\u + 0*\s, 1*\u ) -- (2*\u + 0*\s, 1*\u);
    \draw (1*\u + 0*\s, 0*\u ) -- (1*\u + 0*\s, 2*\u);
    \node at (0.5*\u + 0*\s, 1.5*\u) {1};  \node at (1.5*\u + 0*\s, 1.5*\u) {2};
    \node at (0.5*\u + 0*\s, 0.5*\u) {3};  \node at (1.5*\u + 0*\s, 0.5*\u) {4};
  \end{tikzpicture}
  \caption{\label{fig:interchangebox} Different enumerations corresponding to the different decomposition related by the interchange relation. The basic non--enumerated box diagrams agree while the generalized box diagrams differ as on the left the horizontal line segments may not be moved separately in the generalized case, but in both cases, they may be on the right.}
\end{figure}
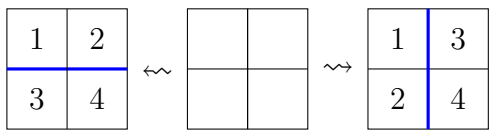

\begin{rmk}\label{rmk:decompcube}
Substitution into the labeled basic/generalized box diagrams form a non--Sigma colored operad structure and hence the morphisms of a non--Sigma Feynman category. There is a morphism of such operads, functor of Feynman categories, from generalized box diagrams to basic box diagrams by classes under the coarser equivalence.
A substitution by composition with a generator \eqref{eq:boxgenerators} in the box picture divides one rectangle with a horizontal of a vertical line.

Similarly, enumerated unlabeled   basic/generalized box diagrams form an operad under substitution. Decomposable tight squares
 form a suboperad of the decomposable squares of \cite{Dunn,Brinkmeier} and there is a morphism
operad to the enumerated basic/box configurations, by  taking the isotopy classes.
\end{rmk}

\begin{prop}\mbox{}
\label{prop:box2cell}
\begin{enumerate}
    \item
Contracting horizontal line segments provides a surjective morphism of the colored non--Sigma operad of basic/generalized box diagrams (with a valid population) to 2--cell diagrams (with a valid population). In particular, the boxes and 2--cells are in bijection.
\item This morphism is an isomorphism for basic box diagrams.
\item
 The morphism from generalized box diagrams to 2--cells factors through the morphism from generalized to basic box diagrams.
 \item The fiber over a fixed 2--cell diagram is the possible decompositions of it.
\end{enumerate}

\end{prop}
\begin{proof} All but the last  statement are straightforward.
For the fiber, we see that these are the generalized box diagrams that  map to the same basic box diagram. These are enumerated by a choice of matching movable horizontal line segments with neighboring movable line segments. Matching or not matching corresponds  to the choice of a side in an application of an interchange relation. This is the only relation leading to different compositions, and thus the last statement follows.
\end{proof}

An enumeration
of the boxes of a basic box diagram is called \emph{compatible/standard} if the enumeration on the 2-cell side is compatible/standard.

\begin{ex}  The box diagram and 2--cell diagram for the morphism
\begin{equation}
\g_{\phi_6,\phi_1\ot \phi_2,\phi_3\ot \phi_5\circ \phi_4}\circ(
id_{\phi_6}\ot \g_{\phi_1\ot \phi_2,\phi_3\ot \phi_5\circ \phi_4})\circ ( id_{\phi_6}\ot id_{\phi_5}\ot
\gamma(\phi_4,\phi_5)\ot id_{\phi_2}\ot id_{\phi_1}):\phi\to \tilde{\phi}
\end{equation}
where $\phi=\phi_6\bt\phi_3\bt \phi_5\bt \phi_4\bt \phi_1\bt \phi_2$
which corresponds to the formula
$
f=\phi_6\circ(\phi_3\ot (\phi_5\circ \phi_4))\circ (\phi_1\ot \phi_2)
$
are given in  Figure~\ref{fig:boxcell}.

%
%In the above formula is the substitution
%$\g_{\phi_1,\phi_2}\mu_{\phi_1,\phi_2}\circ_3\mu_{\phi_1,\phi_2}\circ_1\g_{\phi_1,\phi_2}$

\end{ex}

\begin{cor}
\label{cor:enumerate}
Fully bracketed formulas are equivalent to both:
 (1) enumerated basic box diagrams, and (2) generalized box diagrams (with standard enumeration).
\end{cor}
\begin{proof}
Fixing an iteration, a fully bracketed formula is equivalent to specifying a compatible enumeration for the corresponding 2--cell diagram by Proposition~\ref{prop:associnv}.
This is the same as a compatibly enumerated basic box diagram by Proposition~\ref{prop:box2cell}. Hence the first claim follows. The last claim follows from Proposition~\ref{prop:box2cell}.
\end{proof}

The following algorithm produces a generalized box diagram from a pre--formula.
First write the corresponding diagram by drawing horizontal and vertical line segments.
Then take its class as a generalized box diagram.
Given such a diagram or suspension graph, inversely use the algorithm to remove full horizontal or vertical line segment
and note the operations as $\circ$ or $\ot$ in the pre--formula.
This process will terminate in just one box by decomposablility.

\begin{df}
\label{df:planar-suspension}
The  {\em planar b/w graph} of a basic decomposable box diagram is the planar directed bipartite b/w graph, whose black vertices are the horizontal line segments, whose white vertices are the boxes and there is an edge if a the horizontal sub--line--segment is the boundary of the box. The edges are directed from top to bottom.

The {\em suspension graph} of a generalized decomposable  box diagram is the planar directed bipartite b/w graph whose white vertices are the boxes, whose black vertices are the full line segments. There is an edge if the full line segment has a sub--line segment which is part of the box and the direction is down.
\end{df}

If the box diagram is decorated,
there is an induced decoration on the graph and a direction  from source to target.
For instance, the cell arrangement in Figures \ref{fig:2cat} and \ref{fig:boxcell} (B) translates to the planar b/w graph and the suspension b/w graph in Figure~\ref{fig:boxgraph} both orientated downward.

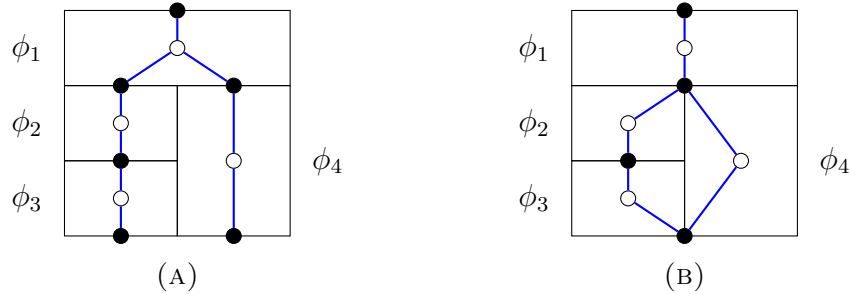
\begin{figure}[h]
  \begin{subfigure}[c]{0.4\linewidth}
    \centering
  \begin{tikzpicture}[scale=.8]
    \tikzstyle{mor}=[draw, circle, inner sep=2pt]
    \tikzstyle{iso}=[draw, circle, inner sep=2pt, fill=black]
    \tikzstyle{line}=[blue, thick]
    % width
    \tikzmath{\w = 1.5;}
    % iso
    \node[iso] (iso01) at (1*\w, 3) {};
    % phi1
    \draw (0*\w,3) rectangle (2*\w,2);
    \node[mor] (phi1) at (1*\w, 2.5) {};
    \node at (-0.5, 2.5) {$\phi_1$};
    % iso
    \node[iso] (iso12) at (0.5*\w, 2) {};
    \node[iso] (iso14) at (1.5*\w, 2) {};
    % phi2
    \draw (0*\w,2) rectangle (1*\w,1);
    \node[mor] (phi2) at (0.5*\w, 1.5) {};
    \node at (-0.5, 1.5) {$\phi_2$};
    % iso
    \node[iso] (iso23) at (0.5*\w, 1) {};
    % phi 3
    \draw (0*\w,1) rectangle (1*\w,0);
    \node[mor] (phi3) at (0.5*\w, 0.5) {};
    \node at (-0.5, 0.5) {$\phi_3$};
    % phi 4
    \draw (1*\w,2) rectangle (2*\w,0);
    \node[mor] (phi4) at (1.5*\w, 1) {};
    \node at (2*\w+0.5, 1) {$\phi_4$};
    % iso
    \node[iso] (iso35) at (0.5*\w, 0) {};
    \node[iso] (iso45) at (1.5*\w, 0) {};
    % lines
    \draw[line] (iso01) -- (phi1);
    \draw[line] (phi1) -- (iso12) -- (phi2) -- (iso23) -- (phi3) -- (iso35);
    \draw[line] (phi1) -- (iso14) -- (phi4) -- (iso45);
  \end{tikzpicture}
  \caption{}
 \end{subfigure}
    \begin{subfigure}[c]{0.4\linewidth}
    \centering
    \begin{tikzpicture}[scale=.8]
    \tikzstyle{mor}=[draw, circle, inner sep=2pt]
    \tikzstyle{iso}=[draw, circle, inner sep=2pt, fill=black]
    \tikzstyle{line}=[blue, thick]
    % width
    \tikzmath{\w = 1.5;}
    % iso
    \node[iso] (iso01) at (1*\w, 3) {};
    % phi1
    \draw (0*\w,3) rectangle (2*\w,2);
    \node[mor] (phi1) at (1*\w, 2.5) {};
    \node at (-0.5, 2.5) {$\phi_1$};
    % iso
    \node[iso] (iso124) at (1*\w, 2) {};
    % phi2
    \draw (0*\w,2) rectangle (1*\w,1);
    \node[mor] (phi2) at (0.5*\w, 1.5) {};
    \node at (-0.5, 1.5) {$\phi_2$};
    % iso
    \node[iso] (iso23) at (0.5*\w, 1) {};
    % phi 3
    \draw (0*\w,1) rectangle (1*\w,0);
    \node[mor] (phi3) at (0.5*\w, 0.5) {};
    \node at (-0.5, 0.5) {$\phi_3$};
    % phi 4
    \draw (1*\w,2) rectangle (2*\w,0);
    \node[mor] (phi4) at (1.5*\w, 1) {};
    \node at (2*\w+0.5, 1) {$\phi_4$};
    % iso
    \node[iso] (iso345) at (1*\w, 0) {};
    % lines
    \draw[line] (iso01) -- (phi1) -- (iso124);
    \draw[line] (iso124) -- (phi2) -- (iso23) -- (phi3) -- (iso345);
    \draw[line] (iso124) -- (phi4) -- (iso345);
  \end{tikzpicture}
  \caption{}
  \end{subfigure}
  \caption{\label{fig:boxgraph} A box diagram with its the dual b/w string diagram (A) and dual suspension graph (B).}
\end{figure}
A more complicated diagram is given in Figure~\ref{fig:boxcell}.
%\begin{rmk}
%\label{rmk:bwgraphs}
The removal of line segments that determine the decomposition translate as follows:
\begin{enumerate}
    \item In the planar black and white graph, bivalent black vertices are those that border bivalent white vertices correspond to freely movable line segments. These are the ones than can be removed in a (de)composition.
    \item If line segments are moved to ``+'' crossings, the effect on the graphs is a merging of two black vertices.

    \item In the b/w graph, the black vertices are decorated by the source and target objects.
    \item in the suspension graph, the black vertices are the tensor product of the objects belonging to the sub--line--segments of the full line segment.
    \end{enumerate}
%    \end{rmk}
The algorithm of \S\ref{par:algpar} in this language reads:
%    \begin{rmk}\mbox{}
    \begin{enumerate}
\item In the b/w graph, the removal of a movable  horizontal line segment in the decomposable box diagram  corresponds to the deletion of the corresponding bivalent black vertex and the contraction of the edge between the two neighboring white vertices.
If there is a labeling, the new label is the composition.

The removal of a horizontal line segment corresponds to the merging of two parallel neighboring black-white-black strands where the white vertex is bivalent. If there is a labeling, upon merging labels are tensored.

\item In the suspension graph, a horizontal line piece can  be removed in the generalized box diagram, if it belongs to a bivalent black vertex. Then the procedure is the same as above.

A removal of a vertical line piece is possible if two white vertices are suspended from the same two black vertices and are neighbors. The removal is then the merging of these two white vertices.
\end{enumerate}
%\end{rmk}

Thus the graph captures the different sides of the interchange equation.
In the basic picture, one has to align the line segments horizontally and then compose to do the $\mu$ operations first, compare with the string diagrams of Figure~\ref{fig:2cat}. The  choice of the prep-step is recorded by the suspension graph.
The basic b/w bipartite graph $\bwb$ is the graph with a black input and output and one white vertex.
\begin{prop}
\label{prop:dualgraph}
A directed planar b/w bipartite graph is the dual of a basic diagram if it has black input and output vertices
and the algorithm above terminates with $\bwb$.
We will call these {\em composition graphs}.

A directed b/w bipartite graph is the dual of a generalized box diagram if it has black input and output vertices
and the algorithm above terminates with $\bwb$.
We will call these {\em suspension graphs}.

\end{prop}
\begin{proof}
That the images are of this type is clear. Conversely, building up the diagram by running the algorithm backwards produces the box diagram.
\end{proof}

The algorithm allows us to read off the bracketed formula by recording  the operations as $\circ$ or $\ot$ in the pre--formula.

\begin{cor}
The set of pre--formulas is bijective with respect to the  decomposable  suspension graphs. This yields an isomorphism of non--Sigma FCs.
\end{cor}

\begin{rmk}
\label{rmk:bwsubstitution}
Substitution for the diagrams leads to substitution for the b/w graphs.
For this, a white vertex is replaced by the b/w graph, minus the input and output black vertices. In suspension graphs, these vertices are merged with the corresponding black vertex above and vertex below the white vertex.
In the basic b/w graph, the gluing splits  these vertices according to the labeling.
This depends on the labeling and the combinatorics of the possibilities is, not by chance, that appearing in the compositions of the discretization of the arc operad \cite{KLP,hoch1,hoch2}.
\end{rmk}

\subsection{The structure of $\monplus{\M,P}$ and its localizaions}
Since equivalent categories yield equivalent plus constructions with given equivalences, in the following we {\em assume that $\M$ is strict  monoidal}.
%\subsubsection{The non--Sigma FC $\spure(\monplus{\M,P}$)}
Let  $S=\Mor(\M)$ for the monoidal category $(\M,\ot)$. This has the partially defined binary operation of composition $\circ$ and that of tensor product $\ot$. The free monoid  $S^\bt$ can be identified with the morphisms of the non--symmetric free monoidal category $\M^{\neg\Sigma \bt}$.
A formula is valid if all the compositions $\circ$ can be performed, that is sources and targets line up correctly.
The basic discrete groupoid of objects is $\V=(\M \da \M)^{disc}$.

\begin{prop}
\label{prop:planarmonplusFC}
%label{prop:planarcFC}
The category $\spure(\monplus{\M,P})$ is isomorphic to $\cF(S)$ for $S=\Mor(\M)$ and
is part of the corresponding Feynman category $\FF_{\it form}(S)$.
This Feynman category has a proper degree function with the $\g$ and $\mu$ morphisms being generators of degree $1$.
\end{prop}
\begin{proof}
By definition, $\spure(\monplus{\M})$
is generated by $\g_{\phi_0,\phi_1},\mu_{\phi_1,\phi_2}$ and the identities under monoidal product and composition modulo inner associativity and the interchange relation.  In particular, any morphism $\phi\to \tilde\phi$ is given by a sequence of tensor products of these operations.
Such a sequence is in bijective correspondence with a fully bracketed formula. Such a formula is not a unique morphism, but is subject to associativity   according to the relations of Definition~\ref{def:cplus}.
% \eqref{assreleq} and \eqref{eq:muassoc}.
The result follows from the fact that the relations  on both sides are identified by the bijection, and any class of valid formulas corresponds to a reduced formula.
The  degree function is well defined as all relations are homogeneous with respect to this degree.
\end{proof}

 Due to the interchange relation, follows from Proposition \ref{dfprop:frmcubical} by observing that the interchange relation does not change the number of vertices.
% {\sc \noindent {\sc NB:}} Note that the internal interchange relation
% % \eqref{eq:intinter}
% is not yet part of the relations as it contains a commutation isomorphism, see also Remark~\ref{genrmk}.
% \RK{Something off with interchange, probably need to include commutators into $\spure$}

% \RK{Should this be a Lemma with proof?}

%
%  with this definition $\bt$ irreducible morphisms correspond to $\bt$ irreducible morphisms.

% \subsubsection{FC of classes of valid  formulas}
% \label{par:formulaoperad}

%fix.
% \begin{rmk}
% The coloring can be viewed as a decoration \cite[\S6.1.4]{decorated} of the Feynman category for $PROs$ which is the non--Sigma version of the one for $PROP$s, cf.\ \cite{feynman,matrix}.
% \end{rmk}

\begin{cor}
\label{cor:planar}
The morphisms of $\spure(\monplus{\M})$ and the corresponding non--Sigma Feynman category  are isomorphically given by the following non--Sigma colored operads with monoidal structure being ordered disjoint union.

\begin{enumerate}
    \item   Planar planted b/w bipartite trees with valid decorations by morphisms of $\M$ on the leaves, where gluing decorated b/w bipartite trees at leaves gives the operad structure and hence the compositions
    \item Compatibly enumerated 2--cell diagrams with a valid population and substitution. Or equivalently, string diagrams with a valid population and compatible enumeration and substitution.
    \item Compatibly enumerated basic box diagrams with a valid population and substitution. Or dually, composition graphs with a valid decoration and a compatible enumeration. The operad structure is substitution of this type of graph, see Remark~\ref{rmk:bwsubstitution}.
    \item Generalized box diagrams with a valid population and substitution. Or dually, suspension graphs with a valid decoration. The operad structure is substitution of this type of graph, see Remark~\ref{rmk:bwsubstitution}.
\end{enumerate}
\end{cor}

\begin{proof}
This follows from regarding a formula as a flow chart, Proposition~\ref{prop:associnv},
Proposition~\ref{prop:box2cell}, Corollary~\ref{cor:enumerate} and  Proposition~\ref{prop:dualgraph}.
\end{proof}

%\end{df}

%\begin{rmk}
%which is a
%non--Sigma decoration of the Feynman category of surjections.
%\end{rmk}
%
% One obtains a suspended b/w string diagram as follows. First put the horizontal line segments in general position, which means that there are no crossings only T intersections.
% The black vertices are the horizontal pieces of the diagram or the 1--cells in
% the 2--category. These are labeled  by objects that are determined by the source and target maps. There is a directed edge from a white vertex to a black vertex, if the black vertex is labeled by the target of the morphism and a directed edge from a black vertex to a white vertex, if it is the source. This corresponds to being the upper or lower boundary of the box respectively 2-cell.

% Allowing also degenerate configurations with  movable horizontal line segments forming a line, like in the interchange relation, we define a

% We will call this the
% A given decomposition of the diagram can be encoded in the generalized suspension graph, where one is allowed
% to merge black vertices on parallel neighboring strands.
% See Figure ?? for an example. This {\em uniquely} determines the decomposition.
% The standard decomposition correspond to only having the top and bottom black vertex being or higher arity than $2$.
% The higher arity black vertices stand for horizonal composition.

%\subsubsection{Symmetric version $\spure(\monplus{\M})\Sigma_{ext}$}

Freely adding the commutators imbues the non--Sigma structures with a free $\SS$-action.
Quotienting out by the interchange relation defines
the symmetric versions of the results above. In this process, entire orbits are identified.
%In the graphical versions this will be a decoration of cells by morphisms of  $\M$, which is a morphism from the set of cells to the morphisms of $\M$.
 This works for the 2-cell diagrams and equivalently for the  basic or generalized box diagrams, as well as their dual graphs.

%That is using r.h.s.\ of \eqref{interchangeeq}.

\begin{prop}
Basic box diagrams with standard enumerations of 2-cells  generate the morphisms of $\spure(\monplus{\M})\I_{ext}$ symmetrically monoidally
 that is under the free symmetric product $\bt$. $\spure(\monplus{\M})\I_{ext}$ is part of a  Feynman category with $\V = \Iso(\M\da \M)$ and $\Indec$ given by the set of basic box diagrams with standard enumeration. The length decrease is a proper degree function.

 Isomorphically, the  $(n,1)$-morphisms together with their composition structure  are:
\begin{enumerate}
 \item Composition graphs with labeled and enumerated white vertices.
    \item Pairs $(f,\s)$ of irreducible valid standard formula of arity $n$ and a permutation $\SS_n$ of arity with $\Sn$ acting on $\s$ on the right.
    \item 2--cell diagrams in $\underline{\M}$ with a decoration of cells by elements of  $\M$ and an arbitrary enumeration of the cells. Or,  equivalently the associated string diagrams with a decoration and an arbitrary enumeration.

\end{enumerate}
\end{prop}
\begin{proof}
The proof is in two steps. First adding arbitrary permutations, we obtain pairs $(b,\s)$
of a box diagram $b$ with $n$ boxes in standard enumeration and an element $\s\in\Sn$ where  $\Sn$ acts on the right. This maps to morphisms in  $\spure(\monplus{\M})\I_{\it ext}$ by $\G(b)\circ \s$ where $\G(f=b)$ is the morphism in $\spure(\monplus{\M})$ defined by $b$. By  Proposition~\ref{prop:sigma}, any morphism is of this type and the map is surjective.
Note that  morphisms $\s$ permute the source $\phi_1\bdb \phi_n$ to
$\phi_{\s(1)}\bdb \phi_{\s(n)}$ and there is only one box diagram in standard enumeration in the orbit. This is due to the fact that any use of the interchange relation \eqref{interchangeeq} changes the numbering by a transposition. This is particularly transparent in the b/w tree picture. Due to the same fact any other decomposition is also in the orbit as any formula having the same diagram, but with possibly different compatible enumeration, can be obtained by a repeated application of the interchange relation in the 2--category.
By internal interchange, this corresponds to the pre--composition by a permutation $\s$.
The element $(b,\s)$ is mapped to $\G(b)\s$ and if $b=b'\sigma'$ is the permutation to a standard form $b$, then $\G(b)\s=\G(b')\s'\s$ by internal interchange.
On the other hand, $b$ being in standard form $\G(b)\s=\G(b')\s'$ means that $\s=\s'$ and $b=b'$ so there is only one standard form in each orbit.
Thus, a full orbit of the standard enumeration is given by an arbitrary enumeration.
The rest of the statements again follow from Propositions \ref{prop:associnv} and \ref{prop:box2cell}, Corollary~\ref{cor:enumerate}, and  Proposition~\ref{prop:dualgraph}. That the degree function is cubical follows from
Proposition~\ref{prop:planarmonplusFC}, since the morphisms of $\I_{ext}$ are the only morphisms of degree $0$.
 %In terms of diagrams this changed the enumeration to an arbitrary enumeration.
\end{proof}

% \begin{rmk}\mbox{}
% The $(n,1)$-morphisms, again by general theory, form a symmetric colored operad. The relation to the operad of decomposable little 2--cubes is apparent in this interpretation.
% The permutation can be viewed as part of the source map, so that morphisms are diagrams together with a source map. This is parallel to the fact \cite{feynman} that morphisms in the graphical Feynman categories are given by their ghost graph together with extra structures defining the source and target maps.
% \end{rmk}

%\subsubsection{Adding internal base morphisms}

Adding in the action of $\scs$,
the generators will be classes of generalized box diagrams  (or equivalently suspension graphs)  decorated by morphisms and internal base morphisms, i.e.\ morphisms of the type $P(\s)$.

 \begin{df}
 A  {\em valid   decoration by morphisms and internal base morphisms}
  for a
 {\em generalized} decomposable box diagram,  is a decoration of the boxes by morphisms
  and the {\em full line} segment by base morphisms, that is morphisms of the form $P(\s)$. These have to be compatible
 in the way that the compositions line up correctly. That is, the source of the base morphism is the tensor product the target boundaries of the two cells bordering this line segment from above and the target of the base morphism is the tensor product of the sources of the 2--cells bordering this line segment from below.
For instance in the example in Figure~\ref{fig:boxcell}, the base morphism in the second line from the top will be from $t(\phi_1)\ot t(\phi_2)\to s(\phi_3) \ot s(\phi_4)$. The top and bottom lines only have constraints on the target and the source of the base morphisms respectively.
This translates to the {\em suspension graph} as
  black vertices being decorated by morphisms of the type $P(\sigma)$ and  white vertices by morphisms in a compatible fashion.

  A  {\em valid   decoration by morphisms and internal base morphisms} for a  {\em basic decomposable box diagram} is a decoration of all cells by morphisms of $\M$
 and all  horizontal line segments by base morphisms.  This has to be compatible, so that the tensor product of the base morphisms of the sub--segments satisfy the conditions above.
This translates to the {\em dual b/w string diagram} as
  black vertices being decorated by base morphisms and  white vertices by morphisms in a compatible fashion, and for a {\em 2--cell diagram} as a decoration of a
the 2-cells by morphisms and of the 1-cells by base morphisms, such that the decoration obtained by censuring the decorations of the 1--cells that form the boundary of a 2--cell the decoration is a valid decoration as above.

  \end{df}

 E.g.\ in Figure~\ref{fig:boxcell}, the base morphism decorating the second full interval will be decorated by a base morphism $t(\phi_1) \ot t(\phi_2) \to s(\phi_3) \ot s(\phi_4)$.
 Substituting a morphism $\scs(\psi)$ into $\phi_4$ would yield the composition with the morphism $id\ot P(\s)$ on the top of the box and $id\ot P(\s')$ on the bottom if the outside base morphism of the box being substituted is $\scs$.

 \begin{prop}
 \label{prop:mncpint}
  $\spure(\monplus{\M}\I_{\it int})$ is a non--Sigma Feynman category with $\V = \Iso(\M\da \M)$.
 The basic morphisms of $\spure(\monplus{\M}\I_{\it int})$ can be described by the groupoid colored non--Sigma operad whose generators are generalized box diagrams (or equivalently suspension graphs) with a valid decoration by morphisms and base morphisms.
 \end{prop}

 \begin{proof} Analogously to Proposition~\ref{prop:cplusint},
it follows from inner equivariance that the generators are the suspension graphs decorated as stated. Substituting these generators, composes  outer action by base morphisms which can be pulled through by outer equivariance, where for longer intervals, these are extended by identities as tensor factors on the subintervals that are not part of the box that is being substituted.
The upshot is all the base morphisms can be
associated to the full line segments and this is stable under composition.
\end{proof}

% \subsubsection{General case}

 \begin{thm}
 \label{thm:mncp-equals-FC}
 %Result!
%  $\monplus{\M}$ is a FC whose basic groupoid of objects is $\V=Iso(\M\da\M)$.
 For a factorizable pointing,
 $\monplus{\M,P}$ is an FC whose basic groupoid of objects is $\V=\Iso(P\da P)$.
 The basic  morphisms of $\monplus{\M,P}$ are equivalently given
 by either (1) box diagrams with a valid decoration by morphisms and base morphisms and an enumeration of their cells,
or (2) composition graphs with a valid decoration by morphisms and base morphisms and an enumeration of their white vertices.
The category is has a proper degree function, with the $\gamma$ and $\mu$ morphisms as degree $1$ generators.

The hyp version is an FC with a proper degree function.
The standard gcp version are Feynman categories with additional generators of degree $-1$ corresponding to the units. The generators are the $\g_{\phi_1,\phi_0}$,  $\g_{\phi_1,\phi_0}$ and $\mu_{\phi_1,\phi_2}$, and $\bar\g_{\phi_1,\phi_0}$ respectively.
 \end{thm}

 \begin{proof}
 If $\M$ has factorizable pointing, then any morphism of $\M$ can be factored as $\G\I_{ext}\I_{int}$ by Corollary~\ref{cor:decomp} and Proposition~\ref{prop:sigma}. The morphisms $\G\I_{ext}$
 are given by box diagrams with a valid decoration by morphisms and an enumeration of the two cells. The decoration is by the morphisms $\I_{int}$ applied to the source.
 As in the proof of Proposition~\ref{prop:mncpint}, the operation of $\I_{int}$ can be represented by
 decorating the boundaries of the boxes compatibly with internal base morphisms.
 The description of graphs again follows from Proposition~\ref{prop:dualgraph}.
  \end{proof}
  %By definition $\ot$ irreducible morphisms generate.

  \begin{cor}
  \label{cor:gensplit}
  If $\M$ has factorizable pointing, then
 every morphism $\bG$ splits as  $\mu_n (\bG_1\bdb\bG_n)\Perm$, if the box diagram has $n-1$ full vertical lines (top to bottom), equivalently  $n$ is the number of connected components of the composition graph, $\bG_i$ are the morphisms corresponding to the connected components and $\Perm$ is a permutation.
  \end{cor}
 \begin{proof}
   This is clear since the algorithm determining the morphisms will perform the multiplication followed by the last outer base morphism, whose permutations can be pulled to second last place by the factorizations of base morphisms.
   In formulas: $\bG=\Sigma\mu(\bG_1\bdb\bG_n)=\mu \Perm\Sigma'(\bG_1\bdb\bG_n)
 =\mu \Perm(\Sigma_1\bG_1\bdb\Sigma_n\bG_n)
 =\mu (\bG'_{\Perm(1)}  \bdb \bG'_{\Perm(n)})\Perm$
 %=
%-%=(\Sigma'_{P(1)}\bG_{P(1)}\bdb\Sigma'_{P(n)}\bG_{P(n)})
%=\mu(\bG_1  \bdb \bG_n)P^{-1}P$
where $\Perm$ is the permutation of the factors determined by $\Sigma=\Perm(\Sigma_1\odo\Sigma_n)\Perm$.
 \end{proof}

 \begin{rmk}
% For a UFC, the base morphisms are a groupoid colored PROP whose com
  Note that  if there no assumption of factorizability, there is no concise answer.
If the  pointing is factorizable, then the interchange relations are compatible with the action of the base morphisms, and the decoration of the generalized diagram can be transferred to the basic diagram in whose class it lies. This effectively divides out all the interchange relations.

In the general case, this procedure yields an equivalence relation, dividing out the possible interchange relation after applying base morphisms ---namely those where the base morphisms factors. Then the morphisms of  $\monplus{\M,P}$  are  classes of diagrams where two diagrams are equivalent if there is a compatibly enumerated basic diagram
decorated by morphisms and base morphisms, which maps to both of the generalized decorated diagrams. For example as in Figure~\ref{fig:movesegments}.
 \end{rmk}

 \begin{lem}
 %\label{prop:hereditarysplit}
 \label{lem:hereditarysplit}
 If $\M$ has a localizable pointing $P$, then if the target $\phi$ of a  morphism $\G:\Phi\to \phi$ in $\monplus{\M,P}$ is $\ot$-decomposable, viz.\ $\phi=\phi_1\odo \phi_n$, then so is $\G$, viz.\ $\G=\G_1\odo \G_n$. That is
 given the solid arrows, the dotted arrows exist
  \begin{equation}
     \begin{tikzcd}
    \Phi_1\bdb \Phi_n\ar[r, dotted, "\mu^{\bt n}"]
    \ar[d,dotted,"(\G_1\bdb \G_n)P"]&\Phi\ar[d,"\G"]\\
   \phi= \phi_1\bdb \phi_n \ar[r,"\mu"]&\phi_1\odo \phi_n
     \end{tikzcd}
 \end{equation}

    In particular, for a \pher{} UFC, each base morphism $\bG$ splits into $n$   irreducible factors if its target is of length $n$.
 \end{lem}

 \begin{proof}
 This follows from the fact that by the conditions, the pullback of $\Phi$ along $\mu$ exists for all generators.
 \end{proof}

 In terms of the decorated box diagrams, this means that they have at least $n-1$ full vertical lines and for the composition graphs, this means that they can be  split into a disjoint union of at least $n$ graphs.

If one does not assume strictness, all arguments
go through  up to equivalence by using a presentation. The results will be independent of the presentation. The argument is parallel to
Lemma~\ref{lem:essential}.

% \begin{prop} For  a hereditary UFC $\M$, any morphism in $\monplus{\M}$ has the form $\mu_Jf$ where $f$ is irreducible.
% \end{prop}

%\subsection{Localization and $\plus{\M}$}
%\label{sec:localization-M-plus}

\begin{thm}
\label{thm:remonplusbox}
If $P$ is a fully factorizable pointing of $\M$, then the  $Hom_{\redmonplus{\M,P}}(\phi,\psi)$
are equivalence classes of roofs given by pairs of a decomposition $\phi=\phi_1\odo \phi_n$ and a morphism $\bG$, where the latter can be taken to be a  box diagram with a valid decoration by the morphism $\phi_1\kdk\phi_n$ and a choice of basic morphisms
such that the target is $\psi$.
Or, equivalently, a composition graph with this decoration.

More generally, if the pointing is only factorizable, the above morphisms generate under the monoidal product and composition.
\end{thm}

\begin{proof}
By Proposition~\ref{prop:roof}, if $\M$ is localizable, then the morphisms of $\locmonplus{\M,P}$ are pairs $(\mu_n,\bG)$ where $\bG\in \monplus{\M}$ and where $\mu_n:s(\bG)=\phi_1\bdb\phi_n\to \phi_1\odo\phi_n$ modulo the equivalence relation on roofs given in Lemma~\ref{lem:roof}.
Using Theorem~\ref{thm:mncp-equals-FC}, the first result follows from Proposition \ref{prop:loc-red-equiv}. The second statement is clear.
\end{proof}

%Note that the permutation is incorporated into the isomorphism.

% \begin{rmk}
% The exact conditions for a UFC to have a concise description of $\plus{\M}$ will be treated in the future.
% We expect that the counter-example \ref{par:counterex} will play an important role.
% \end{rmk}

%\RK{Last reading here}
%\subsection{The unital, gcp and hyp versions}
\label{par:gcpgenpar}
The unital  versions can also be described in this formalism.
Since the unit was strict, by adding morphisms using $i_{id_X}$, we have vertical identities at our disposal.
Like the proof of Proposition~\ref{prop:alternate}, these create a ``brick wall'' diagram and compose horizontally. In the graph picture, this corresponds to leveling the graph, cf.\ \cite[Appendix B.1.4]{feynmanrep} for the corresponding leveling of trees.
We will use labels $i_\s$ as a morphism decoration to indicate the presence of such a map. E.g.\ in the box diagram for $\gamma_{\phi,\s}\circ (id_\phi\bt i_{\s})$ will have  comprised of two vertically stacked boxes the labels $\phi_1$ and $i_\s$. In particular, as $\g_{\phi,id_{s(\phi)}}\circ ( id_\phi\bt i_{id_{s(\phi)}})=id_\phi$, we can introduce horizontal lines into each box and decorate the new boxes with $i_{s(\phi)}$ or $i_{t(\phi)}$. This means  that the  $\bG$, up  to  permutations and precomposition with $\mu$'s, can be  taken to be the $\gen{\phi}{\s}{n}$. We call this the leveling.

\begin{prop}
\label{prop:mncpgcpgen}
For $\M$ with factorizable pointing,
a basic morphism in $\monplusp{\M,P}$ can be written as
\begin{equation}
    \bG=\gen{\phi}{\sigma}{n}\circ [\mu^{m_1}\bdb \mu^{m_n}]\circ \Perm\circ (i\bt id^\bt)
\end{equation} where $i$ is an inclusion of identity maps $i_{id_X}$ applied to added factors of $\unit$.
That is, a basic morphism is
specified by the data of a composite generator
$\gen{\phi}{\sigma}{n}$ together with decompositions $\phi_l=\phi_l^1\odo \phi_l^{n_1}$ and a permutation $P$,
where  the $\phi_l$ may contain factors of $i_{id_X}$ and $\Perm$ is a permutation rearranging the $\phi_l^j$.
The only relations are moving the insertions of the $i_{id_X}$ which is a choice and  may change the level $l$ of $\phi_l^j$.

% For a fully factorizable $\M$, in particular if $\M$ is part of a \pher{} UFC, composite generators $\gen{\phi}{\sigma}{n}$
% can be further split according to \eqref{eq:hereditarygendecomp}.

% In particular, the isomorphisms together with morphisms $\mu$ and the {\em connected} $\gen{\phi}{\sigma}{n}$ generate
%  That is pairs
% that is  $\gen{\phi}{\s}{n}=[\gen{\phi^1}{\sigma^1}{n}\bdb  \gen{\phi^m}{\sigma^m}{n_m}]\circ\circ P)
%  \end{equation}
%   with
% $\ot$--irreducible $\gen{\phi^j}{\sigma^j}{n}$; the source of $\mu_n$ being $\phi^1_1\bdb\phi^m_{m_n}$ and its target being $\phi^1_1\odo\phi^m_{m_n}$.
 \end{prop}
{\sc \noindent {\sc NB:}} The $\phi_l^j$ decorations do not have any factors of $i_\s$ for $\s\neq id_X$.

\begin{proof}
Given a box diagram, using the unit constraints and the morphisms $i_{t(\phi_i)}$ or $i_{s(\phi_i)}$,
we can insert horizontal lines into each box to make the number of horizontal lines in any two adjacent boxes match. This means that the diagram can be put into a position where it has full horizontal lines (left to right). Using the interchange relation, we can first compose the horizontal and then the vertical compositions.  The horizontal compositions at each level are the $\mu^{n_l}$. These yield the morphisms $\phi_i$ and the base morphisms $\sigma_i$.

With the exception of the units, there are no labels $i_\s$
, so there are no  relations stemming from the unit conditions (Definition~\ref{gcpdef}~(c)), except for the identities.
Thus the statement follows from the definition of $\monplusp{\M,P}$ together with Theorem~\ref{thm:mncp-equals-FC}.
The last statement is an application of the \pher{} condition.
\end{proof}

\section{UFCs, Plus Constructions and Cubical Feynman categories}
\label{par:ufcstructure}

\subsection{Essential Generators}

For UFCs, we are using the standard gcp $P$ and omit this in the notation for their plus constructions. Since all the base morphisms are invertible, we make the identification $\scs=(\s^{-1}\Da \s)$ to make contact with known plus constructions.
We do not assume UFCs to be strict in general, but by the equivalence Theorem \ref{thm:endo}, we could assume a strict model. When working with the non--strict version, the following makes this concrete.

\begin{lem}
%\label{lem:genUFCplus}
%\label{genprop}
\label{lem:essential}
For a  UFC $\FM$, the morphisms $\gamma_{\phi_1,\phi_0}$
 with  $\phi_0,\phi_1 \in \jmath^\ot(\Indec^\ot)$  together with the morphisms $\mu_{\phi_1,\phi_2}$ where $\phi_1,\phi_2\in  \jmath^\ot(\Indec^\ot)$ and the isomorphisms $\sdsphi$ with $\phi\in    \jmath^\ot(\Indec^\ot)$ form a set of essential generators under concatenation and monoidal product of $\monplus{\M}$.
\end{lem}
Here essential means that they generate together with the isomorphisms given by $\jmath^\ot$, so in particular these generate an equivalent subcategory. Thus, in the following, we will assume strictness without loss of generality.

\begin{proof}

We first show that the $\g_{\phi_0,\phi_1}$ with $\phi_0,\phi_1\in \jmath^\ot(\Indec^\ot)$ generate together with the isomorphisms given by $\jmath^{\ot}$.
Given a composable pair $(\phi_0,\phi_1)$, there  are isomorphisms $\sds$ and $\tdt$, such that $\sds(\phi_0)=\phi''_0\in \Indec^\ot$
and $\tdt(\phi_1)=\phi'_1\in \Indec$. These fit into a diagram
\begin{equation}
\xymatrix{
X_0\ar[rr]^{\phi_0}\ar[d]_\s&&X_1\ar[rr]^{\phi_1}\ar[dl]_{\s'}\ar[dr]^\t&&X_2\ar[d]^{\t'}\\
X'_0\ar[r]_{\phi'_0}\ar@(dr,dl)@{-->}[rrr]_{\phi''_0}&X''_1\ar@(dr,dl)[rrr]_{\phi'_1}\ar[rr]_{\t(\s')^{-1}}&&X'_1\ar[r]_{\phi''_1}&X'_2
}
\end{equation}
Setting $\phi'_1=(id\Da \t(\s')^{-1})(\phi'')\in \Indec$, we have $\phi'_1=\phi''_1\t (\s'^{-1})=\t'\phi_1\t^{-1}\t\s'^{-1}=(\s'\Da \t')(\phi_1)$, we find that
\begin{equation}
\begin{aligned}
\g_{\phi_1,\phi_0}&=
(\s\Da \t')^{-1} \g_{\t' \phi_1,\phi_0 \s^{-1}} ((id\Da \t')\ot (\s\ot id))\\
&=(\s\Da \t')^{-1} \g_{\t' \phi_1(\s')^{-1},\s' \phi_0 s^{-1}}  (\s'\Da id)\ot (id\Da \s')
 ((id\Da \t')\ot (\s\ot id))\\
&= (\s\Da \t')^{-1} \g_{\phi'_1,\phi'_0} ( (\s'\Da \t')\ot (\s\Da \s'))
\end{aligned}
\end{equation}
The essential reductions for the other morphisms is analogous.
\end{proof}

\begin{lem}
\label{lem:genstructure}
For a  UFC $\FM$ for  essential generators from the the generators of $\monplus{\M}$ from Lemma \ref{lem:essential} the following generators suffice:
 The $\mu_{\phi_1,\phi_2}$ with $\phi_1,\phi_2\in \jmath(\Indec)$. The $\g_{\phi_0,\phi_1}$ with {\em connected} pairs $(\phi_0,\phi_1)$.
 Note, if the underlying $\M$ is \pher{}, then these are the $\g's$ whose output is in $\jmath(\Indec)$.
 The isomorphisms $\sdsphi$ with $\phi\in \jmath(\Indec)$ and permutations. That is those of the form
     $\sdsphi= (\Sigma\Da \Sigma')(\sdsp{\s_1}{\s'_1}{\phi_1}\odo \sdsp{\s_n}{\s'_n}{\phi_n})$ where each $\s$ and $\s'$ is $\s_i=\s_i^1\odo \s_i^{n_i}$ and $\s'_i=\s_i^{\prime 1}\odo s_i^{\prime m_i}$.

    %  \item The isomorphisms $\Sigma\in Mor(\Iso(\plus{\M}))$ are of the form $\Sigma= \sdsphi$ and  have the following decomposition: Fix $\phi\simeq \biotimes_{i=1}^n \phi_i$
    %  \RK{Choose basis get rid of $\simeq$}
    %  \begin{equation}
    %  \sdsphi \simeq   C_P\circ\bigotimes_{i=1}^n \sdsp{\sigma_i}{\sigma'_i}{\phi_i}=
    %   \bigotimes_{i=1}^n\sdsp{\sigma_{P(i)}}{\phi_{P(i)}}{\phi_{P(i)}}\circ C_P
    %  \end{equation}
    %  where $P$ is a permutation of $n$ elements and $C_P$ is the corresponding the commutativity constraint. Up to switching the two factors this decomposition is unique once the decomposition of $\phi$ is fixed. Moreover the $\sigma_i\simeq \bigotimes_{s\in S}\sigma_s$ and $\sigma'=\bigotimes_{t\in T}$  where $\index(\phi_v)=(\pi_S,\pi_T)$ and $\sigma_s,\sigma_t\in Mor(\V)$.

 \begin{proof}
 The statement about the reduction of the generators $\mu_{\phi_1,\phi_2}$ to those with irreducible $\phi_1,\phi_2$ follows from internal associativity.
 The  reduction of the generators $\g_{\phi_1,\phi_2}$ follows from Lemma~\ref{lem:UFCcommonfact}.
  Indeed if $\phi_0\phi_1$ is decomposable, then as $\M$ is a  UFC,
    $\phi_0\phi_1=\bigotimes_{u\in U}\phi_{0,u}\phi_{1,u}$ with $(\phi_{0,u},\phi_{1,u})$ connected.
 In the \pher{} case, this also implies that $\phi_{0,u}\phi_{1,u}$ is in $\Indec$.   By the interchange relations, cf.\ Proposition \ref{prop:loc-red-equiv}, $\gamma_{\phi_0,\phi_1}=\gamma_{\bigotimes_{u\in U}\phi_{0,u},\bigotimes_{u\in U}\phi_{1,u}}=\bigotimes_{u\in U}\gamma_{\phi_{0,u},\phi_{1,u}}$.
The reduction of the isomorphisms follows from the factorizablity of isomorphisms in a UFC.
 First we can factor into $(\Sigma \Da \Sigma)(\sdsp{\s_1}{\s'_1}{\phi_1}\odo \sdsp{\s_n}{\s'_n}{\phi_n})$
 using Remark \ref{rmk:factorizable} and Definition \ref{df:localizable} (1). Finally, applying Definition \ref{df:localizable} (1) to the $\s_i,\s'_i$ and pulling out the permutations yields the result.
 \end{proof}

\end{lem}
%then there is a length on objects and morphisms of $\plus{\M}$ with $|\g_{\phi_n\kdk\phi_0}|=n$ and $|\sds|=0$. This proves the second part --- note that $\g_{\phi_0}=id_{\phi_0}$ by definition.

% By decomposing the sources and targets in $\Indec^\ot$,
% the corresponding lengths of source and target of two composable morphisms must coincide and
% hence again arranging the generators in rows and columns, there is a grid and
% the columns are composable. Composing these first renders show that the $\g_{\phi_n\kdk\phi_0}$ are monoidal generators.

%\RK{maybe insert picture here}

%\RK{finish}
% \RK{move}
% \begin{proposition}
% \label{prop:mp}
% For a \pher{} UFC $\M$, the category $\plus{\M}$ is symmetric monoidally equivalent to $\locmonplus{\M}$.
% \end{proposition}
% \begin{proof} Define a  monoidal functor $F:(\locmonplus{\M}, \bt) \to (\plus{\M}, \ot)$ by $F(\phi_1 \bdb \phi_n)=\phi_1\odo \phi_n$, $F(\sdsphi)=\sdsphi$,
% $F(\mu_{\phi_1,\phi_2})=id_{\phi_1\ot \phi_2}$,
% $F(\g_{\phi_0,\phi_1})=\bar\g_{\phi_0,\phi_1}$.
% Similarly define the functor $G: \plus{\M} \to \locmonplus{\M}$ to be the inclusion on the groupoid part and
% $G(\bar \g_{\phi_0,\phi_1})=\g_{\phi_0,\phi_1}\mu^{-1}_{\phi_0,\phi_1}$.
% Using that any object $\phi_1 \bdb \phi_n$ is isomorphic to $\phi=\phi_1\odo \phi_n$  and that the condition Definition \ref{df:factorizable} (1) identifies the isomorphisms of $\phi$ with those of any decomposition, it is straightforward to show that $F$ and $G$ witness the equivalence.
% \end{proof}
\subsection{Structural results for plus constructions of UFCs}
%\subsection{Indexing for composable morphisms in UFCs}
\begin{df}
\label{nonconnectedindex}
If $\M$ is a not-necessarily \pher{} UFC, then the connected components of a composable pair $(\phi_0,\phi_1)$ give a new cospan on the level of indexing of morphisms.
Let $V_0=\index(\phi_0), V_1=\index(\phi_1)$, $W=\index(\phi_1\circ\phi_0)$ and $U$ be the push-out defining the connected components.
Then this defines a co-span {\em the index of $(\phi_0,\phi_1)$ denoted by  $\index(\phi_0,\phi_1)$}.
We define $\depth(\phi_0,\phi_1):=|U|$.
%   \begin{equation}
% \label{eq:indexingdiagram}
% \begin{tikzcd}
%     &U&\\
%     V\ar[ur]&&W\ar[ul]\\
%     &S \ar[ul]\ar[ur]&
% \end{tikzcd}
% \end{equation}
\end{df}
Similarly,
given composable morphisms $(\phi_0,\dots,\phi_{n-1})$:
$\xnelt{X}{\phi}{n}$, fixing bases $X_i\simeq \bigotimes_{s\in S_i}*_s$ and $\phi_i\simeq \bigotimes_{v\in V_i}\phi_v$, then using indexing, we get the first two rows of a push--out diagram \eqref{pushoutdiageq}.
We will call the isomorphism class of the full diagram $\index(\phi_0\kdk\phi_n)$.
Let $U$ be the full--pushout, that is the combination of all the equivalence relations,
then by iterating the construction above, we obtain composable morphisms $(\phi_{u,0},\dots,\phi_{u,n-1})$ with $\phi_i\simeq \bigotimes_{u\in U} \phi_{u,i}$ and if
$\phi=\phi_{n-1}\circ \dots \circ \phi_{0}$ and
$\phi_u=\phi_{n-1,u}\circ \dots \circ \phi_{u,0}$, then
$\phi_i\simeq \bigotimes_{u\in U} \phi_{u,i}$.
\begin{equation}
\label{pushoutdiageq}
\begin{tikzcd}[column sep = small]
&V_0&&V_1&\cdots&V_{n-1}\\
S_0\ar[ur]&&\ar[ul]S_1\ar[ur]&&
    \ar[ul]S_2\ar[ur, dashed, no head]&\cdots&\ar[ul]S_n
   \end{tikzcd}
\end{equation}

We call the $\phi_u$ the {\em connected components of the composition}  and
$(\phi_{u,0},\dots,\phi_{u,n-1})$ the {\em connected components of the composable functions} $(\phi_0,\dots,\phi_n)$.
A sequence $(\phi_0,\dots,\phi_n)$ is called {\em connected}, if it only has only one connected component. This is equivalent to the statement that $\index(\phi_0 \codco \phi_n)$ is connected,
that is $|U|=1$. More generally, we define
$\depth(\phi_0,\dots,\phi_n)=|U|$.
Note that all the morphisms are uniquely fixed by choosing bases. Different choices of bases, however, only change these morphisms by isomorphisms, so that the condition that the connected components are irreducible is independent of basis.

\begin{prop}
\label{prop:connected}
%\label{connectedlem}
    A  UFC is \pher{} if and only if the connected components $\phi_u$ of any set of composable morphisms $(\phi_0,\dots,\phi_{n-1})$ are irreducible.
\end{prop}

\begin{proof}
  If the connected components are irreducible for any composable sequence, then they are for pairs. The reverse direction can be done by induction. For two morphisms this follows from the same proposition. If the statement is true for composable morphisms
  $(\phi_{0},\dots,\phi_{m})$, $2\leq m\leq n$, consider a composable sequence
  $(\phi_{0},\dots,\phi_{n})$, this gives rise to the pair $(\phi_0,\phi'_1=\phi_{n}\circ\dots\circ \phi_1)$.
  Now the connected components $\phi'_{1,u'}$ of $\phi'_1$, with $u'\in U'$ which is the pushout of the $S_i,V_i,i=1,\dots,n$  are irreducible and are thus part of a diagram in Figure~\ref{Udecompeq} which in turn shows that  the connected components corresponding to the final push--out $U$, that is the $\phi_u$ are irreducible.
  \end{proof}

By a similar argument:
\begin{lem}
 A sequence $(\phi_0,\dots,\phi_n)$ is connected if and only if all $0\leq i_1\leq\dots,\leq i_k\leq n$:
 $(\phi_0\codco \phi_{i_1},\phi_{i_1+1}\codco \phi_{i_2},\cdots,\phi_{i_k+1}\codco \phi_n)$ are connected. \qed
\end{lem}

\begin{rmk}
  Alternatively, this argument can be done by using  $\index$. In $\Cospan$, the corresponding push--out diagram  must have at least one one--point set on the level 2 and above. If there is such an object, it propagates, since it is not possible to have the push--out $\{*\}\leftarrow \emptyset\rightarrow\{*\}$  after the first level. The only cospan whose middle set is the empty set is the unit.
(This is analogous to the fact that $S^0$ is not connected, but all higher spheres are.)
This means that when composing sequences, composing with something connected makes the result connected.
It is possible that composing two non--connected sequences, the result is connected.
This mirrors the connectivity of cobordisms, see Remark \ref{rmk:cospanintepretation} \eqref{item:cobordism}.
\end{rmk}

%\subsection{Some structural results}

Using the functor $\index:\M\to \Cospan$, see \S\ref{par:indexing}, we can also phrase a criterion for a morphism $\bG$ to be $\ot$--irreducible for UFCs.
By the  functionality of the plus construction, there is a graph $\index(\bG)$   obtained functorially from $\bG$. This is given by redecorating the underlying box diagram by $\index(\phi_i)$ and $\index(\sigma_j)$ for the decoration by morphisms and basic morphisms.
In analogy with Proposition~\ref{indecprop}, we have:
\begin{lem}
For a \pher{} UFC,
the underlying graph of a composition graph is the same graph obtained as the image of the morphisms induced by the functor $\index$ and hence connectedness can be checked on the level of source and target objects.\qed
\end{lem}
\begin{prop}
\label{prop:herpgraph}
For a UFC $\FM$,
a morphism $\bG$ of $\monplus{\M,P}$ is $\ot$--irreducible if the  composition graph is connected.
\end{prop}

\begin{proof}
Assume that  $\bG$ is  isomorphic to the $\ot$--product of at least two generators $\G=\mu\circ(\bG'\bt \bG'')$, then
 $\index(\bG)=\mu(\index(\bG)\bt \index(\bG'))$ whose corresponding graph has at least two components.
  Vice--versa, assume that $\index(\bG)=\mu(\index(\bG)\bt \index(\bG'))$ is reducible.
  Then tracing through the diagram $\index(\bG)$, we see that the pre--images must also be disconnected, as they are simply decorations of the diagrams.
\end{proof}

%\subsection{Plus constructions, (cubical) FCs and UFC}

\begin{df}
Consider a \pher{} UFC $\FM$. Extending the definitions of index and depth, set
\begin{eqnarray*}
    \index(\gen{\phi}{\sigma}{n})&=&
\index(\sigma_0 \phi_0,\s_1\phi_1\kdk \s_{n-1}\phi_{n-1}\s_n)\\
\depth(\gen{\phi}{\sigma}{n})&=&\depth(\sigma_0 \phi_0,\s_1\phi_1\kdk \s_{n-1}\phi_{n-1}\s_n)
\end{eqnarray*}

A morphism $\gen{\phi}{\sigma}{n}$ is {\em connected} if
$\depth(\gen{\phi}{\sigma}{n})=1$.
\end{df}
\begin{rmk}
Note that due to the definition of push-outs, up to isomorphisms, we are free to ``rebracket'' the  isomorphisms in the arguments, that is $(\phi_0\sigma_0,  \dots )$ or $(\dots, \phi_i\sigma_{i-1},\dots)$.
\end{rmk}
\begin{prop}
\label{indecprop}
If $\FM$ is a \pher{} UFC, then the target of $\gen{\phi}{\sigma}{n}$ is irreducible  if and only if the morphism itself is connected.
\end{prop}

\begin{proof}
  If $\g=\gen{\phi}{\sigma}{n}$ is connected, assume that it is the product of at least two generators.
  By Lemma~\ref{lem:hereditarysplit}, the morphisms can be iteratively split to $\mu\circ(\g\simeq\gamma'\bt \gamma'')$.
  Since $\depth$ is additive under $\mu$, $1=\depth(\gamma'\ot\gamma'')=\depth(\gamma')+\depth(\gamma'')=\depth(\gamma)$ and thus either $\depth(\g')=0$ or $\depth(\g'')=0$.
  Let's assume $\depth(\g')=0$.
  This means that  the whole diagram of push--outs \eqref{pushoutdiageq} only has entries $\emptyset$,
  which in turn means that $\gamma'\simeq id_\unit$; and, hence $\g$ is irreducible.
  Note as depth is independent under isomorphisms this also holds for any essential decomposition.

  Vice--versa, assume that $\gen{\phi}{\sigma}{n}$ is not connected, then $|U|\geq 2$ and tracing back the pre--images in $\index(\gen{\phi}{\sigma}{n})$ one obtains a monoidal decomposition into connected components with more than one generator and hence it is reducible.
\end{proof}

\begin{rmk}
If $\M$ is just a UFC, $\depth$, see Definition \ref{df:degrees}, is defined and  irreducible, viz.\ $\depth=1$ implies connected, but not vice--versa.
\end{rmk}

\begin{cor}
\label{cor:connetedgen}
For a \pher{} UFC, any composite generator for $\monplus{\M}$  can be decomposed as
\begin{equation}
    \label{eq:hereditarygendecomp}
  % \gen{\phi}{\sigma}{n}
   \mu^k[\gen{\phi^1}{\sigma^1}{n_1}\bdb  \gen{\phi^k}{\sigma^k}{n_k}]\Perm
\end{equation}
where now the   $\gen{\phi^j}{\sigma^j}{n_1}$ are connected and $\Perm$ is the permutation rearranging the factors.
\end{cor}
 \begin{proof}
     The  statement follows from Corollary \ref{cor:decomp}.
 \end{proof}

\begin{cor}
%\label{prop:hufcp}
\label{cor:hufcp}
For  a \pher{} UFC $\M$, after passing to a strict model or restricting to essential generators, every morphism in
$\locmonplus{\M}$ has an essentially unique representative $(\mu_n,\mu_m  (\bG_1\bdb\bG_m)P)$ where the source of $\mu_n$ is $\phi_1\bdb\phi_n$, its target is $\phi_1\odo\phi_n$ with each $\phi_i$  irreducible and the $\bG_i$ are $\ot$--irreducible.  This is unique if $\M$ is a strict UFC.
\end{cor}
\begin{proof} Starting  from the results from Corollary \ref{cor:connetedgen}, we analyze the generating roofs.
  In a strict \pher{} UFC $\Iso(\M\da
\M)=\Indec^\bt$, every $\phi$ can be decomposed {\em essentially uniquely} as $\phi= \phi_1\odo \phi_n$.
Given a morphism $(\mu_1,f_1)$ with $\mu_1:(\phi_1^1 \bdb \phi_1^{m_1})\bdb (\phi_n^1\bdb \phi_n^{m_n})\to \phi_1\bdb\phi_n$ in the obvious indexing, decompose $\phi_j^i$ into irreducibles, let $\mu_n$ be the total map to the source and let $\mu_l$ be the map that factors the total map as $\mu_n=\mu_1\mu_l$. The roof $(\mu_l,id)$ then gives an equivalence to $(\mu,f_2)$ with $f_2=f_1\mu_l$ and we may assume that $\mu_n$ is maximal and thus unique, as any other roof used for an equivalence must have an identity on the left, cf.\ Lemma~\ref{lem:roof}.
Splitting the target $\psi= \psi_1\odo \psi_m$ into irreducibles and applying  Lemma~\ref{lem:hereditarysplit} yields the result.
\end{proof}

\label{par:standard}

\begin{thm}
%Result! This is one of the main results.
\label{thm:ufcp-equals-fc}
If  $\M$ is a \pher{} UFC, then the category $\plus{\M}$ is equivalent to the Feynman category whose underlying groupoid is given by
$\Iso(P\da P)$ and whose basic morphisms out of a given $\phi$ are given by a decomposition of $\phi\simeq \phi_1\odo\phi_n$,
and a connected composition graph validly decorated by these morphisms and isomorphisms.
If $\M$ is just a UFC (not necessarily \pher{}), then these morphisms still generate the morphisms of $\plus{\M}$ under $\ot$ and $\circ$.

  Assigning degree $0$ to isomorphisms and degree $1$ to the morphisms $\bar\gamma_{\phi_0,\phi_1}$ with $\phi_0, \phi_1 \in \jmath{\W}$ defines a proper degree function and the Feynman category is cubical.

Vice--versa, if a Feynman category $(\F,\V,\imath)$ is the standard plus construction  for a monoidal category $\M$, i.e.\ $\M^+=\F$, the category $\M$ is necessarily a \pher{} UFC with $\Indec=\V$ and
    $\jmath=\imath$.
    Moreover, such a Feynman category is cubical, viz.\ it has a proper degree function and is cubical with respect to it.
\end{thm}

\begin{proof}
 Given a \pher{} UFC $\M$ with standard gcp $P$, let $\V=\Indec$, then   $\V^\bt\simeq \Indec^\bt=(P\da P)=Iso(\M^+)$ with $\imath$ defined by the inclusion: $\V\to \Indec$, thus $\M^+$ satisfies
    condition (i) of an FC. For condition (iii), notice that the slice category of $\M$ being essentially
    small means that $(P \da P)$ is essentially small as well. The morphisms $\bar \gamma$ in the small model are then a subset of the product of morphisms which is also a set.
 For condition (ii),
using Proposition \ref{prop:loc-red-equiv}, the fact that the morphisms are described by decorated connected composition graph follows from Corollary~\ref{cor:hufcp} together with Proposition~\ref{prop:herpgraph}.
The claim about being an FC  then follows from Corollary \ref{cor:connetedgen}, as any morphism of $\monplus{\M}$  is given by a not--necessarily connected decorated composition graph by Corollary~\ref{cor:hufcp} which is the disjoint union of its connected components and the ordered disjoint union acts freely, that is without relations, since this is already true on the underlying graphs.
%Note that there is no permutation for $\ot$--$\M^+$ is generated by the
    Morphisms $\scs$ and the $\bar\gamma_{\phi_1,\phi_2}$ with $\phi_0, \phi_1$ are indecomposable by Lemma~\ref{lem:genstructure}. Indeed if $\phi_0\phi_1$ is decomposable, then as $\M$ is a hereditary UFC,
    $\phi_0\phi_1=\bigotimes_{u\in U}\phi_{0,u}\phi_{1,u}$ and by the interchange relations, cf.\ Proposition \ref{prop:loc-red-equiv}, $\bar\gamma_{\phi_0,\phi_1}=\bar\gamma_{\bigotimes_{u\in U}\phi_{0,u},\bigotimes_{u\in U}\phi_{1,u}}=\bigotimes_{u\in U}\bar \gamma_{\phi_{0,u},\phi_{1,u}}$. The relations are among these $\bar\g$ are quadratic by definition, cf.~Definition \ref{def:mloc} and hence the degree function is well--defined. It is also proper as non--isomorphisms have degree $>0$.
    For morphisms
    $\bar\gamma_{\phi_0\kdk \phi_n}$ there are indeed $n!$
    decompositions up to isomorphism, applying one $\bar \gamma_i=id\odo id \ot \bar\gamma_{\phi_i,\phi_{i+1}}\ot id \odo id$ at a time, and hence $\M^+=\F$ is cubical.
   The  statement about UFCs is simply that any morphism is represented by a concatenation of the $\bar \g$.

    Vice--versa, if $\F=\M^+$ for a (symmetric) monoidal $\M$ with gcp $P$, then $(P\da P)=Iso(\M^+)\simeq \V^\bt$ by condition (i) of a FC, and $\M$ is a UFC with $\Indec=\V$ and $\jmath=\imath$. The smallness of the slice categories follows from that of $\F$. Given a composable pair $(\phi_0,\phi_1)$ the morphism $\gamma_{\phi_0,\phi_1}$ is decomposable as $\bigotimes_{u\in U}\gamma_{\phi_{0,u},\phi_{1,u}}$ by virtue of $\M^+$ being an FC and hence having factorizable morphisms. Since $\F$ has only irreducible morphisms of type $(n,1)$,
    the images $\phi_{0,u}\phi_{1,u}$ must be indecomposable and moreover $\phi_0\phi_1\simeq\bigotimes_{u\in U} \phi_{0,u}\phi_{1,u}$ is the decomposition of the target of $\gamma_{\phi_0,\phi_1}$. This in turn means that the irreducible $\phi_u$ are the connected components and the UFC is \pher{}.
    The other statements, then follow from the
first part of the proof.
\end{proof}

\begin{cor}
    The Feynman categories for modular operads, their nc version, and props are not equivalent to a standard plus construction.
\end{cor}

\begin{proof}
   The FCs for modular, nc, modular operads and props  need two sets of degree $1$ generators, cf.\ \cite{feynman}.
\end{proof}
%Note that there are are relations between these generators.
% For $\plus{\M}$ the source is $\mu_n=\phi_{P^{-1}(1)}\bdb \phi_{P^{-1}(n)}$ is a permutation of the decorations given by the enumeration of the corresponding box diagram for $\bG$.
{\sc NB:} Props are related to these constructions through $B_+$ operators, see \cite{Michael} and \cite[\S 3.2.1]{feynman}.

\begin{prop}
\label{prop:mp-equals-fc}
For a \pher{} UFC $\M$, $\monplusp{\M}$ is the underlying monoidal category of  a Feynman category whose basic objects are $\Indec$ and whose  basic morphisms  are given by pairs
$(\mu_m, (\gen{\phi}{\sigma}{n})\Perm)$ with
connected $ \gen{\phi}{\sigma}{n}$ where $\phi_l=\phi^l_1\odo \phi^l_{m_l}$ with $\phi^l_i\in \Indec$ or $\phi^l_i=i_{id_X}$, the source of $\mu_m$ is $\phi^1_1\bdb\phi^n_{m_n}$ and its target is $\phi^n_1\bdb\phi^m_{n_m}$.

% These are in bijection with collections  $\phi_l^j\in \Indec,1=1\kdk n,j=1\kdk n_l$ and a permutation of them.
\end{prop}

\begin{proof}
We can level all  $\bG=\gen{\phi^i}{\s^i}{n}$ in \eqref{eq:hereditarygendecomp} in Corollary~\ref{cor:connetedgen}, such that they are of the same length and the use \ref{Gammaoteq}
to write them in the indicated form.
% Using the $i_{id_X}$,
% By the proposition above and
% Lemma~\ref{lem:roof} any morphism in $\monplusgcp{\M}$
% $(\mu_I,\mu^k\circ (\gen{\phi^1}{\sigma^1}{n_1}\bdb  \gen{\phi^k}{\sigma^k}{n_k}))$ with
% connected $ \gen{\phi^j}{\sigma^j}{n}$ and the source of $\mu_I$
% These are in bijection with collections  $\gen{\phi^1}{\sigma^1}{n_1}\kdk  \gen{\phi^k}{\sigma^k}{n_k}$ together with a enumeration of the $\phi_i^j$.
\end{proof}

 Parallel to Theorem~\ref{thm:ufcp-equals-fc}, we obtain

\begin{thm}
\label{thm  :level}
For a \pher{} UFC the  basic morphisms of $\plusp{\M}$ are represented by
 connected {\em leveled} composition graphs with black vertices
decorated by base morphisms and white vertices decorated by morphism in $\Indec$. \qed
\end{thm}
% \begin{proof}
% By the Proposition~\ref{prop:mncpgcpgen} the generators
% The enumeration fixes the
% permutation $P\in \SS_n$, , which specifies the source as  where the target $\{1\kdk n\}$
% is identified with $n_1^1\kdk n$
% \end{proof}

{\sc NB:} The decoration includes an enumeration of the labels. This enumeration specifies the source of the morphisms:
Identifying the enumeration $n_1^1\kdk n_1^{m_1}\kdk n^m_1\kdk n^m_{m_k}$ with $1\kdk m=\sum_l m_l$,
the enumeration is an element of $\SS_m$, where $m=\sum_{j=1}^k m_j$.
This also fixes the source as $\phi=\phi_{1}\odo\phi_m$. Two representations yield the same morphism if they represent levelings, by insertion of identities, of the same underlying non--leveled graph. These are the only relations.
There is a standard leveling, by inserting all identities at the top.

This leveling is necessary for these morphisms to lead to an indexing. As the corresponding bimodule has to be unital. A  morphism with $m$ levels in $\plusgcp{\M}$ will correspond to an element in the $m$--th level of the nerve. The hyp construction  guarantees that the isomorphisms are not counted, just like in the depth. This corresponds to taking the groupoid
nerve, as is appropriate according to the general philosophy.

\begin{prop}
For a \pher{} UFC, the morphisms of $\monplushyp{\M}$ and $\plushyp{\M}$ are the graphs as specified in  Proposition~\ref{prop:mncpgcpgen}
with the additional restriction that none of the labels $\phi^l_j$ are in $P(\B)$.
\end{prop}

\begin{proof}
Any  occurrence of such a label can be replaced by an isomorphism $(id,\s)$  or $(\s,id)$ and hence removed from a standard form.
\end{proof}

% By standard arguments:
% \begin{thm}
%      The gcp plus construction is a Feynman category with additional generators of type $(0,1)$ with degree $-1$. The hyp version and are cubical in the sense of \cite[Definition 7.2.2]{feynman}.
% \qed
% \end{thm}
\section{Graphical plus construction}
\label{sec:graphical-plus-construction}

In the case where $\M$ is a strict \pher{} UFC or Feynman category, the standard plus construction has a graphical description in more familiar terms.
The formalism of graphs we use is that of \cite{BorMan,feynman}.
We will then build on this by introducing groupoid-colored graphs expanding on \cite[Appendix B]{feynmanrep}.
This will allow us to describe the standard plus construction of a strict UFC $\mathfrak{M}$
via $\Iso(\M)$-colored graphs whose vertices are decorated by $\Indec$.
The graphical standard gcp and hyp constructions will be variations on this idea.

\subsection{Basics about Feynman category of graphs}
A {\em graph} is a tuple $\G=(V_\G, F_\G, \del_\G, i_\G)$ where: $V_\G$ is a set whose elements are called \emph{vertices}, $F_\G$ is a set whose elements are called \emph{flags},
$\del:F \to V$ is a map saying to which vertex a flag is attached,
$i_\G:F\to F$ is an involution so that $i_\G^2=id$, which gives the set of edges as the orbits of length $2$ and the tails as the orbits of length $1$.
A {\em morphism of graphs $\phi:\G\to \G'$} is a triple $(\phi_V,\phi^F,i_\phi)$ where
 $\phi_V:V_\G\ta V_{\G'}$ is a surjection, $\phi^F:\F_{\G'}\hookrightarrow \F_{\G}$ is an injection and $i_\phi$ is a {\em fixed point free} involution of $F_{\G} \setminus \phi^F(F_{\G'})$.
The orbits (all of length $2$) will be called the {\em ghost edges}.
There  are  several  compatibility  conditions  for which we refer to \cite[Appendix B]{feynmanrep}.
The {\em underlying or ghost graph} of a morphism to be $\gh(\phi) = (V_\G, F_\G, \del_\G, \hat\imath_\phi)$, where $\hat\imath_\phi$ is the trivial extension of $\phi$ to all of $F$.
The category is monoidal with respect to disjoint union of graphs.
A {\em corolla} is a one--vertex graph and an {\em aggregate} is a disjoint union of these.
It is clear that, by adding a vertex to the end of a tail, a graph defines an 1-dimensional CW complex and as usual a graph is a connected if the realization is.
A connected graph is a {\em tree} if its realization is contractible, and a {\em root} of a tree is a marked vertex.
The wide subcategory of aggregates $\Agg$ is part of a Feynman category $\mathfrak{G}$ whose groupoid is the subgroupoid $\Crl$ whose objects are corollas, see \cite[Proposition 2.1.2] {feynman}.
This category has $(n,1)$ generators that are connected, i.e.\ whose ghost graph is connected, these are called edge/loop contractions and those whose ghost graph have no edges, called mergers.
Degree one generators contract one edge, one loop, or merge two vertices, cf \cite[\S 5.1]{feynman}. Restricting to tree generators, this yields the Feynman subcategory $\FF^{\rm operads}$ which corepresents non--unital pseudo--operads.

An {\em oriented edge} is a choice of order on the set $\{f, i(f)\}$ defining the edge; this will be denoted by $\vec{e}=(f,i(f))$. Given an oriented edge $\vec{e}=(f,i(f))$, we denote the edge with the opposite orientation by $\cev{e}=(i(f),f)$.
A directed graph is {\em full} if at every vertex either no output flag is a part of an edge or all of them are and either no input flag is part of an edge or all input flags are part of an edge. A {\em two--level contraction} is a basic morphism of corollas whose ghost graph is full and has two levels. That is, it is isomorphic to a map $\phi:a_{i_1,o_1}\amalg a_{i_2,o_2}\to *_{i_1,o_2}$ where $a_{i_1,o_1},a_{i_2,o_2}$ are two aggregates and the ghost graph is connected and $\imath$ is a bijection between $o_1$ an $i_2$.

Restricting to the sub--Feynman category generated by these basic morphisms, one obtains
$\FF^{\rm nu \mdash prop}=(\Crl^{\rm dir}, \Agg^{\rm dir, 2-level},i)$ and
$\FF^{\rm nu \mdash properad}=(\Crl^{\rm dir},\Agg^{\rm dir, 2\mdash level, ctd},i)$
where in the latter case the ghost graphs of the basic morphisms are also connected. Here ``$\mathrm{nu}$'' stands for ``non--unital''', and
the names come from the fact that the categories of strong monoidal functors out of them are non--unital props and non--unital properads.
The generators of $\F^{\rm nu\mdash properads}$ are two-level edge contractions.
Similarly $\F^{\rm nu\mdash prop}$
is generated by 2--level contractions and mergers.
To add units, one proceeds with adding special black bivalent vertices expressing the presence of units, cf.\ \cite[\S 2.8.2]{feynman}, see also
\cite[\S 2.2.1]{feynman} for operads and \cite[B.1.4]{feynmanrep} for the plus
construction in Feynman categories. These encode morphisms $u:\unit\to id_{*_{1,1}}$ where $*_{1,1}$ is the corolla with vertex $*$, one input and one output. To implement
the unit properties, these graphs are taken modulo the relation of removing black vertices. The classes
are then simply the graphs without special vertices and on extra class, that of $[\dottree]$ which can be expressed as the loose edge $|$ following \cite{Markl}.
This is the special generator $u$.
This construction yields the Feynman categories $\FF^{\rm prop}$ and $\FF^{\rm properads}$.

% \subsection{Set-colored graphs}
% Colored graphs are a straightforward decoration of graphs, see \cite{feynman,decorated}. For a set $C$, define a $C$-colored graph to be a graph $\G$ together with a morphism $\clr: F \to C$ called a \emph{coloring} such that $\clr(f)=\clr(i(f))$. In other words, the two flags of an edge have the same color.

% A morphism of $C$-colored graphs is a morphism of the underlying graphs whose ghost graph is a $C$-colored graph.
% The disjoint union is defined as it was before.
% Hence, these assemble into a monoidal category $C \mdash\Graphs$ of $C$-colored graphs.

\subsection{Groupoid-colored graphs}
\label{par:groupoidgraphs}

A \emph{groupoid-colored graph} for a groupoid $\colorcat$ is a triple $(\G,\clr_\G,\bs_\G)$.
The data $\G$ is an ordinary graph.
The data $\clr: F \to \Obj(\colorcat)$ is a function called the \emph{coloring}.
The data $\bs_{\G}$ assigns to each directed edge $\vec{e}=(f, i(f))$ an isomorphism:
$
    \sigma(f,i(f)): \clr(f) \stackrel{\sim}{\to} \clr(i(f))
$.
Moreover, we require $\bs_{\G}$ to be compatible with reorientation so that $\sigma(\vec{e}) = \sigma^{-1}(\cev{e})$.

A morphism of $\colorcat$-colored graphs $\phi:(\G, \clr_\G, \bs_\G) \to (\G', \clr_{\G'}, \bs_{\G'})$ is a triple $(\phi, \bs_\phi, \btau)$. Here, $\phi:\G\to \G'$ is an ordinary graph morphism. The data $\bs_\phi$ is a collection of isomorphisms, one for each directed ghost edge $(f,i_\phi(f))$, $\sigma((f,i_\phi(f))): \clr(f)\stackrel{\sim}{\to} \clr(i_\phi(f))$ --- again satisfying $\sigma(\vec{e})=\sigma^{-1}(\cev{e})$ ---, and $\btau$ is a collection of isomorphisms, one for each $f'\in F'$, $\tau(f):\clr(f')\stackrel{\sim}{\to}\clr (\phi^F(f'))$. We call these \emph{flag recolorings}.
Again, we obtain a category $\colorcat\mdash\Graphs$ of $\colorcat$--colored graphs.

\begin{ex}
Let $\colorcat$ be a discrete category, that is $\colorcat$ only has identity morphisms, and let $V$ be the underlying set of objects. Then $\colorcat\mdash\Graphs$
is just the category of graphs which have $V$--colored flags.
\end{ex}

\begin{prop}
    Any functor $F: \colorcat \to \colorcat'$ induces a functor $\nu: \colorcat \mdash \Graphs \to \colorcat' \mdash \Graphs$ given by $(\G,\nu\circ \clr, \nu\circ\bs )$.
\end{prop}
\begin{proof}
Straightforward.
\end{proof}

% There is a functor $\colorcat \to sk(\colorcat)$ which sends an object to the isomorphism class of that object which induces a functor: $\colorcat \mdash \Graphs \to sk(\colorcat)\mdash \Graphs$.
%\RK{Is this a Grothendieck co-fibration?}
The unique functor from the groupoid $\colorcat$ to the trivial category yields a forgetful functor $\colorcat \mdash \Graphs \to \Graphs$.

\begin{definition}
  For a groupoid $\C$, define $\colorcat \mdash \F^{\rm nu \mdash prop}$ to be the wide subcategory of directed $\colorcat$-colored aggregates such that the morphisms are generated by isomorphisms and morphisms with full  directed ghost graphs for basic morphisms.
  Similarly, define
  $\colorcat \mdash \F^{\rm nu \mdash properad}$ to be the wide subcategory of directed $\colorcat$-colored aggregates such that the morphisms are generated by isomorphism and morphisms with full \emph{connected} directed ghost graphs for basic morphisms.
\end{definition}

\begin{rmk}
  Although (undirected) graphs must be groupoid colored, {\em directed} graphs can be colored by any category $\C$.
  In that case, colorings of in-flags will occur in $\C^{op}$ and colorings of out-flags will occur in $\C$.
  However, the graphical plus construction is most natural for the standard gcp where this extra generality would not be utilized.
\end{rmk}

\subsection{Plus construction as a decoration}
We will give a version of the standard plus construction for a strict \pher{} UFC $\M$ with $\Iso(\M) = \V^\bt$ as a decoration of $\V$-colored graphs.
As in \cite{feynman}, we will be concerned with the wide subcategories generated by aggregates of corollas, that is $\V \mdash \Agg$.

\begin{const}
  \label{tuple}
  For notational brevity, let $\mathrm{Indec}(-,-)$ denote the set of indecomposables morphisms for $\M$.
  Define the following set for each aggregate of $\V$-colored corollas where $\V$ is the groupoid such that $\Iso(\M) = \V^\bt$:
  \begin{equation}
    \label{eq:word}
    \mathrm{Tuple}_{\M}
    \left(
      \coprod_{v \in V} \ast_v
    \right) =
    \prod_{v \in V}
    \mathrm{Indec}
    \left(
      \bigotimes_{s \in F_{in}(v)} \clr(s),
      \bigotimes_{t \in F_{out}(v)} \clr(t)
    \right)
  \end{equation}
  Here $F_{in}(v)$ is the set of in-flags and $F_{out}(v)$ is the set of out-flags of the corolla $\ast_v$.

  Let $\V \mdash \mathcal{P}in$ be the subcategory of $\V \mdash \F^{\rm properad}$ where the objects are indexed disjoint unions of corollas $\coprod_{v \in V} \ast_v$ and the morphisms are generated by isomorphisms of corollas and bijections of indices.
  The morphisms of $\V \mdash \mathcal{P}in$ then act on the set \eqref{eq:word} in the following manner:\\
  {\sc Bijections of flags:} For a bijection of flags $b: \ast \to \ast'$, let $b_{in}(v): F'_{in}(v) \to F_{in}(v)$ and $b_{out}(v): F'_{out}(v) \to F_{out}(v)$ be the induced bijections on in-flags and out-flags.
    These induce commutativity constraints $C_{in}: \bigotimes_{s' \in F'_{in}} \clr(s') \to \bigotimes_{s \in F_{in}} \clr(s)$ and $C_{out}: \bigotimes_{t' \in F'_{out}} \clr(s') \to \bigotimes_{t \in F_{out}} \clr(s)$.
    Define $\O_{\M}(b): \O_{\M}(\ast_0) \to \O_{\M}(\ast_1)$ as:
    \begin{equation}
      \begin{tikzcd}[column sep = large]
        \displaystyle
        \mathrm{Indec}
        \left(
          \bigotimes_{s \in F_{in}} \clr(s),
          \bigotimes_{t \in F_{out}} \clr(t)
        \right) &
        \displaystyle
        \mathrm{Indec}
        \left(
          \bigotimes_{s' \in F'_{in}} \clr(s'),
          \bigotimes_{t' \in F'_{out}} \clr(t')
        \right)
        \arrow["{(C_{in}, C_{out}^{-1})}", from=1-1, to=1-2]
      \end{tikzcd}
    \end{equation}
    {\sc Recolorings:} For a label-recoloring $r: \ast \to \ast'$, let $\clr_{\ast}$ denote the color function for $\ast$ and let $\clr_{\ast'}$ denote the color function for $\ast'$.
    The graph morphism gives morphisms $\sigma_i: \clr_{\ast'}(s_i) \to \clr_{\ast}(s_i)$ for each index $i$.
    Likewise, we have morphisms $\tau_i: \clr_{\ast'}(t_j) \to \clr_{\ast}(t_j)$ for the out-flags.
    Define $\O_{\M}(r): \O_{\M}(\ast) \to \O_{\M}(\ast')$ to be:
    \begin{equation}
      \begin{tikzcd}[column sep = huge]
        \displaystyle
        \mathrm{Indec}
        \left(
          \bigotimes_i \clr_{\ast}(s_i),
          \bigotimes_j \clr_{\ast}(t_j)
        \right) &
        \displaystyle
        \mathrm{Indec}
        \left(
          \bigotimes_i \clr_{\ast'}(s_{i}),
          \bigotimes_j \clr_{\ast'}(t_{j})
        \right)
        \arrow["{(\bigotimes_i \sigma_i, \bigotimes_j \tau^{-1}_j)}", from=1-1, to=1-2]
      \end{tikzcd}
    \end{equation}
    {\sc Bijections of indices:} Straightforward, commutativity constraints $\coprod_{v \in V} \to \coprod_{v' \in V'}$ become commutativity constraints $\prod_{v \in V} \to \prod_{v' \in V'}$.
\end{const}

\begin{prop}
  \label{OM-functor}
  The sets $\mathrm{Tuple}_{\M}$ can be extended into a strong monoidal functor $\O_{\M}: \V \mdash \F^{\rm nu\mdash properad} \to \Set$.
\end{prop}
\begin{proof}
For an aggregate $\G$ with flags $F$ and vertices $V$, define
\begin{equation}
  \O_{\M}(\G) :=
  \int^{\amalg_v \ast_v \in \V \mdash \mathcal{P}in}
  \Iso(\amalg_v \ast_v, \G)
  \times
  \mathrm{Tuple}_{\M}(\amalg_v \ast_v)
\end{equation}
  {\sc Isomorphisms:}
  The composition $\Iso(\otimes_v \ast_v, \G) \times \Iso(\G, \G') \to \Iso(\otimes_v \ast_v, \G')$ induces $\O_{\M}(\sigma): \O_{\M}(\G) \to \O_{\M)}(\G')$ for each $\sigma \in \Iso(\G, \G')$.\\
  {\sc 2-level morphisms:}
  For a two--level graph morphism $(\coprod_i \ast_i) \amalg (\coprod_j \ast'_j) \to \star$
  %whose ghost graph is connected 2-level graph,
  we have:\\
  $\O_{P}
    (
      (\coprod_i \ast_i) \amalg (\coprod_j \ast'_j)
    )
    = \O_{\M}(\coprod_i \ast_i) \times \O_{\M}(\coprod_j \ast'_j)
$.
  Hence an element on the left is a pair of morphisms $(\phi, \phi')$ written below with colors $(X_i)$, $(Y_j)$, $(Y'_j)$, and $(Z_k)$:
  $
      \phi
      \in \O_{\M}(\coprod_i \ast_i)
      \subseteq \Hom(X_1 \odo X_m, Y_1 \odo Y_n)$ and
      $\phi'
      \in \O_{\M}(\coprod_j \ast'_j)
      \subseteq \Hom(Y'_1 \odo Y'_n, Z_1 \odo Z_p)
 $.
  The graph morphism induces a bijection between the out-tails of the aggregate $\coprod_i \ast_i$ and the in-tails of the aggregate $\coprod_j \ast'_j$.
  This gives a commutativity constraint.
  Moreover, each ghost edge of the morphism corresponds to an isomorphism.
  All together, this defines a morphism $\sigma \in \Hom_{\V^\bt}(Y_1 \bdb Y_n, Y'_1 \bdb Y'_n)$.
  Now define the set map $
      \O_{\M}
      (
      (\coprod_i \ast_i) \amalg (\coprod_j \ast'_j)
      )
      \to \O_{\M}(\star)$
      by
    $
        \phi \otimes \phi' \mapsto \phi' \circ P(\sigma) \circ \phi
    $.
\end{proof}

\begin{prop}
    \label{graphical-plus}
    Given a strict \pher{} UFC $\M$, the \emph{graphical plus construction} is defined by the decoration $\plusgr{\M} := (\V \mdash \F^{\rm properad})_{\dec \, \O_{\M}}$ where $\V^{\bt} = \Iso(\M)$.
    Moreover, there is a monoidal equivalence between $\plus{\M}$ and $\plusgr{\M}$.
\end{prop}
\noindent {\sc NB:} For a strong monoidal functor $\O: \F \to \Set$, the category $\F_{\dec \O}$ is the Grothendieck construction for Feynman categories, see \cite{decorated}. Explictly, the objects are pairs $(X, a_X\in \O(X))$ and morphisms are $\phi:X\to Y$ for which $\O(\phi)(a_X)=a_Y$.
\begin{proof}
  As short hand, let $(\G; \Sigma, (\phi_v)_{v \in V})$ denote the graph $\G$ decorated with the class of $(\Sigma, (\phi_v)_{v \in V}) \in \O_{\M}(\G)$.
  By the definition of a UFC, each morphism $\Phi$ in $\plus{\M}$ can be written as $\Phi = (\sigma \Da \sigma')(\phi_1 \odo \phi_n)$ where the $\phi_i$ are basic.
  Define a strong monoidal functor $Graph: \plus{\M} \to \plusgr{\M}$ by
\begin{equation}
    Graph(\Phi) = (\amalg_{i=1}^n \ast_i; (\sigma \Da \sigma'), (\phi_i)_{i=1}^n)
  \end{equation}
  Where $\ast_i$ is the $\V$--colored corolla whose in-flags are colored by the source of $\phi_i$ and the out-flags are colored by the target of $\phi_i$.
  This is well-defined because any other factorization belongs to the same class in $\O_{\M}(\G)$.

  The definition on morphisms is straightforward.
  The $Graph$ functor respects inner equivariance, outer equivariance and internal associativity because composition in the category $\M$ is associative.
  Internal interchange follows from $\O_{\M}$ being a strong monoidal functor.

  For the other direction, first pick a functor $Card: \FinSet \to \FinSet$ that sends every finite set $X$ to $\{1, \ldots, |X|\}$ and a natural isomorphism $\kappa: \id \Rightarrow Card$.
  Now, define a monoidal functor $Mor: \plusgr{\M} \to \plus{\M}$ by
 {\sc Objects:}
    define $Mor(\G; \Sigma, (\phi_v)_{V})
      = \Sigma ( \phi_{\kappa_V^{-1}(1)} \ot \ldots \ot \phi_{\kappa_V^{-1}(|V|)} )$.\\
 {\sc Vertex bijections:}
    If $B$ is a graph morphism which is exclusively a bijection of vertices, then $Word(B)$ is the morphism
 $ (\Gamma; \Sigma, (\phi_v)_{V}) \to (B(\Gamma); B \circ \Sigma, (\phi_v)_{V})
   $.
 {\sc Flag isomorphisms:} reverse the $Graph$ functor.\\
 {\sc 2--level--contractions:} reverse the $Graph$ functor.

  We can always arrange for $\kappa_{\underline{n}}: \underline{n} \to \underline{n}$ to be the identity.
  In that case, the composition $Mor \circ Graph$ is the identity functor.
  On the other hand, $Graph(Mor(\Gamma; \Sigma, (\phi_v)_{V})) = (\Gamma; \Sigma, (\phi_{\kappa_{V}^{-1}(i)})_{i=1}^{|V|})$.
  Let $K_V$ be the vertex bijection induced by $\kappa_V: V \to Card(V)$.
  This gives a morphism of decorated graphs:
$
    (\Gamma; \Sigma, (\phi_v)_{V}) \to (\Gamma; \Sigma \circ K_V^{-1}, (\phi_{\kappa_{V}^{-1}(i)})_{i=1}^{|V|})
 $.
  This is a natural isomorphism $\id_{\gplus{\M}} \cong Graph \circ Mor$.
\end{proof}

\begin{rmk}
  Due to the way the $\Iso(\M)$ actions are handled in Construction~\ref{tuple} and Definition\ref{OM-functor}, it is straight forward to adapt this construction to any pointing by a groupoid.
\end{rmk}

\begin{cor}
  \label{cor:cospan-plus-equals-properad}
  $\plus{\Cospan}$ is equivalent to $\nuproperads$
\end{cor}
\begin{proof}
  Let $\mathsf{Cospan}$ denote the skeletal version of $\Cospan$ whose objects are the sets $\un, n\in \N_0$ and the morphisms are the cospans for these objects.
  Since the plus construction respects the principle of equivalence, $\plus{\Cospan} \simeq \plus{\mathsf{Cospan}}$.
  From the last result, we know $\plus{\mathsf{Cospan}} \simeq \plusgr{\mathsf{Cospan}}$.
  The $\mathrm{Indec}$ sets for $\mathsf{Cospan}$ are always a singleton, so the objects of $\plusgr{\mathsf{Cospan}}$ are all aggregates and the morphisms are those of $\properads$.
  % Check isomorphisms
\end{proof}

\subsection{Graphical gcp construction}

%   And $\V \mdash \F^{prop}$ and
%   $\V \mdash \F^{properad}$ by adding a unit for each $\s\in \V$ that is morphisms $\unit \to *_{X,Y}$ where
%   $*_{s(\sigma),t(\sigma)}$ is the corolla with $s(\sigma)$ as the input color of its lone input flag and $t(\sigma)$ as the color of its unique output flag.
%   These will be graphically encoded by labeled black bivalent vertices where the label is $\sigma$.
Define $\vert^X_Y$ to be a corolla with one in-tail labeled by $X$ and one out-tail labeled by $Y$.
Define $\V \mdash \F^{\rm properad,gcp}$ from $\V \mdash \FF^{\rm nu \mdash properad}$ by adding morphisms $i_{\sigma}: \emptyset \to \vert^X_Y$ for each morphism $\sigma: X \to Y$ in $\V$. The presence of these morphisms will be encoded by bivalent vertices $\dottree \sigma$.
We then quotient by the following relations:
\begin{enumerate}
    \item Compatibility with flag recoloring: For isomorphisms $\sigma: X \to Y$, $\tau: X \to X'$ and $\tau': Y \to Y'$, the following diagram commutes
    \begin{equation}
    \begin{tikzcd}
	\emptyset
	&& {\vert^{X'}_{Y'}} \\
	& {\vert^X_Y}
	\arrow["{i_\sigma}"', from=1-1, to=2-2]
	\arrow["{(\tau, \tau')}"', from=2-2, to=1-3]
	\arrow["{i_{\tau' \sigma\tau}}", from=1-1, to=1-3]
	\end{tikzcd}
    \end{equation}
    \item Compatibility with graph composition:
      \begin{equation}
        \begin{tikzcd}
          \Gamma \arrow[rr, "{(\coprod_j \sigma_j, \coprod_k \tau_k)}"] \arrow[rd, "id \amalg \coprod_j i_{\sigma_j} \amalg \coprod_k i_{\tau_k}"'] & & \Gamma' \\
          & \Gamma \amalg (\coprod_j |^{X_j'}_{X_j}) \amalg (\coprod_k |^{Y_k}_{Y'_k}) \arrow[ru] &
        \end{tikzcd}
      \end{equation}
\end{enumerate}

\begin{prop}
  There is a canonical $\O_{\M}^{gcp}: \V \mdash\F^{\rm properad, gcp} \to \Set$ extending $\O_{\M}$.
\end{prop}
\begin{proof}
    Send $i_\sigma: \emptyset \to \vert^{s(\sigma)}_{t(\sigma)}$ to the pointing $\{ \id_{\unit} \} \to \O_{\M}\left(\vert^{s(\sigma)}_{t(\sigma)}\right)$ which selects the element $\sigma$. This respects the relations imposed on the pointings. Hence this defines a monoidal functor.
\end{proof}

\begin{prop}
\label{prop:graphicalgcp}
    For a strict \pher{} UFC $\M$, define the \emph{graphical gcp construction} as the decoration $\plusgcpgr{\M} := (\V \mdash\F^{\rm properads, gcp})_{\dec \O_{\M}^{gcp}}$.
    Then there is a monoidal equivalence between $\plusgcp{\M}$ and $\plusgcpgr{\M}$.
\end{prop}
\begin{proof}
     This is a routine modification of the proof of Theorem~\ref{graphical-plus}.
\end{proof}

\begin{cor}
  $\plusgcp{\Cospan} = \properads$.
\end{cor}
\begin{proof}
  Similar to Corollary~\ref{cor:cospan-plus-equals-properad}.
\end{proof}

\subsection{Graphical Hyper construction}

Define $\V\mdash\F^{\rm properad, hyp}$ from $\V\mdash\F^{\rm properad, gcp}$ by adding morphisms $r_{\sigma}: \vert^{s(\sigma)}_{t(\sigma)} \to \emptyset$ for each morphism $\sigma$ in $\V$ and quotient by the relation $r_{\sigma} \circ i_{\sigma} = id_{\emptyset}$.

\begin{prop}
    There is a canonical extension $\O_{\M}^{Hyp}: \V\mdash\F^{\rm properad, hyp} \to \Set$ of $\O_{\M}$.
\end{prop}
\begin{proof}
    The only option is to send $r_{\sigma}: \vert^{s(\sigma)}_{t(\sigma)} \to \emptyset$ to the unique map $\O_{\M}(\vert^{s(\sigma)}_{t(\sigma)}) \to \{ \id_{\unit} \}$.
    This clearly respects the required relations.
\end{proof}

\begin{prop}
    For a strict \pher{} UFC $\M$, define  $\plushypgr{\M} := (\V \mdash \mathcal{F}^{properads, hyp})_{\dec \, \O_{\M}^{Hyp}}$.
    Then there is a monoidal equivalence between $\plushyp{\M}$ and $\plushypgr{\M}$.
\end{prop}
\begin{proof}
     This is another routine modification of the proof of Theorem~\ref{graphical-plus}.
\end{proof}

\subsubsection{Graphical monoidal plus construction}

One can extend the functors $\O_{\M}$, $\O_{\M}^{gcp}$ and $\O_{\M}^{hyp}$ above to non--connected graphs by first modifying the previous definitions:
\begin{align}
  \mathrm{MTuple}_{\M}
  \left(
    \coprod_{v \in V} \ast_v
  \right)
  &:=
  \prod_{v \in V}
  \Hom
  \left(
    \bigotimes_{s \in F_{in}(v)} \clr(s),
    \bigotimes_{t \in F_{out}(v)} \clr(t)
  \right) \\
  \O^{\ot}_{\M}(\G)
  &:=
  \int^{\amalg_v \ast_v \in \V \mdash \mathcal{P}in}
  \Iso(\amalg_v \ast_v, \G)
  \times
  \mathrm{MTuple}_{\M}(\amalg_v \ast_v)
\end{align}
We then define the merger $\mu$ so that the labels on the vertices are tensored together by the morphisms $\mu_{\phi_1,\phi_2}$.
This leads to a graphical monoidal plus construction $\monplusgr{\M} := \V\mdash\props_{dec \, \O^{\ot}_{\M}}$ and its gcp and hyp versions $\monplusgcp{\M}$ and $\monplushypgr{\M}$ given by decorating $\V \mdash \F^{\rm props, gcp}$ and $\V \mdash \F^{\rm props, hyp}$.
\begin{prop}
\label{prop:graphnc}
For a \pher{} UFC, there are monoidal equivalences  between
$\monplusgr{\M}$ and $\monplus{\M}$ as well as between $\monplusgcpgr{\M}$ and $\monplusgcp{\M}$ and between $\monplushypgr{\M}$ and $\monplushyp{\M}$.
\end{prop}
\begin{proof}
     This is again a variation of the construction above.
\end{proof}

\begin{ex}
$\monplus{\Cospan} = (\nuprops)^{nc}$ and
$\monplusgcp{\Cospan} = (\props)^{nc}$ which
can be merged to $\nuprops$ and $\props$, cf. \cite[Example 3.2.4]{feynman}.
\end{ex}

\subsection{Summary}
\label{sec:summary}

\begin{thm}
%Result!
  \label{thm:graphs}
  For a \pher{} strict UFC $\M$, the graphical plus constructions $\monplusgr{\M}$ and $\plusgr{\M}$ are equivalent to $\monplus{\M}$ and $\plus{\M}$. Moreover their gcp and hyp versions are also equivalent. \qed
\end{thm}

\begin{cor}
    The FC $\frak{F}^{\it cyc}$ corepresenting cyclic operads is not a standard plus construction.
\end{cor}
\begin{proof}
 $\frak{F}^{cyc}$ as presented in \cite{feynman} has a  presentation, with one type of generator, edge contractions, which are given by one--edge trees, and is cubical, but the graphs  underlying the morphisms are not oriented. To get plus type generators one would need orientation on the graphs. Note the degree $1$ generators are unique up to isomorphism in this case.
 Putting an orientation on trees with one edge, would yield two one--edge 2--level trees corresponding to generators $\g_{\phi_0,\phi_1}$ and $\gamma_{\phi_1,\phi_0}$, duplicating the number of generators.
   In these generators, the category would not be cubical. There are $n!$ enumerations of edges on a tree, but there are an additional $2^{|E|}$ choices of orientations of them.
\end{proof}

\subsubsection{Relation to composition graphs}
\label{par:directgraph}
There is also a direct construction using the composition graphs.
The graphical procedure is as follows:
Given a composition graph,
\begin{enumerate}
    \item split each edge labeled by $X=X_1\odo X_k$ where the $X_i$ are irreducible into $k$ edges labeled by the $X_i$.
 \item   split each
white vertex labeled by $\phi=\phi_1\odo \phi_n$, with $\phi_i\in \Indec$  into $n$ vertices labeled by $\phi_i$
where the flags are assigned to the vertices according to the sources and targets of the $\phi_i$, using that $s(\phi)=s(\phi_1)\odo s(\phi_n)$ an likewise for the target.

\item Split the black labeled with an isomorphisms $\s=\s_1\odo s_n$ in a similar fashion where the isomorphism $\s_i:X_i\to Y_i$ are such that $X_i,Y_i\in \V$.
\end{enumerate}

The result is a labeled graph of the type found in the graphical plus constructions, see Figure~\ref{fig:graphex}.

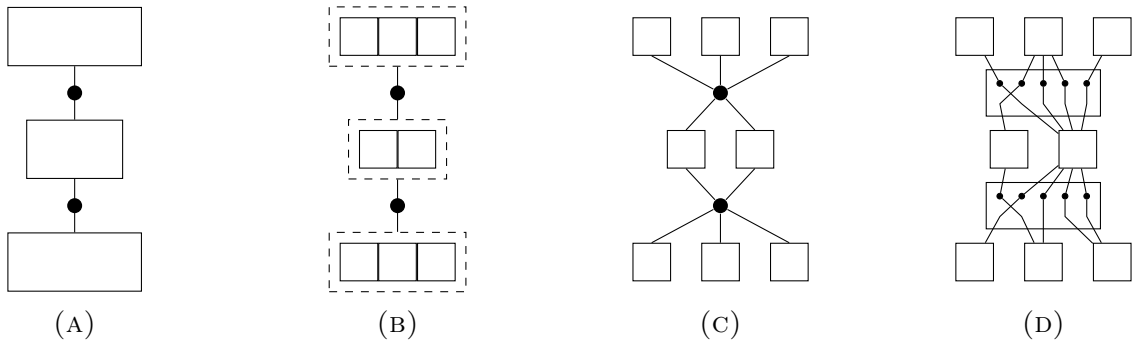
\begin{figure}
  \centering
  \begin{subfigure}[c]{0.22\textwidth}
    \centering
    \begin{tikzpicture}[scale=.9]
      \tikzstyle{basic-mor}=[minimum size=5mm]
      \tikzstyle{mor}=[draw, nodes={basic-mor}, column sep = 0mm]
      \tikzstyle{iso}=[circle, fill=black, inner sep=2pt]
      \matrix(level3) [mor] at (0,3)
      {
        \node(3a) {}; & \node(3b){}; & \node(3c) {}; \\
      };
      \node (iso32) at (0, 2.25) [iso] {};
      \matrix (level2)[mor] at (0,1.5)
      {
        \node(2a) {}; & \node(2b) {}; \\
      };
      \node (iso21) at (0, 0.75) [iso] {};
      \matrix(level1)[mor] at (0,0)
      {
        \node (1a) {}; & \node (1b) {}; & \node (1c) {}; \\
      };
      \draw (level3.south) -- (level2.north);
      \draw (level2.south) -- (level1.north);
    \end{tikzpicture}
    \caption{}
  \end{subfigure}
  \hfill
  \begin{subfigure}[c]{0.22\textwidth}
    \centering
    \begin{tikzpicture}[scale=.9]
      \tikzstyle{basic-mor}=[draw, solid, minimum size=5mm]
      \tikzstyle{mor}=[draw, dashed, nodes={basic-mor}, column sep = 0mm]
      \tikzstyle{iso}=[circle, fill=black, inner sep=2pt]
      \matrix(level3) [mor] at (0,3)
      {
        \node(3a) {}; & \node(3b){}; & \node(3c) {}; \\
      };
      \node (iso32) at (0, 2.25) [iso] {};
      \matrix (level2)[mor] at (0,1.5)
      {
        \node(2a) {}; & \node(2b) {}; \\
      };
      \node (iso21) at (0, 0.75) [iso] {};
      \matrix(level1)[mor] at (0,0)
      {
        \node (1a) {}; & \node (1b) {}; & \node (1c) {}; \\
      };
      \draw (level3.south) -- (level2.north);
      \draw (level2.south) -- (level1.north);
    \end{tikzpicture}
    \caption{}
  \end{subfigure}
  \hfill
  \begin{subfigure}[c]{0.22\textwidth}
    \centering
    \begin{tikzpicture}[scale=.9]
      \tikzstyle{basic-mor}=[draw, minimum size=5mm]
      \tikzstyle{mor}=[draw=none, nodes={basic-mor}, column sep = 4mm]
      \tikzstyle{iso}=[circle, fill=black, inner sep=2pt]
      \matrix(level3) [mor] at (0,3)
      {
        \node(3a) {}; & \node(3b){}; & \node(3c) {}; \\
      };
      \node (iso32) at (0, 2.25) [iso] {};
      \matrix (level2)[mor] at (0,1.5)
      {
        \node(2a) {}; & \node(2b) {}; \\
      };
      \node (iso21) at (0, 0.75) [iso] {};
      \matrix(level1)[mor] at (0,0)
      {
        \node (1a) {}; & \node (1b) {}; & \node (1c) {}; \\
      };
      \draw (3a.south) -- (iso32);
      \draw (3b.south) -- (iso32);
      \draw (3c.south) -- (iso32);
      \draw (iso32) -- (2a.north);
      \draw (iso32) -- (2b.north);

      \draw (2a.south) -- (iso21);
      \draw (2b.south) -- (iso21);
      \draw (iso21) -- (1a.north);
      \draw (iso21) -- (1b.north);
      \draw (iso21) -- (1c.north);
    \end{tikzpicture}
    \caption{}
  \end{subfigure}
  \hfill
  \begin{subfigure}[c]{0.22\textwidth}
    \centering
    \begin{tikzpicture}[scale=.9]
      \tikzstyle{basic-mor}=[draw, minimum size=5mm]
      \tikzstyle{mor}=[draw=none, nodes={basic-mor}, column sep = 4mm]
      \tikzstyle{iso}=[draw, nodes={inner sep=0.75pt}, column sep = 2mm, row sep = 2mm]
      \tikzstyle{iso-dot}=[draw, circle, fill=black]
      \matrix(level3) [mor] at (0,3)
      {
        \node(3a) {}; & \node(3b){}; & \node(3c) {}; \\
      };

      \matrix(iso32) [iso] at (0,2.25)
      {
        \node(32a3) [iso-dot] {}; & \node(32b3) [iso-dot] {}; & \node(32c3) [iso-dot] {}; & \node(32d3) [iso-dot] {}; & \node(32e3) [iso-dot] {}; \\
        \node(32a2)[inner sep = 0mm] {}; & \node(32b2) {}; & \node(32c2) {}; & \node(32d2) {}; & \node(32e2) {}; \\
      };

      \draw (3a) -- (32a3);
      \draw (3b) -- (32b3);
      \draw (3b) -- (32c3);
      \draw (3b) -- (32d3);
      \draw (3c) -- (32e3);

      \matrix (level2)[mor] at (0,1.5)
      {
        \node(2a) {}; & \node(2b) {}; \\
      };

      \draw (32a3) -- (32b2.center) -- (2b);
      \draw (32b3) -- (32a2.center) -- (2a);
      \draw (32c3) -- (32c2.center) -- (2b);
      \draw (32d3) -- (32d2.center) -- (2b);
      \draw (32e3) -- (32e2.center) -- (2b);

      \matrix(iso21) [iso] at (0,0.75)
      {
        \node(21a2) [iso-dot] {}; & \node(21b2) [iso-dot] {}; & \node(21c2) [iso-dot] {}; & \node(21d2) [iso-dot] {}; & \node(21e2) [iso-dot] {}; \\
        \node(21a1) {}; & \node(21b1) {}; & \node(21c1) {}; & \node(21d1) {}; & \node(21e1) {}; \\
      };

      \draw (2a) -- (21a2);
      \draw (2b) -- (21b2);
      \draw (2b) -- (21c2);
      \draw (2b) -- (21d2);
      \draw (2b) -- (21e2);

      \matrix(level1)[mor] at (0,0)
      {
        \node (1a) {}; & \node (1b) {}; & \node (1c) {}; \\
      };

      \draw (21a2) -- (21b1.center) -- (1b);
      \draw (21b2) -- (21a1.center) -- (1a);
      \draw (21c2) -- (21c1.center) -- (1b);
      \draw (21d2) -- (21d1.center) -- (1c);
      \draw (21e2) -- (21e1.center) -- (1c);
    \end{tikzpicture}
    \caption{}
  \end{subfigure}
  \caption{\label{fig:graphex} An example of converting a composition graph into a graph appearing in the graphical plus construction. Here for instructive purposes, we fist split the white vertices splitting the output edges as well, then split the black vertices along with the edges.}
\end{figure}

In this interpretation, the $\g_{\phi_1,\phi_2}$ which correspond to removing horizontal line segments correspond to 2-level contractions.

{\sc NB} Note that the same procedure already works more generally for the suspension graphs, yielding a graphical description for these cases as well, see Figure~\ref{fig:suspensionex}.
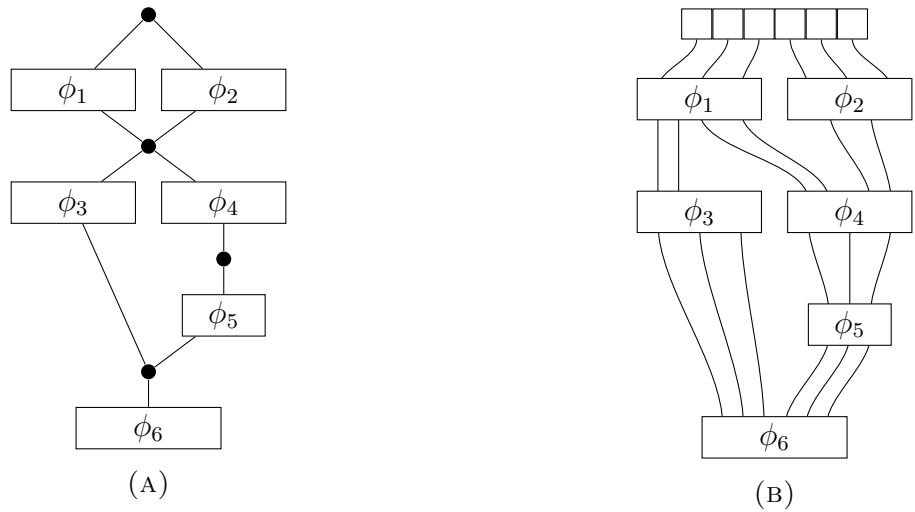
\begin{figure}
  \centering
  \hfill
  \begin{subfigure}[c]{0.40\linewidth}
    \centering
    \begin{tikzpicture}[scale=.6]
      \tikzstyle{iso}=[circle, fill=black, inner sep=2pt]
      \tikzstyle{mor}=[draw, nodes={minimum size=1mm},  fill=white, column sep = 0mm]
      \node[iso] (isoa) at (0, 5) {};
      \matrix [mor] (phi1) at (-1.5, 4)
      {
        \node(phi1a) {}; & \node(phi1b){}; & \node(phi1c) {};  & \node(phi1d) {}; & \node(phi1e) {}; \\
      };
      \node at (phi1) {$\phi_1$} {};
      \matrix[mor] (phi2)  at (1.5, 4)
      {
        \node(phi2a) {}; & \node(phi2b){}; & \node(phi2c) {};  & \node(phi2d) {}; & \node(phi2e) {}; \\
      };
      \node at (phi2) {$\phi_2$};
      \node[iso] (isob) at (0, 3.25) {};
      \matrix [mor] (phi3) at (-1.5, 2.5)
      {
        \node(phi3a) {}; & \node(phi3b){}; & \node(phi3c) {};  & \node(phi3d) {}; & \node(phi3e) {}; \\
      };
      \node at (phi3) {$\phi_3$};
      \matrix[mor] (phi4)  at (1.5, 2.5)
      {
        \node(phi4a) {}; & \node(phi4b){}; & \node(phi4c) {};  & \node(phi4d) {}; & \node(phi4e) {}; \\
      };
      \node at (phi4) {$\phi_4$};
      \node[iso] (isoc) at (1.5, 1.75) {};
      \matrix [mor] (phi5) at (1.5, 1)
      {
        \node(phi5a) {}; & \node(phi5b){}; & \node(phi5c) {}; \\
      };
      \node at (phi5) {$\phi_5$};
      \node[iso] (isod) at (0, 0.25) {};
      \matrix[mor] (phi6)  at (0, -0.5)
      {
        \node(phi6a) {}; & \node(phi6b){}; & \node(phi6c) {};  & \node(phi6d) {}; & \node(phi6e) {};  & \node(phi6f) {};  \\
      };
      \node at (phi6) {$\phi_6$};
      \begin{scope}[on background layer]
        \draw (isoa) -- (phi1) -- (isob);
        \draw (isoa) -- (phi2) -- (isob);
        \draw (isob) -- (phi3) -- (isod);
        \draw (isob) -- (phi4) -- (isoc);
        \draw (isoc) -- (phi5) -- (isod);
        \draw (isod) -- (phi6);
      \end{scope}
    \end{tikzpicture}
    \caption{}
  \end{subfigure}
  \hfill
  \begin{subfigure}[c]{0.40\linewidth}
    \centering
    \begin{tikzpicture}[scale=.6]
      \tikzstyle{mor}=[draw, nodes={minimum size=1mm},  fill=white, column sep = 0mm]
      \matrix [column sep = 0mm, nodes={draw, minimum size=4mm}] (iso) at (0, 5)
      {
        \node(isoa) {}; & \node(isob){}; & \node(isoc) {};  & \node(isod) {}; & \node(isoe) {}; & \node(isof) {}; \\
      };
      \matrix [mor] (phi1) at (-1.5, 4)
      {
        \node(phi1a) {}; & \node(phi1b){}; & \node(phi1c) {};  & \node(phi1d) {}; & \node(phi1e) {}; \\
      };
      \node at (phi1) {$\phi_1$};
      \matrix[mor] (phi2)  at (1.5, 4)
      {
        \node(phi2a) {}; & \node(phi2b){}; & \node(phi2c) {};  & \node(phi2d) {}; & \node(phi2e) {}; \\
      };
      \node at (phi2) {$\phi_2$};
      \matrix [mor] (phi3) at (-1.5, 2.5)
      {
        \node(phi3a) {}; & \node(phi3b){}; & \node(phi3c) {};  & \node(phi3d) {}; & \node(phi3e) {}; \\
      };
      \node at (phi3) {$\phi_3$};
      \matrix[mor] (phi4)  at (1.5, 2.5)
      {
        \node(phi4a) {}; & \node(phi4b){}; & \node(phi4c) {};  & \node(phi4d) {}; & \node(phi4e) {}; \\
      };
      \node at (phi4) {$\phi_4$};
      \matrix [mor] (phi5) at (1.5, 1)
      {
        \node(phi5a) {}; & \node(phi5b){}; & \node(phi5c) {}; \\
      };
      \node at (phi5) {$\phi_5$};
      \matrix[mor] (phi6)  at (0, -0.5)
      {
        \node(phi6a) {}; & \node(phi6b){}; & \node(phi6c) {};  & \node(phi6d) {}; & \node(phi6e) {};  & \node(phi6f) {};  \\
      };
      \node at (phi6) {$\phi_6$};
      \tikzstyle{loose}=[in=90, out=-90, looseness=0.75]
      \begin{scope}[on background layer]
        \draw[loose] (isoa) to (phi1a);
        \draw[loose] (isob) to (phi1c);
        \draw[loose] (isoc) to (phi1e);
        \draw[loose] (isod) to (phi2a);
        \draw[loose] (isoe) to (phi2c);
        \draw[loose] (isof) to (phi2e);
        \draw[loose] (phi1a) to (phi3a);
        \draw[loose] (phi1b) to (phi3b);
        \draw[loose] (phi1c) to (phi4a);
        \draw[loose] (phi1e) to (phi4b);
        \draw[loose] (phi2b) to (phi4d);
        \draw[loose] (phi2d) to (phi4e);
        \draw[loose] (phi4a) to (phi5a);
        \draw[loose] (phi4c) to (phi5b);
        \draw[loose] (phi4e) to (phi5c);
        \draw[loose] (phi3a) to (phi6a);
        \draw[loose] (phi3c) to (phi6b);
        \draw[loose] (phi3e) to (phi6c);
        \draw[loose] (phi5a) to (phi6d);
        \draw[loose] (phi5b) to (phi6e);
        \draw[loose] (phi5c) to (phi6f);
      \end{scope}
    \end{tikzpicture}
    \caption{}
  \end{subfigure}
  \caption{\label{fig:suspensionex} Graphical planar version}
  \hfill
\end{figure}

\bibliography{KMoManin.bib}
\bibliographystyle{halpha}
\end{document}

%% file: preamble.tex
\usepackage{cite}
\usepackage{mathtools}
\usepackage{amsmath,color}
\usepackage[v2,all]{xy}
\usepackage{amsfonts,amssymb}
\usepackage{mathrsfs}
\usepackage{tikz}
\usetikzlibrary{cd,fit}

\usepackage{hyperref}
\usepackage{graphicx}
\newcommand{\cev}[1]{\reflectbox{\ensuremath{\vec{\reflectbox{\ensuremath{#1}}}}}}

\usepackage{amssymb}

\usepackage{bold}
 \def\unit{\Eins}
 \def\gh{\mathbbnew{\Gamma}}

\catcode `\@=11
\def\numberbysection{\@addtoreset{equation}{section}
         \renewcommand{\theequation}{\thesection.\arabic{equation}}}
\numberbysection
\def\subsubsection{\@startsection{subsubsection}{3}%
  \normalparindent{.5\linespacing\@plus.7\linespacing}{-.5em}%
  {\normalfont\bfseries}}

  \def\paragraph{\@startsection{paragraph}{4}%
  \normalparindent{.5\linespacing\@plus.7\linespacing}{-.5em}%
  {\normalfont\bfseries}}

\def\l@paragraph{\@tocline{4}{0pt}{1pc}{7pc}{}}
\def\l@subparagraph{\@tocline{5}{0pt}{1pc}{7pc}{}}  
  
\newtheorem{thm}{Theorem}[section]
\newtheorem{theorem}[thm]{Theorem}
\newtheorem{lem}[thm]{Lemma}
\newtheorem{lemma}[thm]{Lemma}
\newtheorem{prop}[thm]{Proposition}

\newtheorem{cor}[thm]{Corollary}

\newtheorem{dfprop}[thm]{Definition-Proposition}

\theoremstyle{definition}

\newtheorem{df}[thm]{Definition}
\newtheorem{definition}[thm]{Definition}

\newtheorem{rmk}[thm]{Remark}

\newtheorem{nota}[thm]{Notation}
\newtheorem{ex}[thm]{Example}

\newtheorem{const}[thm]{Construction}

\def\C{\mathcal C}

\def\cF{\mathcal{F}rm}
\def\O{\mathcal{O}}
\def\Ab{\mathcal{A}b}

\def\Vect{\mathcal{V}ect}

\def\Set{\mathcal{S}et}

\def\V{\mathcal V}

\def\E{\mathcal{E}}

\def\Pl{\mathcal{P}l}

\newcommand{\bisub}[3]{{\vphantom{#2}}_{#1}{#2}_{#3}}

\makeatother
\newcommand{\EZDIAG}[5]{\xymatrix
@C+=2.5cm{*+[r]{#1}
\ar@(u,l)_(0.62){\displaystyle #5}[]
\ar@<1ex>^-{#3}[r]&\ar@<1ex>^-{#4}[l]#2}}

\def\del{\partial}

\def\id{{\mathrm{id}}}

\def\SS{\mathbb{S}}
\def\odo{\otimes \dots\otimes}
\def\bdb{\boxtimes\dots\boxtimes}

\def\kdk{,\dots,}
\def\bdb{\boxtimes\dots\boxtimes}

\def\codco{\circ\cdots\circ}
\def\eps{\epsilon}

\def\B{\mathscr{B}}

\def\nn{\nonumber}
\def\I{\mathscr {I}}

\def\C{\mathscr {C}}

\def \colim {\mathop {\rm colim}}
\def\F{\mathcal{F}}
\def\FF{\mathfrak F}
\def\asts{{\mathcal V}}
\def\ff{\mathfrak{f}}

\def\clusters{{\mathcal F}}
\def\N{\mathbb{N}}
\def\dottree{\bullet\hskip-4.5pt |}

\def\nn{\nonumber}

\def\eps{\epsilon}
\def\G{\Gamma}

\def\Vect{\mathcal{V}ect}

\def\Graphs{{\mathcal G}\it raphs}

\def\Ties{\mathcal{T}\!\it ies}
\def\CalC{{\mathcal C}}

\def\B{\mathcal B}

\def\Agg{{\mathcal A}gg}
\def\Set{{\mathcal S}et}

\def\Crl{{\mathcal C}\it rl}

\def\props{{\mathfrak F}^{\rm prop}}
\def\properads{{\mathfrak F}^{\rm properad}}
\def\nuproperads{{\mathfrak F}^{\rm nu\mdash properad}}
\def\nuprops{{\mathfrak F}^{\rm nu\mdash prop}}

\def\F{\mathcal F}
\def\FF{\mathfrak F}

\def\C{\CalC}

\def\N{{\mathbb N}}
\def\G{\Gamma}

\def\del{\partial}
\def\colim{\mathrm{colim}}

\def\FinSet{\mathcal{F}in\mathcal{S}et}
\def\FFinSet{\mathfrak{F}in\mathfrak{S}et}

\def\Surj{\mathfrak{FS}}
\def\Surj{\mathfrak{S}urj}

\def\a{\alpha}
\def\s{\sigma}

\def\t{\tau}
\def\g{\gamma}

\def\O{{\mathcal O}}
\def\P{{\mathcal P}}
\def\Sn{{\mathbb S}_n}
\def\SS{{\mathbb S}}

\def\odo{\otimes \cdots \otimes}

\def\day{\circledast}

\def\deg{\mathrm{deg}}

\newcommand{\Obj}{\operatorname{Obj}}
\newcommand{\Hom}{\operatorname{Hom}}
\newcommand{\Iso}{\operatorname{Iso}}
\newcommand{\Aut}{\operatorname{Aut}}
\newcommand{\Mor}{\operatorname{Mor}}

\def\V{\asts}
\def\W{{\mathcal W}}
\def\asts{{\mathcal V}}
\def\F{\clusters}
\def\clusters{{\mathcal F}}

\def\Ab{{\mathcal Ab}}

\def\T{{\mathbb T}}
\def\I{{\mathcal I}}
\def\D{\mathcal D}

\def\ta{\twoheadrightarrow}

\def\trivial{\{1\}}
\def\triv{\underline{*}}

\def\mdash{\text{-}}
\def\mddash{\text{--}}
\def\un{\underline{n}}

\def\um{\underline{m}}

\def\clr{\it clr}
\def\dec{\it dec}
\def\Gpd{\mathcal{G}}
\def\A{\mathcal{A}}
\def\M{\mathcal{M}}
\def\Loc{\mathcal{L}}

\def\bs{\boldsymbol{ \sigma}}
\def\btau{\boldsymbol{ \tau}}
\def\bG{\boldsymbol{\G}}

\def\sds{(\sigma \Da \sigma')}
\def\scs{(\s,\s')}
\def\sdsphi{(\sigma \raisebox{4.5pt}{${\phi}\atop{\Da}$} \sigma')}

\newcommand{\sdsp}[3]{(#1 \raisebox{4.5pt}{${#3}\atop{\Da}$} #2)}
\newcommand{\scsp}[3]{(#1 \raisebox{4.5pt}{${#3}\atop{,}$} #2)}
\newcommand{\scsi}[1]{(\s_{#1} ,\s'_{#1})}
\def\tdt{(\t\Da \t')}
\def\tct{(\t , \t')}
\def\ot{\otimes}
\def\bt{\boxtimes}
\def\sq{\square}

\def\final{\mathcal{T}}

\def\el{\mathfrak{el}}

\def\GtG{\Gpd^{op}\times \Gpd}

\def\Da{\!\Downarrow\!}
\def\da{\!\downarrow\!}

\def\FM{\mathfrak {M}}
\def\M{\mathcal M}

\usepackage{subcaption}

\newcommand{\act}{\mathcal{B}}

\newcommand{\colorcat}{\mathcal{C}} % For graph colors

\newcommand{\Indec}{\mathcal{W}}

\tikzstyle{none}=[]
\tikzstyle{Box}=[fill=white, draw=black, shape=rectangle]
\tikzstyle{Dotted}=[fill=white, draw=black, shape=rectangle, dashed]
\tikzstyle{Gap}=[fill=white, shape=circle]
\tikzstyle{White}=[fill=white, draw=black, shape=circle]
\tikzstyle{Black}=[fill=black, draw=black, shape=circle]
\tikzstyle{Gray}=[fill={rgb,255: red,191; green,191; blue,191}, draw=black, shape=circle]

\tikzstyle{Vertex}=[fill=white, draw=black, shape=circle]
\tikzstyle{Hyperedge}=[fill=black, draw=black, shape=circle,inner sep=0.6mm]

\tikzstyle{myRed}=[fill=red, draw=black, shape=circle]
\tikzstyle{myBlue}=[fill=blue, draw=black, shape=circle]
\tikzstyle{myGreen}=[fill=green, draw=black, shape=circle]

\tikzstyle{Arrow}=[->]
\tikzstyle{Mapsto}=[{|->}]
\tikzstyle{Dash}=[-, dashed]
\tikzstyle{DashArrow}=[->, dashed]
\tikzstyle{Double}=[<->]
\tikzstyle{Red}=[-, draw=red]
\tikzstyle{Blue}=[-, draw=blue]
\tikzstyle{Green}=[-, draw=green]
\tikzstyle{RedArrow}=[draw=red, ->]
\tikzstyle{BlueArrow}=[draw=blue, ->]

\usetikzlibrary{calc}
\usetikzlibrary{decorations.pathmorphing}
\usetikzlibrary{backgrounds}
\usetikzlibrary{math}

\tikzset{curve/.style={settings={#1},to path={(\tikztostart)
    .. controls ($(\tikztostart)!\pv{pos}!(\tikztotarget)!\pv{height}!270:(\tikztotarget)$)
    and ($(\tikztostart)!1-\pv{pos}!(\tikztotarget)!\pv{height}!270:(\tikztotarget)$)
    .. (\tikztotarget)\tikztonodes}},
    settings/.code={\tikzset{quiver/.cd,#1}
        \def\pv##1{\pgfkeysvalueof{/tikz/quiver/##1}}},
    quiver/.cd,pos/.initial=0.35,height/.initial=0}

\tikzset{tail reversed/.code={\pgfsetarrowsstart{tikzcd to}}}
\tikzset{2tail/.code={\pgfsetarrowsstart{Implies[reversed]}}}
\tikzset{2tail reversed/.code={\pgfsetarrowsstart{Implies}}}
\tikzset{no body/.style={/tikz/dash pattern=on 0 off 1mm}}

\def\CCospan{\mathfrak{C}ospan}
\def\Cospan{\mathcal{C}ospan}
\def\Span{\mathcal{S}pan}

\def\connected{\mathcal{C}td}
\def\Ctd{\connected}
\newcommand{\xnelt}[3]{
#1_0\stackrel{#2_0}{\to} #1_1\stackrel{#2_1}{\to}#1_2 \to\dots\to #1_{#3-1}\stackrel{#2_{#3-1}}{\to}#1_#3
}
\newcommand{\gen}[3]
{
\stackrel{#2_0}{-}#1_1\stackrel{#2_1}{-}#1_2\stackrel{#2_2}{-} \cdots \stackrel{#2_{#3-1}}{-}#1_{#3}\stackrel{#2_{#3}}{-}
}
\newcommand{\gent}[3]
{
#1_{#3}\stackrel{#2_{#3}}{-}#1_{#3-1} \stackrel{#2_{#3-1}}{-} \cdots \stackrel{#2_2}{-} #1_{1}\stackrel{#2_1}{-}
}

\def\index{{\rm {idx}}}
\def\findex{\mathfrak {index}}

\def\depth{\rm depth}
\def\spure{{\rm Plr}} %This is the naive graph. Maybe use another name \

\def\type{\it type}

\def\bwb{b\mdash w \mdash b}

\newcommand{\pl}{\mathbin{\square}}

%% file: population.tikz
\begin{tikzpicture}
  \node [style=White] (0) at (-1, 1) {};
  \node [style=none] (1) at (-1.5, 1.75) {$\phi_3$};
  \node [style=none] (2) at (-0.5, 1.75) {$\phi_4$};
  \node [style=Black] (3) at (-1.5, 0.25) {};
  \node [style=none] (4) at (-2, 1) {$\phi_2$};
  \node [style=White] (5) at (-2, -0.5) {};
  \node [style=none] (6) at (-2.5, 0.25) {$\phi_1$};
  \node [style=Black] (7) at (0, -0.5) {};
  \node [style=none] (8) at (-0.5, 0.25) {$\phi_5$};
  \node [style=none] (9) at (0.5, 0.25) {$\phi_6$};
  \node [style=White] (10) at (-1, -1.25) {};
  \node [style=none] (11) at (-1, -2) {$\phi_1 \circ (\phi_2 \otimes (\phi_3 \circ \phi_4)) \circ (\phi_5 \otimes \phi_6)$};
  \node [style=White] (13) at (3.5, -1.25) {};
  \node [style=none] (14) at (3, -0.5) {$\phi_7$};
  \node [style=none] (15) at (4, -0.5) {$\phi_8$};
  \node [style=none] (16) at (3.5, -2) {$\phi_7 \circ \phi_8$};
  \draw (15) to (13);
  \draw (14) to (13);
  \draw (13) to (16);
  \draw (5) to (10);
  \draw (7) to (10);
  \draw (10) to (11);
  \draw (8) to (7);
  \draw (9) to (7);
  \draw (6) to (5);
  \draw (3) to (5);
  \draw (4) to (3);
  \draw (0) to (3);
  \draw (1) to (0);
  \draw (2) to (0);
\end{tikzpicture}

%% file: movesegments.tikz
\begin{tikzpicture}[scale=.6]
\node [style=none] (0) at (-1, 3) {};
\node [style=none] (1) at (-1, 2.25) {};
\node [style=none] (2) at (-1, 1.25) {};
\node [style=none] (3) at (-1, 0.5) {};
\node [style=none] (4) at (0, 0.5) {};
\node [style=none] (5) at (1, 0.5) {};
\node [style=none] (6) at (1, 1.25) {};
\node [style=none] (7) at (1, 2.25) {};
\node [style=none] (8) at (1, 3) {};
\node [style=none] (9) at (0, 3) {};
\node [style=none] (10) at (0, 2.25) {};
\node [style=none] (11) at (0, 1.75) {};
\node [style=none] (12) at (0, 1.25) {};
\node [style=none] (13) at (1, 1.75) {};
\node [style=none] (14) at (2, 3) {};
\node [style=none] (15) at (2, 2.25) {};
\node [style=none] (16) at (2, 1.25) {};
\node [style=none] (17) at (2, 0.5) {};
\node [style=none] (18) at (3, 0.5) {};
\node [style=none] (19) at (4, 0.5) {};
\node [style=none] (20) at (4, 1.25) {};
\node [style=none] (21) at (4, 2.25) {};
\node [style=none] (22) at (4, 3) {};
\node [style=none] (23) at (3, 3) {};
\node [style=none] (24) at (3, 2.25) {};
\node [style=none] (25) at (3, 2.25) {};
\node [style=none] (26) at (3, 1.25) {};
\node [style=none] (27) at (4, 2.25) {};
\node [style=none] (28) at (2, -0.5) {};
\node [style=none] (29) at (2, -1.25) {};
\node [style=none] (30) at (2, -2.25) {};
\node [style=none] (31) at (2, -3) {};
\node [style=none] (32) at (3, -3) {};
\node [style=none] (33) at (4, -3) {};
\node [style=none] (34) at (4, -2.25) {};
\node [style=none] (35) at (4, -1.25) {};
\node [style=none] (36) at (4, -0.5) {};
\node [style=none] (37) at (3, -0.5) {};
\node [style=none] (38) at (3, -1.25) {};
\node [style=none] (39) at (3, -1.25) {};
\node [style=none] (40) at (3, -2.25) {};
\node [style=none] (41) at (4, -1.25) {};
\node [style=none] (42) at (-4, 3) {};
\node [style=none] (43) at (-4, 2.25) {};
\node [style=none] (44) at (-4, 1.25) {};
\node [style=none] (45) at (-4, 0.5) {};
\node [style=none] (46) at (-3, 0.5) {};
\node [style=none] (47) at (-2, 0.5) {};
\node [style=none] (48) at (-2, 1.25) {};
\node [style=none] (49) at (-2, 2.25) {};
\node [style=none] (50) at (-2, 3) {};
\node [style=none] (51) at (-3, 3) {};
\node [style=none] (52) at (-3, 2.25) {};
\node [style=none] (53) at (-3, 1.25) {};
\node [style=none] (54) at (-3, 1.25) {};
\node [style=none] (55) at (-2, 1.25) {};
\node [style=none] (56) at (-4, -0.5) {};
\node [style=none] (57) at (-4, -1.25) {};
\node [style=none] (58) at (-4, -2.25) {};
\node [style=none] (59) at (-4, -3) {};
\node [style=none] (60) at (-3, -3) {};
\node [style=none] (61) at (-2, -3) {};
\node [style=none] (62) at (-2, -2.25) {};
\node [style=none] (63) at (-2, -1.25) {};
\node [style=none] (64) at (-2, -0.5) {};
\node [style=none] (65) at (-3, -0.5) {};
\node [style=none] (66) at (-3, -1.25) {};
\node [style=none] (67) at (-3, -2.25) {};
\node [style=none] (68) at (-3, -2.25) {};
\node [style=none] (69) at (-2, -2.25) {};
\node [style=none] (70) at (1.5, 1.75) {$\leadsto$};
\node [style=none] (71) at (3, 0) {\rotatebox{-90}{$\leadsto$}};
\node [style=none] (72) at (-3, 0) {\rotatebox{-90}{$\leadsto$}};
\node [style=none] (73) at (-1.5, 1.75) {\rotatebox{180}{$\leadsto$}};
\draw (0.center) to (1.center);
\draw (1.center) to (2.center);
\draw (2.center) to (3.center);
\draw (3.center) to (4.center);
\draw (4.center) to (5.center);
\draw (5.center) to (6.center);
\draw (7.center) to (8.center);
\draw (8.center) to (9.center);
\draw (9.center) to (0.center);
\draw (7.center) to (13.center);
\draw (13.center) to (6.center);
\draw (13.center) to (11.center);
\draw (10.center) to (1.center);
\draw (12.center) to (2.center);
\draw (9.center) to (10.center);
\draw (10.center) to (11.center);
\draw (11.center) to (12.center);
\draw (12.center) to (4.center);
\draw (14.center) to (15.center);
\draw (15.center) to (16.center);
\draw (16.center) to (17.center);
\draw (17.center) to (18.center);
\draw (18.center) to (19.center);
\draw (19.center) to (20.center);
\draw (21.center) to (22.center);
\draw (22.center) to (23.center);
\draw (23.center) to (14.center);
\draw (21.center) to (27.center);
\draw (27.center) to (20.center);
\draw (27.center) to (25.center);
\draw (24.center) to (15.center);
\draw (26.center) to (16.center);
\draw (23.center) to (24.center);
\draw (24.center) to (25.center);
\draw (25.center) to (26.center);
\draw (26.center) to (18.center);
\draw (28.center) to (29.center);
\draw (29.center) to (30.center);
\draw (30.center) to (31.center);
\draw (31.center) to (32.center);
\draw (32.center) to (33.center);
\draw (33.center) to (34.center);
\draw (35.center) to (36.center);
\draw (36.center) to (37.center);
\draw (37.center) to (28.center);
\draw (35.center) to (41.center);
\draw (41.center) to (34.center);
\draw (41.center) to (39.center);
\draw (38.center) to (29.center);
\draw (40.center) to (30.center);
\draw (38.center) to (39.center);
\draw (39.center) to (40.center);
\draw (40.center) to (32.center);
\draw (42.center) to (43.center);
\draw (43.center) to (44.center);
\draw (44.center) to (45.center);
\draw (45.center) to (46.center);
\draw (46.center) to (47.center);
\draw (47.center) to (48.center);
\draw (49.center) to (50.center);
\draw (50.center) to (51.center);
\draw (51.center) to (42.center);
\draw (49.center) to (55.center);
\draw (55.center) to (48.center);
\draw (55.center) to (53.center);
\draw (52.center) to (43.center);
\draw (54.center) to (44.center);
\draw (51.center) to (52.center);
\draw (52.center) to (53.center);
\draw (53.center) to (54.center);
\draw (54.center) to (46.center);
\draw (56.center) to (57.center);
\draw (57.center) to (58.center);
\draw (58.center) to (59.center);
\draw (59.center) to (60.center);
\draw (60.center) to (61.center);
\draw (61.center) to (62.center);
\draw (63.center) to (64.center);
\draw (64.center) to (65.center);
\draw (65.center) to (56.center);
\draw (63.center) to (69.center);
\draw (69.center) to (62.center);
\draw (69.center) to (67.center);
\draw (66.center) to (57.center);
\draw (68.center) to (58.center);
\draw (65.center) to (66.center);
\draw (66.center) to (67.center);
\draw (67.center) to (68.center);
\end{tikzpicture}